\documentclass{amsart}

\usepackage{arxiv_math}

\title{Naturality in real Heegaard Floer theory}

\author[Gary Guth]{Gary Guth}
\address{Department of Mathematics\\Stanford University\\
		Building 380\\
		Stanford, California 94305}
\email{gmguth@stanford.edu}

\author[Ciprian Manolescu]{Ciprian Manolescu}
\address{Department of Mathematics\\Stanford University\\
		Building 380\\
		Stanford, California 94305}
\email{cm5@stanford.edu}

\begin{document}
\begin{abstract}
	In a previous paper we defined real Heegaard Floer homology, an invariant of three-manifolds equipped with involutions. Here we prove that real Heegaard Floer homology is natural, and that it admits an action of the equivariant mapping class group. We also establish naturality of real link Floer homology and real sutured Floer homology, and give a definition of involutive real Heegaard Floer homology.
\end{abstract}
\maketitle
\setcounter{tocdepth}{1} 
\tableofcontents

\section{Introduction}\label{sec:Introduction}

A real manifold $(Y, \tau)$ is an oriented smooth manifold together with a smooth involution $\tau:Y \to Y$ with codimension-two fixed point set. Real Seiberg-Witten theory \cite{tian-wang, nakamura1, kato, KMT:Ktheory, KMT:homology, li:HMR, baraglia-hekmati} assigns invariants to real three- and four-manifolds, by lifting the involution in an anti-linear way to the spinor bundle, and counting invariant solutions to the Seiberg-Witten equations. Among its best-known applications are proofs that cables of the figure-eight are not slice \cite{KPT:cables, KPT:cables2}, the existence of an exotic embedding of $\RP^2$ in $S^4$ \cite{miyazawa}, and of various exotic involutions on four-manifolds \cite{hugheskimmiller_branched, baraglia_miyazawa}. 

Ordinary (``unreal'') Seiberg-Witten theory has a symplectic counterpart, Heegaard Floer theory, which is more computable and has numerous applications \cite{os_holodisks, os_properties_apps}. In \cite{guth_manolescu2025real}, the authors developed real Heegaard Floer homology as the symplectic counterpart of real monopole Floer homology. This allowed for the general calculation of the analogues of Miyazawa's knot invariant; see \cite[Section 7]{guth_manolescu2025real} and \cite{srivastava}. Further, using bordered techniques, Lipshitz and Ozsv\'ath \cite{lipshitz_ozsvath:real_bordered} gave an algorithmic procedure for calculating the (hat) real Heegaard Floer homology of any real three-manifold with connected fixed point set. For general real manifolds, the hat version can be computed combinatorially via nice diagrams, a la Sarkar-Wang \cite{sarkarwang}, by the work of \cite{BGX:real_sutured}.

We expect that real Heegaard Floer theory will also lead to invariants of real four-manifolds that are more amenable to computations than the gauge-theoretic ones. However, for this one first needs to prove that the real Heegaard Floer groups are functorial under equivariant cobordisms.

In this paper we take the first step towards functoriality, by proving that it holds under equivariant diffeomorphisms (which can be thought of as the identity cobordism with different identifications at the ends). Most of the work involves establishing naturality: whereas in \cite{guth_manolescu2025real} we only showed that the isomorphism class of real Heegaard Floer homology is an invariant, here we define the group itself as a natural invariant. Specifically, the real Heegaard Floer homology is constructed from a real Heegaard diagram. In \cite{guth_manolescu2025real} we showed that a sequence of moves between such diagrams induces an isomorphism on Floer homology, and here we show that this isomorphism is independent of the sequence. This gives a transitive system of groups (as in \cite[Definition 1.1]{JTZ_naturality_mapping_class_groups}), and its colimit is the natural invariant.

\begin{definition}
A {\em based real manifold} is a triple $(Y, \tau, p)$, where $(Y, \tau)$ is a real manifold and $p$ is a base point on the fixed point locus $Y^{\tau}$. We let $\RMan_*$ be the category whose  objects are closed, connected, based real three-manifolds $(Y, \tau, p)$, and whose morphisms are basepoint-preserving equivariant diffeomorphisms. 
\end{definition}

If $\mathbf{k}$ is a field, we let $\mathbf{k}\text{-}\mathsf{Vect}$ be the category of $\mathbf{k}$-vector spaces. Similarly, if $R$ is a ring,  we let $R\text{-}\mathsf{Mod}$ be the category of modules over $R$. 
Let $\F$ be the field of two elements.

\begin{thm}\label{thm:real-nat}
For $\circ \in \{-, \infty, +, \hat{\phantom{o}} \}$, there are functors 
$$\HFR^{\circ}: \RMan_* \to \Fmod$$
such that for any triple $(Y, \tau, p)$, the groups $\HFR^{\circ}(Y, \tau, p)$ are isomorphic to the real Heegaard Floer homology groups defined in \cite{guth_manolescu2025real}. Moreover, isotopic diffeomorphisms induce identical maps on $\HFR^{\circ}$.
\end{thm}

\begin{definition}\label{def:real-MCG} Let $(Y, \tau)$ be a closed real three-manifold. 
    Let $\mathrm{Diff}(Y, \tau)$ be the group of equivariant diffeomorphisms of $Y$ and let $\mathrm{Diff}_0(Y, \tau)$ be the subgroup consisting of those diffeomorphisms equivariantly isotopic to the identity. Then, the \emph{equivariant mapping class group of $(Y, \tau)$} is the group
    \begin{align*}
        \mathrm{MCG}(Y, \tau) = \mathrm{Diff}(Y, \tau)/\mathrm{Diff}_0(Y, \tau).
    \end{align*}
    For a based real manifold $(Y, \tau, p)$, the \emph{based equivariant mapping class group} is 
    \begin{align*}
        \mathrm{MCG}(Y, \tau,p) = \mathrm{Diff}(Y, \tau,p)/\mathrm{Diff}_0(Y, \tau,p),
    \end{align*}
    where all groups involved consist of maps fixing $p$.
\end{definition}

Theorem~\ref{thm:real-nat} has the following immediate consequence.
\begin{cor}
\label{cor:real-MCG}
   The real Floer homology groups have well-defined actions of $\mathrm{MCG}(Y, \tau,p)$.  
\end{cor}

There are analogues of Theorem~\ref{thm:real-nat} in a few other related settings. 

In \cite{hendricks:real}, Hendricks defined an invariant of strongly invertible knots called {\em real knot Floer homology}. In \cite{xiao} and \cite{xiao2}, Xiao extended this to {\em real link Floer homology}, an  invariant of multi-based generalized strongly invertible  links $\mathbb{L} :=(L, \w, \z)$ with $L \subset Y$ and $\w$ and $\z$ being collections of basepoints. (See \cite[Definition 4.1]{xiao2} for the details.) If $k$ is the number of link components preserved by $\tau$ setwise, and $l$ is the number of pairs of link components interchanged by $\tau$, then the real link Floer homology $\HFLR^-$ takes the form of a module over $\F[u_1,\dots, u_k, U_1, \dots U_l]$. There is also a version $\HFLRhat$, which is an $\F$-vector space. Let $\RLink_*^{k,l}$ be the category of multi-based generalized strongly invertible links  with fixed $k$ and $l$, with morphisms being equivariant diffeomorphisms of $(Y, \t)$ preserving $\mathbb{L}$. 

\begin{thm}\label{thm:links-nat}
There are functors
\begin{align*} \HFLR^-&: \RLink_*^{k,l} \to \F[u_1,\dots, u_k, U_1, \dots U_l]\text{-}\mathsf{Mod},\\
 \HFLRhat &: \RLink_*^{k,l} \to \F\text{-}\mathsf{Vect},
 \end{align*}
that produce, up to isomorphism, the real link Floer homologies from \cite{xiao}. Moreover, isotopic diffeomorphisms induce identical maps. Hence, the mapping class group of $\mathbb{L}$ acts on $\HFLR^-(\mathbb{L})$ and $ \HFLRhat(\mathbb{L}).$
\end{thm}

There is also an extension of real Heegaard Floer homology to real sutured manifolds. The resulting invariant is called real sutured Floer homology and is developed in \cite{BGX:real_sutured}.  Let $\RSut$ be the category of real sutured manifolds. 

\begin{thm}\label{thm:sutured-nat}
There is a functor
$$ \RSFH: \RSut \to  \F\text{-}\mathsf{Vect},$$
that produces, up to isomorphism, the real sutured Floer homologies from \cite{BGX:real_sutured}. Isotopic diffeomorphisms induce identical maps, and therefore the mapping class group of a real sutured manifold acts on the real sutured Floer homology.
\end{thm}

The proofs of Theorems~\ref{thm:real-nat}, \ref{thm:links-nat} and \ref{thm:sutured-nat} follow the same strategy as in the unreal setting, where naturality was proved by Juh\'asz, Thurston, and Zemke in \cite{JTZ_naturality_mapping_class_groups}. We outline below the main steps, emphasizing what is new compared to \cite{JTZ_naturality_mapping_class_groups}.

In Section~\ref{sec:real floer invariants} we set up some of the background. Following \cite{guth_manolescu2025real}, we list the real Heegaard moves between real Heegaard diagrams.  These suffice to define a weak real Heegaard invariant. We then list the axioms (corresponding to elementary loops of real Heegaard moves) that are needed to upgrade to a strong real Heegaard invariant; that is, to prove naturality. These axioms parallel those that appear in \cite{JTZ_naturality_mapping_class_groups}, with one notable exception: We now encounter a {\em simple trade}, a loop in which a pair of trivial orbit stabilizations are canceled by a single $\Z/2$-orbit stabilization. (See \Cref{fig:simple_trade_loop}.) We also formulate Theorem~\ref{thm:naturality}, that real Heegaard Floer homology and its variants are strong Heegaard invariants, and show that it implies Theorems~\ref{thm:real-nat}, \ref{thm:links-nat} and \ref{thm:sutured-nat}.

In Section~\ref{sec:Singularities} we study the singularities of {\em real functions} $f: Y \to [-1,1]$, that is, those that satisfy $f\circ \tau = -f$. We distinguish between critical points that appear in pairs (related by the involution $\tau$) and those that are fixed by $\tau$. The former can be treated as in \cite{JTZ_naturality_mapping_class_groups}, whereas the latter require a new analysis. We classify real singularities up to codimension $2$. The main novelty (compared to the unreal case) is that in codimension $2$ we may encounter an $A_4$ singularity on the fixed point set.  

In Section~\ref{sec:Families of Real gradients} we go from singularities of functions to those of real gradient flows. We classify bifurcations of such flows up to codimension $2$.

In Section~\ref{sec:constructing hd} we describe how gradient flow bifurcations (of codimension $0$, $1$ and $2$) produce real Heegaard diagrams, real Heegaard moves, and loops of such moves. We draw these explicitly. In particular, we encounter a new move called a {\em crossover}, coming from the failure of a transversality condition on the fixed point set. 

In Section~\ref{sec:reducing moves} we decompose  the real Heegaard moves from Section~\ref{sec:constructing hd} into simpler ones. For example, the crossover is a combination of a stabilization, some handleslides, and a destabilization. We also decompose the loops of Heegaard moves into those needed in the definition of a strong real Heegaard invariant. The new $A_4$ singularity is the origin of the simple trade loop.

In Section~\ref{sec:reducing no_monodromy} we prove that having a strong real Heegaard invariant does indeed lead to naturality.

In Section~\ref{sec:real hf} we check that the axioms in the definition of a strong real Heegaard invariant hold for real Heegaard Floer homology (and its link and sutured variants). We thus establish Theorem~\ref{thm:naturality}. Showing that the simple trade induces the identity requires analyzing specific holomorphic curves, and the analysis is somewhat similar (but not the same) as in the proof of handleswap invariance in \cite[Section 9.3]{JTZ_naturality_mapping_class_groups}. In the process  we give a cylindrical reformulation of real Heegaard Floer theory, in the manner of Lipshitz's reformulation of the unreal theory \cite{Lipshitz_cylindrical}. This is, of course, closely related to the real bordered theory of \cite{lipshitz_ozsvath:real_bordered}.

Finally, in Section~\ref{sec:HFRI} we use the naturality for $\HFR^\circ$ to construct an involutive version of real Heegaard Floer homology. Recall that in \cite{hendricks_manolescu_Invol}, Hendricks and the second author used the conjugation involution $\iota$ on the Heegaard Floer complex $\CF^\circ$ to define involutive Heegaard Floer homology, $\HFI^\circ(Y, \s)$, as the homology of the mapping cone of $1+\iota$. Since then, involutive Heegaard Floer homology has found various topological applications, e.g., to homology cobordism \cite{DHST:hcob, HHSZ_surgery_exact_invol} and concordance \cite{DHST:concordance, DKMPS:cable_fig_8}. We can mimic the construction using real Heegaard Floer complexes and define involutive real Heegaard Floer homology, $\HFRI^\circ(Y, \tau, \s)$. Naturality is used in showing that the conjugation involution is well-defined up to chain homotopy. We deduce the following.

\begin{thm}
\label{thm:HFRI}
Let $(Y, \tau)$ be a real three-manifold and $\s$ a self-conjugate real $\SpinC$ structure. For $\circ \in \{-, \infty, +, \hat{\phantom{o}} \}$, the isomorphism class of $\HFRI^\circ(Y, \tau, \s)$ is an invariant of the triple $(Y, \tau, \s)$.
\end{thm}

\bigskip
\noindent {\bf Note on AI use.} We used Gemini Deep Think and Claude to produce first drafts  of the proofs of Propositions~\ref{prop:standard_form_real_sing} and ~\ref{prop:fv}, Corollary~\ref{cor:weaklycon}, Lemmas  \ref{lem:sym_6.20}, \ref{lem:sym_6.21} and \ref{lem:sym_6.23}. 
 The proofs were checked and then rewritten by the authors.

During the revision process we used GPT-5.6 Sol to look for typographical and mathematical errors, which we then fixed.

\bigskip
\noindent {\bf Acknowledgements.} 
We are grateful to Kristen Hendricks, Andr\'as Juh\'asz, Robert Lipshitz, Eha Srivastava, Dylan Thurston, Yonghan Xiao, and Ian Zemke for helpful conversations. We thank Mohammed Abouzaid for computer code that helped us draw some of the figures.

The authors were supported by the Simons Collaboration Grant on New Structures in Low-Dimensional Topology. CM was also supported by a Simons Investigator Award and NSF grant DMS-2522743.

\section{Real Heegaard Floer invariants}\label{sec:real floer invariants}
\subsection{Real sutured manifolds}
\label{sec:sutured-manifolds}
As in \cite{JTZ_naturality_mapping_class_groups}, we will work with sutured manifolds, so that our naturality results will extend to other variants of real Heegaard Floer homology (sutured, knot, and link; see \cite{BGX:real_sutured,xiao}.) We briefly review the relevant objects, closely following \cite[Section 2]{JTZ_naturality_mapping_class_groups}.

\begin{definition}\label{def:sutured}
A \emph{sutured manifold} $(Y,\g)$ is a compact, oriented
3-manifold $Y$ with boundary, together with a set $\g \subset
\partial Y$ of pairwise disjoint annuli. Furthermore, the interior of each component of
$\g$ contains a homologically
nontrivial oriented simple closed curve, called a \emph{suture}. The union of the
sutures will be denoted $s(\g)$.
In addition, every component of $R(\g) = \partial Y \setminus
\text{Int}(\g)$ is oriented. Define $R_+(\g)$ (respectively
$R_-(\g)$) to be the set of those components of $\partial Y \setminus
\text{Int}(\g)$ whose orientations agree (respectively disagree) with the
orientation of $\partial Y$, or equivalently, whose
normal vectors point out of (respectively into) $Y$.
The orientation on $R(\g)$ must be coherent with respect to~$s(\g)$;
i.e., if $\eta$ is a component of $\partial
R(\g)$ and is given the boundary orientation, then $\eta$ must
represent the same homology class in $H_1(\g)$ as some suture.%
\end{definition}

\begin{rem}
    We will always assume our sutured manifolds are \emph{proper}, which is to say that the maps $\pi_0(\g) \to \pi_0(\partial Y)$ and $\pi_0(\partial Y) \to \pi_0(Y)$ are surjective. We will also assume our manifolds are \emph{balanced}, i.e., $\chi(R_+(\g)) = \chi(R_-(\g))$.
\end{rem}

\begin{definition}\label{def:real-sutured}
    A \emph{real sutured manifold} $(Y, \gamma, \tau)$ is a sutured manifold $(Y, \gamma)$ equipped with an orientation-preserving involution $\tau$ whose fixed point set is codimension-2 and such that $\tau$ exchanges $R_+$ and $R_-$.
\end{definition}

We note that because of the symmetry condition, $(Y, \gamma, \tau)$ is necessarily balanced.

\subsection{Real sutured diagrams}
\label{sec:sutured-diagrams}

\begin{definition}
Let $\S$ be a compact oriented surface with boundary. An
\emph{attaching set in $\S$} is a one-dimensional smooth submanifold
$\etas \subset \text{Int}(\S)$ such that each component of $\S \setminus \etas$ contains at least one component
of $\partial \S$. The isotopy class of~$\etas$ will be denoted $[\etas]$.
\end{definition}

\begin{definition} \label{def:scb}
A sutured manifold $(Y,\gamma)$ is a \emph{sutured compression body} if either there is an attaching
set $\etas \subset R_{+}(\gamma)$ such that if we compress
$R_{+}(\gamma)$ inside~$Y$ along all the components of $\etas$, we get
a surface that is isotopic to $R_{-}(\gamma)$ relative to $\gamma$,
or there is an attaching
set $\etas \subset R_{-}(\gamma)$ such that if we compress
$R_{-}(\gamma)$ inside~$Y$ along all the components of $\etas$, we get
a surface that is isotopic to $R_{+}(\gamma)$ relative to $\gamma$. We call
$\etas$ an \emph{attaching set} for $(Y,\gamma)$.\\

Given a set of attaching curves $\etas$ in $\Sigma$, we define $C(\etas)$ to be the sutured compression body obtained by attaching 3-dimensional 1-handles to $\S \times [0,1]$ along $\etas\times\{0\}$ and $\gamma = \partial\S \times [0,1]$. We will write $C_{\pm}(\etas)$ for the  boundary components of $C(\etas)$ which do not include the sutures. If~$\etas$ and~$\etas'$ are two attaching sets in~$\S$, then we say they are
\emph{compression equivalent}, and we write
$\etas \sim \etas'$, if there is a diffeomorphism $d \colon C(\etas) \to C(\etas')$
such that $d|_{C_-(\etas)}$ is the identity. This is an equivalence relation that
descends to isotopy classes of attaching sets.
So we will write $[\etas] \sim [\etas']$ if $\etas \sim \etas'$.
\end{definition}

\begin{definition}
A \emph{real sutured diagram} is a quadruple $(\S, \alphas, \betas, \tau)$, where $\S$ is a compact oriented surface with boundary,
and $\alphas$ and $\betas$ are two attaching sets in~$\S$ with the property that $\tau(\alphas) = \betas$. An \emph{isotopy diagram} is a quadruple $(\S, [\alphas], [\betas], \tau)$, where $(\S,\alphas,\betas, \tau)$ is a real sutured diagram. 
\end{definition}

\begin{rem}
    Following \cite{JTZ_naturality_mapping_class_groups}, given isotopy classes $A = [\alphas]$ and $B = [\betas]$, we will usually write $H = (\Sigma, A, B, \t)$ for an isotopy class of diagrams and $\cH = (\Sigma, \alphas, \betas, \t)$ for a representative of the class.
\end{rem}

\begin{definition}
Let $(Y,\gamma, \tau)$ be a real sutured manifold. Then $(\S, \alphas,
\betas, \tau)$ is an \emph{(embedded) real sutured diagram of $(Y,\gamma, \tau)$}
if
\begin{enumerate}
\item $\S \subset Y$ is an oriented surface with $\partial \S =
  s(\gamma)$ as oriented 1-manifolds,%
\item the components of $\alphas$ bound disjoint disks to the negative side of $\S$,
\item if we compress $\S$ along $\alphas$, we get a surface isotopic to
  $R_-(\gamma)$ relative to $\gamma$.
\end{enumerate}
In short, $\S$ splits $(Y,\gamma, \tau)$ into two symmetric sutured compression bodies, with attaching sets
$\alphas$ and $\betas$.
\end{definition}

\begin{lemma} \label{lem:existence}
Every real sutured manifold $(Y,\gamma, \tau)$ admits a real sutured diagram.
\end{lemma}

\begin{proof}
This follows from \cite{nagase}. Alternatively, this follows from the Morse theoretic analyses of \Cref{sec:Singularities} and \Cref{sec:Families of Real gradients}.
\end{proof}

\begin{definition}
Let $(Y,\gamma, \tau)$ be a real sutured manifold. A \emph{real Heegaard surface of $(Y,\g,\tau)$} is an oriented surface $\S \subset Y$ such that $\partial \S = s(\g)$ and $\S$ divides $(Y,\g,\tau)$ into two sutured compression bodies exchanged by the involution.
\end{definition}

\subsection{Moves on diagrams and weak Heegaard invariants}
\label{sec:moves-diagrams}

Let us introduce a notion of equivalence for isotopy diagrams.

\begin{definition}
We say that the isotopy diagrams $(\S_1, A_1, B_1)$ and $(\S_2, A_2, B_2)$ are \emph{real-equivalent}
if $\S_1 = \S_2$ and $A_1 \sim A_2$. Note that the symmetry condition forces $B_1 \sim B_2$ as well.
\end{definition}

\begin{definition}\label{def:Handleslide} (cf. \cite[Definition 2.10]{JTZ_naturality_mapping_class_groups})
    Let $\eta_1$ and $\eta_2$ be two disjoint simple closed curves in $\S$, and fix an embedded arc~$a$ from $\eta_1$ to~$\eta_2$ whose interior is disjoint from $\eta_1 \cup \eta_2$ and from~$\partial \S$. Then a regular neighborhood of the graph $\eta_1\cup a \cup \eta_2$ is a planar surface with three boundary components: one is isotopic to $\eta_1$, the other is isotopic to $\eta_2$, and the third is a new curve $\eta_1'$ that we call the curve obtained by \emph{handlesliding $\eta_1$ over $\eta_2$ along the arc $a$}. 

    Suppose $\etas$ and $\etas'$ are two systems of attaching circles. We say that $\etas$ and $\etas'$ are \emph{related by a handleslide} if there are components $\eta_1$ and $\eta_2$ of $\etas$ and a component $\eta_1'$ of $\etas'$ such that $\eta_1'$ can be obtained by handle-sliding $\eta_1$ over $\eta_2$ along some arc whose interior is disjoint from~$\etas$, and $\etas' = (\etas \setminus \eta_1) \cup \eta_1'$. If $D$ and $D'$ are isotopy classes of attaching sets, then they are related by a handleslide if they have representatives $\etas$ and $\etas'$, respectively, such that $\etas$ and $\etas'$ are related by a handleslide.
\end{definition}

We note that by \cite[Lemma 2.11]{JTZ_naturality_mapping_class_groups} we have $\etas \sim \etas'$ if and only if they are related by a sequence of handleslides. 

We now recall the basic moves which relate real sutured Heegaard diagrams. 

\begin{definition}\label{def:real-Handleslide}
    We say that the real sutured diagram $(\S_2, \alphas_2, \betas_2, \tau_2)$ is obtained from $(\S_1, \alphas_1, \betas_1, \tau_1)$ by a \emph{real handleslide} if
    \begin{enumerate}
        \item $\S_1 = \S_2$, 
        \item $\tau_1 = \tau_2$,
        \item $\alphas_1$ and $\alphas_2$ are related by a handleslide (and necessarily $\betas_1$ and $\betas_2$ are also related by a handleslide).
    \end{enumerate}
\end{definition}

In the real setting, there are two different stabilization operations, appearing in both trivial and $\Z/2$-orbits. 

\begin{definition} \label{def:Z2-stab}
The real sutured diagram $(\S_2, \alphas_2, \betas_2,\tau_2)$ is obtained from $(\S_1, \alphas_1, \betas_1, \tau_1)$ by a \emph{$\Z/2$-stabilization}
if
\begin{itemize}
\item there is a disk $D \subset \S_1$ and a punctured torus $T \subset \S_2$ such that $\S_1 \setminus D = \S_2 \setminus T$, $D \cap \tau_1(D) = \emptyset$ and  $T \cap \tau_2(T) = \emptyset$,
\item $\alphas_1  = \alphas_2 \cap (\S_2 \setminus (T\cup \tau_2(T)))$,
\item $\alphas_2 \cap T$ and $\betas_2 \cap T$ are simple closed curves that intersect each other transversely in a single point; the same holds for $\alphas_2 \cap \tau_2(T)$ and $\betas_2 \cap \tau_2(T)$.
\end{itemize}
In the reverse direction, we also say that $(\S_1, \alphas_1, \betas_1,\tau_1)$ is obtained from $(\S_2, \alphas_2, \betas_2, \tau_2)$ by a \emph{$\Z/2$-destabilization}. See \Cref{fig:simple_Z2_stab}.

\begin{figure}[h]
\def\svgwidth{.8\linewidth}
\begingroup%
  \makeatletter%
  \providecommand\color[2][]{%
    \errmessage{(Inkscape) Color is used for the text in Inkscape, but the package 'color.sty' is not loaded}%
    \renewcommand\color[2][]{}%
  }%
  \providecommand\transparent[1]{%
    \errmessage{(Inkscape) Transparency is used (non-zero) for the text in Inkscape, but the package 'transparent.sty' is not loaded}%
    \renewcommand\transparent[1]{}%
  }%
  \providecommand\rotatebox[2]{#2}%
  \newcommand*\fsize{\dimexpr\f@size pt\relax}%
  \newcommand*\lineheight[1]{\fontsize{\fsize}{#1\fsize}\selectfont}%
  \ifx\svgwidth\undefined%
    \setlength{\unitlength}{779.52755906bp}%
    \ifx\svgscale\undefined%
      \relax%
    \else%
      \setlength{\unitlength}{\unitlength * \real{\svgscale}}%
    \fi%
  \else%
    \setlength{\unitlength}{\svgwidth}%
  \fi%
  \global\let\svgwidth\undefined%
  \global\let\svgscale\undefined%
  \makeatother%
  \begin{picture}(1,0.24)%
    \lineheight{1}%
    \setlength\tabcolsep{0pt}%
    \put(0,0){\includegraphics[width=\unitlength,page=1]{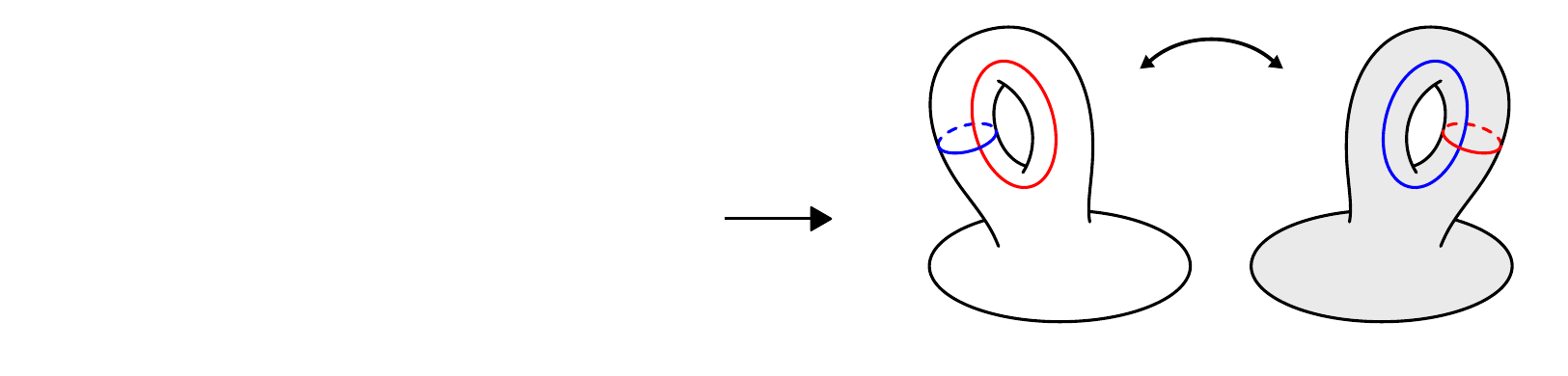}}%
    \put(0.78082025,0.22282774){\color[rgb]{0,0,0}\makebox(0,0)[t]{\smash{\begin{tabular}[t]{c}{\small$\t'$}\end{tabular}}}}%
    \put(0,0){\includegraphics[width=\unitlength,page=2]{simple_Z2_stab.pdf}}%
    \put(0.21509295,0.22282774){\color[rgb]{0,0,0}\makebox(0,0)[t]{\smash{\begin{tabular}[t]{c}{\small$\t$}\end{tabular}}}}%
    \put(0.11172365,0.00009086){\color[rgb]{0,0,0}\makebox(0,0)[t]{\smash{\begin{tabular}[t]{c}{\small$D$}\end{tabular}}}}%
    \put(0,0){\includegraphics[width=\unitlength,page=3]{simple_Z2_stab.pdf}}%
    \put(0.6755269,0.00009086){\color[rgb]{0,0,0}\makebox(0,0)[t]{\smash{\begin{tabular}[t]{c}{\small$T$}\end{tabular}}}}%
  \end{picture}%
\endgroup%

    \caption{A (simple) $\Z/2$-stabilization.}
\label{fig:simple_Z2_stab}
\end{figure}

Let $H_1$ and $H_2$ be isotopy diagrams. Then $H_2$ is obtained from $H_1$ by a \emph{$\Z/2$-(de)sta\-bi\-li\-za\-tion} if
they have representatives $(\S_2, \alphas_2, \betas_2,\tau_2)$ and $(\S_1, \alphas_1, \betas_1,\tau_1)$, respectively, such that $(\S_2, \alphas_2, \betas_2,\tau_2)$ is obtained from $(\S_1, \alphas_1, \betas_1,\tau_1)$ by a $\Z/2$-(de)stabilization.
\end{definition}

\begin{definition} \label{def:1-stab}
The sutured diagram $(\S_2, \alphas_2, \betas_2,\tau_2)$ is obtained from $(\S_1, \alphas_1, \betas_1,\tau_1)$ by a \emph{$\{1\}$-stabilization}
if
\begin{itemize}
\item there is a disk $D \subset \S_1$ which is fixed setwise by $\tau_1$ and a punctured torus $T \subset \S_2$ fixed setwise by $\tau_2$ such that $\S_1 \setminus D = \S_2 \setminus T$,
\item $\alphas_1  = \alphas_2 \cap (\S_2 \setminus T)$,
\item $\alphas_2 \cap T$ and $\betas_2 \cap T$ are simple closed curves that intersect each other transversely in a single point and are exchanged by the involution.
\end{itemize}
We also say that $(\S_1, \alphas_1, \betas_1)$ is obtained from $(\S_2, \alphas_2, \betas_2)$ by a \emph{$\{1\}$-destabilization}. See \Cref{fig:simple_triv_stab}.

Just as before, if $H_1$ and $H_2$ are isotopy diagrams, then $H_2$ is obtained from $H_1$ by a \emph{$\{1\}$-(de)sta\-bi\-li\-za\-tion} if
they have representatives $(\S_2, \alphas_2, \betas_2,\tau_2)$ and $(\S_1, \alphas_1, \betas_1,\tau_1)$, respectively, such that $(\S_2, \alphas_2, \betas_2,\tau_2)$ is obtained from $(\S_1, \alphas_1, \betas_1,\tau_1)$ by a $\{1\}$-(de)stabilization.
\end{definition}

\begin{figure}[h]
\def\svgwidth{.8\linewidth}
\begingroup%
  \makeatletter%
  \providecommand\color[2][]{%
    \errmessage{(Inkscape) Color is used for the text in Inkscape, but the package 'color.sty' is not loaded}%
    \renewcommand\color[2][]{}%
  }%
  \providecommand\transparent[1]{%
    \errmessage{(Inkscape) Transparency is used (non-zero) for the text in Inkscape, but the package 'transparent.sty' is not loaded}%
    \renewcommand\transparent[1]{}%
  }%
  \providecommand\rotatebox[2]{#2}%
  \newcommand*\fsize{\dimexpr\f@size pt\relax}%
  \newcommand*\lineheight[1]{\fontsize{\fsize}{#1\fsize}\selectfont}%
  \ifx\svgwidth\undefined%
    \setlength{\unitlength}{779.52755906bp}%
    \ifx\svgscale\undefined%
      \relax%
    \else%
      \setlength{\unitlength}{\unitlength * \real{\svgscale}}%
    \fi%
  \else%
    \setlength{\unitlength}{\svgwidth}%
  \fi%
  \global\let\svgwidth\undefined%
  \global\let\svgscale\undefined%
  \makeatother%
  \begin{picture}(1,0.26545455)%
    \lineheight{1}%
    \setlength\tabcolsep{0pt}%
    \put(0,0){\includegraphics[width=\unitlength,page=1]{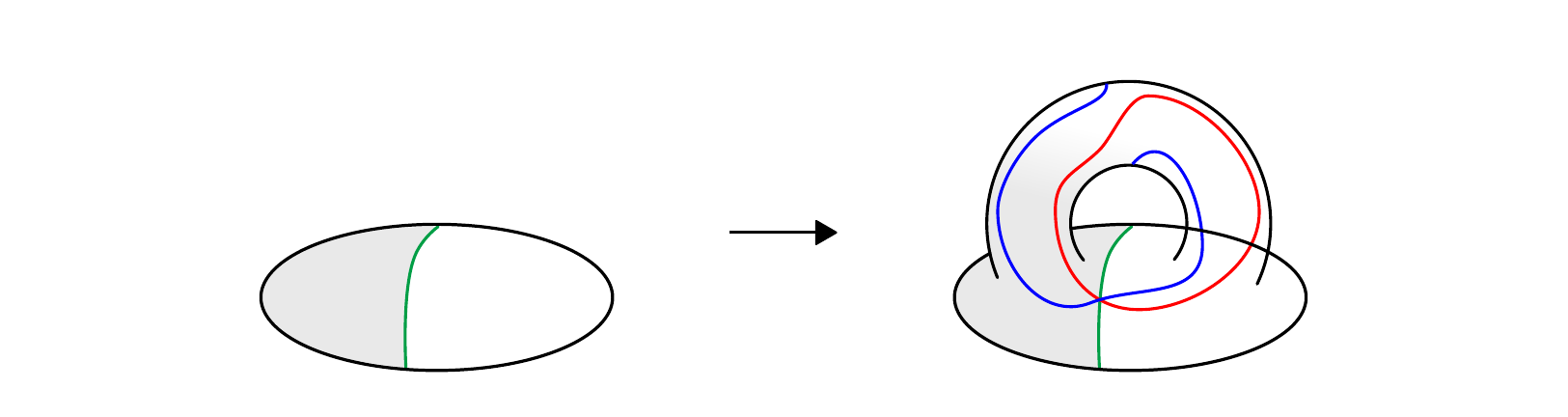}}%
    \put(0.27749825,-0.00199323){\color[rgb]{0,0,0}\makebox(0,0)[t]{\smash{\begin{tabular}[t]{c}{\small$D$}\end{tabular}}}}%
    \put(0.71797985,-0.00281534){\color[rgb]{0,0,0}\makebox(0,0)[t]{\smash{\begin{tabular}[t]{c}{\small$T$}\end{tabular}}}}%
    \put(0,0){\includegraphics[width=\unitlength,page=2]{simple_triv_stab.pdf}}%
    \put(0.72739857,0.24150187){\color[rgb]{0,0,0}\makebox(0,0)[t]{\smash{\begin{tabular}[t]{c}{\small$\t'$}\end{tabular}}}}%
    \put(0,0){\includegraphics[width=\unitlength,page=3]{simple_triv_stab.pdf}}%
    \put(0.2713531,0.24150187){\color[rgb]{0,0,0}\makebox(0,0)[t]{\smash{\begin{tabular}[t]{c}{\small$\t$}\end{tabular}}}}%
  \end{picture}%
\endgroup%

    \caption{A (simple) $\{1\}$-stabilization.}
\label{fig:simple_triv_stab}
\end{figure}

\begin{remark}
    We will often write $\cO$-stabilization when we do not want to distinguish between the two orbit types of stabilizations.
\end{remark}

If $d \colon \S \to \S'$ is an equivariant diffeomorphism of surfaces and $C$ is an isotopy class of attaching sets in $\S$,
then $d(C)$ is defined as $[d(\etas)]$, where $\etas$ is an arbitrary attaching set representing $C$.

\begin{definition}
Given isotopy diagrams $H_1 = (\S_1, A_1, B_1, \tau_1)$ and $H_2 = (\S_2, A_2, B_2,\tau_2)$, a diffeomorphism\footnote{We emphasize that we define diffeomorphisms of real Heegaard diagrams to be equivariant; for this reason we will not always include the adjective ``equivariant''.} $d \colon H_1 \to H_2$
is an orientation-preserving equivariant diffeomorphism $d \colon \S_1 \to \S_2$ such that $d(A_1) = A_2$ and $d(B_1) = B_2$.
\end{definition}

\begin{definition} \label{def:big-graph}
Let $\G$ be the graph
whose class of vertices $|\G|$ consists of real isotopy diagrams and, for
$H_1$,$H_2 \in |\G|$, the edges
$\G(H_1,H_2)$ can be written as a union
\[
\G(H_1,H_2) = \G_{\a\b}(H_1,H_2)
\cup \G_{\{1\}}(H_1, H_2) \cup\G_{\Z/2}(H_1, H_2) \cup \G_{\text{diff}}(H_1,H_2).
\]
The set $\G_{\a\b}(H_1,H_2)$
consists of a single arrow if $H_1$ and $H_2$ are real-equivalent and
is empty otherwise. The set $\G_{\{1\}}(H_1, H_2)$ consists of a single arrow
if $H_2$ is obtained from $H_1$ by a $\{1\}$-stabilization or a
$\{1\}$-destabilization and is empty otherwise. Similarly, set $\G_{\Z/2}(H_1, H_2)$ consists of a single arrow
if $H_2$ is obtained from $H_1$ by a $\Z/2$-stabilization or a
$\Z/2$-destabilization and is empty otherwise.
Finally, $\G_{\text{diff}}(H_1,H_2)$ consists of all diffeomorphisms from $H_1$ to $H_2$.
\end{definition}

\begin{proposition}\label{prop:diag-connected-weak}
The isotopy diagrams $H_1$, $H_2 \in |\G|$ can be connected by an oriented path if and only if they define equivariantly diffeomorphic real sutured
manifolds. Furthermore, the existence of an unoriented path from $H_1$ to $H_2$ implies the existence of an oriented one.
\end{proposition}

\begin{proof}
This follows from the work of Nagase \cite{nagase}; also see \cite[Lemma 3.16]{guth_manolescu2025real}. 
\end{proof}

\begin{definition} \label{def:weak-Heegaard}
Let $\cS^R$ be a set of diffeomorphism types of real sutured manifolds, and let $\cC$ be any category. Let $\G(\cS^R)$ be the full subgraph of $\G$ spanned by those
isotopy diagrams $H$ for which $S(H) \in \cS^R$.
A \emph{weak Heegaard invariant of $\cS^R$}
is a morphism of graphs
$F \colon \G(\cS^R) \to \cC$ such that for every arrow $e$ of $\G(\cS^R)$ the image $F(e)$ is an isomorphism.
\end{definition}

\begin{definition}\label{def:S-cat}
As in \cite{JTZ_naturality_mapping_class_groups} we will consider three different classes.
    \begin{enumerate}
    \item $\cS_{man}^R$, the set of $[Y(p, \xi), \t]$, where $(Y, p,\t)$ is a based real manifold and $\xi = (\xi_1, \xi_2, \xi_3)$ is a framing of $p$ so that the restriction $d\t|_{\mathrm{span}(\xi_1,\xi_2)}$ is orientation reversing (this implies that $\xi_3$ points in the direction tangent to the fixed point set). 
    \item $\cS_{link}^R$, the set of $[Y(L, \w, \z), \t]$, where $(Y, \t)$ is a real manifold and $L \sub Y$ is a generalized strongly invertible link with at least one pair of $\w$ and $\z$ markings on each component (which are also fixed setwise by $\t$).
    \item $\cS^R_{sut}$, the set of all real sutured manifolds $[(Y, \gamma, \t)]$. 
\end{enumerate}
\end{definition}

\begin{thm}[\cite{guth_manolescu2025real}] \label{thm:HF-weak}
 The functors
 \[
 \widehat{\CFR} \text{, }\CFR^- \text{, } \CFR^+ \text{, }\CFR^\infty \colon \G(\cS^R_{man}) \to \cC
 \]
 are weak real Heegaard invariants of $\cS^R_{man}$ (where the $U$-action is trivial on $\widehat{\CFR}$). Here, $\cC$ is the homotopy category of $\F[U]$-complexes.
\end{thm}

There are also variants of real Heegaard Floer homology for knots \cite{hendricks:real}, links \cite{xiao}, and sutured manifolds. The invariance results of \cite{xiao} and \cite{BGX:real_sutured} can also be phrased in this language.

\begin{thm}[\cite{xiao2}] \label{thm:HFL-weak}
 The functors
 \[
 \widehat{\CFLR} \text{, }\CFLR^- \text{, } \CFLR^+ \text{, }\CFLR^\infty \colon \G(\cS^R_{link}) \to \cC
 \]
 are weak real Heegaard invariants of $\cS^R_{link}$ (where the polynomial action is trivial on $\widehat{\CFLR}$). Here, $\cC$ is the homotopy category of curved $\F[u_1, \hdots, u_k, U_1,\hdots, U_\ell]$-complexes.
\end{thm}

\begin{thm}[\cite{BGX:real_sutured}] \label{thm:RSFH-weak}
 The functor
 \[
 \RSFC \colon \G(\cS^R_{sut}) \to \cC
 \]
is a weak real Heegaard invariant of $\cS^R_{sut}$. Here, $\cC$ is the homotopy category of $\F$-complexes.
\end{thm}

\subsection{Strong Real Heegaard Floer Invariants}

We now outline the additional axioms which must be satisfied in order to obtain natural invariants of real manifolds.

\begin{definition}\label{def:distinguished-rect}
Let $H_i = (\S_i,[\alphas_i],[\betas_i], \tau_i)$ be real isotopy diagrams for $1 \le i \le 4$.
A \emph{distinguished rectangle} in $\G$ is a subgraph
\begin{align*}
\begin{tikzcd}[ampersand replacement = \&]
    H_1 \ar[r,"e"] \ar[d,"f"] \& H_2 \ar[d,"g"]\\
    H_3 \ar[r,"h"] \ar[r] \& H_4
\end{tikzcd}
\end{align*}
of $\G$ that satisfies one of the following properties:
\begin{enumerate}
\item\label{item:rect-alpha-stab} Either both $e$ and $h$ are real-equivalences while $f$ and $g$ are both $\cO$-stabilizations.
\item\label{item:rect-alpha-diff}  Both $e$ and $h$ are real-equivalences while $f$ and $g$ are both equivariant diffeomorphisms. In this case, we necessarily have
$\S_1 = \S_2$ and $\S_3 = \S_4$. We require, in addition, that the diffeomorphisms
$f \colon \S_1 \to \S_3$ and $g \colon \S_2 \to \S_4$ are the same.
\item\label{item:rect-stab-stab} The maps $e$ and $h$ are $\cO_1$-stabilizations while $f$ and $g$ are $\cO_2$-stabilizations, such that there are disjoint orbits of disks $\cO_1 \times D_1$ and $\cO_2\times D_2$ and disjoint orbits of punctured tori $\cO_1 \times T_1$, $\cO_2 \times T_2 \subset \S_4$ satisfying
$\S_1 \setminus ((\cO_1 \times D_1) \cup (\cO_2 \times D_2))= \S_4 \setminus (\cO_1 \times T_1) \cup (\cO_2 \times T_2)$, and such that
$\S_2 = (\S_1 \setminus (\cO_1 \times D_1)) \cup (\cO_1 \times T_1)$ and $\S_3 = (\S_1 \setminus (\cO_2 \times D_2)) \cup (\cO_2 \times T_2)$.
\item\label{item:rect-stab-diff} The maps $e$ and $h$ are $\cO$-stabilizations, while $f$ and $g$ are equivariant diffeomorphisms. Furthermore,
there are orbits of disks $(\cO \times D) \subset \S_1$ and $(\cO \times D') \subset \S_3$ and orbits of punctured tori $T\times \cO \subset \S_2$ and
$\cO \times T' \subset \S_4$ such that $\S_1 \setminus (\cO\times D) = \S_2 \setminus (\cO\times T)$ and $\S_3 \setminus (\cO\times D') = \S_4 \setminus (\cO\times T')$,
and the diffeomorphisms $f$, $g$ satisfy $f(\cO\times D) = \cO\times D'$,  $g(\cO\times T) = \cO\times T'$, and $f|_{\S_1 \setminus (\cO\times D)} = g|_{\S_2 \setminus (\cO\times T)}$.
\end{enumerate}
\end{definition}

\begin{figure}[h]
\def\svgwidth{.8\linewidth}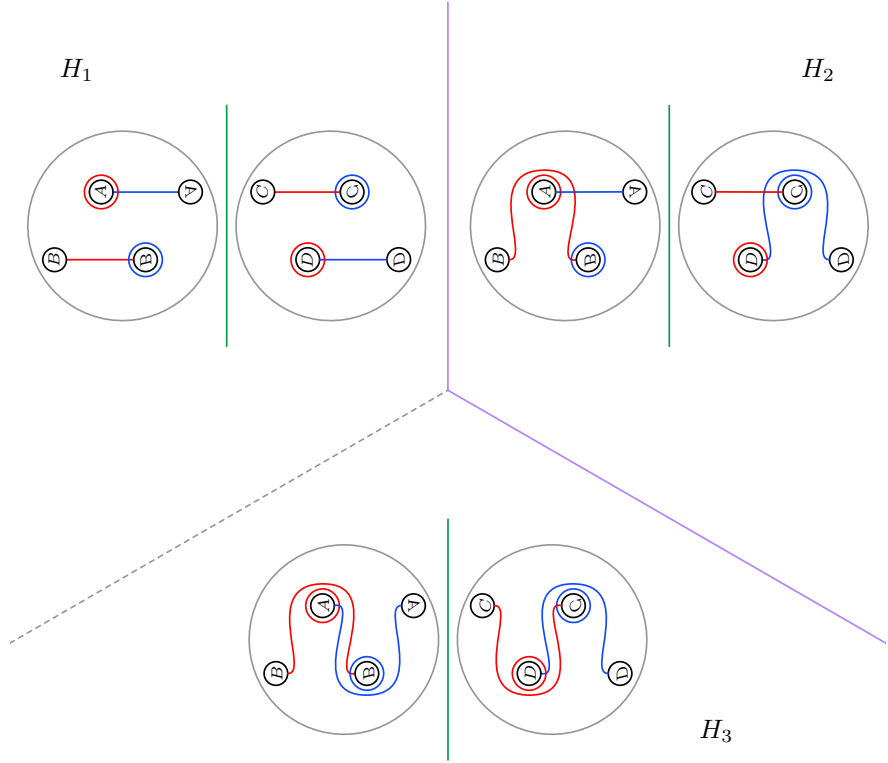
    \caption{A simple real handleswap loop. As usual in Heegaard Floer theory, the alpha curves are in red and the beta curves in blue. Also, the fixed point set is in green. The purple lines represent real equivalences, and the dashed line represents a diffeomorphism. See Remark~\ref{rem:color_code}.}
\label{fig:simple_swap}
\end{figure}

\begin{definition} \label{def:simple-handleswap}
A \emph{simple real handleswap} is a subgraph of $\G$ of the form 
\begin{align*}
\begin{tikzcd}[ampersand replacement = \&]
    H_1 \ar[rr,"e"] \& \& H_2\ar[ld,"f"] \\
    \& H_3 \ar[ul,"g"] \& 
\end{tikzcd}
\end{align*}
such that
\begin{enumerate}
\item $H_i = (\S_0 \# \S\# \S_1,[\alphas_i],[\betas_i], \tau)$ are isotopy diagrams for $i \in \{\,1,2,3\,\}$,
where $\S_0$ is a genus two surface,
\item $e$ and $f$ are real equivalences and $g$
  is an equivariant diffeomorphism,
\item in the surface $P = (\S_0\# \S \# \S_1) \setminus (\S)$ (which is a pair of punctured genus two surfaces), the above triangle
is conjugate to the triangle in \Cref{fig:simple_swap}; i.e.,
there is a diffeomorphism that throws $P \cap H_i$ onto the pictures
in the green circles, sending the $\alpha$-circles in~$P$ to the two red circles,
and the $\beta$-circles in $P$ to the two blue circles,
\item $\S_0$ and $\S_1$ are exchanged by $\tau$,
\item \label{it:identical} in $\S$, the diagrams $H_1$, $H_2$, and $H_3$ are identical.
\end{enumerate}
\end{definition}

\begin{remark}
We could have combined the two real handleslides from \Cref{def:simple-handleswap} into a single real equivalence. We chose to keep them separate to be in line with \cite[Definition 2.32]{JTZ_naturality_mapping_class_groups}. 
\end{remark}

\begin{figure}[h]
\def\svgwidth{.8\linewidth}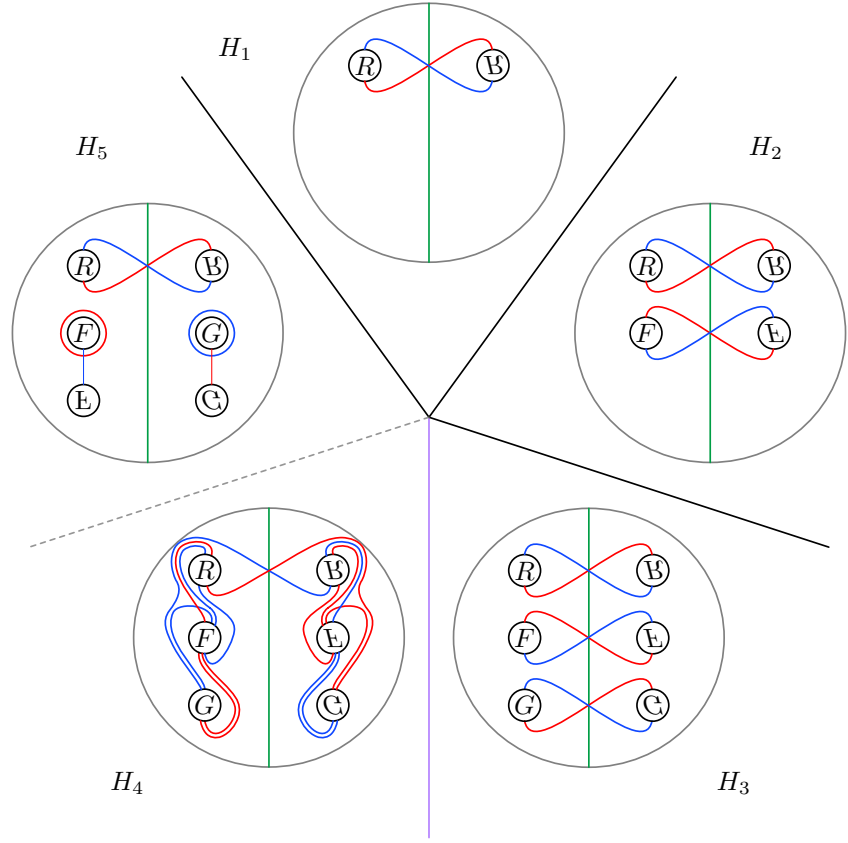
    \caption{A simple trade loop. The black lines indicate (de)stabilizations. The letters in the circles indicate which feet of the handles are identified, and how. (We emphasize that for a $\{1\}$-stabilization, while the attaching circles are symmetric along a vertical axis, the gluing is by reflection in a horizontal axis.) We chose the letters $R$, $F$, $G$ because they have no symmetries, so their position fully describes the gluing.}
\label{fig:simple_trade_loop}
\end{figure}

\begin{definition}\label{def:simple-trade}
    A \emph{simple trade} is a subgraph of $\cG_R$ of the form
        \begin{align*}
    \begin{tikzcd}[ampersand replacement = \&,column sep = small]
        \&\& H_1 \ar[rrd,"e"] \&\& \\
        H_5 \ar[urr,"i"] \&\&\&\& H_2  \ar[dl,"f"] \\
       \& H_4 \ar[ul,"h"] \& \& H_3  \ar[ll,"g"] \&\\
    \end{tikzcd}
    \end{align*}

such that 
\begin{enumerate}
\item $H_1 = (\S \# \S_0,[\alphas_1],[\betas_1], \tau_1)$, $H_2 = (\S \# \S_0\#\S_1,[\alphas_2],[\betas_2], \tau_2)$, $H_i = (\S \# \S_0\#\S_1\#\S_2,[\alphas_i],[\betas_i], \tau_i)$ for $i \in \{3,4\}$, and $H_5 = (\S \# \S_0,[\alphas_5],[\betas_5], \tau_5)$ are isotopy diagrams where $\S_i$ are genus 1 surfaces,
\item $e$ and $f$ are $\{1\}$-stabilizations, $g$
 is a real equivalence (consisting of three real handleslides: first $G$ over $F$, then $F$ over $R$, and then $R$ over $F$), $h$ is an equivariant diffeomorphism, and $i$ is a $\Z/2$-destabilization,
\item in the complement of $\S$, the above pentagon
is conjugate to the pentagon in \Cref{fig:simple_trade_loop}; i.e.,
there is a diffeomorphism that throws $P \cap H_i$ onto the pictures
in the gray circles, sending the $\alpha$-circles in~$P$ to the three red circles,
and the $\beta$-circles in $P$ to the three blue circles.
\end{enumerate}
\end{definition}

\begin{definition}\label{def:strong-Heegaard}
Let $\cS^R$ be a set of equivariant diffeomorphism types of real sutured manifolds.
A \emph{strong real Heegaard invariant of $\cS^R$} is a weak Heegaard
invariant $F \colon \G(\cS^R) \to \cC$ that satisfies the following
axioms:
\begin{enumerate}
\item\label{item:strong-funct} \textbf{Functoriality}: The restrictions 
  of $F$ to $\G_{\a\b}(\cS^R)$ and~$\G_{\text{diff}}(\cS^R)$ are functors to~$\cC$. Furthermore, if
  $e \colon H_1 \to H_2$
  is an $\cO$-stabilization and $e' \colon H_2 \to H_1$ is the corresponding $\cO$-destabilization,
  then $F(e') \sim F(e)^{-1}$.
\item\label{item:strong-commute} \textbf{Commutativity}: For every
  distinguished rectangle
\begin{align*}
\begin{tikzcd}[ampersand replacement = \&]
    H_1 \ar[r,"e"] \ar[d,"f"] \& H_2 \ar[d,"g"]\\
    H_3 \ar[r,"h"] \ar[r] \& H_4
\end{tikzcd}
\end{align*}
in $\G(\cS^R)$, we have $F(g) \circ F(e) \sim F(h) \circ F(f)$.
\item\label{item:strong-cont} \textbf{Continuity}: If $H \in |\G(\cS^R)|$ and $e \in \G_{\text{diff}}(H,H)$ is a diffeomorphism isotopic to $\text{Id}_\S$, then $F(e) \sim \text{Id}_{F(H)}$.
\item\label{item:strong-handleswap} \textbf{Handleswap invariance}: For every simple handleswap
\begin{align*}
\begin{tikzcd}[ampersand replacement = \&]
    H_1 \ar[rd,"e"] \\
    H_3 \ar[u,"g"] \& \ar[l,"f"] H_2
\end{tikzcd}
\end{align*}
in $\G(\cS^R)$, we have $F(g) \circ F(f) \circ F(e) \sim \text{Id}_{F(H_1)}$.
\item\label{item:strong-trade}\textbf{Trade invariance:} For every simple trade pentagon 
    \begin{align*}
    \begin{tikzcd}[ampersand replacement = \&,column sep = small]
        \&\& H_1 \ar[rrd,"e"] \&\& \\
        H_5 \ar[urr,"i"] \&\&\&\& H_2  \ar[dl,"f"] \\
       \& H_4 \ar[ul,"h"] \& \& H_3  \ar[ll,"g"] \&\\
    \end{tikzcd}
    \end{align*}
    in $\G(\cS^R)$, we have $F(i)\circ F(h)\circ F(g) \circ F(f) \circ F(e) \sim \text{Id}_{F(H_1)}$.
\end{enumerate}
\end{definition}

The main result of this paper is the following:
\begin{thm}\label{thm:naturality}
    The following are strong real Heegaard invariants:
    \begin{enumerate}
        \item Real sutured Floer homology;
        \item Real Heegaard Floer homology;
        \item Real link Floer homology.
    \end{enumerate}
\end{thm}

\begin{definition} \label{def:isotopic}
  Suppose that $H_1$ and~$H_2$ are two isotopy diagrams of
  $(Y,\gamma,\tau)$ with $H_i = (\S_i,A_i,B_i,\tau_i)$, and let $\iota_i :
  \Sigma_i \to Y$ be the inclusion for $i \in \{1,2\}$.  Then a
  diffeomorphism $d : H_1 \to H_2$ is \emph{isotopic to the identity
    in~$Y$} if $\iota_2 \circ d \colon \S_1 \to Y$ is equivariantly isotopic to
  $\iota_1 \colon \S_1 \to Y$ relative to $s(\gamma)$.
\end{definition}

\begin{definition} \label{def:subgraph} Let $(Y,\g,\tau)$ be a real sutured
  manifold. Then $\G_{(Y,\g,\tau)}$ is the subgraph of~$\G$ whose vertex
  set $|\G_{(Y,\g,\tau)}|$ consists of all isotopy
  diagrams of $(Y,\g,\tau)$. The set of edges between $H_1$, $H_2 \in
  |\G_{(Y,\g,\tau)}|$ is defined to be
  \[
  \G_{(Y,\g,\tau)}(H_1,H_2) = \G_{\a\b}(H_1,H_2) \cup \G_{\{1\}}(H_1,H_2) \cup
  \G_{\Z/2}(H_1,H_2) \cup (\G)^0_\text{diff}(H_1,H_2),
  \]
  where $(\G)^0_\text{diff}(H_1,H_2)$ is the set of diffeomorphisms from $H_1$ to $H_2$
  isotopic to the identity in~$Y$.
\end{definition}

In Section~\ref{sec:reducing no_monodromy} we will establish a stronger version of Nagase's result:
\begin{proposition}\label{prop:connected-strong}
    Any two vertices in $\G_{(Y,\g,\tau)}$ can be connected by an oriented path.
\end{proposition}

\begin{definition} \label{def:iso}
  Given a weak real Heegaard invariant $F \colon \G(\cS^R) \to \cC$ and an
  oriented path $\eta$ in $\G(\cS^R)$ of the form
  \[
  H_0 \stackrel{e_1}{\longrightarrow} H_1
  \stackrel{e_2}{\longrightarrow} \dots
  \stackrel{e_n}{\longrightarrow} H_n,
  \]
  define $F(\eta)$ to be the isomorphism
  \[
  F(e_n) \circ \dots \circ F(e_1) \,\colon\, F(H_0) \to F(H_n).
  \]
\end{definition}

Naturality of a real Heegaard invariant means that the equivalence $F(\eta)$ is independent of the choice of path $\eta$ up to homotopy. For strong real Heegaard invariants, this is indeed the case; in Section~\ref{sec:reducing no_monodromy} we will prove the following. 

\begin{thm} \label{thm:iso} Let $\cS^R$ be a set of diffeomorphism
  types of real sutured manifolds containing $[(Y,\g, \tau)]$ and let
  $F \colon \G(\cS^R) \to \cC$ be a strong real Heegaard invariant. Given real
  isotopy diagrams $H$, $H' \in |\G_{(Y,\g,\tau)}|$ and any two oriented
  paths $\eta$ and $\nu$ in $\G_{(Y,\g,\tau)}$ from $H$ to $H'$, we
  have
  \[
  F(\eta) = F(\nu).
  \]
\end{thm}

In particular, for any two vertices $H$ and $H'$, there is a canonical homotopy equivalence between them.
\begin{definition}
  If $H$ and $H'$ are isotopy diagrams
  of $(Y,\g,\tau)$, then let
  \[
  F_{H\ra H'} = F(\eta),
  \]
  where $\eta$ is an arbitrary oriented path connecting $H$ to~$H'$.
  By Theorem~\ref{thm:iso}, the map $F_{H\ra H'}$ does not depend on the
  choice of~$\eta$ up to homotopy.
\end{definition}

\subsection{Transitive systems of curved chain complexes} 

Finally, before proceeding to our analysis of the space of real Heegaard diagrams, we recall the notion of a  transitive system.

\begin{definition}\label{def:transitive-sys-curved-cxs}
    Let $\cC$ be a category and $A$ a set. Let $\mathrm{Trans}_A(\cC)$ be the \emph{category of transitive systems over $\cC$ indexed by $A$}. An object of this category consists of a map $S: A \ra \operatorname{Ob }(\cC)$ together with a collection of distinguished morphisms $\Psi_{a\to b}: S(a) \ra S(b)$ for all $a,b \in A$ which satisfy
    \begin{enumerate}
        \item $\Psi_{a \ra a} = \id$;
        \item $\Psi_{b \ra c} \circ \Psi_{a\ra b} = \Psi_{a\ra c}$.
    \end{enumerate}
    We will at times write objects in this category as $\{S(a)\}_{a \in A}, \{\Psi_{a \to b}\}_{a,b \in A})$ if we want to emphasize the collection of morphisms; we will write $S$ when the distinguished morphisms are clear from the context.
    \end{definition}

    \begin{definition}
    If $S_1$ and $S_2$ are two transitive systems over $\cC$ indexed by $A_1$ and $A_2$, then a \emph{morphism of transitive systems} $F: S_1 \to S_2$  is a collection of morphisms $F_{a_1, a_2}: S_1(a_1) \ra S_2(a_2)$ indexed over $A_1\times A_2$ satisfying 
    \begin{align*}
        \begin{tikzcd}[ampersand replacement = \&]
            S_1(a_1) \ar[r,"F_{a_1,a_2}"]\ar[d,"\Psi_{a_1 \ra b_1}" left] \& S_2(a_2) \ar[d,"\Psi_{a_2 \ra b_2}"]\\
            S_1(b_1) \ar[r,"F_{b_1,b_2}"] \& S_2(b_2) 
        \end{tikzcd}
    \end{align*}
    for all pairs $(a_1, b_1), (a_2, b_2) \in A_1 \times A_2$.

    The identity map of a transitive system $(\{S(a)\}_{a \in A}, \{\Psi_{a \to b}\}_{a, b \in A})$ is defined by the diagrams
    \begin{align*}
        \begin{tikzcd}[ampersand replacement = \&]
            S(a) \ar[r,"\Psi{a\to b}"]\ar[d,"\Psi_{a \ra c}" left] \& S(b) \ar[d,"\Psi_{b \ra d}"]\\
            S(c) \ar[r,"\Psi_{c\to d}"] \& S(d).
        \end{tikzcd}
    \end{align*}
    for all $a, b, c, d \in A$.
    A morphism of transitive systems $F: S_1 \to S_2$ is called an {\em isomorphism} if there is a morphism $G: S_2 \to S_1$ such that $F \circ G$ and $G \circ F$ are the identity morphisms.
\end{definition}

\begin{thm}\label{thm:transitive_system}
    Let $F: \cG^R(\cS^R) \ra \cC$ be a strong real Heegaard invariant. Then, the map $|\G_{(Y,\g,\tau)}| \ra \mathrm{Ob}(\cC)$ given by $H \mapsto F(H)$, together with the collection $\{F_{H \to H'}\}_{H, H' \in |\G_{(Y,\g,\tau)}|}$, forms a transitive system. Furthermore, diffeomorphisms of real manifolds induce isomorphisms of transitive systems. 
\end{thm}
\begin{proof}
    Given $H \in |\G_{(Y,\g,\tau)}|$, the $F_{H \ra H}$ can be represented by $F(\eta)$ for any loop $\eta \in \pi_1(\G_{(Y,\g,\tau)}, H)$. In particular, we may take $\eta$ to be the constant map, $*_H$. Therefore, $F_{H \to H} = F(*_H) = \id_{F(H)}$. 
    
    If $H, H', H'' \in |\G_{(Y,\g,\tau)}|$, let $\eta_{H \ra H''}$, $\eta_{H \ra H'}$, and $\eta_{H' \ra H''}$ be paths from $H$ to $H''$, $H$ to $H'$, and $H'$ to $H''$ respectively. From \Cref{thm:iso} it follows that $$F_{H \ra H''} = F(\eta_{H \ra H''}) = F(\eta_{H' \ra H''} \circ \eta_{H \ra H'}).$$ Functoriality implies that $$F(\eta_{H' \ra H''} \circ \eta_{H \ra H'}) = F(\eta_{H' \ra H''}) \circ F(\eta_{H \ra H'}) = F_{H' \ra H''}\circ F_{H \ra H'}.$$ 

    Let $\varphi: (Y_0, \t_0) \to  (Y_1, \t_1)$ be a diffeomorphism of real manifolds and let $H_0$ and $H_1$ be  real isotopy diagrams for $(Y_0, \t_0)$. Let $d_i = \varphi|_{\S_i}$ for $i\in \{0,1\}$ so that $H_i' := d_i(H_i)$ are isotopy diagrams for $(Y_1, \t_1)$ (again, $i\in \{0,1\}$). The diagram 
    \begin{align*}
        \begin{tikzcd}[ampersand replacement = \&]
            F(H_0) \ar[r,"F(d_0)"] \ar[d,"F_{H_0\to H_1}" left] \& F(H_0')\ar[d,"F_{H_0'\to H_1'}" right] \\
            F(H_1) \ar[r,"F(d_1)"] \& F(H_1')
        \end{tikzcd}
    \end{align*}
    is commutative by the Commutativity Axiom (\Cref{def:strong-Heegaard} \Cref{item:strong-commute}). In this way, $\varphi$ induces a morphism $F(\varphi)$ of transitive systems. The map induced by $\varphi^{-1}$ is easily seen to satisfy $F(\varphi^{-1}) = F(\varphi)^{-1}$. 
\end{proof}

\subsection{Naturality of the real Heegaard Floer invariants}
Equipped with Theorems~\ref{thm:naturality},~\ref{thm:iso}, and~ \ref{thm:transitive_system}, we are ready to establish naturality of the real Heegaard Floer invariants. The simplest case is real sutured Floer homology.

\begin{proof}[Proof of \Cref{thm:sutured-nat}]
By Theorem~\ref{thm:naturality}, the morphism $F = \RSFH$ is a strong real Heegaard invariant of $\cS^R_{sut}$. Hence, the groups~$F(H)$ and the isomorphisms $F_{H\to H'}$ form a transitive system by \Cref{thm:transitive_system}.  In the same proof, we showed how a diffeomorphism $\varphi$ of $(Y, \g, \t)$ induced a morphism of transitive systems. 

All that remains is to show that isotopic diffeomorphisms induce identical maps on~$\RSFH$. Equivalently, we show that for any diffeomorphism $\varphi : (Y,\g,\t) \to (Y,\g,\t)$ isotopic to~$\id_{(Y,\g,\t)}$, we have $F(\varphi) =\id_{\RSFH(Y,\g,\t)}$. This is a consequence of \Cref{thm:iso}. Let $H$ be an isotopy diagram of $(Y,\g,\t)$, and let $d = \phi|_H$ and $H' = \phi(H)$.  Since $d$ is isotopic to the identity, it corresponds to an edge $\eta$ of $\G_{(Y,\g,\t)}$ connecting $H$ and $H'$. By \Cref{thm:iso}, $F(\eta) = F_{H \to H'}$. Hence, the map $F(\varphi)$ induces the identity map of the transitive system.

Following \cite[Remark 1.2]{JTZ_naturality_mapping_class_groups}, we define the canonical invariant $\RSFH(Y, \g, \t)$ to be the colimit of this transitive system. 
\end{proof}

The case of the real Heegaard Floer invariants for closed (pointed) three-manifolds is slightly more subtle. In order to associate a real sutured manifold to $(Y, \t, p)$, we must also choose a framing $\xi = (\xi_1, \xi_2, \xi_3)$ of $p$ so that $\xi_3 \in T_pC$, as in \Cref{def:S-cat}. Let $\RMan_{*, \xi}$ be the category of real 3-manifolds equipped with framed basepoints like this. In this category, a morphism from $(Y_0, \t_0, p_0, \xi_0)$ to $(Y_1, \t_1, p_1, \xi_1)$ is a pointed, equivariant diffeomorphism $\varphi: (Y_0, \t_0, p_0) \ra (Y_1, \t_1, p_1)$ satisfying $d\varphi_p(\xi_0) = \xi_1$. 

By way of \cite[Definition 2.3]{JTZ_naturality_mapping_class_groups}, the data $(Y, \t, p, \xi)$ determines a real sutured manifold $(Y(p), \gamma(p, \xi), \t)$: We  replace $p$ with the unit sphere bundle of $T_pY$ equipped with the suture $\gamma(p,\xi)$ which is the unit circle bundle of the plane $\mathrm{span}(\xi_1, \xi_2)$.  

\begin{proposition}\label{prop:framed-HFR-nat}
    There is a functor 
    \begin{align*}
        \CFR^\circ: \RMan_{*,\xi} \ra \mathrm{Trans}_{\cG^R_{(Y, \t, p)}}(\cC),
    \end{align*}
    where $\cC$ is the homotopy category of chain complexes over the ring $\F[U]$.
\end{proposition}
\begin{proof}
    As in the sutured case, Theorems~\ref{thm:naturality} and \ref{thm:iso} give rise to a functor $\cS_{man}^R \to \mathrm{Trans}_{\cG^R_{(Y, \t, p)}}(\cC)$. By precomposing with the functor taking $(Y, \t, p, \xi)$ to $(Y(p), \gamma(p, \xi), \t)$ we obtain a natural invariant of $(Y, \t, p, \xi)$, which we denote by $\CFR^\circ(Y, \t, p, \xi)$.
\end{proof}
\begin{rem}
    More generally, there is a functor 
    \begin{align*}
        \CFR^\circ: \RMan_{*,\vec{\xi}} \ra \mathrm{Trans}_{\cG^R_{(Y, \t, p)}}(\cC)
    \end{align*}
    where $\RMan_{*,\vec{\xi}}$ is the category of multi-pointed real three-manifolds equipped with a multi-framing $\vec{\xi}$ (i.e., a framing for each basepoint) and $\cC$ is the homotopy category of curved  chain complexes over the ring $\F[U_{\w}]:=\F[U_{w_1}, \hdots, U_{w_\ell}]$, where $\ell = |\w|$.
\end{rem}

In order to remove the framing dependence, we must define canonical isomorphisms $$\iota_{\xi, \xi'}: \CFR^\circ(Y, \t, \w, \xi) \to \CFR^\circ(Y, \t, \w, \xi').$$ This is more subtle than in the unreal case. Note that a framing of $T_pY$ is equivalent to a choice of an \emph{oriented} 2-plane bundle $V \sub T_pY$. Forgetting the symmetry, the fiber of the forgetful map $\Man_{*,\xi} \to \Man_*$ is $\Gr_2(T_pY)$, the space of \emph{oriented} 2-planes in $T_pY$, which is diffeomorphic to $S^2$. However, in the present situation, the fiber of the forgetful map $$\RMan_{*,\xi} \to \RMan_*$$ is not simply connected. Indeed, the fiber is $\Gr_2^R(T_pY)$, the subspace of $\Gr_2(T_pY)$ consisting of planes $V$ satisfying $d\t(V) = -V$; we call such planes \emph{anti-invariant}. The map assigning $V \in \Gr_2^R(T_pY)$ its positive unit normal vector establishes a homeomorphism $\Gr_2^R(T_pY) \ra S^1$. Therefore, we must prove that the monodromy about this loop is trivial. 

\subsubsection{Equivariant basepoint twist}
Let us establish some notation. 

\begin{definition}\label{def:pointed-diffeos}
    Fix a based real $3$-manifold $(Y,\tau,p)$ as well as an anti-invariant plane $V_0 \in \Gr_2^R(T_pY)$.
    \begin{enumerate}
        \item Define $\mathrm{Diff}^\t(Y,\t, p)$ to be the group of $\tau$-equivariant, orientation-preserving diffeomorphisms of $Y$ fixing $p$, with the $C^\infty$ topology, and let $\mathrm{Diff}^\t_0(Y,\t, p)$ be the subgroup of those diffeomorphisms that are equivariantly isotopic to $\id_Y$ through elements of $\mathrm{Diff}^\t(Y,\t, p)$.  
        \item Let $G \leq GL^{+}(T_pY)$ be the centralizer of $d\tau_p$ in the group of orientation-preserving automorphisms of $T_pY$, and let $G_{0}$ be its identity component.
    \end{enumerate}
\end{definition}

\begin{rem}
    We note that $G$ consists of the automorphisms preserving the splitting $T_pY = \ell_p \oplus N_p$ where $\ell_p := T_pC$ and $N_p$ is the $(-1)$-eigenspace of $d\t$. Writing $A = (a, B) \in GL(\ell_p) \times GL(N_p)$, we have $G = \{\, (a,B) : a \cdot \det B > 0 \,\}$ and $G_{0} = \{\, (a,B) : a > 0,\ \det B > 0 \,\} \cong \mathbb{R}_{>0} \times GL^{+}(N_p)$.
\end{rem}

For $\chi \in \mathrm{Diff}^\t(Y,\t, p)$, the differential $d\chi_p$ lies
in the group $G$ of Definition~\ref{def:pointed-diffeos}, so, having fixed
$V_0 \in\Gr_2^R(T_pY)$, we obtain an evaluation map
\[
\ev_{V_0} \colon \mathrm{Diff}^\t(Y,\t, p) \longrightarrow\Gr_2^R(T_pY),
\qquad
\ev_{V_0}(\chi) := d\chi_p(V_0).
\]
\begin{remark}[Conventions]
\label{rem:conventions}
We fix once and for all an orientation of $\ell_p$ near $p$ as well as a positive vector $\nu_p$; orient $N_p$ by declaring that an ordered basis $(v_1, v_2)$ of $N_p$ is positive if $(\nu_p,v_1, v_2)$ is a positive basis of $T_pY$. Orient $\Gr_2^R(T_pY)$ as the unit circle of $N_p$. We may therefore identify $\pi_1(\Gr_2^R(T_pY), V_0)$ with $\Z$. 
\end{remark}

\begin{definition}
\label{def:twist}
On an open set $U \sub Y$ containing $p$, fix $\t$-invariant coordinates
$\kappa \colon U \to \R^3$ such that $\kappa(p) = 0$, $\t$
acts as $R_\pi = \mathrm{diag}(1,-1,-1)$, the chart is
orientation-preserving, and $d\kappa_p$ carries the positive direction of
$\ell_p$ to $+\partial_x$.  Here $R_\theta \in SO(3)$ denotes the rotation
by angle $\theta$ about the $x$-axis, counterclockwise with respect to the
induced orientation of the $(y,z)$-plane.  Fix a smooth function
$\lambda \colon [0,\infty) \to [0,1]$ with $\lambda \equiv 1$ on $[0,r_1]$
and $\lambda \equiv 0$ on $[r_2,\infty)$, where $0 < r_1 < r_2 < 1$.  For
$\alpha \in \R$ define 
\[
\Delta_{\alpha} :=
\begin{cases}
\kappa^{-1} \circ \bigl( x \mapsto R_{\alpha \lambda(|x|)}\, x \bigr)
\circ \kappa & \text{on } U, \\[2pt]
\id_Y & \text{on } Y \setminus U,
\end{cases}
\]
and for $n \in \Z$ define the \emph{$n$-fold equivariant basepoint twist}
\[
\delta^{n} := \Delta_{2\pi n},
\qquad
\delta := \delta^{1} .
\]
\end{definition}

Since $\lambda \equiv 0$ on $[r_2,\infty)$, the two formulas glue smoothly
and $\Delta_\alpha$ is supported in $\kappa^{-1}(\overline{B}_{r_2})$;
since $\lambda \equiv 1$ on $[0,r_1]$, the map $\Delta_\alpha$ agrees with
the rotation $R_\alpha$ near $p$, and in particular $\delta^n$ is the
identity near $p$.

We record the following facts about the map $\alpha \mapsto \Delta_\alpha$.

\begin{lemma}
\label{lem:twist}
\begin{enumerate}
\item The map $\alpha \mapsto \Delta_\alpha$ is a smooth homomorphism
$\R \to \mathrm{Diff}^\t(Y,\t,p)$; that is, each $\Delta_\alpha$ is a
$\t$-equivariant, orientation-preserving diffeomorphism fixing $p$ (and
$C \cap U$ pointwise), and
$\Delta_{\alpha+\beta} = \Delta_\alpha \circ \Delta_\beta$.  
\item Identifying $T_pY$ with $\R^3$ via $d\kappa_p$, we have
$d(\Delta_\alpha)_p = R_\alpha$.  Hence $d(\delta^{n})_p = \id_{T_pY}$, so
$\delta^{n}$ is a morphism $(Y,\t,p,V) \to (Y,\t,p,V)$ of\, $\RManV$ for
\emph{every} $V \in \Gr_2^R(T_pY)$; moreover $s \mapsto \Delta_{2\pi n s}$
is a path in $\mathrm{Diff}^\t(Y,\t,p)$ from $\id_Y$ to $\delta^{n}$, so
$\delta^n \in \mathrm{Diff}^\t_0(Y,\t,p)$, and
$[\ev_{V_0} \circ \Delta_{2\pi n \,\cdot\,}] = n \in
\pi_1(\Gr_2^R(T_pY)) \cong \Z$ for every $V_0$.
\end{enumerate}
\end{lemma}
\begin{proof}
    Straightforward.
\end{proof}

\begin{definition}
\label{def:winding}
Let $V_0, V_1 \in \Gr_2^R(T_pY)$, and let $\Phi = \{\Phi_s\}_{s\in[0,1]}$
be an isotopy through elements of $\mathrm{Diff}^\t(Y,\t,p)$ whose
endpoints satisfy
$d(\Phi_0)_p(V_0) = d(\Phi_1)_p(V_0) = V_1$ in the oriented sense.  Then
$\ev_{V_0} \circ \Phi$ is a loop based at $V_1$, and we define the
\emph{winding number}
\[
w(\Phi) :=
\bigl[\, \ev_{V_0} \circ \Phi \,\bigr]
\;\in\; \pi_1\bigl(\Gr_2^R(T_pY), V_1\bigr) \cong \Z ,
\]
the identification with $\Z$ being given by the orientation of
\Cref{rem:conventions}.
\end{definition}

The following is the analogue of
\cite[Lemma~2.45]{JTZ_naturality_mapping_class_groups} in the real
setting. It plays the role of the calibration of
identifications by basepoint twists.

\begin{lemma}
\label{lem:rel}
Let $(Y,\t,p)$ be a based real $3$-manifold, let
$V_0, V_1 \in \Gr_2^R(T_pY)$, and let $\Phi = \{\Phi_s\}_{s \in [0,1]}$ be
an isotopy through elements of $\mathrm{Diff}^\t(Y,\t,p)$ whose endpoints
$\phi_0 := \Phi_0$ and $\phi_1 := \Phi_1$ satisfy
$d(\phi_0)_p(V_0) = d(\phi_1)_p(V_0) = V_1$ in the oriented sense.  Set
$w := w(\Phi)$.  Then $\phi_0$ is isotopic to $\delta^{-w} \circ \phi_1$
through $\t$-equivariant diffeomorphisms fixing $p$ and carrying $V_0$ to
$V_1$ in the oriented sense; that is, through morphisms
$(Y,\t,p,V_0) \to (Y,\t,p,V_1)$ of\, $\RManV$.
\end{lemma}

\begin{proof}
Identify $T_pY$ with $\R^3$ via $d\kappa_p$, so that
$d(\Delta_\alpha)_p = R_\alpha$ by \Cref{lem:twist}(2).  The map
\[
e \colon \R \longrightarrow \Gr_2^R(T_pY),
\qquad
e(\theta) := R_\theta(V_1),
\]
is a universal covering with deck group $2\pi\Z$. The path $s \mapsto \ev_{V_0}(\Phi_s) =
d(\Phi_s)_p(V_0)$ begins and ends at $V_1$; let
$\Theta \colon [0,1] \to \R$ be its unique lift through $e$ with
$\Theta(0) = 0$, so that
\[
d(\Phi_s)_p(V_0) = R_{\Theta(s)}(V_1) \quad \text{for all } s.
\]
Clearly $\Theta$ is smooth, and, by \Cref{rem:conventions},
$w = w(\Phi) = \Theta(1)/2\pi \in \Z$.

Consider the smooth family
$g_s := \Delta_{-\Theta(s)} \circ \Phi_s$ in
$\mathrm{Diff}^\t(Y,\t,p)$.  Since $\Phi_s(p) = p$,
\[
d(g_s)_p(V_0)
= R_{-\Theta(s)}\bigl( d(\Phi_s)_p(V_0) \bigr)
= R_{-\Theta(s)} R_{\Theta(s)}(V_1)
= V_1
\quad \text{for all } s ,
\]
so $g$ is an isotopy through morphisms
$(Y,\t,p,V_0) \to (Y,\t,p,V_1)$.  Its endpoints are
$g_0 = \Delta_{0} \circ \phi_0 = \phi_0$ and, since
$\Theta(1) = 2\pi w$,
\[
g_1 = \Delta_{-2\pi w} \circ \phi_1
= (\Delta_{2\pi})^{-w} \circ \phi_1
= \delta^{-w} \circ \phi_1 ,
\]
as desired.
\end{proof}

\begin{thm}\label{thm:bspt-twist-trivial}
    Fix a real pointed three-manifold $(Y, \t, p)$ as well as a framing $\xi$ for $p$. Let $\delta$ be the associated equivariant basepoint twist. Then, the induced map
    \begin{align*}
        \CFR^\circ(\delta): \CFR^\circ(Y, \t, p, \xi) \to \CFR^\circ(Y, \t, p, \xi)
    \end{align*}
    acts by the identity map.
\end{thm}
\begin{proof}
    Intuitively, $\CFR^\circ(\delta)$ acts by the identity because we may arrange that it has no effect on the alpha and beta curves. Hence, we may work in a small neighborhood $U \sub \S$ of $p$, and exhibit a sequence of real Heegaard moves from $\delta(U)$ to $U$. This can be done entirely with stabilizations and isotopies. 

    Let $F = \CFR^\circ$. Recall that $F(Y, \t, p, \xi)$ is by definition the transitive system of $\F[U]$-complexes indexed by the set of real Heegaard diagrams for $(Y, \t, p)$. The proof amounts to showing that $F(\delta_H) = F_{H \to \delta(H)}$, or equivalently, that $F_{\delta(H) \to H}\circ F(\delta_H) = \id_{F(H)}$ for all diagrams $H$. By the commutativity axiom (\Cref{def:strong-Heegaard} \Cref{item:strong-commute}), 
    \begin{align*}
        F_{\delta(H) \to H}\circ F(\delta_H) = F_{H_0 \to H} \circ (F_{\delta(H_0) \to H_0}\circ F(\delta_{H_0}))\circ F_{H \to H_0}
    \end{align*}
    for $H_0$ any other diagram. It follows that $F_{\delta(H) \to H}\circ F(\delta_H) = \id_{F(H)}$ if and only if $F_{\delta(H_0) \to H_0}\circ F(\delta_{H_0}) = \id_{F(H_0)}$, and therefore it suffices to prove the claim for any particular diagram $H_0$. 

    Let $H_0$ be any diagram such that $\delta|_{\S_0}$ is supported on a small disk $D \sub \S_0$ which is disjoint from all attaching curves. Apply $\delta$ and consider the surface $\delta(D)$. The generators of $\CFR^\circ(\delta(H_0))$ and $\CFR^\circ(H_0)$ are identical. \Cref{fig:twist_diffeo} shows cross-sections of $\delta(\S)$ near $p$; we first perform a $\Z/2$-stabilization as shown in the transition between the first and second columns of the diagram. The remaining columns show an isotopy of the stabilization of $\delta(\S_1)$ to a surface which is obtained from $\S_0$ by a $\Z/2$-stabilization. This sequence of real Heegaard moves taking $\delta(H_0)$ to $H_0$ consists entirely of (de)stabilizations and diffeomorphisms isotopic to the identity. Hence, the composition $F_{\delta(H_0)} \circ F(\delta_{H_0}) = \id_{F(H_0)}$ as claimed.
\end{proof}

\begin{figure}[h]
\def\svgwidth{1\linewidth}
\begingroup%
  \makeatletter%
  \providecommand\color[2][]{%
    \errmessage{(Inkscape) Color is used for the text in Inkscape, but the package 'color.sty' is not loaded}%
    \renewcommand\color[2][]{}%
  }%
  \providecommand\transparent[1]{%
    \errmessage{(Inkscape) Transparency is used (non-zero) for the text in Inkscape, but the package 'transparent.sty' is not loaded}%
    \renewcommand\transparent[1]{}%
  }%
  \providecommand\rotatebox[2]{#2}%
  \newcommand*\fsize{\dimexpr\f@size pt\relax}%
  \newcommand*\lineheight[1]{\fontsize{\fsize}{#1\fsize}\selectfont}%
  \ifx\svgwidth\undefined%
    \setlength{\unitlength}{1000.62992126bp}%
    \ifx\svgscale\undefined%
      \relax%
    \else%
      \setlength{\unitlength}{\unitlength * \real{\svgscale}}%
    \fi%
  \else%
    \setlength{\unitlength}{\svgwidth}%
  \fi%
  \global\let\svgwidth\undefined%
  \global\let\svgscale\undefined%
  \makeatother%
  \begin{picture}(1,0.6203966)%
    \lineheight{1}%
    \setlength\tabcolsep{0pt}%
    \put(0,0){\includegraphics[width=\unitlength,page=1]{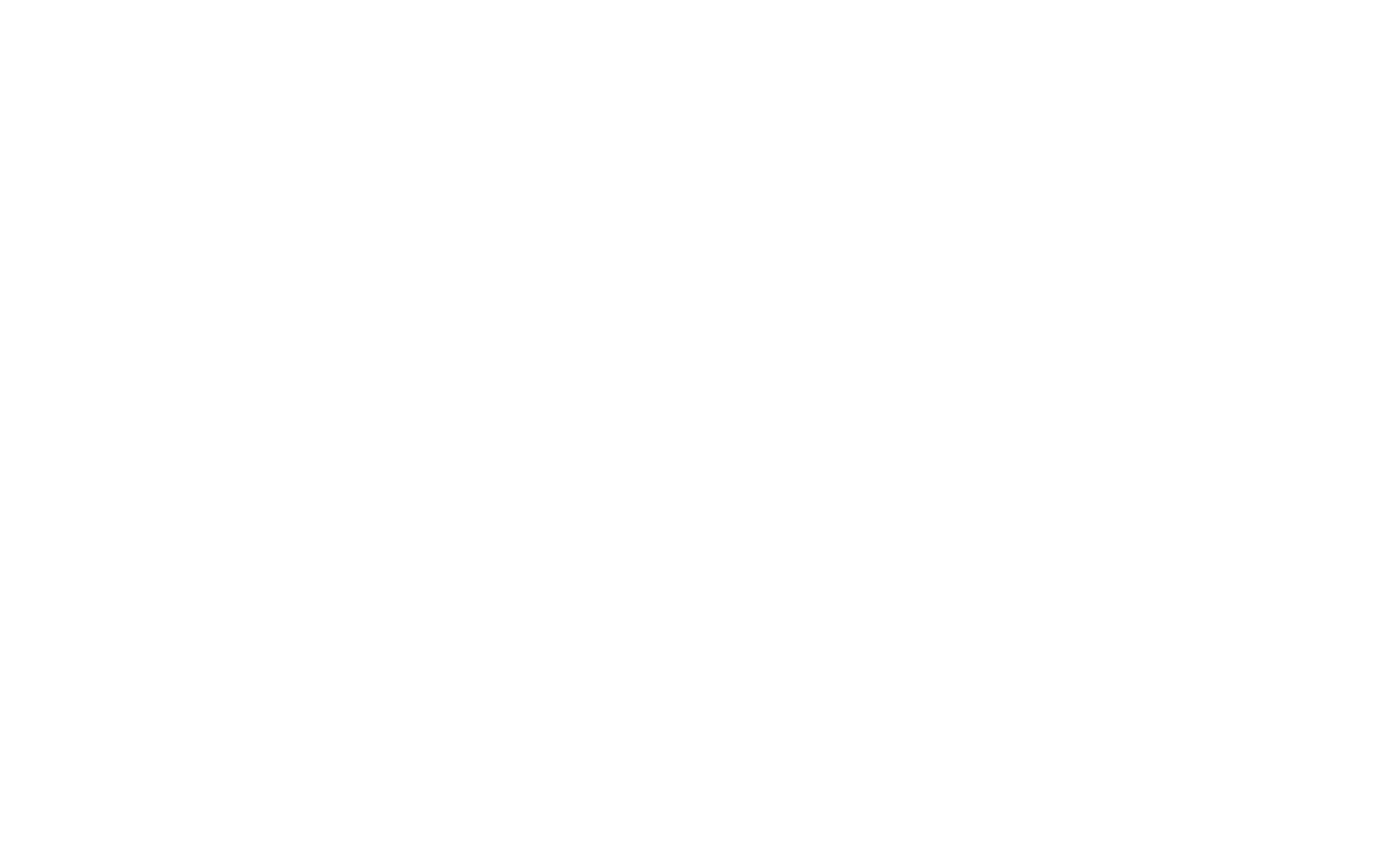}}%
    \put(0.00977133,0.33632193){\color[rgb]{0,0,0}\makebox(0,0)[t]{\smash{\begin{tabular}[t]{c}{$x$}\end{tabular}}}}%
    \put(0.12737573,0.40094571){\color[rgb]{0,0,0}\makebox(0,0)[t]{\smash{\begin{tabular}[t]{c}{\tiny$p$}\end{tabular}}}}%
    \put(0.51618966,0.00989643){\color[rgb]{0,0,0}\makebox(0,0)[t]{\smash{\begin{tabular}[t]{c}{$t$}\end{tabular}}}}%
    \put(0,0){\includegraphics[width=\unitlength,page=2]{twist_diffeo.pdf}}%
    \put(0.32365035,0.40094571){\color[rgb]{0,0,0}\makebox(0,0)[t]{\smash{\begin{tabular}[t]{c}{\tiny$p$}\end{tabular}}}}%
    \put(0.52059483,0.40094571){\color[rgb]{0,0,0}\makebox(0,0)[t]{\smash{\begin{tabular}[t]{c}{\tiny$p$}\end{tabular}}}}%
    \put(0.71686941,0.40094571){\color[rgb]{0,0,0}\makebox(0,0)[t]{\smash{\begin{tabular}[t]{c}{\tiny$p$}\end{tabular}}}}%
    \put(0.91381386,0.40094571){\color[rgb]{0,0,0}\makebox(0,0)[t]{\smash{\begin{tabular}[t]{c}{\tiny$p$}\end{tabular}}}}%
  \end{picture}%
\endgroup%

    \caption{Here is the equivariant basepoint twist displayed in the coordinates described in \Cref{def:twist}; the vertical direction displays slices of $\S_0$ along the $x$-axis of the image of $\delta(\S_0)$. The involution is given by rotation about the axis normal to the page. Horizontal slices show a sequence of real Heegaard moves relating $\delta(\S_0)$ to $\S_0$. The gray lines are used to keep track of the attaching regions of the bands amid the isotopies. In the transition between the first and second columns we  performed a stabilization, which we  indicated by the four bands. (Each pair of bands makes up a tube.) The remaining columns depict an isotopy from a stabilization of $\delta(\S_0)$ to a $\Z/2$-stabilization of $\S_0$.}
\label{fig:twist_diffeo}
\end{figure}

We are at last situated to prove \Cref{thm:real-nat}.

\begin{proof}[Proof of \Cref{thm:real-nat}]
Let $\CFR^\circ$ be any version of real Heegaard Floer homology. Since $\CFR^\circ$ is a strong real Heegaard invariant, it determines a transitive system. It remains to show that elements of $\mathrm{Diff}^\t_0(Y, \t, p)$ act by the identity. 

Let $\xi_0, \xi_1 \in \Gr_2^R(T_pY)$ be two framings of $p$. Let $\phi \in \mathrm{Diff}^\t_0(Y, \t, p)$ such that $d \phi_p(\xi_0) = \xi_1$. According to \Cref{prop:framed-HFR-nat}, there is an induced map
\begin{align*}
    \iota_{\xi_0 \to \xi_1}:=\CFR^\circ(\phi): \CFR^\circ(Y, \t,p,\xi_0) \to \CFR^\circ(Y, \t,p,\xi_1).
\end{align*}
This map is independent of $\phi$. Indeed, if $\psi$ is another such map, we may choose an equivariant isotopy $\Phi = \{\Phi_s\}_{s \in [0,1]}$ in $\mathrm{Diff}^\t_0(Y, \t, p)$ from $\phi$ to $\psi$ (these exist by the definition of $\mathrm{Diff}^\t_0(Y, \t, p)$). For $w = w(\Phi)$, Lemma~\ref{lem:rel} shows that $\phi$ and $\delta^{-w}\circ \psi$ are isotopic as morphisms $(Y,\tau,p,\xi_0) \to (Y,\tau,p,\xi_1)$. Hence, by \Cref{prop:framed-HFR-nat} and functoriality,
\[
\CFR^\circ(\phi)
= \CFR^\circ(\delta^{-w}\circ \psi)
= \CFR^\circ(\delta)^{-w} \circ \CFR^\circ(\psi)
= \CFR^\circ(\psi),
\]
where the last equality is by Theorem~\ref{thm:bspt-twist-trivial}. It follows that $$(\{\CFR^\circ(Y, \t, p, \xi)\}_{\xi \in \Gr^R(T_pY)}, \{\iota_{\xi_0\to\xi_1}\}_{\xi_0,\xi_1 \in \Gr^R(T_pY)})$$ form a transitive system, which we define to be $\CFR^\circ(Y, \t, p)$. By passing to homology and taking the colimit as in \cite[Remark 1.2]{JTZ_naturality_mapping_class_groups}, we obtain the desired functor
\begin{align*}
    \HFR^\circ: \RMan_* \to \Fmod.
\end{align*}
This completes the proof.
\end{proof}
\begin{rem}
    The case of multiple basepoints  follows similarly. The only subtlety is that in the case that $|\w| > 1$, we must work in the category of curved complexes. Therefore, the invariant takes the form of a functor 
    \begin{align*}
        \CFR^\circ: \RMan_{*} \to \mathrm{Trans}_{\cG^R_{(Y, \t, \w)}}(\cC)
    \end{align*}
    where $\cC$ is the homotopy category of chain complexes over the ring $\F[U_\w]$. The curvature is given by $\partial^2 = \sum_{i=1}^n U_i^2,$ and in these cases, we may instead work over the quotient ring $\F[U_1, \hdots, U_n]/(U_1^2 + \hdots U_n^2)$ and then pass to homology and take the colimit.
\end{rem}

\begin{proof}[Proof of \Cref{thm:links-nat}]
    By \Cref{thm:transitive_system}, $\CFLR^\circ$ defines a functor from the category $\cS^R_{link}$ to the category of transitive systems over the homotopy category of curved chain complexes over $\F[u_\z, U_\w]$. By composing with the functor $(Y, L, \t, \w, \z) \mapsto (Y(L,\w, \z), \t)$, we may assign invariants to multi-based generalized strongly invertible  links. By quotienting out the curvature and then passing to homology, we obtain the desired invariants $\HFLR^\circ$.
\end{proof}

\section{Singularities of real functions}\label{sec:Singularities}
In this section, we recall some basic singularity theory and its equivariant version. We will follow the exposition of \cite[Section 4]{JTZ_naturality_mapping_class_groups}, as well as \cite{bao2025morsehomologyequivariance,borodzik2025familiesmorsefunctionsmanifolds,siersma,golubitsky1985singularities,wall_equivariant_jets}

Throughout, $\Yt$ is a real manifold with fixed set $C$ of codimension 2. We allow $Y$ to have boundary, but this will not play a significant role. 

\begin{defn}
    A \emph{real function} on $(Y, \tau)$ is a smooth function $f: Y \ra [-1,1]$ satisfying $f\circ \tau = -f$. Let $C^\infty_R(Y,\tau)$ denote the space of real functions on $(Y, \tau)$.
\end{defn}

\begin{defn}
    Let $f$ be a real function on $\Yt$. We say that an orbit $\Z/2 \cdot p \subset Y$ is a \emph{critical orbit} if $df_{g\cdot p} = 0$ for $g \in \Z/2$. A critical orbit $\Z/2\cdot p$ is \emph{non-degenerate} if the Hessian of $f$ is non-singular at $p$ (and therefore is non-degenerate at each $g\cdot p$). The \emph{index} of a critical point $p$, which we denote $\cI(p)$, is the number of negative eigenvalues of the Hessian at $p$. 
    
    We note that in the real setting, the index is not preserved by the group action: rather $\cI(\tau(p)) = 3 - \cI(p)$.
\end{defn}

Critical points of real functions come in either $(\Z/2)/\{1\}$- or $(\Z/2)/(\Z/2)$-orbits (which we henceforth write as $\Z/2$- or $\{1\}$-orbits). Since critical points are generically  isolated, analyzing critical points appearing in $\Z/2$-orbits can be reduced immediately to the non-equivariant setting. Critical points which are fixed by the involution are more subtle and require a careful analysis. 

For a real 3-manifold with codimension-2 fixed point set, there are two relevant local models. Near a point which is fixed by the $\Z/2$-action, there is a neighborhood equivalent to $(\R^3, \tau)$ where $\t(x,y,z) = (y,x,-z)$. For points on which $\Z/2$ acts freely, the model is simply $(\R^3, \id)$.

\begin{defn}
    Let $\cE_3(\tau)$ be the space of germs at 0 of $\Z/2$-equivariant functions $f: (\R^3, \tau) \ra (\R, -\id)$. We call elements of $\cE_3(\tau)$ \emph{real germs}. Let $\cE_3(\id)$ be the space of germs at 0 of $\Z/2$-equivariant functions $f: (\R^3, \id) \ra (\R, \id)$. Let $\cD_3(\Z/2)$ be the group of germs of $\Z/2$-equivariant diffeomorphisms $g: (\R^3, 0) \ra (\R^3, 0)$. The group $\cD_3(\Z/2)$ acts on $\cE_3(\rho)$ by $g\cdot f:= f \circ g^{-1}$ for $\rho \in \{\id, \tau\}$. Two real germs $f$ and $f'$ are \emph{equivalent} if they lie in the same $\cD_3(\Z/2)$-orbit. 
    
    In our setting, we will sometimes refer to equivalence classes of germs of $\cE_3(\tau)$ as \emph{$\{1\}$-orbit singularities} and equivalence classes of germs of $\cE_3(\id)$ as \emph{$\Z/2$-orbit singularities}. 
\end{defn}

In the following subsections we classify singularities of real functions of codimension up to 2 in the space of real functions. 

\subsection{$\Z/2$-singularities}

\begin{defn}
    A real function $f: Y \ra \R$ is a \emph{real Morse function} if all of its critical orbits are nondegenerate.
\end{defn}

If $p$ is a nondegenerate point of index $k$, the condition $f\circ \tau = -f$ forces $\tau(p)$ to be a critical point of index $3-k$. In particular, the critical points of real Morse functions always come in $\Z/2$-orbits. Therefore, the usual Morse lemma implies the following. 

\begin{lemma}[Morse Lemma for $\Z/2$-critical orbits]
    Let $f: Y \ra \R$ be a real function and let $\Z/2\cdot p$ be a  nondegenerate $\Z/2$-critical orbit. In a neighborhood of $p$ in $Y$, there is a coordinate system in which $f$ has the form 
    \begin{align*}
        f(x) = f(p) +\varepsilon_1 \cdot x_1^2 + \varepsilon_2 \cdot x_2^2 + \varepsilon_3\cdot x_3^2,
    \end{align*}
    where $\varepsilon_i \in \{-1,1\}.$ There is a symmetric coordinate system near $\tau(p)$. The class of such $\Z/2$-singularities is called $A_1(\Z/2)$.
\end{lemma}

\begin{lemma}
    Real Morse functions are generic in $C_R^\infty\Yt$. 
\end{lemma}
\begin{proof}
    This follows as in \cite{bao2025morsehomologyequivariance}, so we only sketch the proof.  Let $f \in C_R^\infty \Yt$. We claim that any such $f$ can be made Morse by a small (equivariant) perturbation. In a neighborhood of a point $p$ of $C$, we may choose coordinates in which $\tau(x,y,z) = (y,x,-z)$. The space of real linear functionals, $\cA = \{\phi \in (\R^3)^*: \phi \circ \tau = - \phi \}$ is two dimensional. A standard argument shows that $f|_{\nu(p)} + A$ is Morse for a generic $A\in \cA$. By covering $C$ by finitely many balls and using a partition of unity, we may perturb $f$ so that it is a real Morse function in a neighborhood of $C$. Furthermore, it is clear that we can ensure that the perturbed function is close to $f$. 

    We then cover $Y \smallsetminus \nu(C)$ by pairs of disjoint balls  $B_i \cup \tau(B_i)$. Each $f|_{B_i}$ can be made Morse by a small perturbation; we apply the symmetric perturbation to $f|_{\tau(B_i)}$. Again, these functions can be stitched together with a partition of unity. The resulting function can be made arbitrarily close to $f$ and is real Morse.
\end{proof}

We turn now to 1- and 2-parameter families of real functions. 

\begin{thm}\label{thm:real singularities Z/2 orbit}
    In a generic 1-parameter family of smooth functions, the only degenerate $\Z/2$-critical orbits $\Z/2 \cdot p$ that appear have normal form 
    \begin{align*}
        f(x) = f(p) +\varepsilon_1 \cdot x_1^2 + \varepsilon_2 \cdot x_2^2 + x_3^3,
    \end{align*}
    where $\varepsilon_i \in \{-1,1\}.$ We refer to these as $A_2(\Z/2)$ singularities.
    
    In generic 2-parameter families of smooth functions, there also appear degenerate $\Z/2$-critical orbits $\Z/2 \cdot p$ of the form 
    \begin{align*}
        f(x) = f(p) +\varepsilon_1 \cdot x_1^2 + \varepsilon_2 \cdot x_2^2 + \varepsilon_3 \cdot x_3^4,
    \end{align*}
    where $\varepsilon_i \in \{-1,1\}.$ We refer to these as $A_3^\pm(\Z/2)$ singularities (depending on the sign of $x_3^4$).
\end{thm}
\begin{proof}
    Singularities appearing in $\Z/2$-orbits behave exactly as individual singularities in the non-equivariant setting. See \cite{siersma} for an elementary proof.
\end{proof}

\subsection{Fixed set singularities}
We now turn to the classification of fixed set singularities up to codimension 2. We note that $A_1$- and $A_3^\pm$-singularities cannot appear in $\{1\}$-orbits. 

Let $f$ be a real function with a critical point at the origin. Recall that the \emph{corank} of $f$ is the dimension of the kernel of the Hessian of $f$. In the non-equivariant setting, the space of functions with corank greater than 1 has codimension greater than two in $C^\infty(\R^n)$. The same is true in the real setting. For simplicity, we specialize to dimension 3.

\begin{lemma}\label{lem:hessian_corank}
    Let $f \in C_R^\infty(\R^3)$ have a critical point at the origin (which is fixed by the group action). Then, the Hessian $H_f(0)$ has corank 1 or 3. Furthermore, the space of $f \in C_R^\infty(\R^3)$ with a critical point at the origin with Hessian of corank 3 has codimension 4. 
\end{lemma}
\begin{proof}

    Choose coordinates near zero so that $\tau(x_1,x_2,x_3) = (x_2,x_1,-x_3).$ Differentiating the condition 
    \begin{align*}
        f \circ \tau = - f,
    \end{align*}
    twice at fixed points yields the relation  
    \begin{align*}
        (d\tau)^\perp \circ H_f \circ (d\tau) = - H_f. 
    \end{align*}
    In our chosen coordinates, we are therefore interested in matrices satisfying
    \begin{align*}
        \begin{bmatrix}
            d & b & -e \\
            b & a & -c \\
            -e & -c & f
        \end{bmatrix}
        = -\begin{bmatrix}
            a & b & c \\
            b & d & e \\
            c & e & f
        \end{bmatrix},
    \end{align*}
    i.e. matrices of the form 
    \begin{align*}
        \begin{bmatrix}
            a & 0 & c \\
            0 & -a & c \\
            c & c & 0
        \end{bmatrix}.
    \end{align*}
    These matrices have eigenvalues $\{0, \pm\sqrt{a^2 + 2 c^2}\}$, and therefore have corank either 1 or 3. In particular, these matrices are always singular. 

    Any real function must satisfy $f_{x_1} + f_{x_2} =0$. Having a critical point on the fixed set is thus specified by the condition $f_{x_1}=0 = f_{x_3}$, which is a codimension 2 constraint. The condition that $a = 0 = c$ (i.e. that the Hessian is corank 3) gives two further constraints. This completes the proof.
\end{proof}

In view of Lemma~\ref{lem:hessian_corank}, if we consider the space of odd functions having a critical point  \emph{somewhere} along a 1-dimensional fixed point set, its codimension is $4-1=3$. In particular, critical points with Hessians of corank 3 do not appear in generic 1- or 2-parameter families of real functions.

We now prove an equivariant splitting lemma, and establish the normal forms for the singularities we will encounter. 

\begin{prop}\label{prop:standard_form_real_sing}
    Let $f: (\R^3, 0) \ra (\R, 0)$ be a real germ with a critical point at the origin whose  Hessian has corank 1. Then, 
    \begin{enumerate}
        \item there is an equivariant change of coordinates $(x_1, x_2, x_3) \mapsto (\hat x_1, \hat x_2, \hat x_3)$ and a smooth, odd germ $h: (\R,0) \ra (\R, 0)$ such that 
        \begin{align*}
            f = \hat{x}_1^2 - \hat x_2^2 + h(\hat{x}_3);
        \end{align*}
        \item further, if $h$ has finite order $2k+1$ at the origin\footnote{As $h$ is odd, its order is necessarily odd as well.}, then the coordinates may be chosen so that 
        \begin{align*}
            f = \hat{x}_1^2 - \hat x_2^2 + \hat x_3^{2k+1}.
        \end{align*}
    \end{enumerate}
\end{prop}

The normal form stated in \Cref{prop:standard_form_real_sing} can be obtained formally, just as in \cite[Splitting Lemma]{siersma}. The non-equivariant statement then follows from the fact that any germ of finite codimension is finitely determined (see  \cite{Wassermann_unfolding}, for instance). Rather than formulate and prove an equivariant version of finite determinacy, we give an elementary proof of Proposition~\ref{prop:standard_form_real_sing}. This is modeled on the proof of the equivariant Morse lemma; cf. \cite[Section 2]{Arnold} and \cite[Lemma 4.1]{WassermanET}.

We make use of the following well-known fact, for which we refer to \cite[Theorem 1.29]{Higham}. 

\begin{lemma}[Principal square root]\label{lem:sqrt}
There is a neighborhood $\mathcal U$ of the identity in $\mathrm{GL}_n(\R)$ and a
smooth map $\sqrt{\,\cdot\,}\colon\mathcal U\to\mathrm{GL}_n(\R)$ with
$\bigl(\sqrt N\bigr)^2=N$ and $\sqrt I=I$, characterized as the \emph{unique} square
root of $N$ whose spectrum lies in $\{\operatorname{Re}>0\}$. It satisfies:
\[
  \sqrt{\,PNP^{-1}\,}=P\,\bigl(\sqrt N\bigr)\,P^{-1}
  \qquad\text{whenever } N,\,PNP^{-1}\in\mathcal U .
\]
\end{lemma}

\begin{proof}[Proof of \Cref{prop:standard_form_real_sing}] 
    It will be more convenient to use coordinates $(u,v,z) = (x_1 +x_2,x_1-x_2,x_3)$ so that $\tau(u, v, z) = (u, -v, -z).$   Let $W=\R^2 \times 0 \subset \R^3$ and write $w=(u, v)$ for a typical element of $W$.

    In these coordinates, the Hessian is given by
    $$H = \begin{bmatrix}
        0&b & c\\b&0 & 0 \\ c& 0 & 0
    \end{bmatrix}.$$
    
    After a linear change of variables, we may assume that $b=1$ and $c=0$. Then, the kernel of the Hessian is the $z$-axis and its restriction to $W$ is the hyperbolic form $$H = \begin{bmatrix}
        0&1\\1&0
    \end{bmatrix}.$$ Therefore, the 2-jet of $f$ is $uv$.
    
   Let 
    $$\sigma: W \to W, \ \sigma(u,v):=(u,-v),$$ so that $\t(w, z) = (\sigma(w), -z).$ 
    
    Define also $A_0 = \frac{1}{2}H$, so that $w^TA_0w = uv = x^2 - y^2$ and $\sigma A_0 \sigma = -A_0$.

    Let $G(w,z)$ be the gradient of $f$ is the $w$-directions, $$G(w, z) := \nabla_w f(w, z) \in \R^2.$$ Note that since $f$ has a critical point at the origin we have $G(0,0) = 0$.  Also, $D_w G(0,0) = D^2_wf(0) = H$ is invertible, so the inverse function theorem supplies a unique smooth function $z \mapsto \xi(z)$ with $\xi(0) = 0$ and 
    \begin{equation}\label{eqn:G-zeros}
        G(\xi(z), z) = 0
    \end{equation}
    for $z$ near $0$. The symmetry of $f$ yields the following symmetry of $G$:
    \begin{align*}
        G \circ \t(w, z) = -\sigma \circ G(w,z).
    \end{align*}
    It follows that $z \mapsto\sigma(\xi(-z))$ also solves \eqref{eqn:G-zeros} and vanishes at the origin; hence the uniqueness of $\xi$ forces the relation 
   \[\xi(-z) = \sigma\, \xi(z).\]

    After the equivariant change of coordinates $(w, z) \mapsto (w - \xi(z), z)$, we may assume that $G(0, z)=0$ for all $z$. Let us keep the symbol $f$ for the function on the new coordinates.  
    
    We define
    \begin{equation}\label{eqn:define-h}
        h(z) := f(0, z).
    \end{equation}
Observe that $h$ is odd, $h(-z) = -h(z)$.

    Consider the germ $\Phi(w, z) := f(w,z) - h(z)$. By \eqref{eqn:define-h}, we have $\Phi(0,z) = 0$ and $\nabla_w \Phi(0,z) = 0$ for all $z$. Taylor's theorem with remainder  in the $w$-variables yields 
    \begin{align*}
        \Phi(w,z) = w^T A(w,z) w, \quad A(w, z) = \int_0^1 (1-s)D^2_w\Phi(sw,z) d s,
    \end{align*}
    which is a smooth family of symmetric matrices. Furthermore, $A(0,0) = \frac{1}{2}H$ which is precisely $A_0$. For points $p$ in a sufficiently small neighborhood of $(0,0)$, the matrix $$M(p) = A_0^{-1}A(p)$$ is defined and lies in 
    the neighborhood $\mathcal U$ from Lemma~\ref{lem:sqrt}. Let $C(p) = \sqrt{M(p)}$.   The symmetry of $f$ implies that:
    \begin{align}\label{eqn:C-symmetry}
        A(\t p) = - \sigma A(p) \sigma, \quad M(\t p) = \sigma M(p) \sigma, \quad C(\t p) = \sigma C(p) \sigma.
    \end{align}
    Here, the relation for $C(p)$ is a consequence of the uniqueness claim of \Cref{lem:sqrt}.

    Notice that the matrix $M=M(p)$ is self-adjoint with respect to $A_0$:
    \begin{align*}
        A_0^{-1}M^TA_0 =  A_0^{-1}(A_0^{-1}A)^TA_0 = A_0^{-1}A = M.
    \end{align*}
    Since both $C=C(p)$ and its adjoint are square roots of $M(p)$ with spectrum in $\{\operatorname{Re}>0\}$, the uniqueness part of \Cref{lem:sqrt} implies that $C$ is self-adjoint. Therefore, 
    \begin{align}\label{eqn: C-conjugate-A}
        C^T A_0 C = A.
    \end{align}

    Finally, define $\hat w = (\hat u, \hat v) =  C(w,z) w$. One can easily verify that this is a diffeomorphism germ. Further, it is equivariant by \eqref{eqn:C-symmetry}, and $\t (\hat w, z) = (\hat u, - \hat v, -z)$. Putting everything together, 
    \begin{align*}
        f(w,z)
        &= h(z)+w^TA(w,z)w \\
        &= h(z)+w^TC(w,z)^TA_0C(w,z)w  \\
        &= h(z)+\widehat w^TA_0\widehat w\\
        &= h(z) + \hat{u}\hat{v}.
    \end{align*}
   
    Taking $\hat{u} \mapsto \hat{x}_1 + \hat{x}_2$ and $\hat{v} \mapsto \hat{x}_1 - \hat{x}_2$ proves part (1) of the proposition. 

    Part (2) is standard. If $h$ has finite order $2k+1$, we can write $h(z) = z^{2k+1} \rho(z)$ with $\rho(0) \neq 0$. Finally, we define $\hat z = z\rho(z)^{1/(2k+1)}$, which yields the desired normal form. 
\end{proof}

Following \cite{golubitsky1985singularities}, we will write $\cE_x$ for the ring of smooth germs in a single variable $x$. As noted in \cite[Chapter VI, Section 2]{golubitsky1985singularities}, the space of odd germs $\cE_x(\Z/2)$ is simply
\begin{align*}
    \cE_x(\Z/2) = \cE_{u}\cdot \{x\},
\end{align*}
where $u = x^2$. Note that $\cE_x(\Z/2)$ is not a ring (the product of two odd germs is not odd) but rather a module over $\cE_u$, the ring of even germs.

Let $f(x)$ be an element of $\cE_x(\Z/2)$. Write $f(x) = r(u)x$. The \emph{tangent space at $f$} is the space
\begin{align*}
    T(f, \Z/2):= \langle r, u\, \partial_u r(u) \rangle \cdot \{x\},
\end{align*}
according to \cite[Chapter VI, Equation 3.2]{golubitsky1985singularities}. The codimension of $f$ agrees with the codimension of $T(f, \Z/2)$ in $\cE_x(\Z/2)$. It is clear that $T(x^{2k+1})$ has codimension $k$ in $\cE_x(\Z/2)$, as 
\begin{align}\label{eqn:xk-z2-codimenion}
    \dim_\R(\cE_x(\Z/2)/T(x^{2k+1}, \Z/2)) = \dim_\R((\R[x^2] \cdot \{x\})/(x^{2k+1})).
\end{align}
    
Combining the previous results yields the following classification result. 

\begin{thm}\label{thm:real singularities trivial orbit}
    In a generic 1-parameter family of smooth functions, the only degenerate $\{1\}$-critical orbits that appear are isolated, index 1 birth-death singularities and have normal form 
    \begin{align*}
        f(x) = f(p) + x_1^2 - x_2^2 + x_3^3, \quad \tau(x_1,x_2,x_3) = (x_2,x_1,-x_3),
    \end{align*}
    which we refer to as $A_2(\{1\})$-singularities.
    
    In generic 2-parameter families of smooth functions, there also appear isolated singularities of the form 
    \begin{align*}
        f(x) = f(p) + x_1^2 - x_2^2 + x_3^5, \quad \tau(x_1,x_2,x_3) = (x_2,x_1,-x_3),
    \end{align*}
    which we refer to as $A_4(\{1\})$ singularities.
\end{thm}

We now analyze the behavior question of real functions near these types of singularities.

\begin{defn}
    A \emph{real deformation} of $f \in \cE_n(\tau)$ with base $\Lambda = \R^k$ is a smooth map 
    \begin{align*}
        F: \R^n \times \R^k \ra \R
    \end{align*}
    such that $F(x, 0) = f(x)$ and $F(\tau(x), \lambda) = - F(x, \lambda).$

    A real deformation $F(x, \lambda)$ of $f$ is called \emph{versal} if every other real deformation $F': \R^n \times \R^\ell \ra \R$ of $f$ can be represented in the form 
    \begin{align*}
        F'(x,\lambda') = F(g(x, \lambda'), \theta(\lambda')),
    \end{align*}
    where $g$ is a germ at zero of an $\ell$-parameter family of diffeomorphisms $\R^n \times \R^\ell \ra \R^n$ so $g(x,0) = x$ and $g(\tau(x), \lambda) = -g(x,\lambda)$ and $\theta: (\R^\ell, 0) \ra (\R^k, 0)$ satisfies $\theta(0) = 0$. 

    A versal (real) deformation whose base is of minimal dimension is called \emph{miniversal}. There is a useful criterion for determining whether a deformation is miniversal.
\end{defn}

\begin{thm}[Theorem 3.3 in \cite{golubitsky1985singularities}]\label{thm:miniversal-deformations}  
    Let $g(x) = R(u)x$, where $u = x^2$, be a germ in $\cE(\Z/2)$ and let $G(x, \lambda_1, \hdots,  \lambda_k) = R(x, \lambda_1, \hdots \lambda_k)x$ be a $k$-parameter real deformation of $g$. Then $G$ is a miniversal real deformation if and only if the codimension of $T(g, \Z/2)$ in $\cE(\Z/2)$ is $k$ and 
    \begin{align*}
        \cE(\Z/2) = T(g,\Z/2) \oplus \R \left \{\frac{\partial R}{\partial \lambda_1}, \hdots, \frac{\partial R}{\partial \lambda_k}  \right\}\cdot x.
    \end{align*}
\end{thm}

Miniversal real deformations of $A_2(\{1\})$- and $A_4(\{1\})$-singularities have the following forms. 

\begin{cor}
    The $\{1\}$-orbit singularities $A_2$ and $A_4$ have miniversal deformations
    \begin{align*}
        f(x_1, x_2, x_3) = x_1^2 - x_2^2 + \lambda_1 x_3 + x_3^3
    \end{align*}
    and 
    \begin{align*}
        f(x_1, x_2, x_3) = x_1^2 - x_2^2 + \lambda_1 x_3 + \lambda_2 x_3^3 + x_3^5,
    \end{align*}
    respectively. 
\end{cor}
\begin{proof}
    By \Cref{prop:standard_form_real_sing}, it suffices to consider unfoldings of the odd, one-variable germ $g(x_3) = x_3^{2k+1}$. 
    As noted above, its tangent space is 
    \begin{align*}
        T(g,\Z/2) = \langle u^k \rangle \cdot \{x_3\},
    \end{align*}
    which clearly has codimension $k$ in $\cE(\Z/2)$. Indeed, its complement is given by 
    \begin{align*}
        \mathrm{span}_\R\left( \frac{\partial R}{\partial \lambda_1}, \hdots, \frac{\partial R}{\partial \lambda_{k-1}}\right)\cdot x = \mathrm{span}_\R\left( u^{k-1}, \hdots , u, 1\right) \cdot x.
    \end{align*}
    It follows that $G(x, \lambda_0, \hdots, \lambda_{k-1}) = (u^{k} + \lambda_{k-1} u^{k-1} + \hdots + \lambda_1 u + \lambda_0)\cdot x$ is a miniversal real deformation. 
\end{proof}

In particular, we can explicitly describe the bifurcation diagrams of these singularities. The diagrams for $A_3^\pm$-singularities are well-known, and in our situation are the same, as these singularities appear in $\Z/2$-orbits; see \Cref{fig:a3_a4_bifurcations}. Let us briefly describe the bifurcation diagram of an $A_4(\{1\})$-singularity.

\begin{figure}[h]
\def\svgwidth{.8\linewidth}
\begingroup%
  \makeatletter%
  \providecommand\color[2][]{%
    \errmessage{(Inkscape) Color is used for the text in Inkscape, but the package 'color.sty' is not loaded}%
    \renewcommand\color[2][]{}%
  }%
  \providecommand\transparent[1]{%
    \errmessage{(Inkscape) Transparency is used (non-zero) for the text in Inkscape, but the package 'transparent.sty' is not loaded}%
    \renewcommand\transparent[1]{}%
  }%
  \providecommand\rotatebox[2]{#2}%
  \newcommand*\fsize{\dimexpr\f@size pt\relax}%
  \newcommand*\lineheight[1]{\fontsize{\fsize}{#1\fsize}\selectfont}%
  \ifx\svgwidth\undefined%
    \setlength{\unitlength}{836.22047244bp}%
    \ifx\svgscale\undefined%
      \relax%
    \else%
      \setlength{\unitlength}{\unitlength * \real{\svgscale}}%
    \fi%
  \else%
    \setlength{\unitlength}{\svgwidth}%
  \fi%
  \global\let\svgwidth\undefined%
  \global\let\svgscale\undefined%
  \makeatother%
  \begin{picture}(1,0.71186441)%
    \lineheight{1}%
    \setlength\tabcolsep{0pt}%
    \put(0,0){\includegraphics[width=\unitlength,page=1]{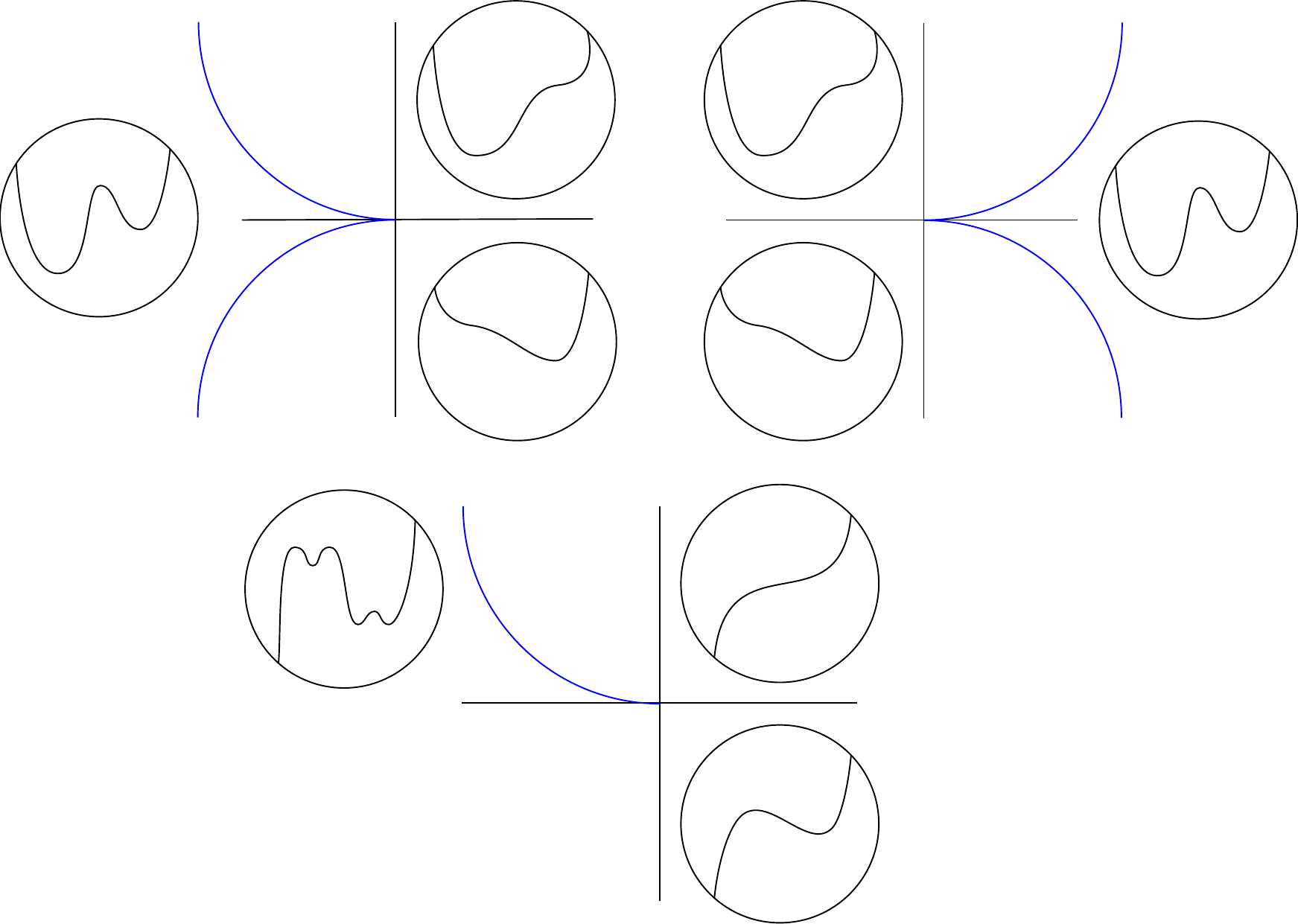}}%
    \put(0.45715001,0.13696434){\color[rgb]{0,0,0}\makebox(0,0)[lt]{\smash{\begin{tabular}[t]{l}{$A_4$}\end{tabular}}}}%
    \put(0.24333907,0.49271532){\color[rgb]{0,0,0}\makebox(0,0)[lt]{\smash{\begin{tabular}[t]{l}{$A_3^+$}\end{tabular}}}}%
    \put(0.74021793,0.49271532){\color[rgb]{0,0,0}\makebox(0,0)[lt]{\smash{\begin{tabular}[t]{l}{$A_3^-$}\end{tabular}}}}%
  \end{picture}%
\endgroup%

    \caption{$A_3^\pm(\Z/2)$ and $A_4(\{1\})$ bifurcation diagrams for $n = 1$.}
\label{fig:a3_a4_bifurcations}
\end{figure}
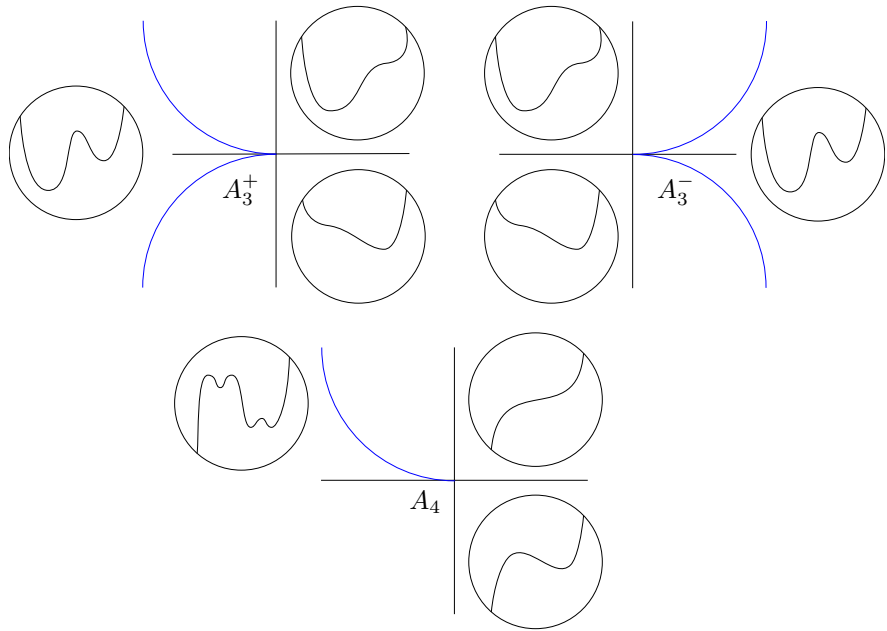

\begin{lemma}
    The bifurcation diagram of an $A_4(\{1\})$-singularity is as in the last frame of  \Cref{fig:a3_a4_bifurcations}.
\end{lemma}
\begin{proof}
    We consider the real roots of the polynomial $f(x) = 5x^4 + 3\lambda_1 x^2 + \lambda_2$.  When $9\lambda_1^2 < 20 \lambda_2$, the polynomial $g(y) = 5y^2 + 3\lambda_1 y + \lambda_2$ has no real roots. When $9\lambda_1^2 > 20 \lambda_2$, $g(y)$ has two real roots $a$ and $b$, and therefore $f(x) = g(x^2)$ has either zero, two, or four real roots, depending on the signs of $a$ and $b$. The product $ab$ is equal to $\lambda_2/5$. When $\lambda_2 < 0$, exactly one root is negative, so $f(x)$ has two real roots. When $\lambda_2 > 0$ the roots have the same sign. The sum $(a+b)$ is equal to $-3\lambda_1/5$; therefore, when $\lambda_1 > 0$ both roots are negative and $f(x)$ has no real roots and when $\lambda_1 < 0$ both roots are positive and $f(x)$ has four real roots. 
\end{proof}

In particular, while traversing a loop around an $A_4$-singularity, one witnesses a $\{1\}$-orbit $A_2$-singularity when the path crosses the positive $\lambda_1$ axis and a second as it crosses the negative $\lambda_1$ axis; then, as the path crosses the graph of $\lambda_2 = 9\lambda_1^2/20$, one witnesses a $\Z/2$-orbit $A_2$-singularity.

\section{Families of real gradients}\label{sec:Families of Real gradients}

We now turn to families of real gradients, following \cite[Section 5]{JTZ_naturality_mapping_class_groups}. See also Palis-Takens  \cite{palis_takens}, Carneiro-Palis \cite{carneiro_palis}, Vegter \cite{vegter}, and Bao-Lawson \cite{bao2025morsehomologyequivariance}. 

\begin{defn}
    An \emph{invariant manifold} of a vector field $v$ on $Y$ is a submanifold that is tangent to the vector field at each of its points. 
\end{defn}

Suppose that $p$ is a singularity of a vector field $v$ on $Y$. Then, $d v_p$ is an endomorphism of $T_pY$. The tangent space at $p$ decomposes into three $d v_p$-invariant subspaces, $T^s$, $T^u$, and $T^c$, with the property that every eigenvalue of $d v_p|_{T^s}$ has negative real part, $d v_p|_{T^u}$ has positive real part, and $d v_p|_{T^c}$ has zero real part. 

\begin{defn}
    A vector field $v$ has a \emph{hyperbolic singularity} at $p$ if none of the eigenvalues of $dv_p$ are purely imaginary, i.e. if $T^c = 0$.
\end{defn}

\begin{thm}[Center Manifold Theorem]
    Let $v$ be a $C^{r+1}$ vector field on $Y$ with a singular point at $p$. Then, $v$ has invariant manifolds $\cW^s$, $\cW^u$, and $\cW^c$ of class $C^{r+1}$, $C^{r+1}$, and $C^{r}$, respectively, that go through $p$ and are tangent at $p$ to $T^s$, $T^u$, and $T^c$,  respectively.
\end{thm}

The manifolds $\cW^s$, $\cW^u$, and $\cW^c$ are called the (strong) stable, (strong) unstable, and center manifolds of the singular point $p$. The center manifold is usually not unique. 

\begin{defn}
    A smooth vector field $v$ on $Y$ has a \emph{saddle-node} at $p$ (or a \emph{quasi-hyperbolic singularity of type 1}) if $\dim T^c = 1$ and $v|_{\cW^c}$ has the form $v = a x^2 \partial_x + O(|x|^3)$ with $a \neq 0$ for some local coordinate $x$ on some center manifold $\cW^c$ around $p$.

    If $v^{\overline{\mu}}$ belongs to a 1-parameter family of vector fields $\{v^\mu\}$ and has a saddle-node at $p$, we say that it \emph{unfolds generically} if there is a center manifold for $\{v^\nu\}$ passing through $p$ at $\mu = \overline{\mu}$ so that the restriction of $v^\mu$ has the form 
    \begin{align*}
        v^\mu = (ax^2 + b(\mu - \overline{\mu})) \partial_x + O(|x^3| + |x(\mu-\overline{\mu})| + |\mu-\overline{\mu}|^2),
    \end{align*}
    with $a, b \neq 0$.
\end{defn}

\begin{defn}
    Let $p$ be a singular point of a vector field $v$ on $Y$ with flow domain $D\sub Y \times \R$ and let $\varphi: D \ra Y$ be the maximal flow of $v$. The \emph{stable set} of $p$ is 
    \begin{align*}
        W^s(p) = \{y \in Y: \lim_{t \ra \infty}\varphi(y,t) = p\},
    \end{align*}
    and the \emph{unstable set} of $p$ is 
    \begin{align*}
        W^u(p) = \{y \in Y: \lim_{t \ra -\infty}\varphi(y,t) = p\}.
    \end{align*}
\end{defn}
When $p$ is a hyperbolic singularity of $v$ then both the stable and unstable sets are injectively immersed sumbanifolds of $Y$ whose tangent spaces agree with $T^s$ and $T^u$ respectively. If $p$ is a saddle-node, $W^{s}(p)$ and $W^u(p)$ are injectively immersed submanifolds with boundary, denoted $W^{ss}(p)$ and $W^{uu}(p)$ respectively, which we call the strong stable and strong unstable manifolds. The tangent spaces of $W^{ss}(p)$ and $W^{uu}(p)$ at $p$ are equal to $T^s$ and $T^u$.

\begin{defn}\label{def:dynamic-mfld}    
    Let $p^{\mu}$ be a hyperbolic singularity of a 1-parameter family of gradients $X^\mu = (f^\mu, g^\mu)$ with a 1-dimensional stable manifold $W^s(p^{\mu})$. Let us write $\gamma_1^\mu \cup \gamma_2^\mu$ for the two components of $W^s(p^{\mu}) \smallsetminus p^\mu$. (We choose these components consistently for different values of $\mu$, so that they form two smooth families.) For some $t_0\in \R$ and $\ep > 0$ small so that $\ep > f^\mu(p^\mu)-t_0 > 0$, consider the slice $(f^\mu)^{-1}(t_0)$. Let $q^\mu=q^\mu(t_0) \in \gamma_1^\mu\cap (f^\mu)^{-1}(t_0)$. Let
    $$v^\mu_0 = \frac{dq^{\mu}}{d\mu} \in  T_{q^\mu}Y$$ 
    be the velocity of the path traced out by $q^\mu$. Define $$v^\mu = \lim_{t \ra -\infty} \dfrac{(d\phi_t^\mu)(v_0^\mu)}{||(d\phi_t^\mu)(v_0^\mu)||} \in T_{p^\mu}Y$$ where $\phi_t^\mu$ is the flow of the gradient-like vector field for $X^\mu$. Consider the projection $T_{p^\mu}Y \to  T_{p^\mu}W^u(p^\mu)$ with kernel $T_{p^\mu}W^s(p^\mu)$. Let $w^\mu$ be the image of $v^\mu$ under this projection. Define the \emph{directional manifold of $p^\mu$ relative to $\gamma_1^\mu$}  be 
    \begin{align*}
        D(p^\mu,\gamma_1^\mu) = \left\{y \in W^u(p^\mu) : \lim_{t \ra \infty} \frac{d}{dt}\left(\frac{\phi(y, t)}{\|\phi(y, t)\|} \right) \in \mathrm{span}(w^\mu) \right\}.
    \end{align*}
    See Figure~\ref{fig:directional}. The directional manifold of $p^\mu$ relative to $\gamma_2^\mu$ is defined similarly.
\end{defn}

Said another way, in a 1-parameter family, the manifold $\gamma_1^{\mu} \subset W^s(p^{\mu})$ has some normalized velocity vector in each $t$-slice. We push this vector to $p^\mu$ via the flow of $X^\mu$ to obtain a vector $v^\mu$ in the tangent space at $p^\mu$. We then project to the unstable directions to get $w^\mu$. The directional manifold is the invariant manifold of $v$ tangent to $\mathrm{span}(w^\mu)$. It can be either zero- or one-dimensional. (In our examples, it will always be the latter.)

\begin{figure}[h]
\def\svgwidth{.8\linewidth}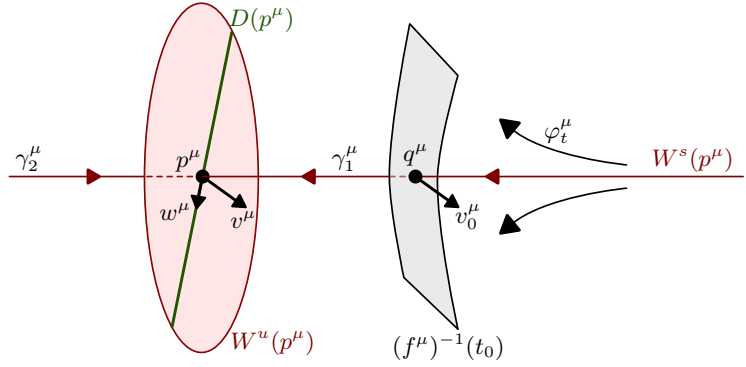
    \caption{The directional manifold of a hyperbolic critical point $p^\mu$, as well as the various vectors.}
\label{fig:directional}
\end{figure}

\begin{defn}
    A vector field $v$ on $Y$ is called \emph{real} if $d\tau(v) = -v$. 
\end{defn}

\subsection{Bifurcations of gradient-like vector fields on 3-manifolds}

Given a Riemannian manifold $(Y, g)$, a gradient vector field $X = \grad_g(f)$ is determined by a smooth function $f: Y \ra \R$ and $g$ according to the relation $g(X,Y) = df(Y)$. Gradient vector fields have no periodic orbits, and their derivative at critical points has only  real eigenvalues. We note that if $g$ is a $\tau$-invariant metric and $f$ is a real function on $Y$, then $X = \grad_g(f)$ is a real vector field:
\begin{align*}
    g(d\tau(X), Y) = g(X, d\tau(Y)) = df\circ d\tau(Y) = -df(Y).
\end{align*}

\begin{defn}
    A gradient vector field $X$ on $Y$ is \emph{Morse-Smale} if 
    \begin{enumerate}
        \item (Hyperbolicity) all singular points of $X$ are hyperbolic;
        \item (Transversality) all stable and unstable manifolds are transverse.
    \end{enumerate}
\end{defn}

Let $X^g(Y)$ be the set of all gradient vector fields on $Y$. It is well known that the set of Morse-Smale vector fields on $Y$ is open and dense in $X^g(Y)$. Let $X^g(Y,\tau)$ be the set of real gradient vector fields on $Y$. We will show that, as in the unreal case, real gradient-like vector fields are generically Morse-Smale.

Before turning to the proof, we introduce some notation and make several observations which will be useful later. 

\begin{defn}
    Suppose $A$ and $B$ are submanifolds of $Y$ which are invariant under the flow of $v$. If $A$ and $B$ have a point of tangency, there is necessarily an entire flow line of tangencies, which we call an \emph{orbit of tangency}. 

    Say $\gamma$ is an orbit of tangency which is fixed setwise by the involution $\tau$. Note that $C \cap \gamma$ is necessarily a single point, $r$. Near $r$, we may choose coordinates $(y_1, y_2, y_3)$ with respect to which $d \tau_r$ is diagonal with eigenvalues $(+1, -1, -1)$. We write $E^+$ for the positive eigenspace and $E^-$ for the negative eigenspace. Note that $E^+$ is precisely $T_rC$. 

    Suppose that $\dim A = 2$ and $B = \tau(A)$. Assume that $\gamma$ is an orbit of tangency between $A$ and $B$. Let $P$ be the plane $T_rA = T_r B$. There are two possibilities. Either $P$ is the sum of the positive eigenspace $E^+$ and $\R \langle v \rangle$, or $P$ is the negative eigenspace $E^-$. We refer to orbits of tangency of the first kind as \emph{type 1 tangencies} and orbits of the second kind as \emph{type 2 tangencies}.
\end{defn}

\begin{lemma}\label{lem: real MS is generic}
    Generic real gradient-like vector fields are Morse-Smale. 
\end{lemma}
\begin{proof}
    We begin as in \cite{bao2025morsehomologyequivariance}. Suppose $\gamma$ is an orbit of tangency from $p$ to $q$. If $\tau(p) \neq q$, we can simply perturb the metric in a small neighborhood of $p$ (and symmetrically in a small neighborhood of $q$) to ensure that the ascending/descending manifolds of $p$ and $q$ intersect transversely. 
    
    If $\tau(p) = q$, we may potentially run into equivariant transversality issues, as modifying the metric near $p$ also alters the metric near $q$. Note that in this case, $\cI(q)= 3 - \cI(p)$, and hence it suffices to consider the case that $\cI(p) = 2$ and $\cI(q) = 1$.  

    Standardize an equivariant neighborhood of $\gamma \cong D^2 \times I$. Consider the slice $S = D^2 \times \{0\}$ which intersects the fixed point set $C$. In this slice, $A = W^u(p)\cap S$ and $\tau(A) = W^s(q)\cap S$ appear as curves which are tangent at the point $r = \gamma \cap S$ which is contained in $C$. Choose one more curve $D\sub S$ such that $\tau(D) = D$ and $T_rD = (T_r C)^\perp$. Notice that curves $A$ and $\tau(A)$ are tangent if and only if $A$ is tangent to $C$ or $D$ at $r$ (these correspond to type 1 and 2 orbits of tangency respectively). Non-equivariant transversality guarantees that generically $A$ intersects both $C$ and $D$ transversely, and hence, $A$ and $\tau(A)$ will generically be transverse. See the central frames of \Cref{fig:codim1_isotopy}.

    Hence, we may proceed as in the previous case; we simply perturb the metric near $p$ (and symmetrically near $q$) to ensure that $A$ is transverse to $C$ and $D$; doing so will guarantee that $W^u(p)$ and $W^s(q)$ intersect transversely.
\end{proof}
\begin{figure}[h]
\def\svgwidth{.8\linewidth}
\begingroup%
  \makeatletter%
  \providecommand\color[2][]{%
    \errmessage{(Inkscape) Color is used for the text in Inkscape, but the package 'color.sty' is not loaded}%
    \renewcommand\color[2][]{}%
  }%
  \providecommand\transparent[1]{%
    \errmessage{(Inkscape) Transparency is used (non-zero) for the text in Inkscape, but the package 'transparent.sty' is not loaded}%
    \renewcommand\transparent[1]{}%
  }%
  \providecommand\rotatebox[2]{#2}%
  \newcommand*\fsize{\dimexpr\f@size pt\relax}%
  \newcommand*\lineheight[1]{\fontsize{\fsize}{#1\fsize}\selectfont}%
  \ifx\svgwidth\undefined%
    \setlength{\unitlength}{822.04724409bp}%
    \ifx\svgscale\undefined%
      \relax%
    \else%
      \setlength{\unitlength}{\unitlength * \real{\svgscale}}%
    \fi%
  \else%
    \setlength{\unitlength}{\svgwidth}%
  \fi%
  \global\let\svgwidth\undefined%
  \global\let\svgscale\undefined%
  \makeatother%
  \begin{picture}(1,0.65517241)%
    \lineheight{1}%
    \setlength\tabcolsep{0pt}%
    \put(0,0){\includegraphics[width=\unitlength,page=1]{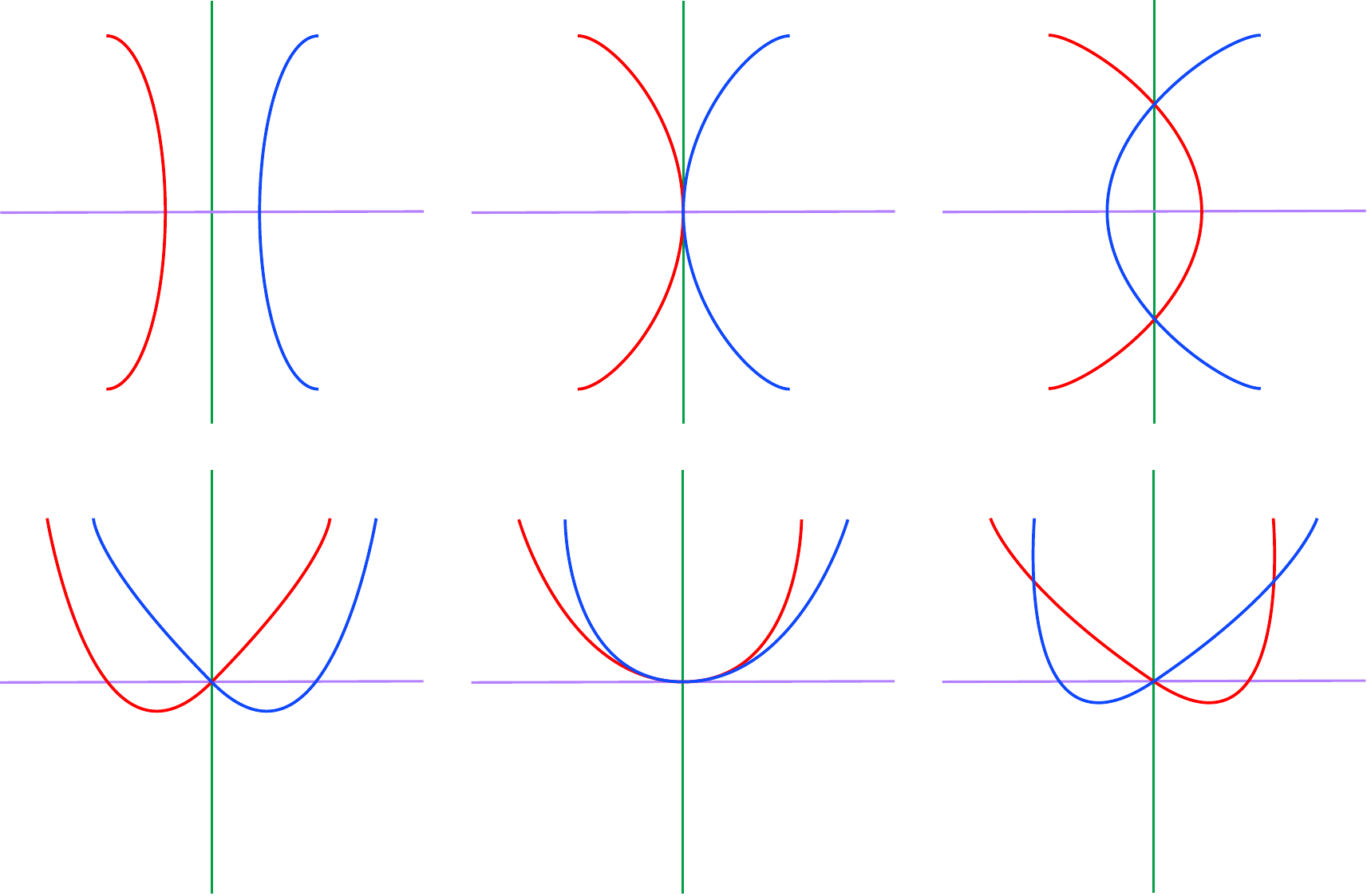}}%
    \put(0.17753831,0.64056387){\color[rgb]{0,0,0}\makebox(0,0)[t]{\smash{\begin{tabular}[t]{c}{$C$}\end{tabular}}}}%
    \put(0.00992816,0.5056753){\color[rgb]{0,0,0}\makebox(0,0)[t]{\smash{\begin{tabular}[t]{c}{$D$}\end{tabular}}}}%
    \put(0.05967345,0.39492487){\color[rgb]{0,0,0}\makebox(0,0)[t]{\smash{\begin{tabular}[t]{c}{$A$}\end{tabular}}}}%
    \put(0.25309298,0.39492487){\color[rgb]{0,0,0}\makebox(0,0)[t]{\smash{\begin{tabular}[t]{c}{$\t(A)$}\end{tabular}}}}%
  \end{picture}%
\endgroup%

    \caption{Two 1-parameter families of equivariant isotopies. The fixed set of the involution, $C$, is drawn here in green; the involution is given by reflection. The red and blue curves $A$ and $\tau(A)$ represent slices of the stable and unstable manifolds which are interchanged by the involution. In 1-parameter families, these may be tangent to one another at isolated times, as shown in the central figures. The purple line, $D$, represents an arc intersecting $C$ in a single point $p$, whose tangent space at $p$ is isomorphic to $(T_{p} C)^\perp$.  The figure illustrates that in 1-parameter families, the two failures of transversality in the equivariant setting correspond either to tangencies between $A$ and $C$, or to tangencies between $A$ and $D$ (in the non-equivariant setting).}
\label{fig:codim1_isotopy}
\end{figure}
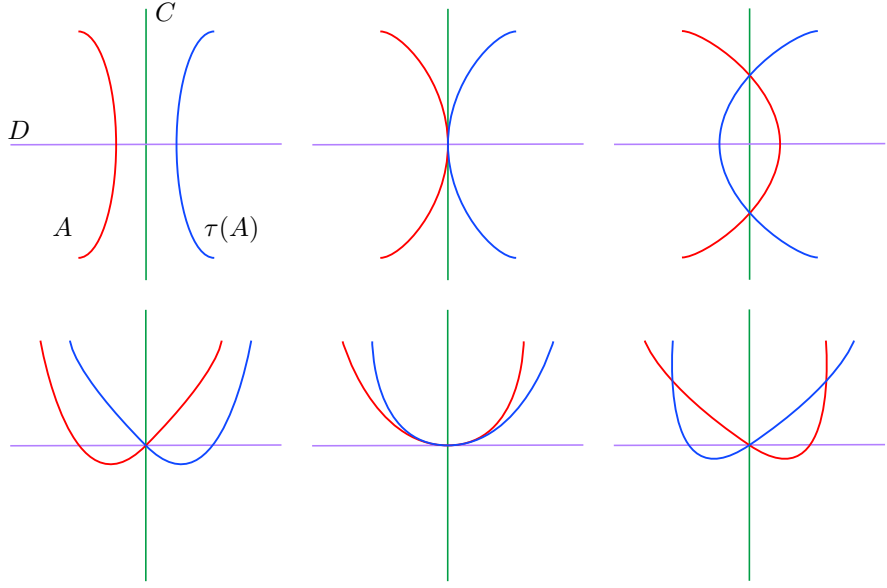

\begin{defn}
    A \emph{$k$-parameter family of real gradients} on $\Yt$ is a family of pairs $(g^\mu, f^\mu)$ for $\mu \in \R^k$ where $\{g^\mu\}$ is a family of $\tau$-invariant metrics and $\{f^\mu\}$ is a family of real functions on $\Yt$. Let $X^\mu := \grad_{g^\mu}(f^\mu)$ for $\mu \in \R^k$. 

    Both $g^\mu$ and $f^\mu$ depend smoothly on $(x, \mu) \in Y \times \R^k$ and the set of such pairs is endowed with the Whitney topology. The topological space of such $k$-parameter families is denoted $X^g_k\Yt$.
\end{defn}

\begin{defn}
    A parameter value $\overline{\mu}\in \R^k$ is called a \emph{bifurcation value} for $\{X^\mu\}$ in $X^g_k \Yt$ if $X^{\overline{\mu}}$ fails to be Morse-Smale. 
\end{defn}

At a bifurcation value $\mub$, the vector field $X^{\mub}$ either has an \emph{orbit of tangency} between stable and unstable manifolds (i.e. an integral curve along which a stable and unstable manifold are tangent) or at least one non-hyperbolic singular point. The \emph{bifurcation set} (or \emph{bifurcation diagram}) of $\{X^\mu\}$ in $X^g_k \Yt$ is the subset of $\R^k$ consisting of all bifurcation values. 

\subsection{Codimension-1 bifurcations}\label{subsec:codim 1 birfurcations}

In codimension 1 families of gradients, we expect tangencies of the various ascending and descending manifolds. However, we require that these tangencies be as generic as possible. 

\begin{defn}{\cite[Definition 5.11]{JTZ_naturality_mapping_class_groups}}
    Let $v$ be a vector field on a 3-manifold $Y$. Let $A$ and $B$ be invariant submanifolds of $v$. Let $\varphi_t$ be the flow of $v$ and let $t_x := \gamma^{-1}(x)$ for $x \in \gamma$. We say that $A$ and $B$ have a \emph{real quasi-transversal tangency} if their intersection contains a connected integral curve $\gamma$ (which is not a single point) such that:
    \begin{enumerate}
        \item[(QT-1)] Either $\dim(T_r A + T_r B) = 2$ for all $r \in \gamma$ or, if $B = \tau(A)$ and $A$ (and therefore $B$) is 1-dimensional, then $\dim(T_{r_0}A + T_{r_0}C) = 2$ for $\gamma \cap C = \{r_0\}$.
        \item[(QT-2)] If $\dim A = 2$ and $\dim B = 2$, we impose the following additional conditions. Let $S$ be a smooth 2-dimensional cross-section for $v$ which contains $r$.
        \begin{enumerate}
            \item If $B\neq \tau(A)$ or $B = \tau(A)$ and $\gamma$ is a type 1 tangency, take coordinates $(x_1,x_2)$ near $r$ in $S$ where $A \cap S = \{x_2 = 0\}$ and $B \cap S = \{x_2 = F(x_1)\}$ for some smooth function $F$. Condition (QT-1) forces $F(0) = 0$ and $\partial_{x_1}F(0) = 0$. We require that $\partial^2_{x_1}F(0) \neq 0$.
            \item If $B = \tau(A)$ and $\gamma$ is a type 2 tangency, we choose local coordinates $(x_1,x_2)$ as above, but require that $\partial^3_{x_1}F(0) \neq 0$. 
        \end{enumerate}
    \end{enumerate}
\end{defn}

\begin{rem}
    Real quasi-transversality differs from the usual notion in two ways. First, in a 1-parameter family a 1-dimensional stable (or unstable) manifold  $A$ will generically intersect the fixed point set a finite number of times. At such an intersection, $A$ and $\t(A)$ are necessarily tangent to one another, though they are both 1-dimensional. Therefore, we impose a condition on $\dim(T_{r_0}A + T_{r_0}C)$ rather than on $\dim(T_{r_0}A + T_{r_0}B)$.
    
    Second, in the case that $\gamma$ is a type 2 tangency, $F$ is generically cubic, rather than quadratic. This can be seen as follows. Since $\gamma$ is type $2$, the manifold $A$ can be expressed as the graph of a function $F: \R \langle v \rangle \oplus E^- \ra \R$ with $F(0) = 0$ and $\partial_{x_1}F(0)= 0$. In the slice $S$, we can choose coordinates $(y_1, y_2)$ so that $A\cap S = \{y_2 = F(y_1)\}$; the symmetry implies that $B \cap S = \{y_2 = F(-y_1)\}$. Fix coordinates 
    \begin{align*}
        x_1 & = y_1 \\
        x_2 & = y_2 - F(y_1).
    \end{align*}
    In these coordinates, $A \cap S = \{x_2 = 0\}$ and $B \cap S = \{x_2 = F(-x_1)-F(x_1)\}$. Let $H(x_1) = F(-x_1)-F(x_1)$. Expanding $F(x_1) = ax_1^2 + bx_1^3 + O(x_1^4)$, we see:
    \begin{align*}
        H(x_1) = (ax_1^2 - bx_1^3 + O(x_1^4)) - (ax_1^2 + bx_1^3 + O(x_1^4)) = -2b x_1^3 + O(x_1^4).
    \end{align*}
    For this reason, we cannot expect the version of (QT-2) stated in \cite{JTZ_naturality_mapping_class_groups} to hold in all cases in the real setting. 
\end{rem}

In the equivariant setting, we of course consider \emph{orbits of orbits of tangencies}. To avoid this awkward language, we will use the term $\{1\}$-orbits of tangencies to refer to orbits of tangency which are fixed by the involution and $\Z/2$-orbits of tangency to refer to pairs of orbits of tangency which are exchanged by the involution.

\begin{rem}
\label{rem:color_code}
    In Figures~\ref{fig:codim1_bifurcations} through \ref{fig:E_codim2_bifurcations}, we provide schematics for the various codimension 1 and 2 bifurcations we will encounter. We have color-coded these figures for the sake of readability. In each figure, we have drawn the fixed point set of $Y$ as a horizontal green line. Gray lines indicate orbits corresponding to transverse intersections between stable and unstable manifolds; these can occur in three situations: between the unstable manifold of an index 1 critical point and the stable manifold of an index 2 critical point; between the unstable manifold of an index 1 critical point and the stable manifold of an index 1-2 saddle node; or between the unstable manifold of an index 1-2 saddle node and the stable manifold of an index 2 critical point. Red (respectively blue) lines correspond to quasi-transverse orbits of tangencies between stable and unstable manifolds of index 1 (respectively index 2) critical points. Finally, lines corresponding to real quasi-transverse orbits of tangency will be colored green, emphasizing that these flow lines intersect the fixed point set.
\end{rem}

\begin{lemma}
    A generic 1-parameter family $X^t$ of real gradient-like vector fields is Morse-Smale at all but finitely values of $t$, where exactly one of the following phenomena can occur:
    \begin{enumerate}
        \item $X^t$ has a single non-hyperbolic $\Z/2$-singularity;
        \item $X^t$ has a single non-hyperbolic $\{1\}$-singularity;
        \item $X^t$ has a single $\Z/2$-orbit real quasi-transversal tangency;
        \item $X^t$ has a single $\{1\}$-orbit real quasi-transversal tangency.
    \end{enumerate}
\end{lemma}
\begin{proof}
    With regard to singularities, the generic occurrence of only (1) and (2) follows from \Cref{thm:real singularities Z/2 orbit} and \Cref{thm:real singularities trivial orbit}. Case (3) involving orbits of tangency appearing in $\Z/2$-orbits can easily be reduced to the unreal case. Hence, it suffices to consider $\{1\}$-orbits between critical points.

    First, suppose that $\gamma$ is an orbit of tangency between the ascending manifold of an index $1$ critical point $p$ and the descending manifold of an index 2 critical point $q$. Note that it must be the case that $q = \tau(p)$. As in \Cref{lem: real MS is generic}, we choose a standard neighborhood and consider now families of curves $A_t$ and $\tau(A_t)$ for $t \in [-1,1]$ appearing in a slice $S$ containing the fixed point set $C$. Let $D$ be a curve in $S$ with $\tau(D) = D$ and $T_rD = (T_rC)^\perp$. Again, see \Cref{fig:codim1_isotopy}. In a generic 1-parameter family, $A_t$ will generically be transverse to both $C$ and $D$ at all but finitely many times $t$ (with isolated tangencies) and hence the same holds for the families $A_t$ and $\tau(A_t)$. See \Cref{fig:codim1_bifurcations}.

    Next, we consider the case that $\gamma$ is an orbit of tangency between the descending manifold of an index $1$ critical point $p$ and the ascending manifold of an index 2 critical point $q$ (which must be equal to $\tau(p)$). First, note that $W^s(p_t)$ and  $W^u(q_t)$ are tangent precisely when $W^s(p_t)$ intersects $C$. In a 1-parameter family, at isolated parameters $t$ the manifold $W^s(p_t)$ will generically intersect $C$ in a single point. 
\end{proof}

\begin{figure}[h]
\def\svgwidth{.8\linewidth}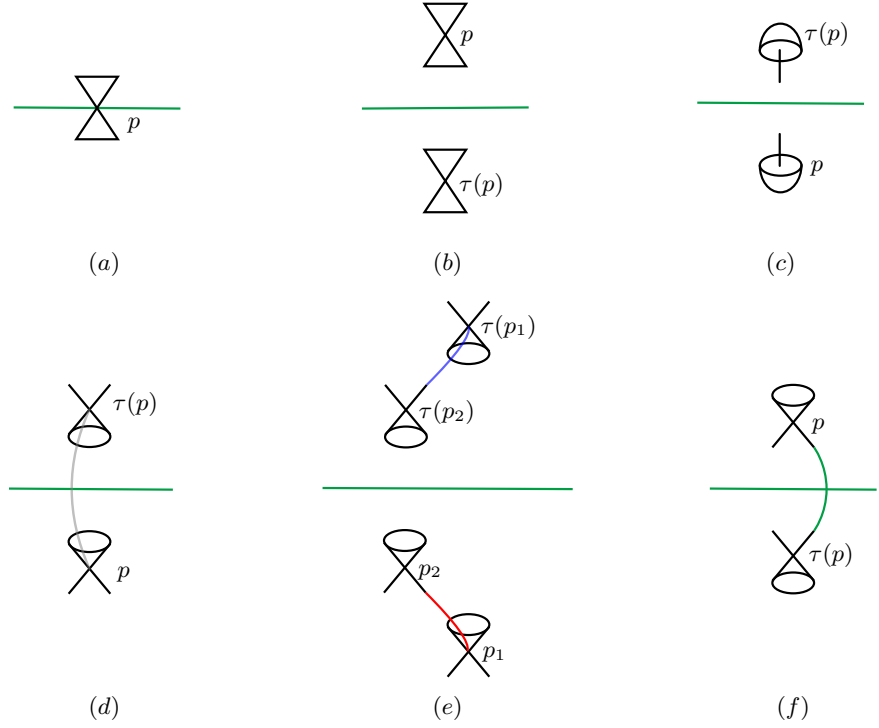
    \caption{Some possible bifurcations appearing in 1-parameter families of real gradient vector fields. The gradient flow always goes upwards. (a) an index 1-2 birth-death in a $\{1\}$-orbit; (b) an index 1-2 birth-death in a $\Z/2$-orbit; (c) an index 0-1 birth-death and a symmetric 2-3 birth-death; (d) a quasi-transversal orbit of tangency from an index 1 critical point to an index 2 critical point (these can appear in both $\{1\}$ and $\Z/2$-orbits); (e) a $\Z/2$-orbit of tangency between index 1 critical points $p_1$ and $p_2$, as well as between their index 2 images $\t(p_1)$ and $\t(p_2)$; (f) a quasi-transversal orbit of tangency from an index 2 critical point $p$ and its image $\tau(p)$.}
\label{fig:codim1_bifurcations}
\end{figure}

\begin{rem}\label{rem: nondegeneracy conditions}
    In the unreal setting, quasi-transversal orbits can be assumed to satisfy certain non-degeneracy conditions. These are denoted (ND-1)-(ND-4) in \cite{JTZ_naturality_mapping_class_groups} (similar conditions are also discussed in \cite{vegter} and \cite{carneiro_palis}). As we will ultimately pass to isotopy diagrams, these conditions do not play any substantial role. The only condition which is relevant to us is (ND-2), which states that, generically, the contracting eigenvalues of the linear part of $X^{\overline{t}}$ at a critical point $p_1^{\overline{\mu}}$ of index 1 are distinct. This implies that there is a unique 1-dimensional manifold $W^{ss}(p_1^{\overline{t}}) \sub W^s(p_1^{\overline{t}})$ so that $T_{p_1^{\overline{t}}}W^{ss}(p_1^{\overline{t}})$ is the eigenspace of $L_{p_1^{\overline{\t}}} X^{p_1^{\overline{t}}}$ corresponding to the strongest contracting eigenvalue. This condition is still generic in the equivariant setting.
\end{rem}

As in the unreal case, non-hyperbolic singularities translate to (de)stabilizations of Heegaard diagrams and flows between critical points of the same index translate to handle slides. However, as noted above, in the real case, we can also have flows from critical points of index 2 to their index 1 images under the action of the involution. This codimension-1 bifurcation gives rise to a new move of real Heegaard diagrams, which we call a \emph{crossover}. See \Cref{def:foot_swap}. While index 0-1 and index 2-3 saddle nodes do appear in the real setting, the analysis in these cases follows directly from \cite{JTZ_naturality_mapping_class_groups} since these only appear in $\Z/2$-orbits.

\subsection{Codimension-2 bifurcations}\label{subsec:codim 2 birfurcations}

We now turn to 2-parameter families of gradients. Codimension-1 bifurcations appear in open and dense families in these situations, and form smooth curves in the parameter space $\R^2$. Moreover, for a generic $\{X^\mu\}\in X_2^g\Yt$, at isolated points $\overline{\mu}$, we have the following additional bifurcations. Throughout, $\cO \in \{\{1\},\Z/2\}$. 

\begin{enumerate}[(A)]
    \item The vector field $X^\mu$ has exactly two real quasi-transversal $\cO$-orbits of tangency and all singularities are hyperbolic. There are three possible configurations: two $\Z/2$-orbits, two $\{1\}$-orbits, or one of each. 
    \item The vector field $X^\mu$ has exactly one non-hyperbolic critical orbit, which is an orbit of saddle-nodes, and exactly one real quasi-transversal $\cO$-orbit of tangency. There are several possible configurations. 
    \item The vector field $X^\mu$ has exactly two non-hyperbolic critical orbits, which are orbits of saddle-nodes, and all stable and unstable manifolds intersect transversely.
    \item The vector field $X^\mu$ has exactly one codimension 2 non-hyperbolic critical orbit (i.e. a $A_3(\Z/2)$ singularity or an $A_4(\{1\})$ singularity).
    \item All singularities of $X^\mu$ are hyperbolic and there is a violation of the quasi-transversality condition.
\end{enumerate}

As in \cite{JTZ_naturality_mapping_class_groups}, we will largely ignore bifurcations corresponding to isotopies, as ultimately, we will be interested in real Heegaard diagrams only up to isotopy. 

Next, we describe the bifurcation sets in $\R^2$ near a parameter $\overline{\mu}$ corresponding to one of the codimension 2 bifurcations (A)-(E) described above. We will also consider \emph{secondary bifurcations}, in which the singularities contributing to the failures of the transversality and hyperbolicity conditions interact with other singularities, giving rise to additional codimension-1 strata. 

Type (A) secondary bifurcations involve two simultaneous real quasi-transversal $\cO$-orbits of tangency. Necessarily, these bifurcations involve index 1 and 2 critical points. We shall largely ignore orbits of tangency involving the unstable manifold of index 1 critical points and the stable manifold of an index 2 critical point, because these correspond to isotopies of attaching curves and we ultimately pass to isotopy diagrams. 

Throughout this section, critical points of index $1$ will be written as $p^\mu$ (perhaps with some decoration) and so the index 2 points will be written $\tau(p^\mu)$. 

\begin{enumerate}
    \item[(A1)] There are four hyperbolic orbits $(\Z/2)\cdot p_1^{\overline{\mu}}$, $(\Z/2)\cdot p_2^{\overline{\mu}}$, $(\Z/2)\cdot p_3^{\overline{\mu}}$, and $(\Z/2)\cdot p_4^{\overline{\mu}}$ such that the orbits of tangencies are contained in $(\Z/2) \cdot (W^u(p_1^{\overline{\mu}})\cap W^s(p_2^{\overline{\mu}}))$ and $(\Z/2)\cdot (W^u(p_3^{\overline{\mu}})\cap W^s(p_4^{\overline{\mu}}))$. Each critical point is of index 1 or 2; the orbits are not necessarily distinct. Generically, the bifurcation diagram consists of two smooth curves which intersect transversely in a point. Such bifurcations are shown schematically in \Cref{fig:A_codim2_bifurcations}.
    
    \item[(A2)] There are three hyperbolic singular orbits $(\Z/2)\cdot p_1^{\overline{\mu}}$, $(\Z/2)\cdot p_2^{\overline{\mu}}$, and $(\Z/2)\cdot p_3^{\overline{\mu}}$ of $X^{\overline{\mu}}$ so that the orbits of tangency are contained in $(\Z/2)\cdot (W^u(p_1^{\overline{\mu}})\cap  W^s(p_2^{\overline{\mu}}))$ and $(\Z/2)\cdot (W^u(p_2^{\overline{\mu}})\cap  W^s(p_3^{\overline{\mu}}))$. Again, see \Cref{fig:A_codim2_bifurcations} for a schematic.

    \item[(A3)] There are hyperbolic orbits $(\Z/2)\cdot p_{1}^{\bar{\mu}}$ and $(\Z/2)\cdot p_{2}^{\bar{\mu}}$ of $X^{\bar{\mu}}$ which each have a $\{1\}$-orbit of tangency from $W^u(\tau(p_i))$ to $W^s(p_i)$ for $i \in \{1,2\}$; both manifolds are 1-dimensional. We allow the case that $p_{1}^{\bar{\mu}} = p_{2}^{\bar{\mu}}$, in which case $W^u(\tau(p_i))$ and $W^s(p_i)$ each intersect $C$ twice. Generically, the bifurcation diagram consists of two smooth curves which intersect transversely at $\bar{\mu}$. See \Cref{fig:A_codim2_bifurcations2}.
    
    \item[(A4)] There are three hyperbolic orbits $(\Z/2)\cdot p_{1}^{\bar{\mu}}$, $(\Z/2)\cdot p_{2}^{\bar{\mu}}$, $(\Z/2)\cdot p_{3}^{\bar{\mu}}$ of $X^{\bar{\mu}}$ such that there is a $\{1\}$-orbit of tangency from $W^u(\tau(p_1))$ to $W^s(p_1)$ as well as a $\Z/2$-orbit of tangency from $W^u(p_2)$ to $W^s(p_3)$ (and symmetrically for $\tau(p_1)$ and $\tau(p_2)$). We allow $p_{1}^{\bar{\mu}} = p_{2}^{\bar{\mu}}$ or $p_{1}^{\bar{\mu}} = p_{3}^{\bar{\mu}}$. The possible configurations are shown in \Cref{fig:A_codim2_bifurcations2}. In case (A4c), secondary bifurcations may be present due to tangencies between $W^s(p_2^{\overline{\mu}})$ and $W^u(\tau(p_2^{\overline{\mu}}))$ if $W^u(p_1^{\overline{\mu}})\cap W^s(p_2^{\overline{\mu}}) \neq \emptyset.$ See \Cref{fig:A_codim2_bifurcations2}.
\end{enumerate}

\begin{figure}[h]
\def\svgwidth{.8\linewidth}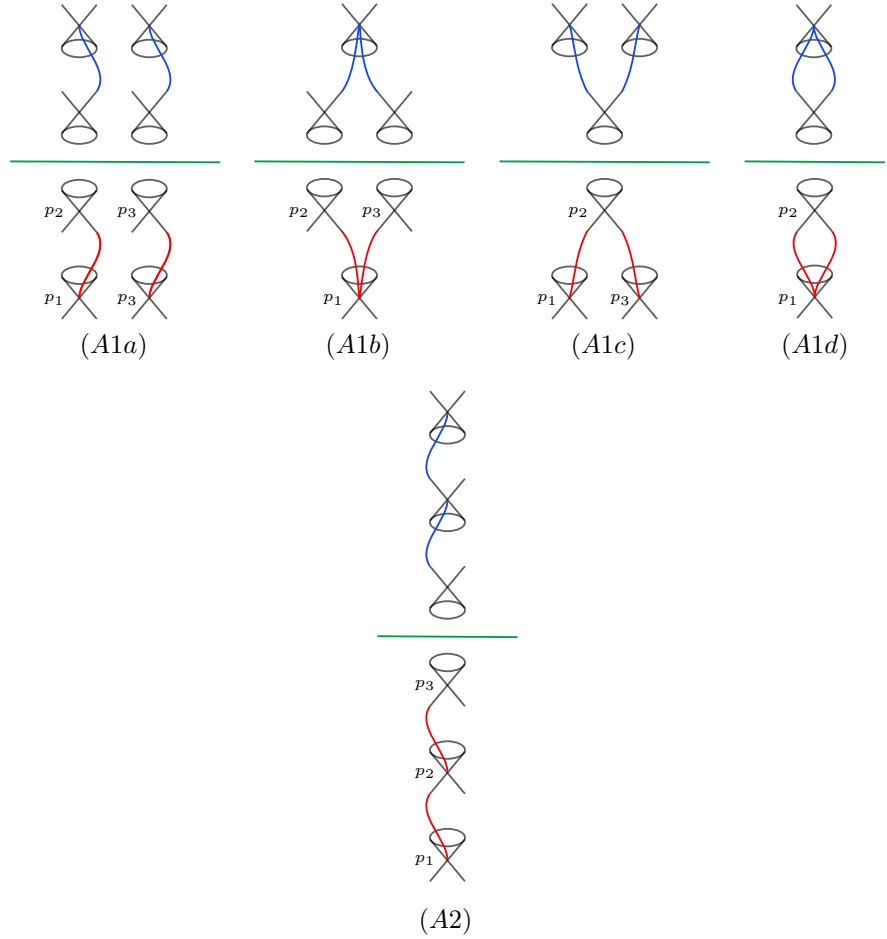
    \caption{Codimension 2 bifurcations of types (A1) and (A2), which involve simultaneous orbits of tangency between a one-dimensional and a two-dimensional submanifold.}
\label{fig:A_codim2_bifurcations}
\end{figure}

\begin{figure}[h]
\def\svgwidth{.8\linewidth}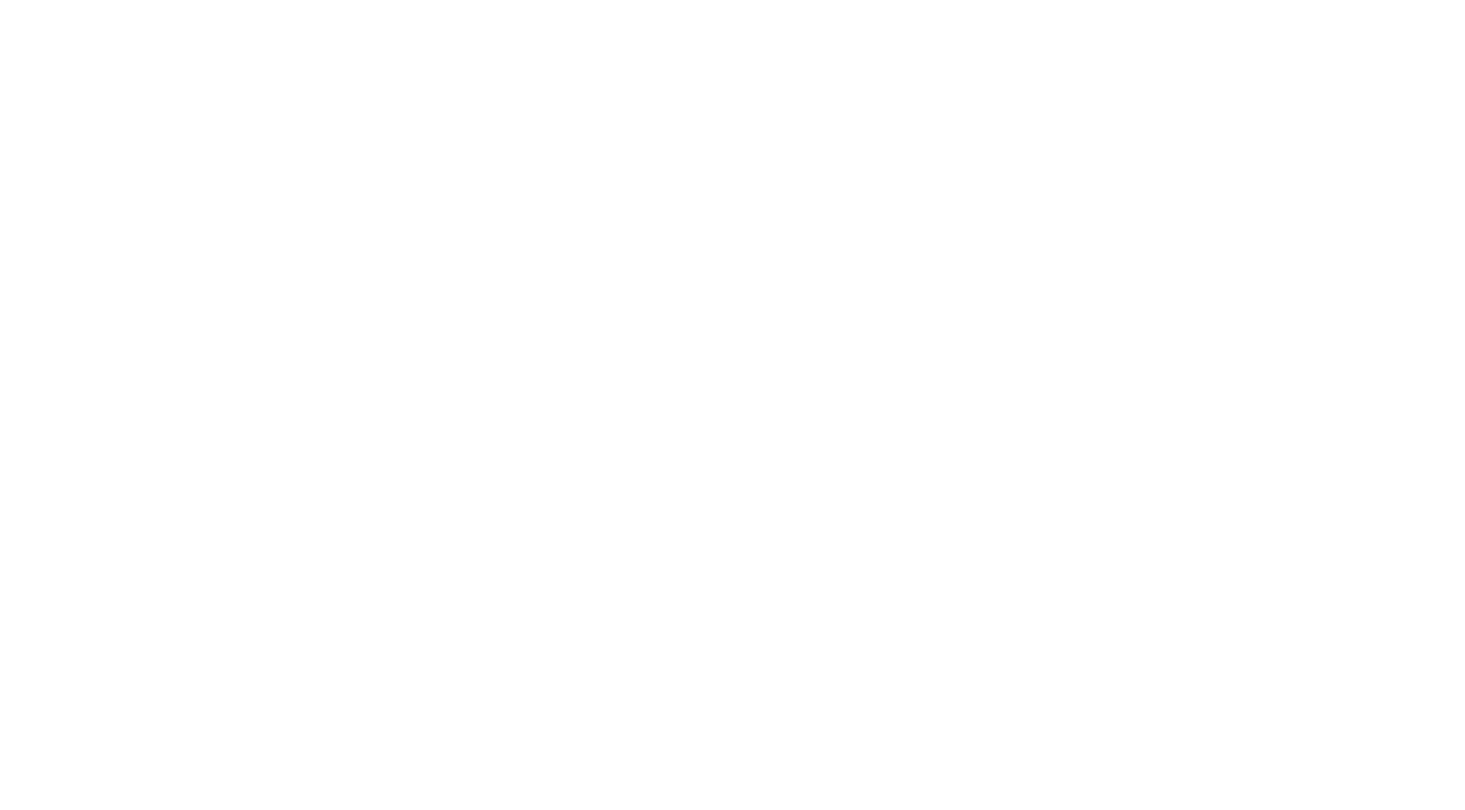
    \caption{Codimension 2 bifurcations of types (A3) and (A4) that consist of simultaneous orbits of tangency involving at least one 1-dimensional submanifold.}
\label{fig:A_codim2_bifurcations2}
\end{figure}

Type (B) secondary bifurcations involve a single saddle-node orbit and a single real quasi-transversal $\cO$-orbit of tangency. All singularities have index 1 and 2. Unlike the previous case, saddle-nodes can appear in both $\Z/2$ and $\{1\}$-orbits, which makes the analysis of secondary bifurcations in this case somewhat more involved.

\begin{enumerate}
    \item[(B1)] The vector field $X^{\overline{\mu}}$ has a saddle-node orbit $(\Z/2)\cdot p^{\overline{\mu}}$ and one quasi-transverse $\Z/2$-orbit of tangency between $(\Z/2)\cdot W^u(p_1^{\overline{\mu}})$ and $(\Z/2)\cdot W^u(p_2^{\overline{\mu}})$, where $p_1^{\overline{\mu}}$ and $p_2^{\overline{\mu}}$ are hyperbolic saddle points of index 1. The bifurcation curve consists of two curves intersecting transversely at $\overline{\mu}$. See  \Cref{fig:B_codim2_bifurcations1}.

    \item[(B2)] The vector field $X^{\overline{\mu}}$ has a saddle-node orbit $(\Z/2)\cdot p^{\overline{\mu}}$ and a hyperbolic saddle-orbit $(\Z/2)\cdot \overline{p}^{\overline{\mu}}$ and a single orbit of tangency between $W^{ss}(p^{\overline{\mu}})$ and $W^{u}(\overline{p}^{\overline{\mu}})$ (and a symmetric orbit between $W^{uu}(\tau(p^{\overline{\mu}}))$ and $W^{s}(\tau(\overline{p}^{\overline{\mu}}))$). Secondary bifurcations arise due to the occurrence of tangencies in between $W^u(p^\mu_*)$ and the unstable manifolds of index 2 critical points. If there are $s$ such tangencies, the bifurcation diagram consists of $s + 3$ codimension-1 strata meeting at $\overline{\mu}$. See \Cref{fig:B_codim2_bifurcations1}.

    \item[(B3)] The vector field $X^{\overline{\mu}}$ has a saddle-node orbit $(\Z/2)\cdot p^{\overline{\mu}}$ and a hyperbolic orbit $(\Z/2)\cdot p^{\overline{\mu}}$ with a single orbit of tangency between $W^{ss}(p^{\overline{\mu}})$ and $W^{u}(p^{\overline{\mu}})$ (and a symmetric orbit between $W^{uu}(\tau(p^{\overline{\mu}}))$ and $W^{s}(\tau(\overline{p}^{\overline{\mu}}))$). The bifurcation diagram consists of three codimension-1 strata meeting at $\overline{\mu}$. See \Cref{fig:B_codim2_bifurcations2} (where $s=1$).

    \item[(B4)] The vector field $X^{\overline{\mu}}$ has a saddle-node orbit $(\Z/2)\cdot p^{\overline{\mu}}$ and one quasi-transverse $\Z/2$-orbit of tangency between $W^s(q^{\overline{\mu}})$ and $W^u(\tau(q^{\overline{\mu}}))$, where $q^{\overline{\mu}}$ is a hyperbolic singularity of index 1. The bifurcation curve consists of two curves intersecting transversely at $\overline{\mu}$. See \Cref{fig:B_codim2_bifurcations2}.

    \item[(B5)] The vector field $X^{\overline{\mu}}$ has a saddle-node orbit $(\Z/2)\cdot p^{\overline{\mu}}$ with a single orbit of tangency between $W^{uu}(p^{\overline{\mu}})$ and $W^{ss}(\tau(p^{\overline{\mu}}))$ (this occurs when $W^{uu}(p^{\overline{\mu}})$ intersects $C$). The bifurcation diagram consists of three codimension-1 strata meeting at $\overline{\mu}$. See \Cref{fig:B_codim2_bifurcations2}.
    
\end{enumerate}

\begin{figure}[h]
\def\svgwidth{.8\linewidth}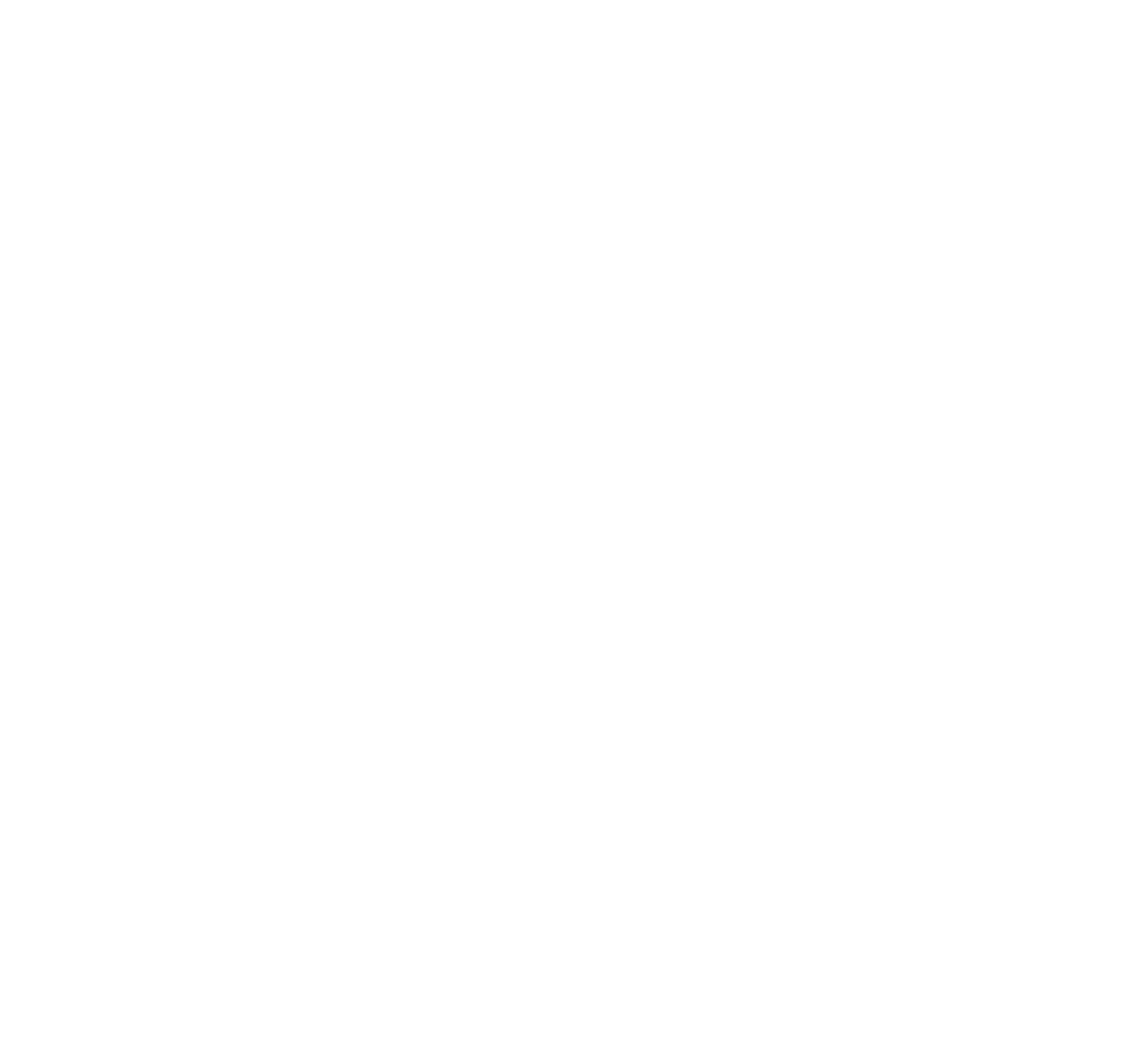
    \caption{Codimension-2 bifurcations of type (B) involving pairs of quasi-transversal orbits of tangency.}
\label{fig:B_codim2_bifurcations1}
\end{figure}

\begin{figure}[h]
\def\svgwidth{.8\linewidth}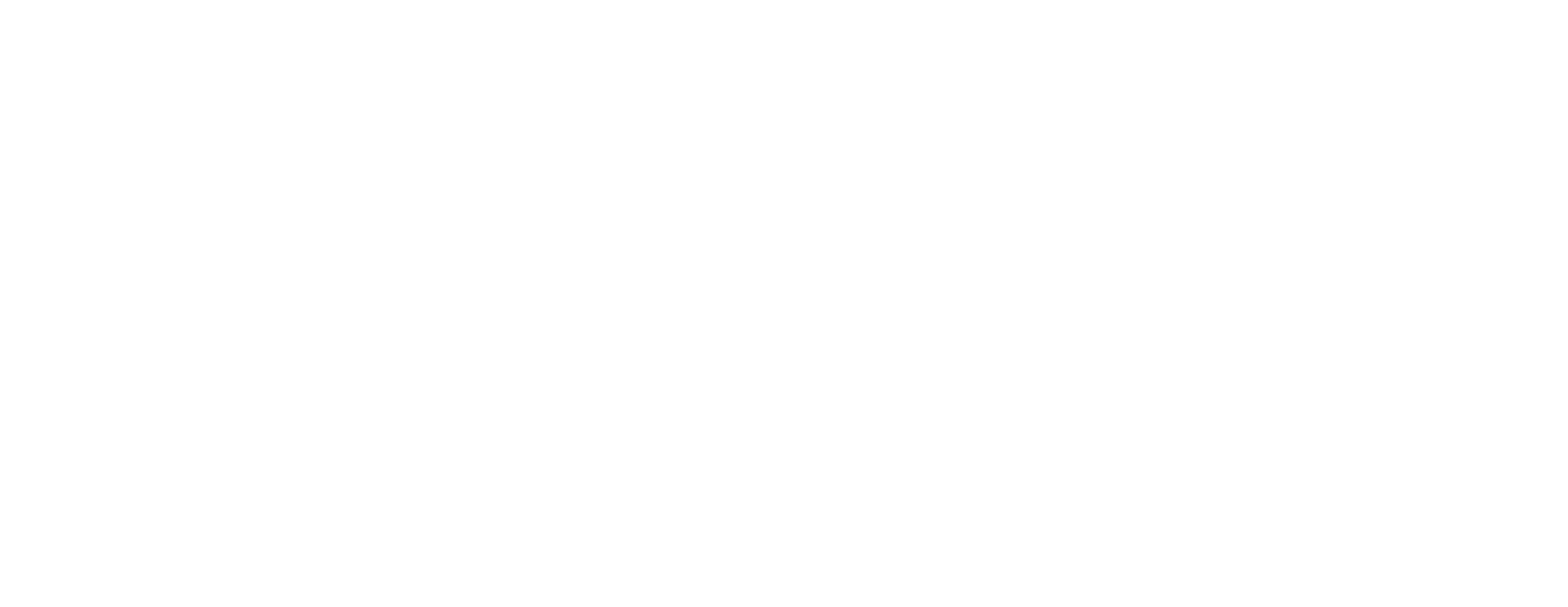
    \caption{Codimension-2 bifurcations of type (B) involving a single real quasi-transversal orbit of tangency.}
\label{fig:B_codim2_bifurcations2}
\end{figure}

\begin{rem}
    We note that bifurcations (B4) and (B5) are analogous to bifurcations (B1) and (B3), respectively, though they involve crossovers (see \Cref{def:foot_swap}), rather than handleslides. Bifurcation (B2) has no codimension 2 analogue for crossovers. 
\end{rem}

Type (C) bifurcations consist of two simultaneous saddle-node orbits. These, of course, can appear in several configurations. 

\begin{enumerate}
    \item[(C)] The vector field $X^{\overline{\mu}}$ has two saddle-node orbits $(\Z/2)\cdot p_1$ and $(\Z/2)\cdot p_2$. The bifurcation diagram consists of two curves corresponding to $(\Z/2)\cdot p_1$ and $(\Z/2)\cdot p_2$ which intersect transversely. When at least one of $\{p_1,p_2\}$ is not fixed by $\tau$, we assume that $f(p_1) < f(p_2)$; in the case that both $p_1$ and $p_2$ are fixed by $\tau$, necessarily $f(p_1) = f(p_2) = 0$. See \Cref{fig:C_codim2_bifurcations}.
\end{enumerate}

\begin{figure}[h]
\def\svgwidth{.8\linewidth}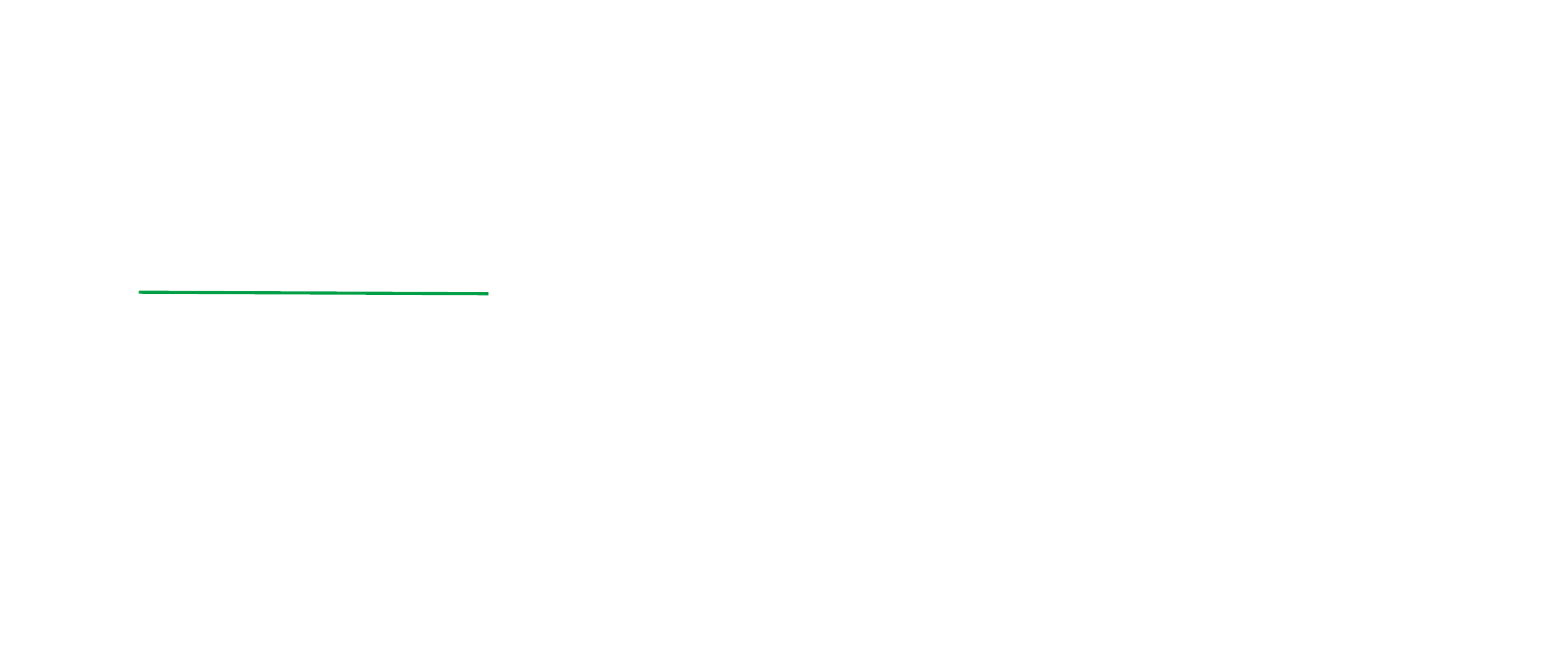
    \caption{Codimension-2 bifurcations of type (C): these come in three varieties, depending on the orbit types of the saddle nodes. In the first frame, we have varied the shading to emphasize the different kinds of transverse orbits that may appear.}
\label{fig:C_codim2_bifurcations}
\end{figure}

Type (D) bifurcations correspond to (real) codimension 2 failures of the Morse condition. In the real setting, there are two possibilities.

\begin{enumerate}
    \item[(D1)] In a neighborhood of $\overline{\mu}$, the bifurcation diagram consists of values $\mu$ for which $X^\mu$ (and therefore $f^\mu$) has a degenerate critical point near $(\Z/2)\cdot p$. For an open and dense class of families $\{f^{\mu}\}$ for which $\grad(f^\mu)$ has an $A_3^\pm$ singularity appearing in a $\Z/2$-orbit, there are $\mu$-dependent local coordinates $(x, y, z)$ in which 
    \begin{align*}
        f^\mu(x, y, z) = \pm x^4 + \mu_1 x^2 + \mu_2 x \pm y^2 \pm z^2.
    \end{align*}
    The bifurcation diagram is the familiar cusp. See the left frame of \Cref{fig:D_codim2_bifurcations}.
    \item[(D2)] In a neighborhood of $\overline{\mu}$, the bifurcation diagram consists of values $\mu$ for which $X^\mu$ (and therefore $f^\mu$) has a degenerate critical point near $p$. For an open and dense class of families $\{f^{\mu}\}$ for which $\grad(f^\mu)$ has an $A_4$ singularity appearing in a trivial orbit, there are $\mu$-dependent local coordinates $(x, y, z)$ in which 
    \begin{align*}
        f^\mu(x, y, z) = x^5 + \mu_1 x^3 + \mu_2 x + y^2 - z^2.
    \end{align*}
    See the right frame of \Cref{fig:D_codim2_bifurcations}. 
\end{enumerate}

\begin{figure}[h]
\def\svgwidth{.8\linewidth}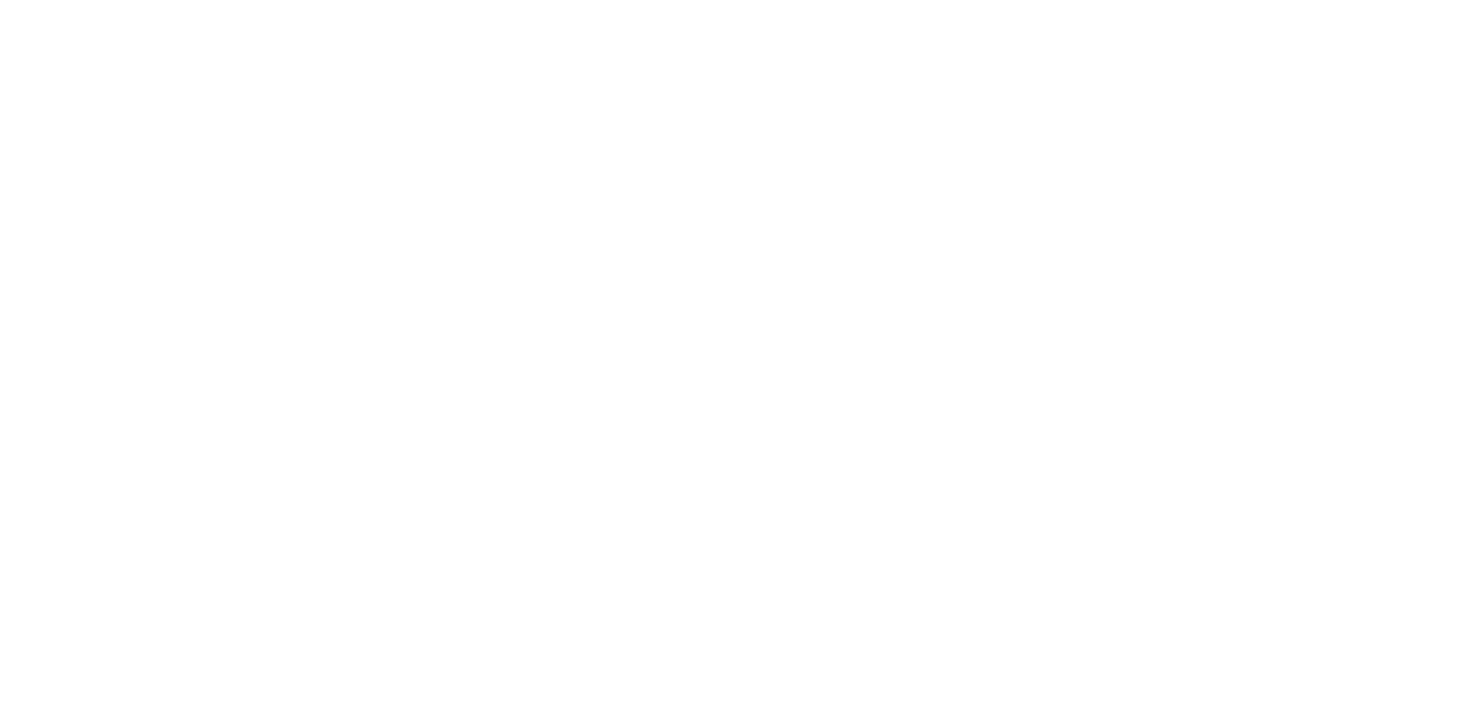
    \caption{Codimension-2 bifurcations of type (D) involving a real singularity of codimension $2$.}
\label{fig:D_codim2_bifurcations}
\end{figure}

Finally, we consider failures of real quasi-transversality.

\begin{enumerate}
    \item[(E1)] There is a violation of (QT-1), i.e. there is a parameter value $\overline{\mu}$ and hyperbolic orbits $(\Z/2)\cdot p_1$ and $(\Z/2)\cdot p_2$ for which there is an orbit of tangency between $W^s(p_1^{\overline{\mu}})$ and $W^{u}(\tau(p_2^{\overline{\mu}}))$ which are both 1-dimensional, and,  symmetrically, there is an orbit of tangency between $W^u(\tau(p_1^{\overline{\mu}}))$ and $W^s(p_2^{\overline{\mu}})$ (which is also 1-dimensional). Here, $p_1^{\mu}$ and $p_2^{\mu}$ are critical points of index 1. There are, of course, two possible configurations, depending on whether or not $p_1 = p_2$. Secondary bifurcations appear in the presence of orbits of tangency between $(\Z/2) \cdot  W^s(\tau(p_2^{\overline{\mu}}))$ and a stable manifold of dimension 2 intersecting $(\Z/2) \cdot W^u(p_1^\mu)$ or between $(\Z/2) \cdot W^u(p_1^\mu)$ and a stable manifold of dimension 2 intersecting $(\Z/2) \cdot W^u(\tau(p_2^\mu))$. See \Cref{fig:E_codim2_bifurcations}.

     \item[(E2)] There is a hyperbolic orbit $(\Z/2)\cdot p^{{\mu}}$ of $X^{\bar{\mu}}$ which has a $\{1\}$-orbit of tangency $\gamma$ between $W^u(\tau(p^\mu))$ and $W^s(p^\mu)$ as well as a quasi-transverse orbit of tangency between the directional manifold $D(p^\mu, \gamma)$ and the unstable manifold of an index 2 critical point. See \Cref{fig:E_codim2_bifurcations}.
\end{enumerate}

\begin{rem}
    In 2-parameter families, we will usually not need to consider both directional manifolds. In cases in which $\gamma$ is clear from the context, we will write $D(p^\mu)$ and suppress $\gamma$ from the notation.
\end{rem}

\begin{figure}[h]
\def\svgwidth{.8\linewidth}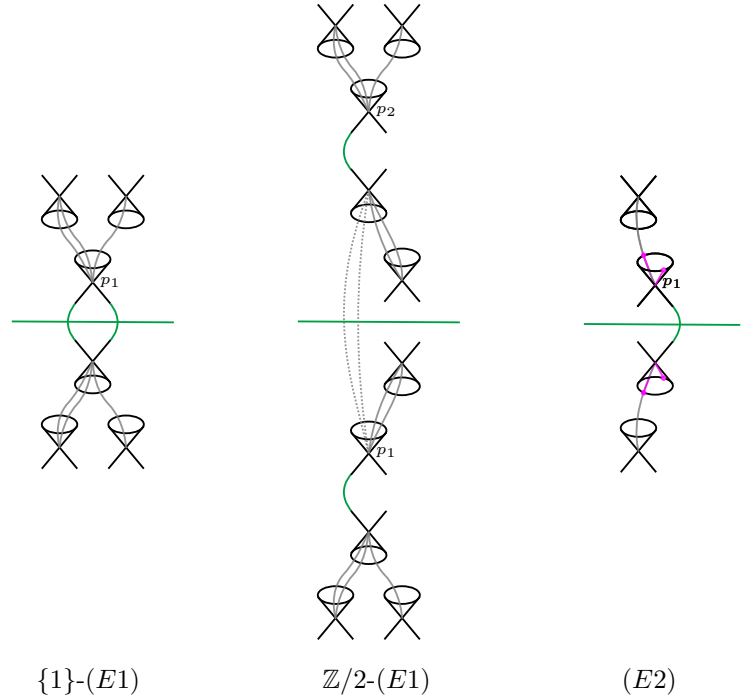
    \caption{Type (E) codimension-2 bifurcations. In the third frame, we have drawn the directional manifold in pink.}
\label{fig:E_codim2_bifurcations}
\end{figure}

\subsection{Real sutured functions and gradient-like vector fields.}

Let $\Yt$ be a real sutured 3-manifold.

\begin{defn}\label{def:sutured-function}
    A \emph{real sutured function} on $\Yt$ is a smooth function $f: Y \ra [-1,1]$ which satisfies:
    \begin{enumerate}
        \item $f^{-1}(\pm1) = R_\pm(\gamma)$ and $f^{-1}(0) \supset s(\gamma)$;
        \item $f$ has no critical points along $R(\gamma)$;
        \item $f|_{\gamma}$ has no critical points;
        \item $f \circ \tau = -f$.
    \end{enumerate}
\end{defn}

For such a function $f$, we write $C(f)$ for the set of critical points of $f$. Note that $C(f)$ consists of points rather than orbits. Necessarily, $C(f)$ lies in the interior of $Y$.

\begin{defn}
    Given a real sutured function $f$ on $\Yt$, we say that a vector field $v$ on $Y$ is \emph{a real gradient-like vector field of $f$} if 
    \begin{enumerate}
        \item $v(f) > 0$ on $Y \smallsetminus C(f)$,
        \item $C(f)$ has a neighborhood $U$ so that $v|_U = \grad_g(f|_U)$ for some $\tau$-invariant Riemannian metric $g$ on $U$.
    \end{enumerate}
\end{defn}

Let $\FV \Yt$ be the space of pairs $(f, v)$ where $f$ is a real function on $\Yt$ and $v$ is a real gradient-like vector field for $f$. The space $\FV\Yt$ is given the $C^\infty$-topology. We write $\FV_0 \Yt$ for the subspace of Morse-Smale pairs. 

\begin{defn}
    We say that $(f, v) \in \FV\Yt$ is \emph{codimension-1} if $(f, v) \not \in \FV_0\Yt$, but $v$ appears as $X^{\overline{\mu}}$ for some 1-parameter family $\{X^{\mu}\}\in X_1^g\Yt$ which is generic in the sense of \Cref{subsec:codim 1 birfurcations}. Let $\FV_1\Yt$ be the space of codimension-1 pairs and $\FV_{\le1}\Yt$ be the union $\FV_0\Yt\cup \FV_1\Yt$.

    Likewise, we say  $(f, v) \in \FV\Yt$ is \emph{codimension-2} if $(f, v) \not \in \FV_{\le1}\Yt$, but $v$ appears as $X^{\overline{\mu}}$ for some 2-parameter family $\{X^{\mu}\}\in X_2^g\Yt$ which is generic in the sense of \Cref{subsec:codim 2 birfurcations}. We write $\FV_2\Yt$ for the space of codimension-2 pairs and $\FV_{\le2}\Yt$ for the union $\FV_0\Yt\cup \FV_1\Yt \cup \FV_2\Yt$.
\end{defn}

\begin{prop}
\label{prop:fv}
    Let $(f,v)\in \FV\Yt$. Then, the space $G(f, v)^\tau$ of $\tau$-invariant Riemannian metrics on $Y$ for which $v = \grad_g(f)$ is non-empty and contractible. 
\end{prop}
\begin{proof}
    This is the real analogue of \cite[Proposition 5.19]{JTZ_naturality_mapping_class_groups}. By that result, there exists a metric $g'$ (not necessarily $\tau$-invariant) such that $v = \grad_g(f)$. We construct a $\tau$-invariant one by averaging:
    $$g = \frac{1}{2}(g' + \tau^*g').$$  
    We verify the gradient condition $g(v, \cdot) = df$:
    \begin{align*}
        g(v, Y) &= \frac{1}{2}(g'(v, Y) + (\tau^*g')(v, Y)) \\
        &= \frac{1}{2}(df(Y) + g'(d\tau(v), d\tau(Y))).
    \end{align*}
    Since $v$ is a real vector field, $d\tau(v) = -v$. Thus,
    \begin{align*}
        g(v, Y) 
        &= \frac{1}{2}(df(Y) - g'(v, d\tau(Y))) \\
        &= \frac{1}{2}(df(Y) - df(d\tau(Y))).
    \end{align*}
    Since $f$ is a real function, $f\circ \tau = -f$, which implies $df\circ d\tau = -df$. We deduce that  $g(v, Y)  = df(Y)$. Hence, $v = \grad_g(f)$, and $G(f, v)^\tau$ is non-empty.
    
    To show contractibility, observe that the space of all $\tau$-invariant Riemannian metrics is convex. The condition $g(v, \cdot) = df$ defines an affine subspace in the space of symmetric tensors. It follows that $G(f, v)^\tau$ is convex, and thus contractible.
\end{proof}

\begin{cor}
\label{cor:weaklycon}
    The space $\FV\Yt$ is weakly contractible. 
\end{cor}
\begin{proof}
    This follows just as in \cite[Corollary 5.20]{JTZ_naturality_mapping_class_groups}. Let $\mathcal{F}\Yt$ be the space of real sutured functions and $\mathcal{G}\Yt$ be the space of $\tau$-invariant Riemannian metrics. Both spaces are convex and thus contractible. Therefore, their product $\mathcal{F}\Yt \times \mathcal{G}\Yt$ is contractible.
    
    There is a Serre fibration
    \[
        \pi: \mathcal{F}\Yt \times \mathcal{G}\Yt \ra \FV\Yt
    \]
    defined by $\pi(f, g) = (f, \grad_g(f))$. By \Cref{prop:fv}, the map $\pi$ is surjective and has contractible fibers.
    
    We consider the long exact sequence of homotopy groups associated with the fibration $\pi$. Since the total space $\mathcal{F}\Yt \times \mathcal{G}\Yt$ and the fibers $G(f,v)^\tau$ are contractible, the homotopy groups of the base space $\FV\Yt$ also vanish. Therefore, $\FV\Yt$ is weakly contractible.
\end{proof}

\section{Constructing Heegaard diagrams}\label{sec:constructing hd}

This section follows \cite[Section 6]{JTZ_naturality_mapping_class_groups}, translating generic singularities of \Cref{subsec:codim 1 birfurcations} and \Cref{subsec:codim 2 birfurcations} to real Heegaard diagrams: generic gradients produce real Heegaard diagrams, codimension-1 singularities give real Heegaard moves between them, and codimension-2 singularities give loops of moves. 

\begin{defn}
    Let $(f, v) \in \FV_{\le2}\Yt$. We define a partition $C(f) = C_{01}(f) \cup C_{23}(f)\cup C_{fix}(f)$ as follows: a point $p \in C(f)$ belongs to $C_{01}(f)$ if one of the following holds:
    \begin{enumerate}[(1)]
        \item $p \in C_0(f) \cup C_1(f)$;
        \item $p$ is an index 0-1 birth-death appearing in a $\Z/2$-orbit;
        \item $p$ is an index 1-2 birth-death appearing in a $\Z/2$-orbit, $(f, v)$ is codimension-2 of type (B1), and $\cI(p_1) = \cI(p_2) = 1$;
        \item $p$ is an index 1-2 birth-death appearing in a $\Z/2$-orbit, $(f, v)$ is codimension-2 of type (B2) or (B3), and $\cI(\overline{p}) = 1$;
        \item $p$ is an index 1-2 birth-death appearing in a $\Z/2$-orbit, $(f, v) \in \FV_2\Yt$ is of type (C), and $q$ is another birth-death appearing in a $\Z/2$-orbit with $q \neq \tau(p)$ and $f(p) < f(q)$;
        \item $(f, v)$ is codimension-2 of type (D1), and $p$ is an index 1-0-1, 0-1-0, or a 1-2-1 birth-death-birth. 
        \item If $p$ is an index 1-2 birth-death appearing in a $\Z/2$-orbit and $(f, v)$ is of codimension-1 of type (NH), $p$ may be placed in either $C_{01}(f, v)$ or $C_{23}(f, v)$. 
    \end{enumerate}
    A point $p \in C(f)$ belongs to $C_{fix}$ if one of the following holds:
    \begin{enumerate}[(1)]
        \item $p$ is an index 1-2 birth-death appearing in a trivial orbit.
        \item $(f, v)$ is codimension-2 of type (D2), and $p$ is an $A_4$ singularity. 
    \end{enumerate}
    The remaining points are placed in $C_{23}(f).$
\end{defn}

\begin{defn}
\label{def:separable}
    A pair $(f, v) \in \FV_{\le2}\Yt$ is \emph{separable} if:
    \begin{enumerate}
        \item $C_{fix}(f) = \emptyset$;
        \item it does not involve a tangency of any 1-dimensional strong manifolds, i.e., is not codimension-2 of type (A3), (A4), (B4), (B5), (E1), or (E2);
        \item if it is codimension-2 of type (C), has birth-death singularities at $(\Z/2)\cdot p$ and $(\Z/2)\cdot q$), and at least one of $(\Z/2)\cdot p$ and $(\Z/2)\cdot q$ is a $\Z/2$-orbit, then $f(p) \neq f(q)$.
    \end{enumerate}
\end{defn}

\begin{defn}
    Let $(f, v) \in \FV\Yt$. We say that a properly embedded surface $\Sigma$ separates $(f, v)$ if
    \begin{enumerate}
        \item $\Sigma \pitchfork v$;
        \item $Y = Y_- \cup Y_+$ so that $Y_- \cap Y_+ = \Sigma$;
        \item $C_{01}(f, v) \sub Y_-$ and $C_{23}(f,v) \sub Y_+$.
    \end{enumerate}
\end{defn}
Let $\Sigma(f, v)$ be the set of separating surfaces for $(f, v)$ and let $\Sigma^R(f, v)$ be the set of real-invariant separating surfaces for $(f, v)$, i.e. surfaces $\Sigma$ such that $\tau(\Sigma) = \Sigma$ and $\tau$ is orientation-reversing on $\Sigma$. 

\begin{rem}
    We emphasize that gradients with degenerate critical points on the fixed point set $C$ cannot be separated by real surfaces. Indeed, if $\Sigma$ is real and separating, it necessarily contains $C$. Therefore, if $(f, v)$ has a degenerate critical point on $C$, the condition that $\Sigma \pitchfork v$ necessarily fails. 
\end{rem}

\begin{defn}
    Let $(f, v) \in \FV_{\le 2}\Yt$ be separable. Then, define $\Gamma_{01}(f,v)$ to be the CW complex:
    \begin{align*}
        \Gamma_{01}(f,v) = \bigcup_{p \in C_{01}(f,v)} W^s(p).
    \end{align*}
    We define $\Gamma_{23}(f, v) = \tau(\Gamma_{01}(f,v))$. 
\end{defn}

As $\Gamma_{01}(f, v)$ and $\Gamma_{23}(f, v)$ are symmetric, we will always work with the former, and simply write $\Gamma(f, v)$.

\begin{lem}\label{lem:real splittings contractible}
    If $(f, v) \in \FV_{\le 2}\Yt$ is a separable pair, then the space $\Sigma^R(f, v)$ is non-empty and contractible. 
\end{lem}
\begin{proof}
    The corresponding fact for $\Sigma_\pm(f, v)$ is \cite[Proposition 6.7]{JTZ_naturality_mapping_class_groups}. The existence of such surfaces is constructive: one may simply take the boundary of a small regular neighborhood of $\Gamma(f, v)$. To see the space is contractible, they note that if $(f,v)$ is separable and $\Sigma \in \Sigma^R(f, v)$, then $\Sigma$ intersects every flow line in a point. It follows that any other separating surface $\Sigma'$ determines a homeomorphism $d_{\Sigma, \Sigma'}: \Sigma \ra \R$ where $d_{\Sigma, \Sigma'}(x)$ is the time $t$ such that $\phi_t(x) \in \Sigma'$ (here, $\phi_s$ is the flow of $v$). This establishes a homeomorphism $\Sigma(f, v) \cong C^\infty(\Sigma)$, and $C^\infty(\Sigma)$ is clearly contractible.

    In the real setting, we can employ a similar strategy. To prove existence, as above we let $\Sigma_0$ be the boundary of a small regular neighborhood of $\Gamma(v, f)$. The surface $\S_0$ is not invariant, but by reparametrizing $v$, we can ensure that $\tau(\S_0)$ is equal to $\phi_1(\S_0)$, where $\phi_t$ is the flow of $v$. It will follow that $\phi_{1/2}(\S_0)$ is an invariant separating surface.
    
    More precisely, let $W = Y \smallsetminus (N(\Gamma(v, f))\cup \tau(N(\Gamma(v, f))))$. Since $\S_0$ is separating, so is $\t(\S_0)$, and therefore, $W$ is diffeomorphic to a product. Furthermore, we have a projection 
    \begin{align*}
        \pi_{\S_0}: W \ra \S_0
    \end{align*}
    by assigning to a point $x$ the unique point in $\S_0$ which is contained in the same flow line. Define $T: \S_0 \ra \R_{\ge 0}$ by the condition that $\phi_{T(x)}(x) \in \t(\S_0)$. Extend $T$ to a function $\widetilde T: W \ra \R_{>0}$ by requiring that $\widetilde T$ is constant on flow lines: i.e., define $\widetilde T = T\circ\pi_{\S_0}$. Note that if $x \in \S_0$ flows to a point $y$ in $\t(\S_0)$, then 
    \begin{align*}
        \t(x) = \t(\phi_{-T(x)}(y)) = \phi_{T(x)}(\t(y)).
    \end{align*}
    It follows that $T(\t(x)) = T(x)$, and hence, $\widetilde T$ is $\t$-invariant as well. 

    Therefore, we can define a vector field $w_x = T(x) \cdot v_x$, with flow $\psi_t$. As we have only rescaled $v$, separating surfaces for $w$ are also separating surfaces for $v$. Now, for any $x \in \S_0$, we have that $\psi_1(x) = \t(x)$, since
    \begin{align*}
        \psi_{1}(x) = \phi_{T(x)}(x) \in \t(\S_0).
    \end{align*}
    We define $\S := \psi_{1/2}(\S_0).$ Then $\S$ is $\tau$-invariant:
    \begin{align*}
        \t(\S) &= \t(\psi_{1/2}(\S_0)) \\
        &= \psi_{-1/2}(\t(\S_0)) \\
        &= \psi_{-1/2}(\psi_{1}(\S_0))\\ 
        &= \psi_{1/2}(\S_0) = \S.
    \end{align*}
    Hence, $\S$ is a separating surface for $w$, and therefore one for $v$ as well. 
    
    The second claim follows by an identical argument as in the unreal case. If $\Sigma'$ were another separating surface, then the homeomorphism $d_{\Sigma, \Sigma'}$ is a real-invariant map $\Sigma \ra \R$, and gives a homeomorphism from $\Sigma^R(f, v)$ to the space of real invariant functions on $\Sigma$, i.e. $C^\infty_R(\Sigma)$, which is contractible.
\end{proof}


\subsection{Codimension-0}

Given a separable pair $(f, v) \in \FV_{\le2}\Yt$, we have seen that there is a contractible space of real Heegaard splittings. Of course, the process is reversible: given a real Heegaard diagram $(\Sigma, \bm \alpha, \bm \beta, \tau)$, there is an essentially unique pair $(f, v) \in \FV_0\Yt$ which realizes it. This follows exactly as in \cite[Section 6.2]{JTZ_naturality_mapping_class_groups}, so our treatment will be terse. 

Given a separable pair $(f, v) \in \FV_{\le2}\Yt$, we can choose a separating real Heegaard diagram $\Sigma \in \Sigma^R(f, v)$. If $(f, v)$ is Morse-Smale, for each $p \in C_1(f)$, we have that $W^u(p)\cap \Sigma$ is an embedded circle, as is $\tau(p) \in C_2(f)$. These curves all intersect transversely. 

\begin{defn}
    Say $(f, v) \in \FV_0\Yt$ and $\Sigma \in \Sigma^R(f, v)$. Then, we define $H(f, v, \Sigma) = (\Sigma,\bm \alpha, \bm\beta, \tau|_{\Sigma})$ to be the Heegaard diagram by taking the alpha curves to be $W^u(p)\cap \Sigma$ for  $p \in C_1(f)$ and the beta curves to be $W^s(q)\cap \Sigma$ for  $q \in C_2(f)$. Of course, the alpha curves are interchanged with the beta curves under $\tau|_{\Sigma}.$
\end{defn}

\begin{rem}
    We will usually abuse notation, and write $\t$ for the involution on $\S$ rather than $\t_\S.$
\end{rem}

Of course, $|C_1(f)|$ may be larger than the genus of $\Sigma$, in which case we obtain an \emph{overcomplete diagram} \cite[Definition 6.13]{JTZ_naturality_mapping_class_groups}. These diagrams specify handle decompositions of $\Yt$ with extra 0- and 3-handles. To obtain a Heegaard diagram in the traditional sense, one starts with the graph $\Gamma_-(f, v)$, removes the vertices $p \in C_1(f)$, and chooses a spanning tree $T$ of $\Gamma_-(f, v)$. We define $H(f, v, \Sigma, T)$ to be the Heegaard diagram obtained by choosing the alpha curves to be $W^u(p) \cap \Sigma$ for $p \in C_1(f)$ and such that $W^s(p)$ is not an edge of $T$; the symmetry determines the beta curves. Different choices of $T$ give rise to real-equivalent diagrams. 

The process can be reversed: given a real diagram $(\Sigma, \alphas, \betas,\t)$, standard Morse theory allows one to construct a pair $(f, v) \in \FV_0\Yt$ which satisfies $H(f, v, \Sigma) = (\Sigma, \alphas, \betas, \t)$. For more details, see  \cite[Section 6.2]{JTZ_naturality_mapping_class_groups}; the discussion there can be applied \emph{mutatis mutandis} in the real setting.

\subsection{Codimension-1}\label{subsec:gradient-to-HD-codim1}

In this subsection, we prove equivariant analogues of the various isotopy extension lemmas of \cite[Section 6.3]{JTZ_naturality_mapping_class_groups}.

\begin{lemma}[Symmetric Isotopy Extension, cf. Lemma 6.19 in \cite{JTZ_naturality_mapping_class_groups}]\label{lem:sym_isotopy_extension}
Suppose $\{\mathcal{H}_t = (\Sigma_{t},\alpha_{t},\beta_{t},\t_t):t\in I\}$ is a smooth 1-parameter family of real Heegaard diagrams in $(Y, \tau)$ such that $\alpha_t \pitchfork \beta_t$.
Then there exists a $\Z/2$-equivariant isotopy $d_t: Y\times I\rightarrow Y$ (i.e., $d_t \circ \tau = \tau \circ (d_t\times \id)$) such that $d_{t}(\mathcal{H}_0)=\mathcal{H}_t$. The space of such equivariant isotopies is contractible.
\end{lemma}

\begin{proof}
This follows just as in \cite[Lemma 6.19]{JTZ_naturality_mapping_class_groups}: we construct an isotopy by integrating a time-dependent vector field on $Y \times I$. We consider the submanifolds $\Sigma_* = \bigcup_{t\in I} \Sigma_t\times \{t\}$, $\bm \alpha_* = \bigcup_{t\in I} \bm \alpha_t\times \{t\}$, and $\bm \beta_* = \bigcup_{t\in I} \bm \beta_t\times \{t\}$. Let $\tilde{\tau}(x,t) = (\tau(x), t)$ be the involution on $Y \times I$. We seek a vector field $\nu$ on $Y \times I$ which is $\tilde{\tau}$-invariant ($d\tilde{\tau}(\nu)=\nu$), positively transverse to the horizontal foliation $\cF$ (e.g., $dt(\nu)=1$), and tangent to the submanifolds $\Sigma_{*}, \alpha_{*},$ and $\beta_{*}$.

We construct $\nu$ by successive extensions, ensuring invariance using the averaging technique: $\nu = (\omega + d\tilde{\tau}(\omega))/2$. We first define $\nu$ on $T(\alpha_{*}\cap\beta_{*})$ transverse to $\cF$, which can clearly be chosen to be invariant. Then, we extend $\nu$ over $T\alpha_{*}$ in any way so that it is positively transverse to $\cF$; $\nu$ is extended over $T\beta_{*}$ by pushing forward the vector field on $T\alpha_{*}$. We then extend to $T\Sigma_{*}$: since $\Sigma_*$ is invariant, we can choose any extension and then apply the averaging trick. Finally, extend $\nu$ to $Y \times I$ invariantly.

The flow of the invariant vector field $\nu$ generates the required equivariant isotopy $D$. The space of such invariant vector fields is convex, hence contractible.
\end{proof}

\begin{lemma}[Symmetric version of Lemma 6.20 in \cite{JTZ_naturality_mapping_class_groups}]\label{lem:sym_6.20}
Let $\{\mathcal{H}_{t}\}$ and $\{\mathcal{H}_{t}^{\prime}\}$ be 1-parameter families of real Heegaard diagrams (possibly overcomplete) connecting $\mathcal{H}_{0}$ and $\mathcal{H}_{1}$. If the two families are homotopic relative to their endpoints through real diagrams, then the induced equivariant diffeomorphisms $d_{1}$, $d_{1}^{\prime}$: $Y\rightarrow Y$ are equivariantly isotopic through equivariant diffeomorphisms mapping $\mathcal{H}_{0}$ to $\mathcal{H}_{1}$.
\end{lemma}

\begin{proof}
Let $\mathcal{H}_{t,s}$ be the homotopy on $I \times I$. Equip $Y \times I \times I$ with the involution $\hat{\tau}(x, t, s) = (\tau(x), t, s)$. We construct a $\hat{\tau}$-invariant vector field $\nu$ such that $\nu(t)=1, \nu(s)=0$, and $\nu$ is tangent to the submanifolds defined by $\mathcal{H}_{t,s}$. This construction follows the generalization of Lemma~\ref{lem:sym_isotopy_extension} to the 2-parameter setting, using the averaging technique to ensure invariance.

The flow of $\nu$ defines a 1-parameter family of equivariant diffeomorphisms $g_s: Y \to Y$, obtained by flowing from $t=0$ to $t=1$ at fixed $s$. By construction, $g_s(\mathcal{H}_0) = \mathcal{H}_1$. We know that $g_0$ is equivariantly isotopic to $d_1$, and $g_1$ to $d_1'$. The family $\{g_s\}$ provides the required equivariant isotopy.
\end{proof}

\begin{lemma}[Symmetric version of Lemma 6.21 in \cite{JTZ_naturality_mapping_class_groups}]\label{lem:sym_6.21}
Let $\{(f_{t},v_{t})\in\mathcal{FV}_{0}(Y, \tau):t\in I\}$ be a family of real gradient-like vector fields. Let $\Sigma_{i}\in\Sigma^R(f_{i},v_{i})$. Let $T^0$ be a spanning tree for $\Gamma(f_0, v_0)$. Note that $\Gamma(f_t, v_t)$ is an isotopy which takes $T^0$ to a spanning tree $T^1$ for $\Gamma(f_1,v_1)$. Let $\mathcal{H}_0, \mathcal{H}_1$ be the corresponding real diagrams. Then there is an induced equivariant diffeomorphism $d:\mathcal{H}_{0}\rightarrow\mathcal{H}_{1}$, equivariantly isotopic to the identity. The space of such equivariant diffeomorphisms is path-connected.
\end{lemma}

\begin{proof}
We first show there exists a smooth family of symmetric surfaces $\{\Sigma_t\}$ connecting $\Sigma_0$ and $\Sigma_1$ such that $\Sigma_t \in \Sigma^R(f_t, v_t)$. This follows from the fact that the fiber $\Sigma^R(f,v)$ is contractible by \Cref{lem:real splittings contractible}. We can adapt the patching argument from the original Lemma 6.21 in \cite{JTZ_naturality_mapping_class_groups}: when interpolating between two symmetric surfaces, the difference function is anti-invariant, ensuring that the interpolated path consists of symmetric surfaces. 

This yields a smooth family of symmetric diagrams $\{\mathcal{H}_t\}$. Applying Lemma~\ref{lem:sym_isotopy_extension}, we obtain an equivariant isotopy $D=(d_t)$ realizing this family. The end point $d=d_1$ is the desired equivariant diffeomorphism.

If $\{\Sigma'_t\}$ is another such family, we construct a homotopy of symmetric surfaces $\Sigma_{t,u}$ connecting them, utilizing the convexity of the fibers $\Sigma^R(f_t, v_t)$. This yields a homotopy of symmetric diagrams $\mathcal{H}_{t,u}$. Applying Lemma~\ref{lem:sym_6.20} shows that the induced diffeomorphisms are equivariantly isotopic.
\end{proof}

\begin{lemma}[Symmetric version of Lemma 6.23 in \cite{JTZ_naturality_mapping_class_groups}]\label{lem:sym_6.23}
Let $\{(f_{t},v_{t})\in\mathcal{FV}_{1}\Yt:t\in I\}$ be a 1-parameter family of separable gradients. Let $\Sigma_{i}\in\Sigma^R(f_{i},v_{i})$ be real Heegaard surfaces for $i \in \{0,1\}$. This family induces an equivariant diffeomorphism d: $\Sigma_{0}\rightarrow\Sigma_{1}$ that is well-defined up to equivariant isotopy.
\end{lemma}

\begin{proof}
This follows just as in the proof of \Cref{lem:sym_6.21}.
\end{proof}

\begin{lemma}[Symmetric version of Lemma 6.24 in \cite{JTZ_naturality_mapping_class_groups}]\label{lem:sym_6.24}
Let $A: I\rightarrow\mathcal{FV}_{0}^\tau(Y)$ be a loop based at $(f_0, v_0)$. Let $\mathcal{H}$ be a symmetric diagram for $(f_0, v_0)$. The loop A induces an equivariant diffeomorphism $d: \mathcal{H}\rightarrow\mathcal{H}$. If A is null-homotopic in $\mathcal{FV}_{0}^\tau(Y)$, then $d$ is equivariantly isotopic to $\id_{\mathcal{H}}$.
\end{lemma}

\begin{proof}
Let $L: I \times I \to \mathcal{FV}_0^\tau(Y)$ be a null-homotopy of $A$. Since the fiber $\Sigma^R(f,v)$ is contractible, we can lift $L$ to a 2-parameter family of symmetric surfaces $\Sigma_{t,u}$, which is constant equal to the base surface $\Sigma$ when $t\in\{0,1\}$ or $u=1$.

This induces a family of symmetric diagrams $\mathcal{H}_{t,u}$, providing a homotopy between the path $P_0=\{\mathcal{H}_{t,0}\}$ (induced by $A$) and the constant path $P_1=\{\mathcal{H}_{t,1}\}$.

Let $d_0, d_1$ be the induced equivariant diffeomorphisms. We have $d_0=d$ and $d_1$ is induced by the constant path, so $d_1$ is equivariantly isotopic to $\id_{\mathcal{H}}$. By Lemma~\ref{lem:sym_6.20}, $d_0$ and $d_1$ are equivariantly isotopic. Therefore, $d$ is equivariantly isotopic to $\id_{\mathcal{H}}$.
\end{proof}

Next, we translate the codimension 1 bifurcations to moves between real Heegaard diagrams. There are moves appearing in both $\Z/2$ and trivial orbits.

\begin{figure}[h]
        \def\svgwidth{.9\linewidth}
\begingroup%
  \makeatletter%
  \providecommand\color[2][]{%
    \errmessage{(Inkscape) Color is used for the text in Inkscape, but the package 'color.sty' is not loaded}%
    \renewcommand\color[2][]{}%
  }%
  \providecommand\transparent[1]{%
    \errmessage{(Inkscape) Transparency is used (non-zero) for the text in Inkscape, but the package 'transparent.sty' is not loaded}%
    \renewcommand\transparent[1]{}%
  }%
  \providecommand\rotatebox[2]{#2}%
  \newcommand*\fsize{\dimexpr\f@size pt\relax}%
  \newcommand*\lineheight[1]{\fontsize{\fsize}{#1\fsize}\selectfont}%
  \ifx\svgwidth\undefined%
    \setlength{\unitlength}{779.52755906bp}%
    \ifx\svgscale\undefined%
      \relax%
    \else%
      \setlength{\unitlength}{\unitlength * \real{\svgscale}}%
    \fi%
  \else%
    \setlength{\unitlength}{\svgwidth}%
  \fi%
  \global\let\svgwidth\undefined%
  \global\let\svgscale\undefined%
  \makeatother%
  \begin{picture}(1,0.30909091)%
    \lineheight{1}%
    \setlength\tabcolsep{0pt}%
    \put(0,0){\includegraphics[width=\unitlength,page=1]{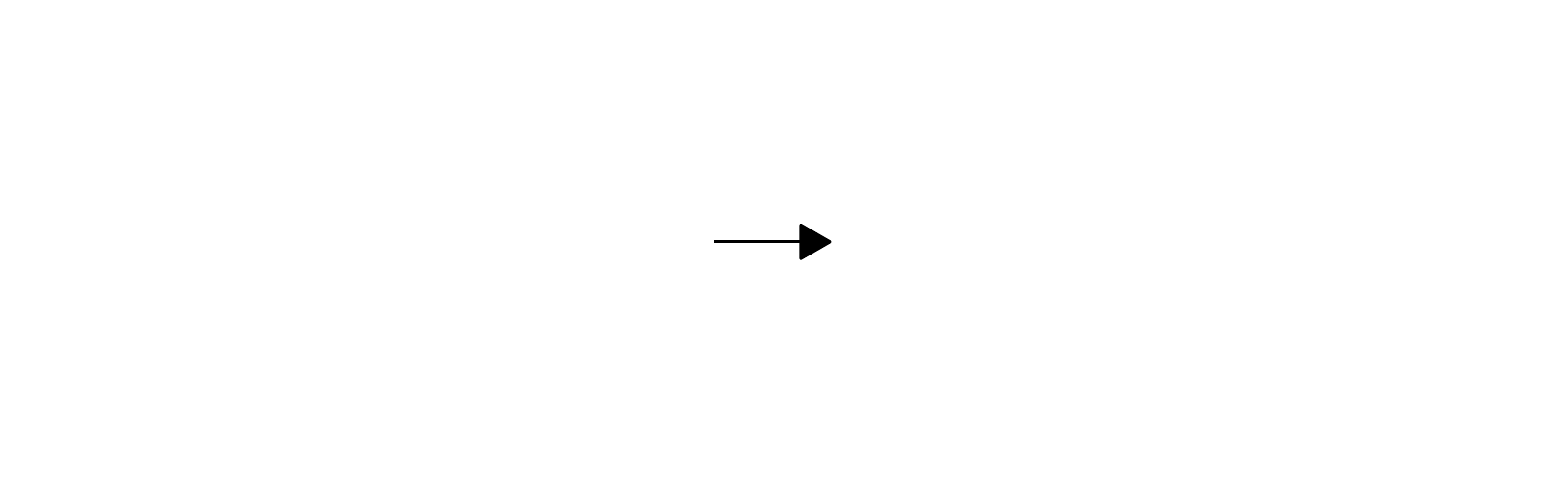}}%
    \put(0.48624023,0.18433567){\color[rgb]{0,0,0}\makebox(0,0)[t]{\smash{\begin{tabular}[t]{c}{\small$(2,3)$}\end{tabular}}}}%
    \put(0.05134849,0.1929394){\color[rgb]{0,0,0}\makebox(0,0)[t]{\smash{\begin{tabular}[t]{c}{\small$\alpha_1$}\end{tabular}}}}%
    \put(0,0){\includegraphics[width=\unitlength,page=2]{generalized_handleslide.pdf}}%
    \put(0.15101516,0.17283334){\color[rgb]{0,0,0}\makebox(0,0)[t]{\smash{\begin{tabular}[t]{c}{\small$\alpha_2$}\end{tabular}}}}%
    \put(0.05897278,0.10840449){\color[rgb]{0,0,0}\makebox(0,0)[t]{\smash{\begin{tabular}[t]{c}{\small$a$}\end{tabular}}}}%
    \put(0,0){\includegraphics[width=\unitlength,page=3]{generalized_handleslide.pdf}}%
    \put(0.61322726,0.1929394){\color[rgb]{0,0,0}\makebox(0,0)[t]{\smash{\begin{tabular}[t]{c}{\small$\alpha_1'$}\end{tabular}}}}%
  \end{picture}%
\endgroup%

            \caption{A generalized real handleslide of type $(m, n) = (2,3)$.}
        \label{fig:generalized_handleslide}
        \end{figure}

\begin{defn}\label{def:(m,n) handleslide}
    A real Heegaard diagram $(\Sigma', \bm \alpha', \bm \beta', \t')$ is obtained from $(\Sigma, \bm \alpha, \bm \beta, \t)$ by a \emph{generalized real handleslide of type $(m,n)$} if there are curves $\alpha_1, \alpha_2 \in \alphas$, a curve $\alpha_1'\in \alphas'$, and an embedded arc $a \in \Sigma$ so that 
    \begin{enumerate}
        \item $\bm \alpha'_1 \smallsetminus \alpha'  = \bm \alpha\smallsetminus \alpha_1$;
        \item $\partial a \sub \alpha_2$ and its interior is disjoint from $\bm \alpha$;
        \item there is a regular neighborhood $N$ of $\alpha_2 \cup a$ so that $\partial N = \alpha_1 \cup \alpha_1' \cup \Tilde{\alpha}_2$, where $ \Tilde{\alpha}_2$ is a small parallel pushoff of $\alpha_2$, and the interior of $N$ is disjoint from $\bm \alpha\cup \bm \alpha' \smallsetminus \{\alpha_2\}$ and 
        \item if $\alpha_2 \smallsetminus \partial a = \alpha_2^0 \cup \alpha^1_2$ so that $\alpha_2^0 \cup a$ is parallel to $\alpha_1$ and $\alpha_2^1 \cup a$ is parallel to $\alpha_1'$, then $|\alpha_2^0 \cap \bm \beta| = m$ and $|\alpha_2^1 \cap \bm \beta| = n$.
    \end{enumerate}
    Symmetric conditions hold for the beta curves. See \Cref{fig:generalized_handleslide}.
\end{defn}

\begin{figure}[h]
        \def\svgwidth{.8\linewidth}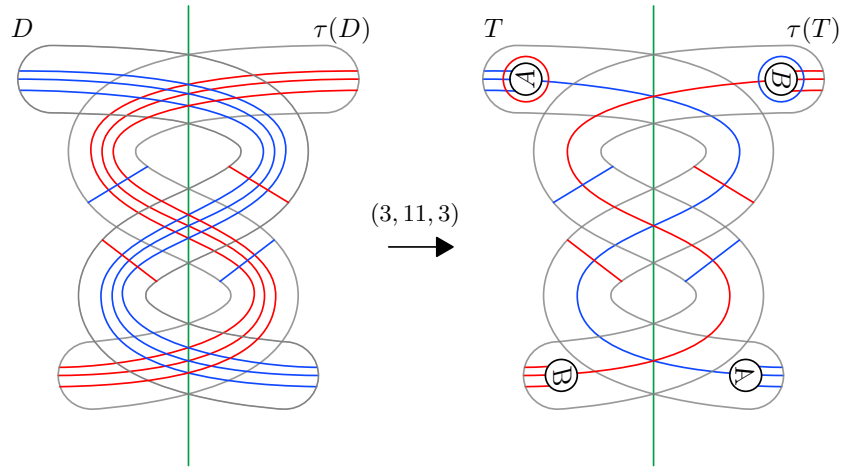
            \caption{A $\Z/2$-stabilization of type $(k, \ell, c) = (3,11,3)$.}
        \label{fig:Z2_stab}
        \end{figure}

\begin{defn}\label{def: Z/2 (k,l) stabilization}
    A real Heegaard diagram $(\Sigma', \bm \alpha', \bm \beta', \t')$ is obtained from $(\Sigma, \bm \alpha, \bm \beta, \t)$ by a \emph{generalized $\Z/2$-stabilization of type $(k,\ell, c)$} if there is a disk $D \sub \Sigma$ and a punctured torus $T \sub \Sigma'$ as well as curves $\alpha \in \bm \alpha'$ and $\beta \in \bm\beta'$ so that:
    \begin{enumerate}
        \item $D \not = \t'(D)$ and the intersection of $\partial D$ and $\t(\partial D)$ is transverse. Similarly for $T$ and $\t'(T)$;
        \item $\Sigma \smallsetminus D = \Sigma' \smallsetminus T$;
        \item $\bm \alpha \smallsetminus D = \bm \alpha' \smallsetminus T$, $\bm \beta \smallsetminus D = \bm \beta' \smallsetminus T$;
        \item $\bm \alpha \cap D$ and $\bm \beta \cap D$ consist of $\ell$ and $k$ arcs, respectively, and each component of $\bm \alpha \cap D$ intersects each component of $\bm \beta \cap D$ transversely in a single point;
        \item $\alpha$ and $\beta$ are contained in $T$ and they intersect in a single point;
        \item $(\bm \alpha \smallsetminus \alpha) \cap T$ consists of $\ell$ parallel arcs, each intersecting $\beta$ transversely;
        \item  $(\bm \beta \smallsetminus \beta) \cap T$ consists of $k$ parallel arcs, each intersecting $\alpha$ transversely;
        \item the intersection $|D\cap \t(D)|$ consists of $c$ components; 
        \item for each component of $\bm \alpha \cap D$, there is a corresponding component of $\bm \alpha'\cap T$ with the same endpoints, and similarly for the beta curves.
    \end{enumerate}
    These of course impose symmetric relations on $\t(D)$, $\t'(T),$ $\t'(\alpha),$ and $\t'(\beta)$. See \Cref{fig:Z2_stab}.
\end{defn}
\begin{rem}\label{rem: failure of separablility}
    We emphasize that $D$ and $\t(D)$ are generally not disjoint (similarly for $T$ and $\t'(T)$); it is for this reason we need to keep track of a third parameter, $c$.  See \Cref{fig:Z2_stab} for an example. 

    We also note that this operation technically depends on the choice of disk $D$; by replacing $D$ and $T$ with $\t(D)$ and $\t'(T)$, the same description above would specify a $\Z/2$-stabilization of type $(\ell, k, c)$ rather than of type $(k, \ell, c)$. For the sake of consistency, we will always label the 1-handle of $T$ by $A$ and $\tau'(T)$ by $B$ (or, in the case of multiple $\Z/2$-stabilization, the label for the $T$ handle will always come first alphabetically). Therefore, $k$ denotes the number of beta curves which crosses over the tube labeled by $A$.
\end{rem}

Now, let us turn to stabilizations appearing in trivial orbits.

\begin{defn}\label{def: trivial k-stabilizations}
    A real Heegaard diagram $(\Sigma', \bm \alpha', \bm \beta', \t')$ is obtained from $(\Sigma, \bm \alpha, \bm \beta, \t)$ by a \emph{generalized $\{1\}$-stabilization of type $(k)$} if there is a disk $D \sub \Sigma$ and a punctured torus $T \sub \Sigma'$ as well as curves $\alpha \in \bm \alpha'$ and $\beta \in \bm\beta'$ so that:
    \begin{enumerate}
        \item $D = \t(D)$, $T = \t'(T)$, $\t'(\alpha) = \beta$;
        \item $\Sigma \smallsetminus D = \Sigma' \smallsetminus T$;
        \item $\bm \alpha \smallsetminus D = \bm \alpha' \smallsetminus T$ and $\bm \beta \smallsetminus D = \bm \beta' \smallsetminus T$;
        \item $\bm \alpha \cap D$ consists of $k$ arcs and each component of $\bm \alpha \cap D$ intersects each component of $\bm \beta \cap D$ transversely in a single point;
        \item $\alpha$ and $\beta$ are contained in $T$ and intersect in a single point;
        \item $(\bm \alpha \smallsetminus \alpha) \cap T$ consists of $k$ parallel arcs, each intersecting $\beta$ transversely;
        \item for each component of $\bm \alpha \cap D$, there is a corresponding component of $\bm \alpha'\cap T$ with the same endpoints, and similarly for the beta curves.
    \end{enumerate}
    These, of course, impose symmetric relations the beta curves. See \Cref{fig:triv_stab} for an example.
\end{defn}

\begin{figure}[h]
        \def\svgwidth{.8\linewidth}
\begingroup%
  \makeatletter%
  \providecommand\color[2][]{%
    \errmessage{(Inkscape) Color is used for the text in Inkscape, but the package 'color.sty' is not loaded}%
    \renewcommand\color[2][]{}%
  }%
  \providecommand\transparent[1]{%
    \errmessage{(Inkscape) Transparency is used (non-zero) for the text in Inkscape, but the package 'transparent.sty' is not loaded}%
    \renewcommand\transparent[1]{}%
  }%
  \providecommand\rotatebox[2]{#2}%
  \newcommand*\fsize{\dimexpr\f@size pt\relax}%
  \newcommand*\lineheight[1]{\fontsize{\fsize}{#1\fsize}\selectfont}%
  \ifx\svgwidth\undefined%
    \setlength{\unitlength}{779.52755906bp}%
    \ifx\svgscale\undefined%
      \relax%
    \else%
      \setlength{\unitlength}{\unitlength * \real{\svgscale}}%
    \fi%
  \else%
    \setlength{\unitlength}{\svgwidth}%
  \fi%
  \global\let\svgwidth\undefined%
  \global\let\svgscale\undefined%
  \makeatother%
  \begin{picture}(1,0.26545455)%
    \lineheight{1}%
    \setlength\tabcolsep{0pt}%
    \put(0,0){\includegraphics[width=\unitlength,page=1]{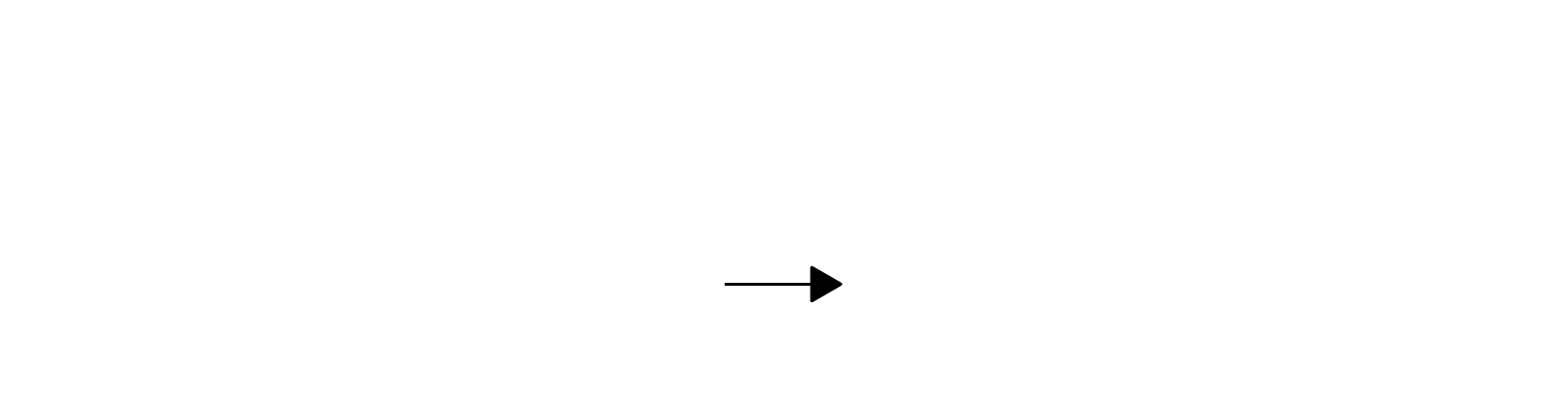}}%
    \put(0.49321661,0.11387999){\color[rgb]{0,0,0}\makebox(0,0)[t]{\smash{\begin{tabular}[t]{c}{\small$(4)$}\end{tabular}}}}%
    \put(0,0){\includegraphics[width=\unitlength,page=2]{triv_stab.pdf}}%
    \put(0.69231307,0.1144573){\color[rgb]{0,0,0}\makebox(0,0)[t]{\smash{\begin{tabular}[t]{c}{$F$}\end{tabular}}}}%
    \put(0,0){\includegraphics[width=\unitlength,page=3]{triv_stab.pdf}}%
    \put(0.83876959,0.13525377){\color[rgb]{0,0,0}\rotatebox{-180}{\makebox(0,0)[t]{\smash{\begin{tabular}[t]{c}{$\reflectbox{$F$}$}\end{tabular}}}}}%
    \put(0,0){\includegraphics[width=\unitlength,page=4]{triv_stab.pdf}}%
    \put(0.07738964,0.22595014){\color[rgb]{0,0,0}\makebox(0,0)[t]{\smash{\begin{tabular}[t]{c}{$D$}\end{tabular}}}}%
    \put(0.61617759,0.22595014){\color[rgb]{0,0,0}\makebox(0,0)[t]{\smash{\begin{tabular}[t]{c}{$T$}\end{tabular}}}}%
  \end{picture}%
\endgroup%

            \caption{A generalized $\{1\}$-stabilization of type $(4)$.}
        \label{fig:triv_stab}
        \end{figure}

\begin{defn}\label{def:foot_swap}
    A real Heegaard diagram $(\Sigma', \bm \alpha', \bm \beta', \tau')$ is obtained from $(\Sigma, \bm \alpha, \bm \beta, \t)$ by a \emph{crossover of type $(k,\ell)$} if there are twice-punctured disks $P \sub \Sigma$ and $P'\sub \Sigma'$ as well as distinguished curves $\alpha \in \bm \alpha$, $\alpha' \in \bm \alpha'$ such that:
    \begin{enumerate}
        \item $\Sigma \smallsetminus P = \Sigma' \smallsetminus P'$;
        \item $\bm \alpha \smallsetminus P = \bm \alpha' \smallsetminus P'$, $\bm \beta \smallsetminus P = \bm \beta' \smallsetminus P'$;
        \item $P$ and $P'$ are fixed setwise by $\t$ and $\t'$ (and therefore necessarily intersect $C$ in an arc);
        \item $\partial P = S \cup A \cup B$ where $S$ is fixed set-wise by $\t$, $A$ and $B$ are exchanged, and $\alpha$ and $\beta = \t(\alpha)$ are push-offs of $A$ and $B$ (similarly for $\partial P' = S'\cup A'\cup B'$);
        \item $(\bm \alpha \smallsetminus \alpha) \cap P$ consists of two collections of parallel arcs from $S$ to $B$ avoiding $C$ of size $k$ and $\ell$ respectively, each intersecting $\beta$ transversely once. The arcs $(\bm \beta \smallsetminus \beta) \cap P$ are symmetric;
        \item $(\bm \alpha' \smallsetminus \alpha') \cap P'$ consists of two collections of parallel arcs from $S'$ to $B'$  of size $k$ and $\ell$,  respectively, each intersecting $C$ transversely once, and intersecting $\beta'=\tau(\alpha')$ transversely once. The arcs $(\bm \beta' \smallsetminus \beta') \cap P'$ are symmetric;
        \item for each arc of $(\bm \alpha \smallsetminus \alpha) \cap P$, there is a corresponding component of $(\bm \alpha' \smallsetminus \alpha') \cap P'$ with the same endpoints, and similarly for the beta curves.
    \end{enumerate}
    See \Cref{fig:crossover} for an example. 
\end{defn}

\begin{figure}[h]
        \def\svgwidth{.8\linewidth}
\begingroup%
  \makeatletter%
  \providecommand\color[2][]{%
    \errmessage{(Inkscape) Color is used for the text in Inkscape, but the package 'color.sty' is not loaded}%
    \renewcommand\color[2][]{}%
  }%
  \providecommand\transparent[1]{%
    \errmessage{(Inkscape) Transparency is used (non-zero) for the text in Inkscape, but the package 'transparent.sty' is not loaded}%
    \renewcommand\transparent[1]{}%
  }%
  \providecommand\rotatebox[2]{#2}%
  \newcommand*\fsize{\dimexpr\f@size pt\relax}%
  \newcommand*\lineheight[1]{\fontsize{\fsize}{#1\fsize}\selectfont}%
  \ifx\svgwidth\undefined%
    \setlength{\unitlength}{779.52755906bp}%
    \ifx\svgscale\undefined%
      \relax%
    \else%
      \setlength{\unitlength}{\unitlength * \real{\svgscale}}%
    \fi%
  \else%
    \setlength{\unitlength}{\svgwidth}%
  \fi%
  \global\let\svgwidth\undefined%
  \global\let\svgscale\undefined%
  \makeatother%
  \begin{picture}(1,0.26181818)%
    \lineheight{1}%
    \setlength\tabcolsep{0pt}%
    \put(0,0){\includegraphics[width=\unitlength,page=1]{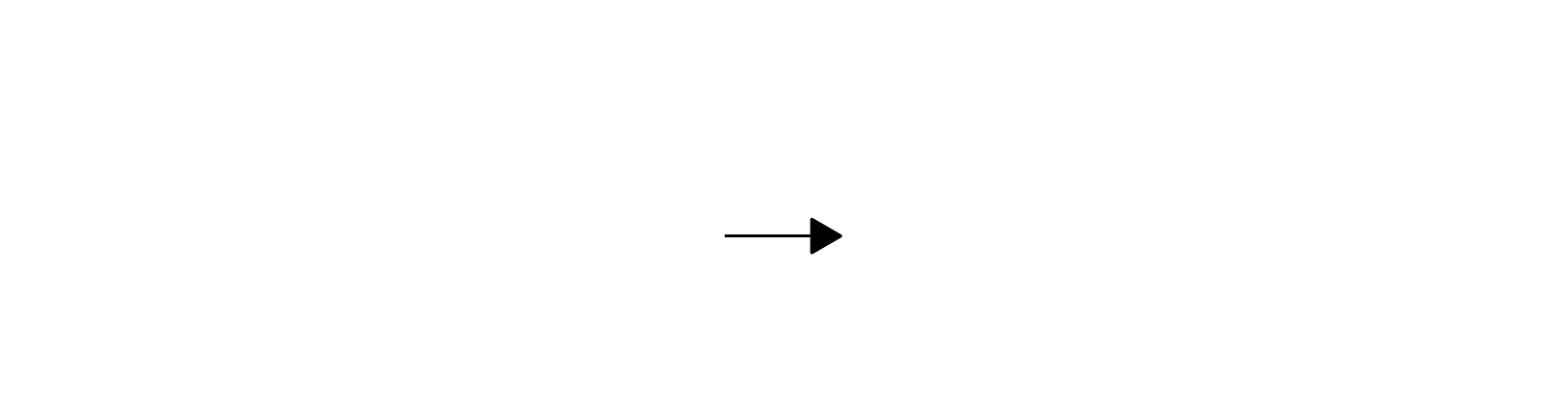}}%
    \put(0.49321661,0.1408874){\color[rgb]{0,0,0}\makebox(0,0)[t]{\smash{\begin{tabular}[t]{c}{\small$(3,2)$}\end{tabular}}}}%
    \put(0,0){\includegraphics[width=\unitlength,page=2]{crossover.pdf}}%
    \put(0.09266543,0.22074477){\color[rgb]{0,0,0}\makebox(0,0)[t]{\smash{\begin{tabular}[t]{c}{$P$}\end{tabular}}}}%
    \put(0.63145338,0.22074478){\color[rgb]{0,0,0}\makebox(0,0)[t]{\smash{\begin{tabular}[t]{c}{$P'$}\end{tabular}}}}%
  \end{picture}%
\endgroup%

            \caption{A crossover of type $(3,2)$.}
        \label{fig:crossover}
        \end{figure}

\begin{prop}\label{prop: codim 1 to heegaard moves}
    Let $ \{(f_t, v_t)\}_{t \in [-1,1]}$ be a generic 1-parameter family of real functions and real gradient-like vector fields on $\Yt$ that has a bifurcation at $t = 0$. If the bifurcation at $t = 0$ is not an index 1-2 birth-death nor a real quasi-transversal orbit of tangency of 1-dimensional submanifolds, then there is a real separating surface $\Sigma \in \Sigma^R(f_0,v_0)$. There is an $\ep > 0$ so that $\Sigma \pitchfork v_t$ for all $|t| < \ep.$ Moreover, for all $x \in (-\ep, 0)$ and $y \in (0, \ep)$ the following hold.

    \begin{enumerate}
        \item $\Sigma \in \Sigma^R(f_x, v_x) \cap \Sigma^R(f_y, v_y)$;
        \item If the bifurcation consists of an index 0-1 birth-death and an index 2-3 birth-death (which always occur together), the two diagrams differ by adding or removing a redundant alpha curve and its corresponding beta curve;
        \item If the bifurcation is a $\Z/2$-orbit of tangency between $W^u(p)$ and $W^s(q)$ for $p \in C_1(f)$ and $q \in C_2(f)$, the diagrams are related by an isotopy creating or canceling two symmetric pairs of intersection points away from $C$;
        \item If the bifurcation is a $\{1\}$-orbit of tangency between $W^u(p)$ and $W^s(\tau(p))$ for $p \in C_1(f)$ and $\tau(p) \in C_2(f)$, the diagrams are related either by an isotopy creating or canceling a pair of intersection points on $C$ or the diagrams are related by an isotopy in which a symmetric pair of intersection points are created or canceled near an intersection point on $C$;
        \item If the bifurcation is a $\Z/2$-orbit of tangency between $\Z/2 \cdot W^u(p)$ and $\Z/2 \cdot W^s(q)$ for $p$ and $q$ in $C_1(f)$, the diagrams are related by a generalized real handleslide. 
    \end{enumerate}
    
    If the bifurcation is an index 1-2 birth-death, then $\Sigma \in \Sigma^R(f_x, v_x)$ and there is a surface $\Sigma' \in \Sigma^R(f_y, v_y)$ such that diagrams $(\Sigma, \bm \alpha, \bm \beta) = H(f_x, v_x, \Sigma)$ and $(\Sigma', \bm \alpha', \bm \beta') = H(f_y, v_y, \Sigma')$ are related by either a generalized $\Z/2$-stabilization of type $(k, \ell, c)$ if there are $\ell$ flows from index 1 critical points into one of the degenerate critical points $p$ (and therefore symmetrically $\ell$ flows from $\tau(p)$ to index 2 critical points), $k$ flows from $p$ to critical points of index 2 (and symmetric flows from index 1 critical points to $\tau(p)$), and $c$ flow lines between $p$ and $\t(p)$ or a generalized $\{1\}$-stabilization of type $(m)$ if there are $m$ flows from index 1 critical points to $p$ (and $m$ symmetric flows from $p$ to index 2 critical points). Index 1-2 deaths are similar, with the roles of $x$ and $y$ reversed.

    If the bifurcation is a real quasi-transversal orbit of tangency from $W^u(\tau(p))$ to $W^s(p)$ for a critical point $p$ of index 1, then $\Sigma \in \Sigma^R(f_x, v_x)$ and there is a surface  $\Sigma' \in \Sigma^R(f_y, v_y)$ such that the diagrams $(\Sigma, \bm \alpha, \bm \beta) = H(f_x, v_x, \Sigma)$ and $(\Sigma', \bm \alpha', \bm \beta') = H(f_y, v_y, \Sigma')$ are related by a crossover of type $(k,\ell)$.
\end{prop}
\begin{proof}
    Much of this follows the same strategy as \cite[Proposition 6.28]{JTZ_naturality_mapping_class_groups}, so we sketch the argument, pointing out where adaptations must be made in the symmetric setting.

    In Case (2), we consider the case of a birth, so that $f_0$ has a degenerate critical orbit $\{p, \tau(p)\}$ which splits into orbits $\{p^t_0,\tau(p^t_0)\}$ and $\{p^t_1, \tau(p^t_1)\}$ for $p^t_0 \in C_0(f)$ and $p^t_1\in C_1(f)$ and $t > 0$. For $\ep$ small, $\Sigma \pitchfork v_t$ for all $|t| < \ep$ and $p^t_0$ and $p^t_1$ are contained in $Y_-$ while $\tau(p^t_0)$ and $\tau(p^t_1)$ are contained in $Y_+$. The attaching set $(\bm \alpha', \bm \beta')$ contains exactly one more $\alpha$- and beta curve than $(\bm \alpha, \bm \beta)$, and the remaining curves are isotopic. The new $\alpha$-circle is isotopic to $W^{uu}(p_0)\cap \Sigma$ and the new $\beta$-circle is isotopic to $W^{ss}(\tau(p_0))\cap \Sigma$.

    In Cases (3) and (4), we simply consider $(\Sigma_t, \bm \alpha_t,  \bm \beta_t, \t_t) = H(f_t, v_t, \Sigma)$ for $t \in [x,y]$. At $t = 0$, we see a generic tangency. Tangencies in $\Z/2$-orbits behave exactly as in the unreal case, and tangencies in $\{1\}$-orbits as described in \Cref{subsec:codim 1 birfurcations}.
    
    Similarly, in Case (5), we fix $p$ and $q$ in $C_1(f_0)$. For all $|t| < \ep$, the surface $\Sigma$ is contained in $\Sigma^R(f_t, v_t)$. At $t = 0$, the union of the alpha curves $W^u(p) \cap \Sigma$ and  $W^u(q) \cap \Sigma$ has a regular neighborhood which is a pair of pants. The boundaries are isotopic to the $\alpha$-attaching curves before and after a generalized $(m, n)$-handleslide. The symmetric slide of beta curves occurs simultaneously.

    We turn now to index 1-2 birth-death bifurcations. We will address the case of births. Suppose that $p$ is a degenerate critical point appearing in a $\Z/2$-orbit which splits into $p^t_0 \in C_1(f_t)$ and $p^t_1 \in C_2(f_t)$; we will write $q$ for the symmetric degenerate critical point which splits into $q^t_0 \in C_1(f_t)$ and $q^t_1 \in C_2(f_t)$. Let us place $p$ in $C_{01}(f_0,v_0)$ and $q$ in $C_{23}(f_0,v_0)$. In this case, $(f_0, v_0)$ is separable, so we may choose an invariant splitting surface $\Sigma$. For small $\ep > 0$, we have that $p^t_0$ and $p^t_1$ are both contained in $Y_-$ and $q^t_0$ and $q^t_1$ are both contained in $Y_+$. The unstable manifold $W^u(q_0^y)$ intersects $\Sigma$ in two points, as does the stable manifold $W^s(p_1^y)$. Define 
    \begin{align*}
        \Sigma' = \partial (Y_-\cup N(W^u(p_0^y)) \cup W^s(q^y_1)).
    \end{align*}
    It is easy to see that the attaching sets for $\Sigma'$ are given exactly by the description of generalized stabilizations, as in \cite{JTZ_naturality_mapping_class_groups}. In the case of $\Z/2$-orbits of index 1-2 deaths, we simply reverse the roles of $x$ and $y$.

    Next, we consider the case that $p$ is a degenerate critical point which is fixed by the involution. As noted in \Cref{rem: failure of separablility}, the gradient $(f_0, v_0)$ is separable in the unreal setting (as in \cite[Definition 6.1]{JTZ_naturality_mapping_class_groups}), but not separable in the sense of Definition~\ref{def:separable} (that is, not separable by a real surface). Fix $\ep > 0$ small and choose $- \ep < x < 0$ and $0 < y < \ep$. Both $(f_x, v_x)$ and $(f_y, v_y)$ are separable, so we can choose real separating surfaces $\Sigma_x \in \Sigma^R(f_x, v_x)$ and $\Sigma_y \in \Sigma^R(f_y, v_y)$. Unlike the previous case, there is no clear way to construct $\Sigma_y$ from $\Sigma_x$. Nevertheless, we can choose an equivariant diffeomorphism of $Y$ which takes $\Sigma_x \smallsetminus N(p)$ to $\Sigma_y \smallsetminus N(p)$. In this way, we can construct separating surfaces which differ by a $\{1\}$-stabilization. Alternatively, we can choose an unreal splitting surface $\widetilde\Sigma_x$, and construct $\widetilde\Sigma_y$ as in the unreal case (or as we did for $\Z/2$-stabilizations). Then, as in the proof of \Cref{lem:real splittings contractible}, we can reparametrize $v$ and obtain real splitting surfaces $\Sigma_x$ and $\Sigma_y$ from $\widetilde\Sigma_x$ and $\widetilde\Sigma_y$ by pushing them forward by the gradient flow.

    Finally, we consider the case of a real quasi-transversal orbit of tangency $\gamma$ from $W^u(\tau(p_0))$ to $W^s(p_0)$ for a nondegenerate critical point $p^t$ of index 1 (see \Cref{fig:codim1_bifurcations} (f)). (We emphasize that, in the unreal case, such an orbit fails the quasi-transversal condition.) Define $\widetilde{\Sigma}_0$ to be the boundary of a small regular neighborhood of 
    \begin{align*}
        \bigcup \left \{ W^s(q): q \in C_0(f_0) \cup C_1(f_0) \smallsetminus \{p_0\} \right\}.
    \end{align*}
   This surface is not invariant, nor is it separating, as the index 1 critical point $p^t$ is contained in $Y_+$ .  

    Fix $\ep > 0$ small, and some $x \in (-\ep, 0)$. The manifold $W^s(p^x)$ intersects $\widetilde{\Sigma}_0$ in a pair of points $a^x$ and $b^x$. Let $\gamma_1^x$ and $\gamma_2^x$ be the associated flow lines. As $t$ varies from $-\ep$ to $0$, one of these points ($a^x$, say) traces out an arc $\tilde\lambda_1:= a^t$ for $t \in (-\ep, 0)$ in $\widetilde{\Sigma}_0$; likewise, for $t\in (0, \ep)$ we have a path $\tilde\lambda_2 :=a^t$ in $\widetilde{\Sigma}_0$. Provided $\ep$ is sufficiently small, we may assume that $\tilde\lambda_1$ and $\tilde\lambda_2$ do not intersect the stable or unstable manifolds of any critical points other than the unstable manifolds of $\tau(p_0)$. 
    
    Define a separating surface
    \begin{align*}
        \widetilde{\Sigma}_x = \partial ( Y_- \cup N(W^s(p^x))).
    \end{align*}
    The neighborhood $N(W^s(p^x))$ intersects $\Sigma$ in two disks, $D_{a^x}$ and $D_{b^x}$. Pick a neighborhood $N(\tilde\lambda_1)$ which contains $D_{a^x}$ (shrinking  $N(W^s(p^x))$ if necessary). By construction, we have that $\widetilde{\Sigma}_x = (\widetilde{\Sigma}_0 \smallsetminus(D_{a^x} \cup D_{b^x}))\cup \widetilde{A}$, where $\widetilde{A}$ is a tube (annulus) around $W^s(p^x)$. 
    
    In a similar manner, for $y \in (0,\ep)$, define
    \begin{align*}
        \widetilde{\Sigma}_y = \partial ( Y_- \cup N(W^s(p^y))),
    \end{align*}
    and so that $\widetilde{\Sigma}_y = (\widetilde{\Sigma}_0 \smallsetminus(D_{a^y} \cup D_{b^y}))\cup \widetilde A'$, where $\widetilde A'$ is a tube around $W^s(p^y)$. 

    Once again, we obtain invariant surfaces $\Sigma_x$ and $\Sigma_y$ as in the proof of \Cref{lem:real splittings contractible}, by pushing $\widetilde{\Sigma}_x$ and $\widetilde{\Sigma}_y$ forward under the flow of $v_t$ (reparametrizing $v_t$ if necessary). Similarly, the annuli $\widetilde A$ and $\widetilde A'$ can be pushed forward to annuli $A$ and $A'$ in $\Sigma_x$ and $\Sigma_y$. In the symmetric surfaces, the paths $\lambda_1$ and $\lambda_2$ limit to a point in $C$, and can be completed to a single connected path $\lambda$ in $\Sigma_x$ and $\Sigma_y$. A  neighborhood of this path meets both $A$ and $A'$: define the first pair of pants $P$ to be the union of this neighborhood together with $A$ and $\tau(A)$ and the second $P'$ to be a neighborhood of $\lambda$ together with $A'$ and $\tau(A')$. See \Cref{fig:crossover_splitting_surface}. 
    
    Let $\alpha_x := \Sigma_x \cap W^u(p^\mu)$ and $\beta_x = \t(\alpha_x)$. Define $\alpha_y$ and $\beta_y$ similarly. The directional manifold $D(p^\mu, \gamma_1^x) \sub W^u(p^\mu)$ intersects $\alpha_x$ in a pair of points; in particular, this gives a partition of the intersection points of $\alpha_x$ with $\bm \beta$ into sets of size $k$ and $\ell$. The situation is symmetric for $\Sigma_y$. In particular, the diagrams $H(f_{x}, v_x,\Sigma_x)$ and $H(f_{y}, v_y,\Sigma_y)$ differ by a crossover of type $(k, \ell).$ 
\end{proof}

\begin{figure}[h]
\def\svgwidth{.8\linewidth}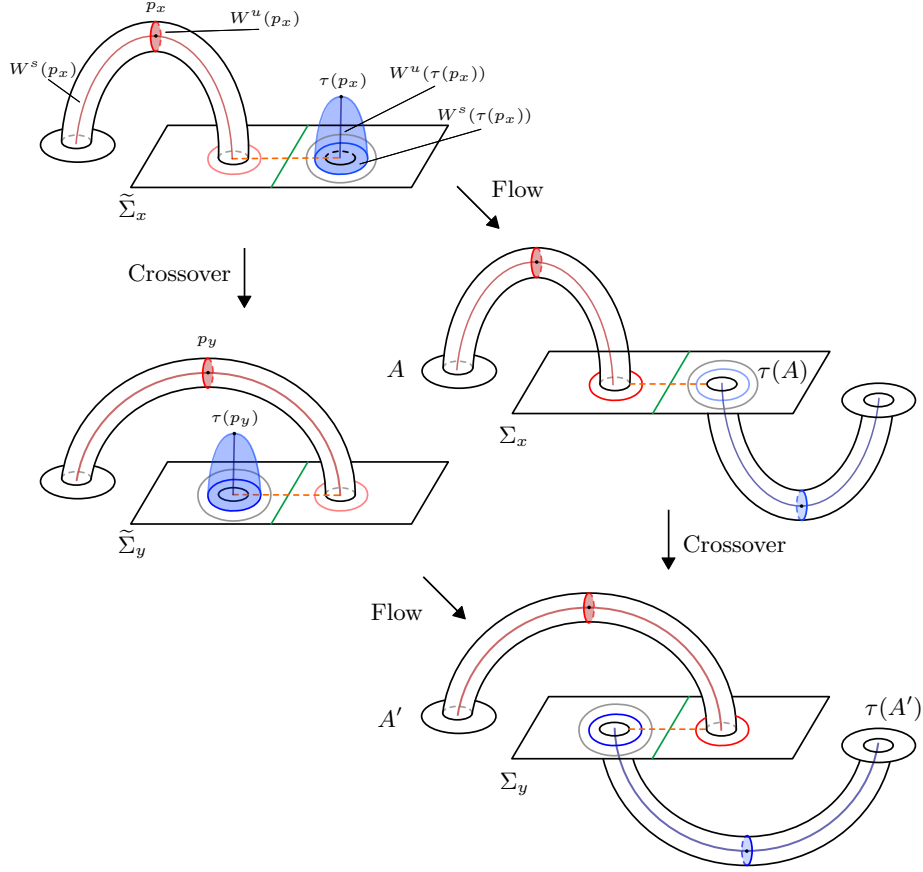
            \caption{The construction of the two invariant splitting surfaces on either side of a bifurcation associated to a real quasi-transversal orbit of tangency.}
        \label{fig:crossover_splitting_surface}
        \end{figure}

\begin{rem}
    A real Heegaard splitting of $\Yt$ determines a framing of the fixed set $C$. The main theorem of \cite{nagase} states that any two real Heegaard splittings which determine the same framing on $C$ are $\Z/2$-stably diffeomorphic. It is not hard to see that $\{1\}$-stabilizations change the framing on $C$ by $\pm 1$. We have defined $\{1\}$-stabilization of type $k$ to include both the positive and negative $\{1\}$-stabilizations we considered in \cite{guth_manolescu2025real}, and exchanging the roles of $\alpha$ and $\beta$ changes the sign of the stabilization. 
\end{rem}

Essentially the same is true in 2-parameter families near codimension-1 bifurcations.  

\begin{prop}
    Let $\{(f_\mu,v_\mu) \in \FV\Yt: \mu \in \R^2 \}$ be a generic 2-parameter family with a codimension-1 bifurcation at $\mu = 0$. Let $S$ be the (one-dimensional) stratum of the bifurcation passing through the origin. Then, there is an $\ep > 0$ so that $D_\ep^2 \smallsetminus S$ consists of two components $C_1$ and $C_2$, and for every $x \in C_1$ and $y \in C_2$, the same conclusion holds as in \Cref{prop: codim 1 to heegaard moves}.
\end{prop}
\begin{proof}
    This follows from essentially the same proof as \Cref{prop: codim 1 to heegaard moves}.
\end{proof}

\begin{defn}\label{def:gen-rectangle}
    Let $H_i = (\Sigma_i, [\bm \alpha_i], [\bm \beta_i], \t_i)$ be real isotopy diagrams for $1 \le i \le 4$. A \emph{generalized distinguished rectangle} is a graph
    \begin{align*}
        \begin{tikzcd}[ampersand replacement = \&]
            H_1 \ar[r,"e"]\ar[d,"f"] \& H_2\ar[d,"g"] \\
            H_3 \ar[r,"h"] \& H_4
        \end{tikzcd}
    \end{align*}
    satisfying any of the conditions (1)-(4) of \Cref{def:distinguished-rect} with ``simple $\cO$-stabilization'' replaced with ``generalized $\cO$-stabilization'', and allowing for one extra possibility:
    \begin{enumerate}
        \item[(5)] Both $e$ and $h$ are crossovers of type $(k, \ell)$ while $f$ and $g$ are both real handleslides. In this case, $\Sigma_1 = \Sigma_3$ and $\Sigma_2 = \Sigma_4$ and $f$ and $g$ are required to be the same map.
    \end{enumerate}
    We also allow overcomplete diagrams.
\end{defn}

\begin{prop}\label{prop:sym_prop_6.32}
    Suppose that $\{(f_\mu, v_\mu) \in \FV_{\le 1}\Yt: \mu \in \R^2\}$ is a generic 2-parameter family and let 
    \begin{align*}
        V_1 = \{\mu\in\R^2: (f_\mu, v_\mu) \in \FV_1\Yt\}
    \end{align*}
    be the codimension-1 bifurcation set. Let $a$ be an arc in $V_1$ with endpoints $\mu_0$ and $\mu_1$ and let $R$ be a small rectangular neighborhood of $a$. Let $b_0$ and $b_1$ be the boundary arcs of $R$ which are transverse to $V_1$ and oriented so that they have the same intersection sign with $V_1$. Let $\partial b_i = y_i - x_i$ for $i \in \{0,1\}$. As this family is generic, we may assume that at $x_i$ and $y_i$ the gradients $(f_{x_i}, v_{x_i})$ and $(f_{y_i}, v_{y_i})$ are real separable; fix such separating surfaces $\Sigma_{i} \in \Sigma^R(f_{x_i}, v_{x_i})$.

    Suppose we are given surfaces $\Sigma_i'\in \Sigma^R(f_{y_i},v_{y_i})$ for $i \in \{0,1\}$ such that:
    \begin{enumerate}
        \item $\Sigma_i'$ is obtained from $\Sigma_i$ by an $\cO$-stabilization if $(f_{\mu_i},v_{\mu_i})$ is an index 1-2 birth,
        \item $\Sigma_i'$ is obtained from $\Sigma_i$ by a crossover if $(f_{\mu_i},v_{\mu_i})$ is a non-quasi-transversal orbit of tangency,
        \item $\Sigma_i = \Sigma_i'$ otherwise.
    \end{enumerate}
    Then the isotopy diagrams $H_1 = [H(f_{x_0}, v_{x_0}, \Sigma_0)]$, $H_2 = [H(f_{y_0}, v_{y_0}, \Sigma_0')]$, $H_3 = [H(f_{x_1}, v_{x_1}, \Sigma_1)]$, and $H_4 = [H(f_{y_1}, v_{y_1}, \Sigma_1')]$, fit into a generalized distinguished rectangle
    \begin{align*}
        \begin{tikzcd}[ampersand replacement = \&]
            H_1 \ar[r,"e"]\ar[d,"f"] \& H_2\ar[d,"g"] \\
            H_3 \ar[r,"h"] \& H_4
        \end{tikzcd}
    \end{align*}
    of type (2) if $(f_{\mu_i}, v_{\mu_i})$ is a handleslide, of type (4) if $(f_{\mu_i}, v_{\mu_i})$ is an index 1-2 birth-death, or type (5) if $(f_{\mu_i}, v_{\mu_i})$ is a non-quasi-transversal orbit of tangency (compare to \Cref{def:gen-rectangle}). For other types of bifurcations, we have a rectangle with arrows $e$ and $h$ representing either the identity or a simultaneous index 0-1/ index 2-3 birth-death, and $f = g$ a diffeomorphism.
\end{prop}
\begin{proof}
    In the cases that the component of $V_1$ containing $a$ is not a $\{1\}$-stabilization or a crossover, this follows exactly as in \cite[Proposition 6.32]{JTZ_naturality_mapping_class_groups}; the gradients along $a$ are separable, and it is straightforward to verify that the relevant diagrams fit into distinguished rectangles. 

    In the case that $a$ corresponds to a $\{1\}$-stabilization, separability fails. Still, we can apply the same strategy as before. Fix parametrizations of the arcs $b_i = b_i(t)$ so that $b_i(t) = \mu_i$. For $\ep>0$ sufficiently small, $(f_{b_i(-\ep)}, v_{b_i(-\ep)})$ and $(f_{b_i(\ep)}, v_{b_i(\ep)})$ are both real separable, and we can choose real separating surfaces $\Sigma_{b_i(\pm\ep)}$ so that $\Sigma_{b_i(\ep)}$ is obtained from $\Sigma_{b_i(-\ep)}$ by attaching a tube around the descending manifold of the new index 1 critical point which is born at $\mu_i$. The previous lemma furnishes diffeomorphisms $\Sigma_{y_i} \cong \Sigma_{b_i(-\ep)}$ and $\Sigma_{x_i} \cong \Sigma_{b_i(\ep)}$. The same argument appearing in \cite[Proposition 6.32]{JTZ_naturality_mapping_class_groups} can be applied to show that this gives rise to a generalized distinguished rectangle.

    The same strategy applies to crossovers.
\end{proof}

Before moving to codimension-2 bifurcations, we note that Juh\'asz-Thurston-Zemke show that one may pass from over-complete diagrams to unreal diagrams ``without altering the relationships of the diagrams before and after the bifurcation in an essential way''. Their strategy holds in our case without change, since the behavior of index 2-3 birth-deaths is dictated by the behavior of index 0-1 birth-deaths by the symmetry. 

As in the codimension-0 case, sequences of real Heegaard moves can be translated to paths of real gradients. This also follows just as in \cite[Section 6.5]{JTZ_naturality_mapping_class_groups}, so we will not discuss it. 

\subsection{Codimension-2}\label{subsec:gradient-to-HD-codim2}

In this section, we turn to loops of real Heegaard moves. We shall consider generic 2-parameter families of real gradients $(f_\mu, v_\mu)$ on $\Yt$ with a codimension-2 singularity at $\mu = 0$. The bifurcation set will be denoted $S$. For $\ep>0$ sufficiently small, $(S \cap D_\ep^2) \smallsetminus \{0\}$ is a disjoint union of strata $S_1, \hdots, S_r$ which each have one boundary component at $0$ and the other on $S_\ep^1 = \partial D_\ep^2$. The regions between these strata will be called chambers $C_1, \hdots, C_r,$ labeled so that $C_i$ lies between $S_{i-1}$ and $S_i$ (here, we define $S_0 = S_r$). 

\begin{remark}
    Following \cite[Section 6.6]{JTZ_naturality_mapping_class_groups}, we will color-code these strata: a stratum $S_i$ will be colored purple if it corresponds to a real equivalence, black if it corresponds to a birth-death, green if it corresponds to a crossover, and dashed gray if it corresponds to a diffeomorphism.
\end{remark}

\begin{defn} (cf.\cite[Definition 6.36]{JTZ_naturality_mapping_class_groups})
    Say $\{(f_\mu, v_\mu): \mu \in \R^2\}$ is a generic 2-parameter family so that $(f_0, v_0) \in \FV_2\Yt$. For $\ep > 0$ as above, the \emph{link of the bifurcation at $0$} is an embedded polygonal curve $P \in D^2_\ep$ such that 
    \begin{enumerate}
        \item the bifurcation value $0$ lives in the interior of $P$;
        \item $P \pitchfork S$ and $|S_i \cap P| = 1$ for all $i$;
        \item each $C_i$ contains exactly 1 or 2 vertices of $P$.
    \end{enumerate}

    A \emph{surface-enhanced link} of the bifurcation at $p$ is a link $P$ together with a choice of surface $\Sigma_\mu \in \Sigma^R(f_\mu, v_\mu)$ for each vertex $\mu$ of $P$.
\end{defn}

If $C_i$ contains a single vertex of $P$, we denote it by $\mu_i$ and write $a_i$ for the edge of $P$ which intersects $S_i$. If $C_i$ contains two vertices, they are denotes $\mu_i$ and $\mu_i'$ ordered so that $\partial a_i = \mu_i' - \mu_i.$ In the case that $P$ is minimal, its vertices are $\mu_1, \hdots, \mu_r$ and its edges are $a_1, \hdots, a_r$. 

In the case that $P$ is not minimal, i.e., there is a chamber $C_i$ which contains two vertices of $P$, we will draw an extra ray in $C_i$ separating these two vertices (as in \cite{JTZ_naturality_mapping_class_groups}). We will label the edge between them by $a_i'$. On either side of the ray, we will choose two \emph{different} splitting surfaces which are related by a diffeomorphism. This is used in \cite{JTZ_naturality_mapping_class_groups} in the cases of simultaneous stabilizations and handleswaps, as separability fails. In our setting, this failure is much more frequent: indeed, we will need to include diffeomorphism strata any time there is a bifurcation that involves a $\{1\}$-stabilization or a crossover. For this reason, our links will usually not be minimal (see \cite[Definition 6.36]{JTZ_naturality_mapping_class_groups}). We will not emphasize this, departing from the convention in \cite{JTZ_naturality_mapping_class_groups}.

\begin{thm}\label{thm:codim2-to-heegaard-loops}
    Suppose $\{(f_\mu, v_\mu): \mu \in \R^2\}$ is a generic 2-parameter family so that $(f_0, v_0) \in \FV_2\Yt$. For all $\ep > 0$ sufficiently small, there is a surface-enhanced link $P \sub D_\ep^2$ of the bifurcation at $0$ so that the following hold:
    \begin{enumerate}
        \item For $i \in \{1, \hdots, r\}$, there is a point $x_i \in \partial a_i$ so that $\Sigma_{x_i}\pitchfork v_\mu$ for all $\mu \in a_i$;
        \item Consecutive isotopy diagrams $H_i$ and $H_{i+1}$ are related by the move corresponding to the type of stratum $S_i$;
        \item Each edge $a_i'$ induces a diffeomorphism $d_i:H_i \ra H_i'$ isotopic to the identity as in \Cref{lem:sym_6.21}.
    \end{enumerate}
    The surface-enhanced links for the bifurcations discussed in \Cref{subsec:codim 2 birfurcations} are described in the proof.
\end{thm}

\begin{proof}
    As in \cite[Theorem 6.37]{JTZ_naturality_mapping_class_groups}, we ignore strata corresponding to index 0-1/ index 2-3 birth-deaths as well as those corresponding to tangencies between unstable manifolds of index 1 critical points and stable manifolds of index 2 critical points. 

    In the cases that $(f_0, v_0)$ is separable, all surfaces can be constructed from some $\Sigma \in \Sigma^R(f_0, v_0)$. When $(f_0, v_0)$ is not separable, we will construct the surfaces on either side of stabilization and crossover strata as in the proof of \Cref{prop: codim 1 to heegaard moves}. In these cases, for each stratum $S_i$, we will pick a small arc $a^{i \ra i+1}$ transverse to $S_i$ and add vertices $\partial a^{i \ra i+1} = \{\mu_i, \mu_{i+1}'\}$ to the link of the singularity. Since the gradient is not separable, the Heegaard splittings $\Sigma_{\mu_i}$ and $\Sigma_{\mu_i'}$ need not be identical, but are simply diffeomorphic according to \Cref{lem:sym_6.21}. Therefore, in these cases we are forced to add in a diffeomorphism stratum relating the two diagrams. Since these strata appear so frequently, we will not always draw them explicitly in the accompanying figures. \\

    \noindent \textbf{Type (A) bifurcations:} Bifurcations of type (A) involve simultaneous orbits of tangency, i.e. simultaneous handleslides and crossovers. Bifurcations of type (A1) and (A2) only involve real handleslides, and therefore necessarily appear in $\Z/2$-orbits, and the analysis is identical to that of \cite[Theorem 6.37]{JTZ_naturality_mapping_class_groups}. For this reason, we will focus on describing the links for bifurcations of types (A3) and (A4) and their corresponding Heegaard diagrams. When we refer to directional manifolds, we will typically not specify which flowline we are considering, as it should be clear from the context. The only bifurcation which involves both directional manifolds of the same critical point is $(A3)$ and in this case we will take care in distinguishing the two.

    \textbf{(A3):} Suppose $p^\mu_1$ and $p^\mu_2$ are critical points (not necessarily distinct) of index 1 and there is a $\{1\}$-orbit of tangency $\gamma_1^0$ from $\tau(p_1^0)$ to $p_1^0$ as well as one $\gamma_2^0$ from $\tau(p_2^0)$ to $p_2^0$. The bifurcation diagram $S$ consists of two smooth curves which intersect transversely at 0. Suppose there are $(k_1 + \ell_1)$ and $(k_2 + \ell_2)$ flows of index 2 critical points into $p_1^\mu$ and $p_2^\mu$, respectively (recall that the flow lines into $p_1^\mu$ are partitioned into two kinds by $D(p_1^\mu, \gamma_1^\mu)$ and flows into $p_2^\mu$ are partitioned by $D(p_2^\mu, \gamma_2^\mu)$). The strata $S_1$ and $S_3$ correspond to a crossover of type $(k_1,\ell_1)$ while $S_2$ and $S_4$ correspond to crossovers of type $(k_2, \ell_2)$. (See \Cref{def:foot_swap} and \Cref{fig:crossover} for an example of a crossover.) Note that when $p_1^\mu = p_2^\mu$, the flows $\gamma_1^0$ and $\gamma_2^0$ are still distinct, and we make use of both directional manifolds to produce the two different partitions.

    The link of the singularity consists of eight vertices, $\mu_i, \mu_i' \in C_i$ for $i \in \{1, 2, 3, 4\}$ obtained as follows. Choose short arcs $a^{i\ra i+1}$ from $C_i$ to $C_{i+1}$ transverse to $S_i$ (here, $i, i + 1$ are taken to be elements of $\Z/4$). We define the vertices of the link by $\partial a^{i\ra i+1} = \{\mu_i, \mu_{i+1}'\}$. As in the proof of \Cref{prop: codim 1 to heegaard moves}, we construct Heegaard diagrams $H_i, H_i'$ for $i \in \{1, 2, 3, 4\}$ so that $H_i$ and $H_{i+1}'$ differ by a crossover; the type of crossover is determined by the directional manifold of $p_1^\mu$ for $\mu \in a^{1\ra 2}$ and $p_2^\mu$ for $\mu \in a^{4\ra 1}$. See \Cref{fig:hd_a3}.

        \begin{figure}[h]
        \def\svgwidth{.8\linewidth}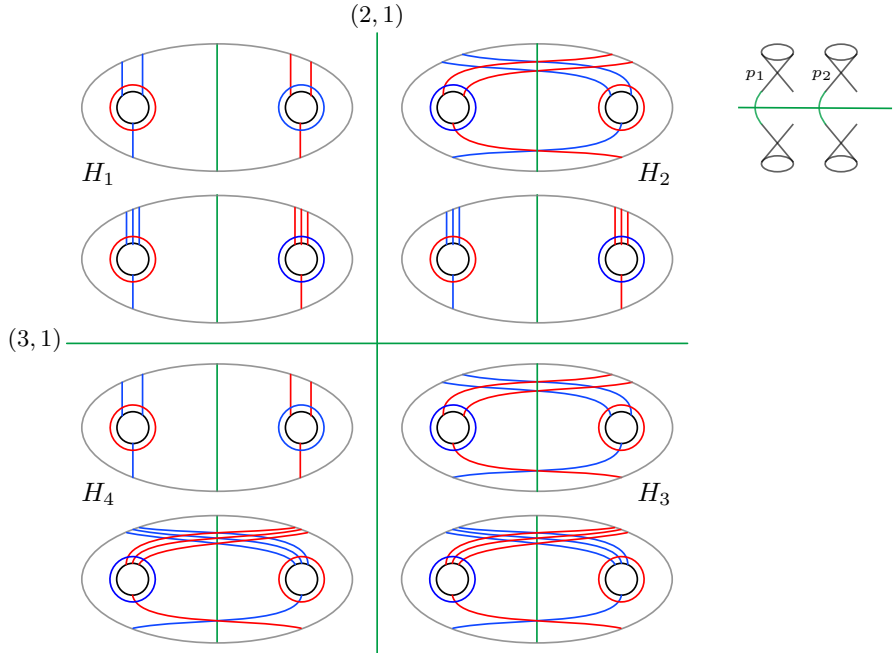
            \caption{Commuting crossovers, i.e. the link of a bifurcation of type (A3) with $(k_1, \ell_1) = (2, 1)$ and $(k_2, \ell_2) = (3, 1)$. We see two commuting crossovers.}
        \label{fig:hd_a3}
        \end{figure}

    \textbf{(A4a):} Here, $p_1^\mu$ is a hyperbolic singularity of index 1 with a $\{1\}$-orbit of tangency from $\tau(p_1^0)$ to $p_1^0$; additionally, $p_2^\mu$ and $p_3^\mu$ are hyperbolic singularities of index 1 distinct from $p_1^\mu$, and there is an orbit of tangency from $p_2^0$ to $p_3^0$. Assume there are $k+\ell$ flows out of $p_1^\mu$ to index 2 critical points. The bifurcation diagram consists of two smooth curves intersecting at the origin. The Heegaard diagrams are constructed just as in case (A3), though things are slightly simpler in this context; we need only choose six vertices, as we can take $\Sigma_1$ and $\Sigma_4$ to be a common surface (likewise for $\Sigma_2$ and $\Sigma_3$); we choose small arcs $a^{1 \ra 2}$ and $a^{3\ra 4}$ transverse to $S_1$ and $S_3$; these arcs determine four Heegaard splittings $H_1$, $H_2'$, $H_3$, $H_4'$ and we define $H_2 := H_3$ and $H_4:=H_1$ and choose diffeomorphisms relating $H_1$ with $H_4'$ and $H_3$ with $H_2'$. The stratum $S_4$ is a crossover of type $(k, \ell)$, where the partition of the flows out of $p_1^\mu$ to index 2 critical points comes from the directional manifolds of $p_1^\mu$ for $\mu \in a^{2 \ra 3}$, and $S_2$ is a crossover of type $(k + 1, \ell)$, with the partition coming from the directional manifold of $p_1^\mu$ for $\mu \in a^{4 \ra 1}$. See \Cref{fig:hd_a4a}.

        \begin{figure}[h]
        \def\svgwidth{.8\linewidth}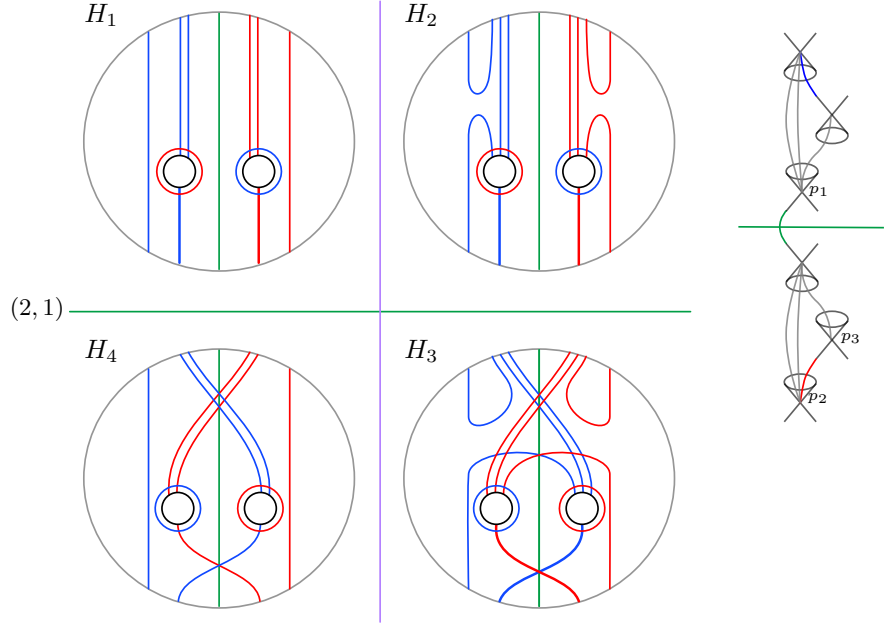
            \caption{A commutation of a real handleslide and a crossover, or the link of a bifurcation of type (A4a).}
        \label{fig:hd_a4a}
        \end{figure}

    \textbf{(A4b):} In this case, $p_1^0$ is a hyperbolic singularity of index 1 with a $\{1\}$-orbit of tangency from $\tau(p_1^0)$ to $p_1^0$; additionally, $p_2^\mu$ and $p_3^\mu$ are hyperbolic singularities of index 1 (not necessarily distinct from $p_1^0$). The bifurcation diagram consists of two smooth curves intersecting transversely at 0; the strata $S_1$ and $S_3$ correspond to crossovers while strata $S_2$ and $S_4$ are real handleslides. The Heegaard diagrams are constructed just as in case (A4a). $H_1$ and $H_4$ differ by a real handleslide, as do  $H_2$  and $H_3$, while $H_1$ and $H_2'$ differ by a crossover, as do $H_3$ and $H_4'$. Recall that in a crossover, there are two distinguished tubes (annuli) which are neighborhoods of curves $\alpha_0$ and $\beta_0 = \tau(\alpha_0)$, and in a crossover the feet (attaching regions) of the tubes are exchanged. In \Cref{fig:hd_a4b}, we show both feet of these annuli, as one foot is involved in the crossover, while the other is involved in the real handleslide.

        \begin{figure}[h]
        \def\svgwidth{.8\linewidth}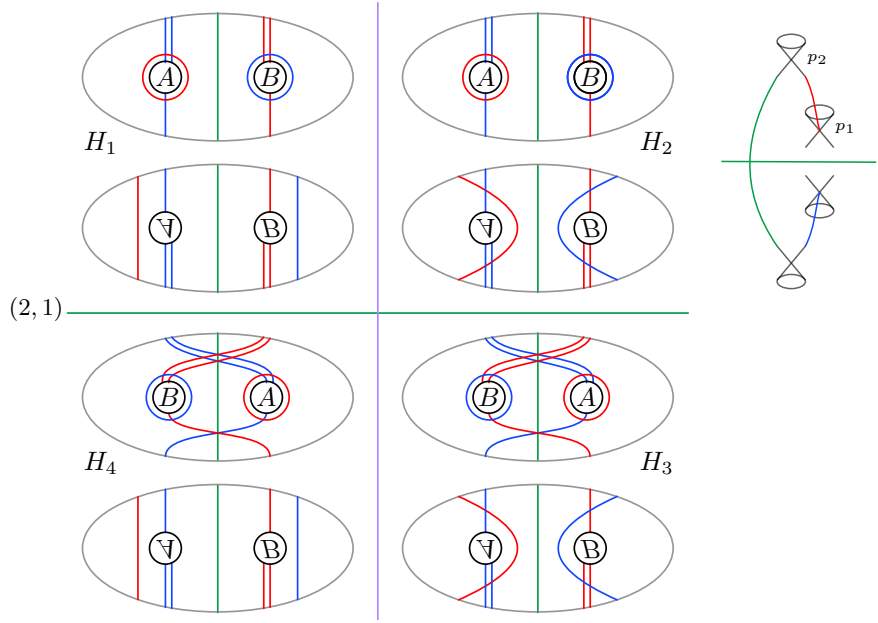
            \caption{A commutation of a real handleslide and a crossover corresponding to the link of a bifurcation of type (A4b).}
        \label{fig:hd_a4b}
        \end{figure}
    
    \textbf{(A4c):} In this case, $p_1^0$ and $p_2^0$ are both hyperbolic singularities and there is a $\Z/2$-orbit of tangency between $W^s(p_2^0)$ and $W^u(p_1^0)$ as well as a $\{1\}$-orbit of tangency between $W^u(\tau(p_1^0))$ and $W^s(p_1^0)$. Suppose that there are $k_1+\ell_1$ flows out of $p_1^0$ to index 2 critical points partitioned into two parts by the directional manifold of $p_2^\mu$. Similarly, assume there are $k_2+\ell_2$ flows from $p_2^0$ to index 2 critical points. On the right side of the diagram, there is a generalized real handleslide of type $(k_1, \ell_2)$, creating or annihilating $(\ell_2 - k_1)$ flows from $p_1^\mu$ to index 2 critical points.
        
    In this case, the bifurcation diagram consists of at least $5 + (\ell_1 - k_2) \ge 5$ curves meeting at the origin. The strata $S_1$ and $S_3$ correspond to handleslides and $S_2$ corresponds to a crossover of type $(k_1, \ell_2 + (\ell_1 - k_2))$.  Pick a point $x \in C_1$, and define $H_1$ using a separating surface $\Sigma_1 \in \S^R(f_x, v_x)$. Since $S_1$ is a handleslide stratum, we can define $\S_2 = \S_1$ as the separating surface for $H_2$ as well. To construct the remaining splitting surfaces, we use the arc $a^{2 \ra 3}$ to construct a surface $\S_3$ so that $H_3$ is obtained from $H_2$ by a crossover of type $(k_1, \ell_2 + (\ell_1 - k_2))$; the type of crossover is determined by the directional manifolds associated to  the arcs $a^{2 \ra 3}$ and $a^{5 \ra 1}$. Of the remaining strata, $S_i$ and $S_{i+1}$ are crossover strata of types $(k_2, \ell_2)$ and $(k_1, \ell_1)$, and the others are real handleslides. For the diagrams $H_4, \hdots, H_{i+1}$, we take $\S_3$ to be the splitting surface. As above, we choose an arc $a^{i-1 \ra i}$ and construct $\S_i$ so that it is obtained as a crossover of $\S_{i-1}' \cong \S_{i-1}$, producing diagrams $H_{i-1}'$ and $H_i$ (and inserting diffeomorphism strata). Likewise, we choose an arc $a^{i\ra i+1}$ to construct diagrams $H_{i}'$ and $H_{i+1}$. We take $H_{i+1}', H_{i+2}, \hdots H_{5 + (\ell_1 - k_2)}$ to have splitting surface $\S_1$. \\
    
    \begin{figure}[h]
        \def\svgwidth{.8\linewidth}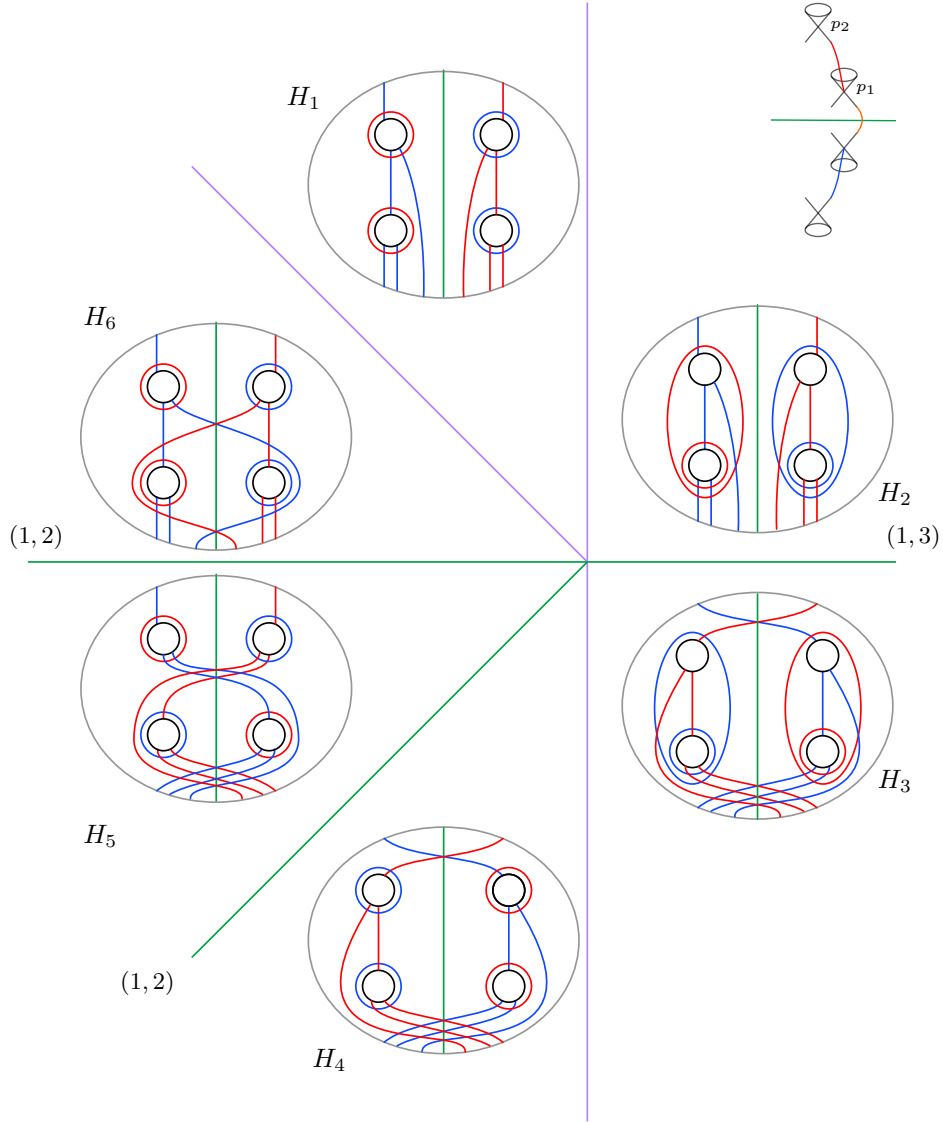
            \caption{The link of a bifurcation of type (A4c). In this example $k_1 =  k_2  = 1$, and $\ell_1 = \ell_2 =2$.}
        \label{fig:hd_a4c}
    \end{figure}

    \noindent \textbf{Type (B) bifurcations:} Bifurcations of type (B) involve stabilizations. The link $P$ contains two stabilization strata $S_1$ and $S_3$, and there is a single stratum $S_2$ on the stabilized side, which is either a real handleslide or a crossover stratum. 

    For each bifurcation, we fix a separating surface $\Sigma \in \Sigma^R(f_\mu, v_\mu)$ for a parameter $\mu$ close to $0$ on the unstabilized side of the parameter space. Let $p$ be the index 1-2 birth-death singularity which breaks into critical points (or orbits) $q^\nu_1 \in C_1(f_\nu)$ and $q^\nu_2 \in C_2(f_\nu)$. We will assume that the surface on the stabilized side, $\Sigma_\nu \in \Sigma^R(f_\nu, v_\nu)$, is obtained by attaching a tube around $W^u(q^\nu_2)$ (though we could just as well have considered the surface obtained by attaching a tube around $W^s(q^\nu_1)$). 

    \textbf{(B1):}
    There are four strata; $S_1$ and $S_3$ are stabilizations while $S_2$ and $S_4$ are real handleslides. By the non-degeneracy condition (ND-2) (see \cite[Definition 5.11]{JTZ_naturality_mapping_class_groups}, \Cref{rem: nondegeneracy conditions}), there is a 1-dimensional invariant submanifold $W^{uu}(p_2^0) \sub W^u(p_2^0)$ which partitions flows from $p_2^0$ to $p_0$ into two parts; we assume there are $k_1$ and $k_2$ such flows. 
        
    Say there are $m$ flows to $p$ from index 1 critical points and $\ell$ flows from $p$ to index 2 critical points (if $p$ is a $\{1\}$-stabilization, i.e. $\tau(p) = p$, then $m = \ell$). If $p$ is a $\Z/2$-stabilization and there are $t$ flows between $p$ and $\t(p)$, then the stabilizations $H_1 \ra H_2$ and $H_3 \ra H_4$ are of type $(\ell + k_1 +t, m, t)$ and $(\ell + k_2+t, m, t)$, respectively; if $p_0$ is a $\{1\}$-stabilization the stabilizations are of types $(\ell + k_1)$ and $(\ell + k_2)$. See Figures \ref{fig:hd_b1_Z2} and \ref{fig:hd_b1_triv}.

        \begin{figure}[h]
        \def\svgwidth{.8\linewidth}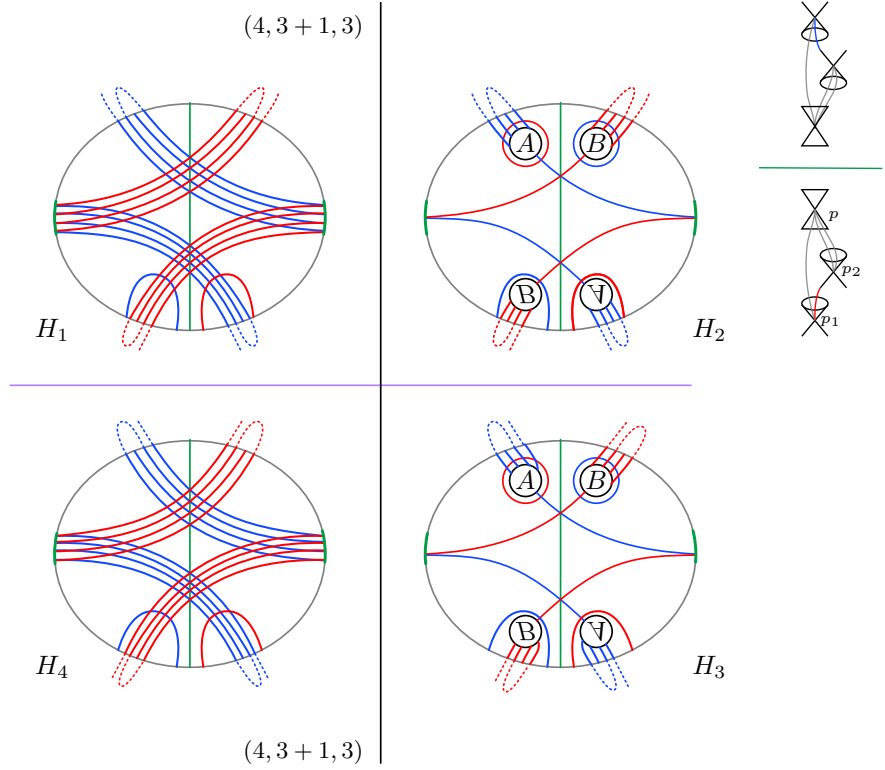
            \caption{A commutation of a real handleslide and a $\Z/2$-stabilization corresponding to the link of a bifurcation of type (B1) appearing in a $\Z/2$-orbit. Here, the two green arcs on either side of the diagrams are identified and represent a portion of the fixed point set. This example has $m = \ell = 3$, $k_1 = k_2 = 1$, and $c = 3$.}
        \label{fig:hd_b1_Z2}
        \end{figure}

        \begin{figure}[h]
        \def\svgwidth{.8\linewidth}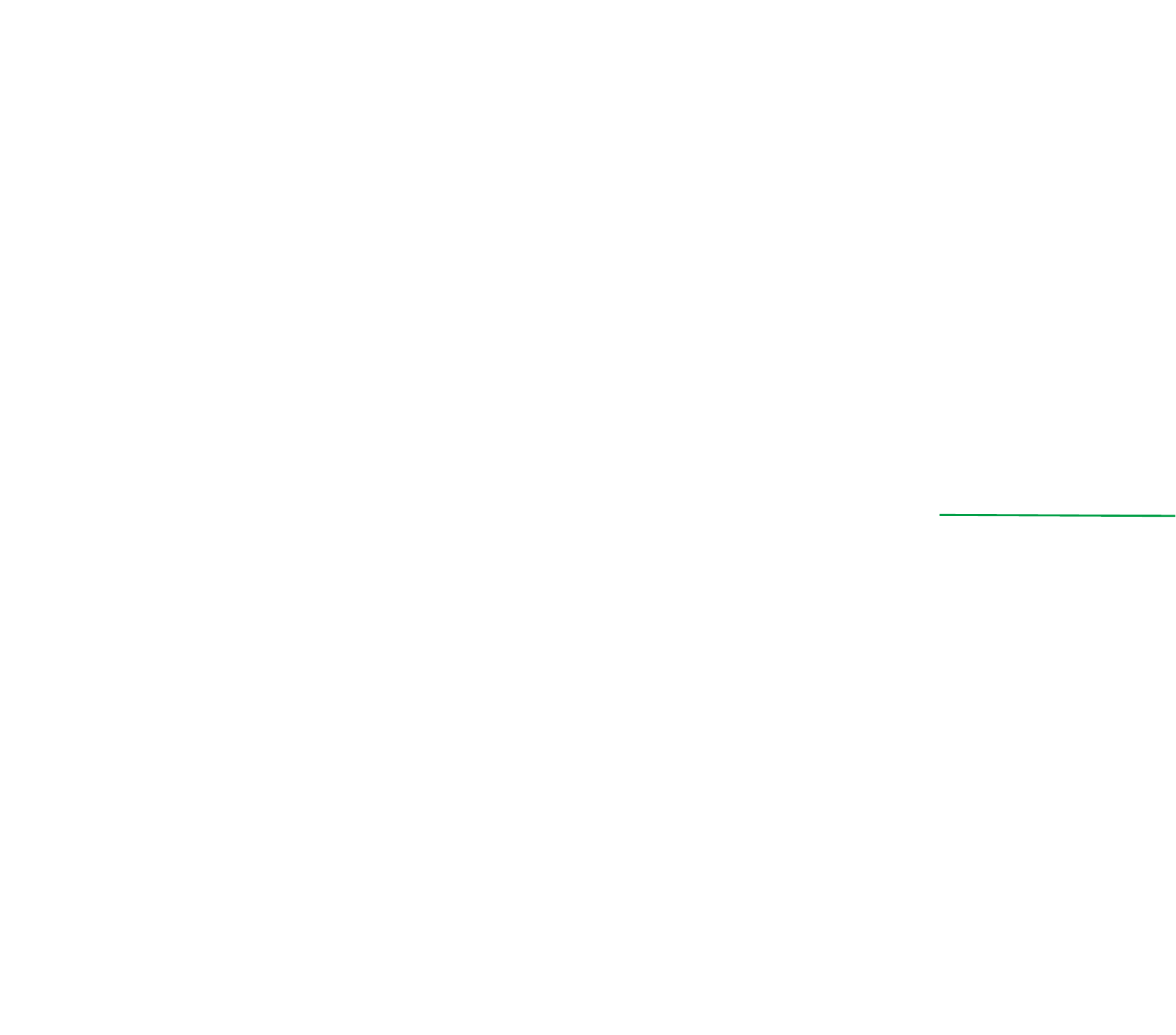
            \caption{A commutation of a real handleslide and a $\{1\}$-stabilization, which corresponds to the link of a bifurcation of type (B1) appearing in a trivial orbit. This example has $m = \ell = 3$ and $k_1 = k_2 = 1$.}
        \label{fig:hd_b1_triv}
        \end{figure}

    \textbf{(B2):} In this case, there is an orbit of tangency from $p$ to an index 1 critical point $\overline{p}^0$. For every flow from an index 1 critical point $p_*^0$ to $p$, we can perturb near $p$ on the death side of $S_1 \cup S_3$ so that there is a flow from $p_*^\mu$ to $\overline{p}^\mu$ (likewise for the symmetric critical points). In the case that $p_0$ is a $\{1\}$-orbit birth-death singularity, we can also perturb to get a flow from $\tau(\overline{p}^\mu)$ to $\overline{p}^\mu$, which corresponds to a crossover stratum. There are at least three strata; again $S_1$ and $S_3$ are stabilizations and $S_2$ is a real equivalence. 

    First, let us assume that $p$ appears in a $\Z/2$-orbit. To understand the stabilizations in this case, we assume there are $k$ flows from index 1 critical points to $p$ and $\ell$ flows from $p$ to index 2 critical points as well as $t$ flows between $p$ and $\tau(p)$. Further, there are $m_1 + m_2$ flows from the hyperbolic singularity $\overline{p}^0$ to index 2 critical points, partitioned into two sets by $W^{uu}(\overline{p}^0)$. The two stabilizations $H_1 \ra H_2$ and $H_3 \ra H_4$ are of types $(k + m_1+t, \ell, t)$ and $(k + m_2+t, \ell, t)$, respectively; on the stabilized side of the diagram, there is a single real handleslide stratum, whereas on the unstabilized side, there are $k$ real handleslide strata. See \Cref{fig:hd_b2_Z2}.

        \begin{figure}[h]
        \def\svgwidth{.8\linewidth}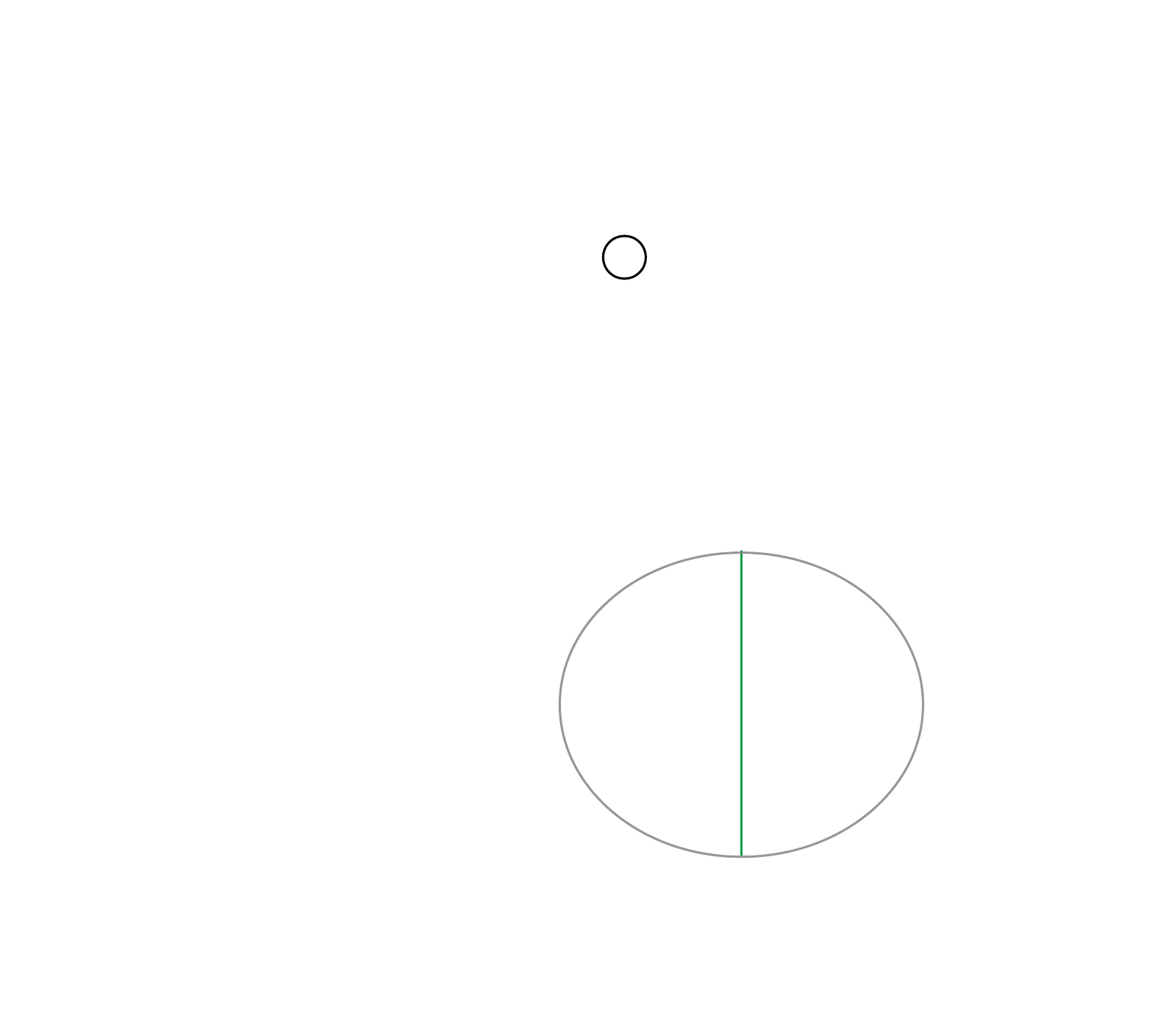
            \caption{The link of a bifurcation of type (B2) appearing in a $\Z/2$-orbit. In this example, $k = 3$, $\ell = 2$, $c=1$, $m_1 = 2$, and $m_2 = 1$.}
        \label{fig:hd_b2_Z2}
        \end{figure}
        
    Next, assume that $p$ appears in a $\{1\}$-orbit. By symmetry, there are an equal number of flows ($k$, say) in and out of $p$. Say that there are $m_1 + m_2$ flows from the hyperbolic singularity $\overline{p}^0$ to index 2 critical points, partitioned into two sets by $W^{uu}(\overline{p}^0)$. Strata $S_1$ and $S_{3}$ are $\{1\}$-stabilizations $H_1\ra H_2$ and $H_4 \ra H_3$ of type $(k+m_1)$ and $(k+m_2)$ respectively. On the stabilized side of the diagram, there is a single real handleslide stratum. On the destabilized side of the bifurcation diagram, there are a total of $k$ real handleslide strata separated into two collections by a single crossover of type $(m_1, m_2)$. See Figure \ref{fig:hd_b2_triv}.
       
    We need to be slightly careful in constructing the splittings. As above, we fix a surface $\Sigma_1$ for $H_1$ and take $\Sigma_2$ to be obtained by a stabilization. We can choose $\Sigma_3 = \S_2$. We cannot, however, take $\S_4 = \S_1$, as on the unstabilized side of the diagram, they are separated by the crossover stratum. We therefore choose $\S_4$ so that it is obtained by a destabilization of $\S_3$. If $S_i$, for $i > 3$, is the crossover stratum, we define the remaining surfaces as follows. Let $\Sigma_5 = \hdots = \Sigma_i:= \Sigma_4$ and $\Sigma_{i+1}' = \Sigma_{i+2} = \hdots \Sigma_{r-1}:= \Sigma_1$. To define $\Sigma_i'$ and $\Sigma_{i+1}$, we choose a small arc $a^{i \ra i+1}$ transverse to $S_i$, and choose $\Sigma_{i+1} \in \Sigma^R(f_{\nu}, v_\nu)$, where $\nu = \partial a^{i \ra i+1} \cap C_{i+1}$ and $\Sigma_{i}' \in \Sigma^R(f_{\eta}, v_\eta)$, where $\eta = \partial a^{i \ra i+1} \cap C_{i}$ so that $H_{i}$ is obtained from $H_{i+1}$ via a crossover. In particular, $C_i$ and $C_{i+1}$ each contain two vertices, and therefore each contains a diffeomorphism stratum.

        \begin{figure}[h]
        \def\svgwidth{.8\linewidth}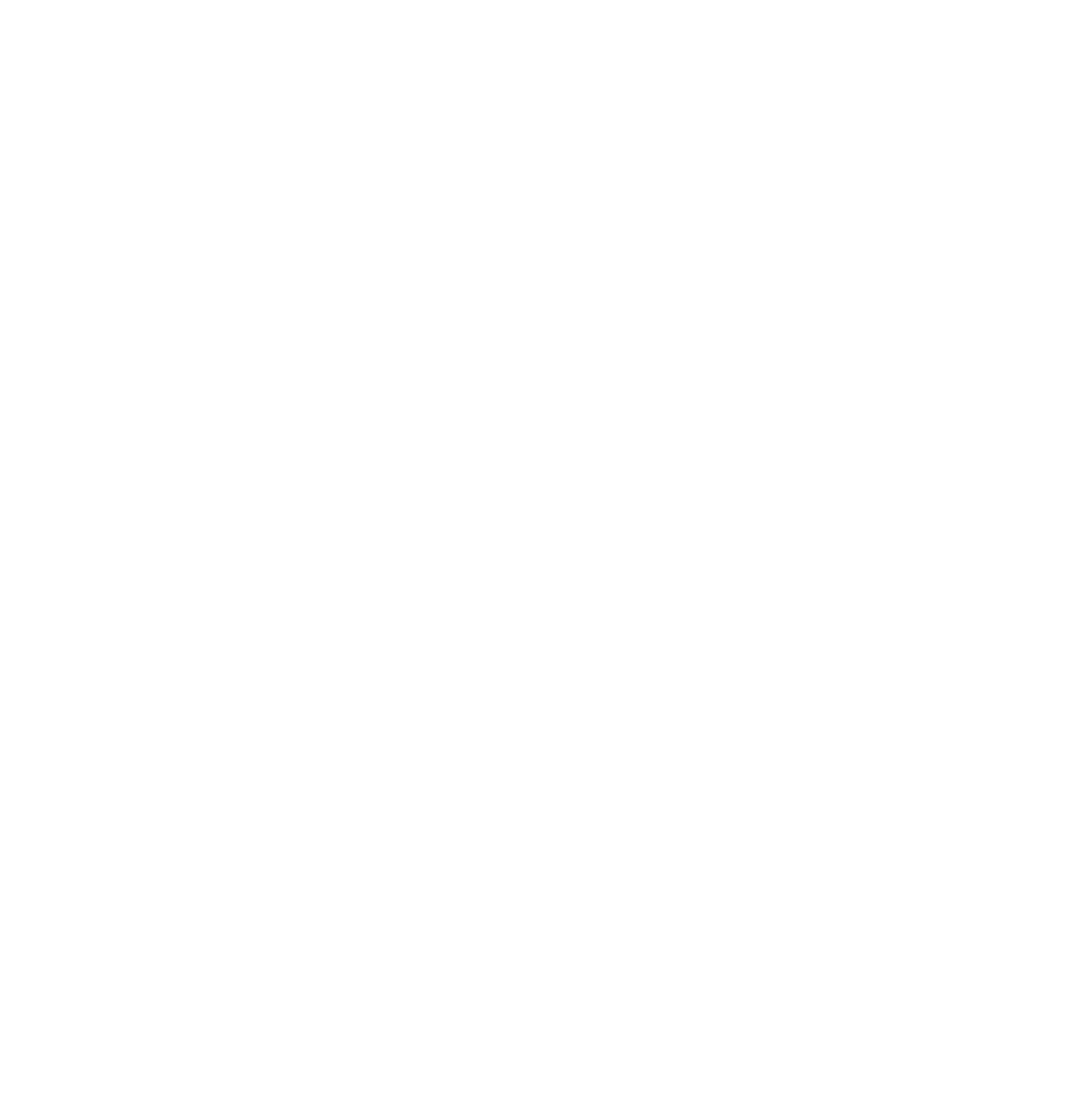
            \caption{The link of a bifurcation of type (B2) appearing in a trivial orbit. In this example, $k = 1$, $m_1 = 2$ and $m_2 = 1$. In this example, there are no additional handleslides after the crossover occurs.}
        \label{fig:hd_b2_triv}
        \end{figure}

    \textbf{(B3):} There is an orbit of tangency between the strong stable manifold of an index 1-2 birth-death $p$ and the unstable manifold of an index 1 critical point $\overline{p}$. In this case $r = 3$, and $S_1$ and $S_3$ are again stabilizations and $S_2$ is a handleslide. 
    
    Consider the case when $p$ appears in a trivial orbit. Say there are $k$ flow lines from index 1 critical points to $p$ (and symmetric flows to index 2 critical points); then the stabilizations $H_1 \ra H_2$ and $H_1 \ra H_3$ are of type $(k+1)$ and $(k)$ respectively. 

    If $p$ appears in a $\Z/2$-orbit, suppose that there are $k$ flow lines from index 1 critical points to $p$, $\ell$ flow lines from $p$ to index 2 critical points, and $c$ flows between $p$ and $\tau(p)$. Then the stabilization $H_1 \ra H_2$ and $H_1 \ra H_3$ are of type $(k+c, \ell+1,c)$ and $(k+2c, \ell, c)$ respectively.

    \begin{figure}[h]
        \def\svgwidth{.8\linewidth}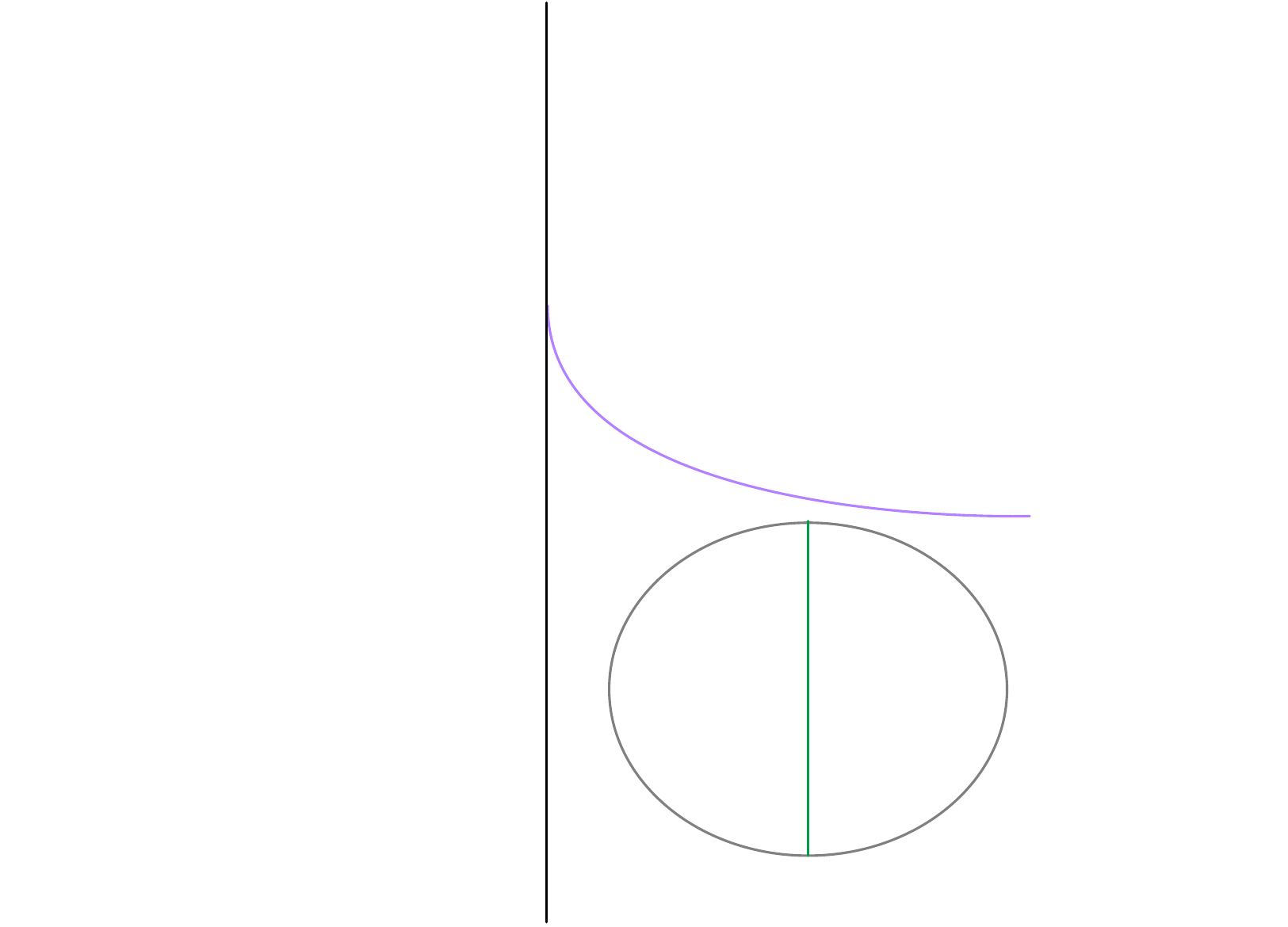
            \caption{The link of a bifurcation of type (B3) appearing in a trivial orbit. In this example, $k = 3$.}
        \label{fig:hd_b3_triv}
        \end{figure}

    \begin{figure}[h]
        \def\svgwidth{.8\linewidth}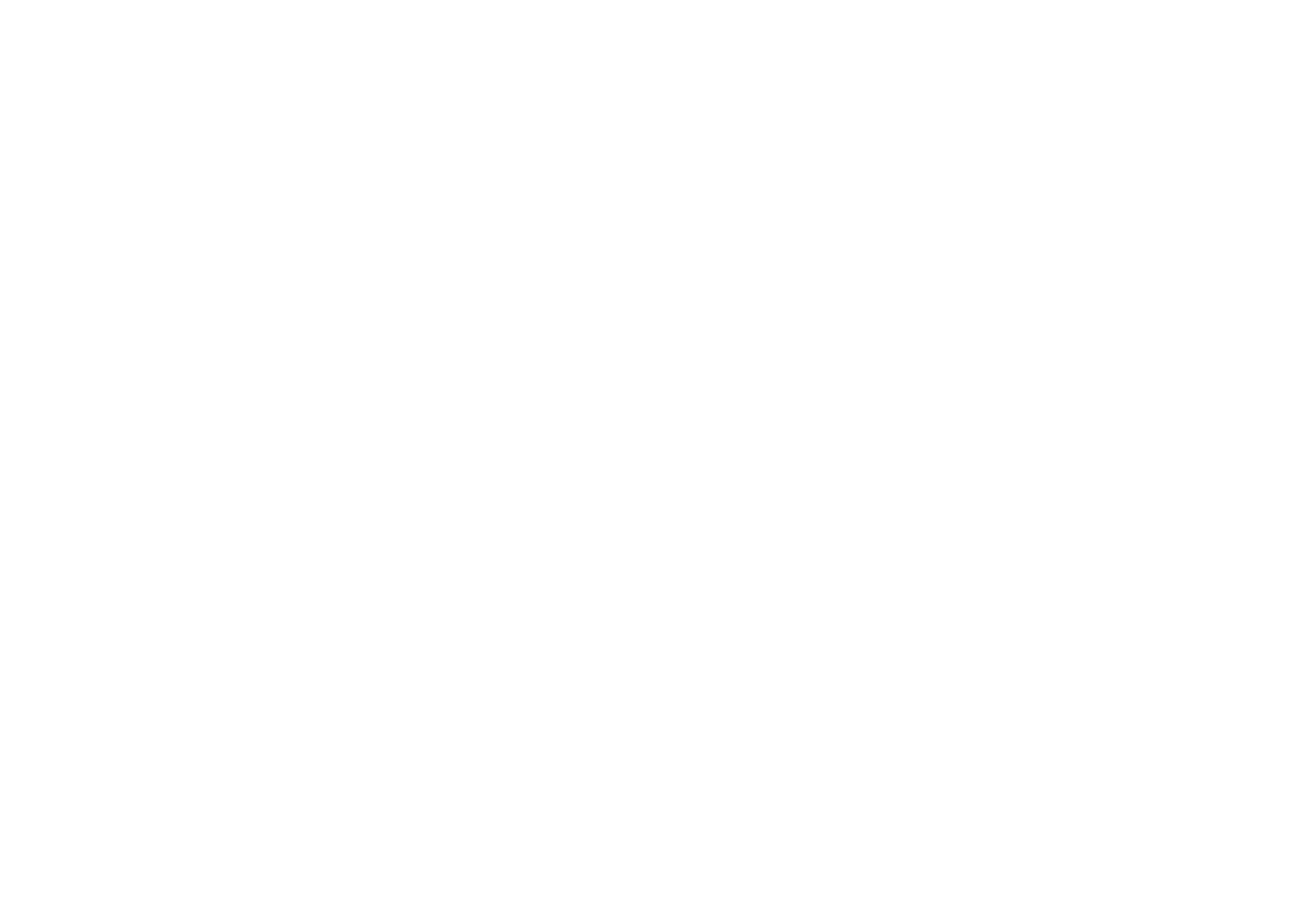
            \caption{The link of a bifurcation of type (B3) appearing in a $\Z/2$-orbit. Here, the two green arcs on either side of the diagrams are identified and represent a portion of the fixed point set. In this example, $k = 4$, $\ell=3$, and $c = 3$.}
        \label{fig:hd_b3_Z2}
        \end{figure}

    \textbf{(B4):} Suppose $\{p, \tau(p)\}$ is a $\Z/2$-orbit of index 1-2 birth-death singularities, that $p_1^\mu$ is an index 1 critical point, and that there is a real quasi-transversal orbit of tangency from $\tau(p_1^0)$ to $p_1^0$. In this case, $r = 4$, the strata $S_1$ and $S_3$ are stabilizations while $S_2$ and $S_4$ are crossovers. 

    Again, flows out of $p_1^\mu$ are partitioned into two subsets. Suppose there are $n$ flows from $p$ to critical points of index $2$, $r_1 + r_2$ flows from $p_1^0$ to $p_1$ (partitioned as such), and $m$ additional flows from index 1 critical points to $p$; additionally, say there are $c$ flows from $\tau(p)$ to $p$. Finally, assume there are $\ell_1 + \ell_2$ flows from $p^\mu_1$ to critical points of index 2. The strata $S_1$ and $S_3$ correspond to $\Z/2$-stabilizations of type $(m + r_1 + r_2 + c, n, c)$ and $(m + r_1 + r_2 + c + r_1(r_1 +\ell_1) + r_2(r_2+\ell_2), n, c + r_1 + r_2)$. The strata $S_2$ and $S_4$ are crossovers of types $(r_1n +\ell_1, r_2n+\ell_2)$ and $(r_1 +\ell_1, r_2+\ell_2)$ respectively. See \Cref{fig:hd_b4_Z2}.

    To construct the various splittings, we choose short arcs $a^{i \ra i+1}$ transverse to $S_i$ and use the parameter values $\partial a^{i\ra i+1}$ to choose $\Sigma_i'$ and $\Sigma_{i+1}$. This gives us eight vertices on our link, and each chamber contains a diffeomorphism stratum.
\begin{figure}[h]
        \def\svgwidth{.8\linewidth}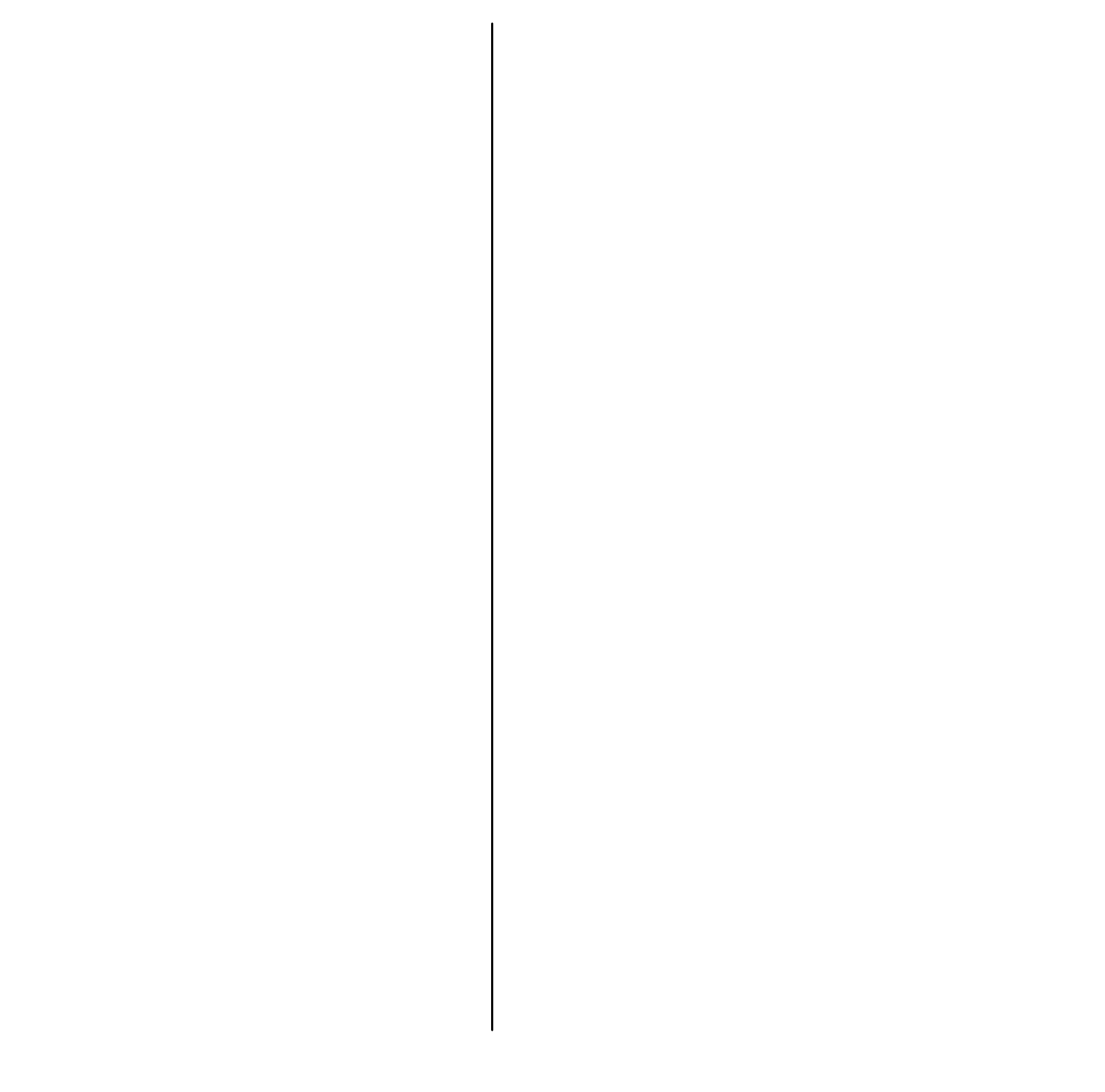
            \caption{A commutation of a crossover and a $\Z/2$-stabilization corresponding to a bifurcation of type (B4) involving a $\Z/2$-stabilization. In this example, $n = 2$, $r_1 = 1$, $r_2 = 0$, $m = 2$, $c = 0$, $\ell_1 = 1,$  $\ell_2 =2$. We highlight some of the more subtle aspects of the figure. The stable manifold of the index 2 critical point which splits off from $p$ is drawn in light blue; the unstable manifold of the symmetric index 1 critical point is drawn in pink. In $H_2$, we have indicated the pair of intersection points corresponding to the $r_1$ flows from $p_1$ to $p$ which give rise to the change in the type of crossover on the two sides of the diagram. Similarly, in $H_3$ we have indicated the intersection point corresponding to the $r_1 + r_2$ additional flow lines from $\t(p)$ to $p$ which is created when the parameter passes through the crossover stratum while moving along the stabilization stratum.}
        \label{fig:hd_b4_Z2}
        \end{figure}

    The case that $p$ is fixed by $\tau$ is much simpler, as there can be no flows from $p$ to $p_1^0$. If there are $n$ flows from $p$ to index 2 critical points and $k_1 + k_2$ flows from $p_1^0$ to index 2 critical points, then $S_1$ and $S_3$ both correspond to $\{1\}$-stabilizations $H_1 \ra H_2$ and $H_4 \ra H_3$ of type (m) while $S_2$ and $S_4$ correspond to crossovers $H_2 \ra H_3$ and $H_1 \ra H_4$ of type $(k_1, k_2)$. See \Cref{fig:hd_b4_triv}.

    \begin{figure}[h]
        \def\svgwidth{.8\linewidth}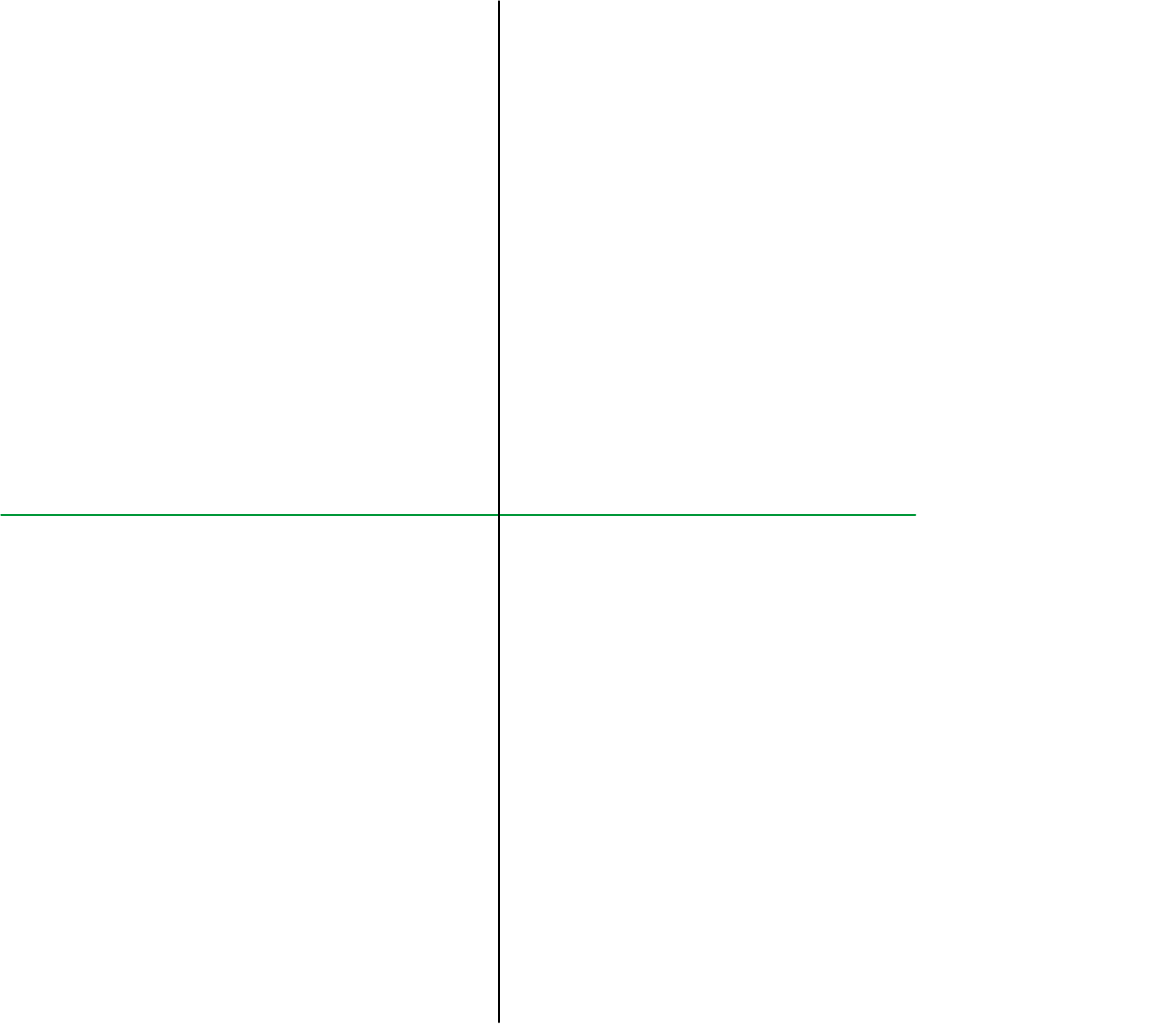
            \caption{A commutation of a crossover and a $\{1\}$-stabilization corresponding to a bifurcation of type (B4) involving a $\{1\}$-stabilization.}
        \label{fig:hd_b4_triv}
        \end{figure}

    \textbf{(B5):} Let $\{p, \tau(p)\}$ be a pair of index 1-2 birth-death singularities with a quasi-transversal orbit of tangency from $W^{uu}(\tau(p))$ to $W^{ss}(p)$. Say there are $m$ flows from $p$ to critical points of index 2 and also that there are $n$ flows from critical points of index 1 and $c$ flows from $\tau(p)$ to $p$. 

    The bifurcation diagram has three strata $S_1, S_2$, and $S_3$. Strata $S_1$ and $S_3$ are $\Z/2$-stabilizations of type $(m+(c+1), n,c+1)$ and $(m +c, n,c)$ while $S_2$ is a crossover of type $(n, 1)$. See \Cref{fig:hd_b5}. Fix $\nu \in C_1$ to choose $\S_1 \in \S^R(f_\nu, v_\nu)$. We construct $\S_2$ and $\S_3'$ by attaching tubes to $\S_1$. We define $\S_2'$ and $\S_3$ as before, by picking surfaces in $\S^R(f_\eta, v_\eta)$ for $\eta \in \partial a^{2 \ra 3}$. There are two diffeomorphism strata relating $\S_2$ and $\S_2'$ as well as $\S_3$ and $\S_3'$.\\

        \begin{figure}[h]
        \def\svgwidth{.8\linewidth}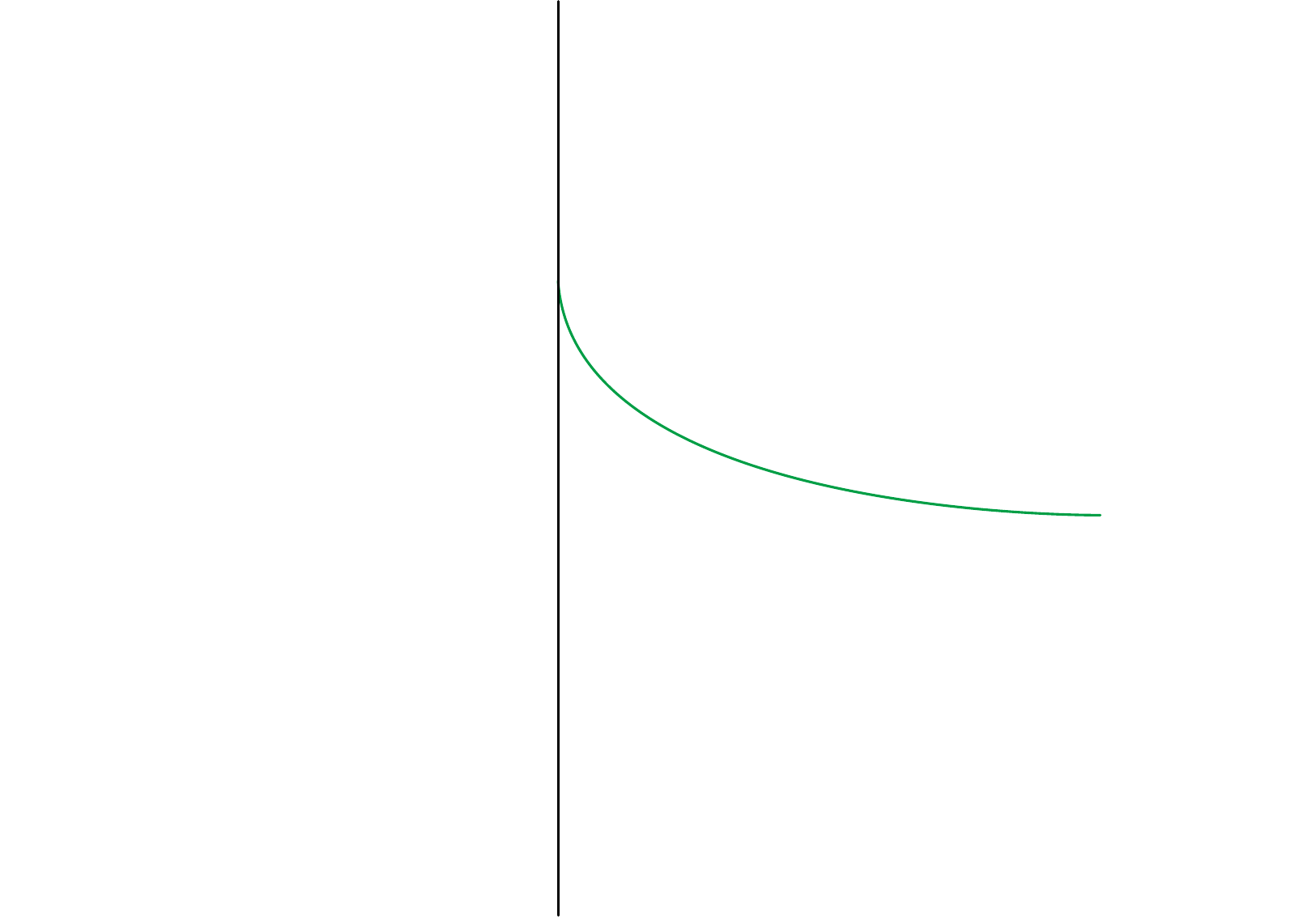
            \caption{A bifurcation of type (B5) involving a $\{1\}$-stabilization. The two green arcs on either side of the diagrams are identified and represent a portion of the fixed point set. In this example, $m = 1$, $n = 2$, $c = 1$.}
        \label{fig:hd_b5}
        \end{figure}
    
    \noindent \textbf{Type (C) bifurcations:} Type (C) bifurcations involve simultaneous stabilizations, and therefore come in three varieties, depending on the combination of orbit types.
    
    Consider the case that $(f_0, v_0)$ has a pair of $\Z/2$-orbit index 1-2 birth-death bifurcations, $\{p_1, \tau(p_1)\}$ and $\{p_2, \tau(p_2)\}$. Generically, $f(p_1) < f(p_2)$. In this case, $(f_0, v_0)$ is separable: indeed, we place $p_1, p_2 \in C_{01}(f)$ and $\tau(p_1), \tau(p_2) \in C_{23}(f)$ and choose $\Sigma_0 \in \Sigma^R(f_0, v_0)$ so that for $\ep > 0$ small enough, $\Sigma_0 \pitchfork v_\mu$ for all $|\mu| < \ep$. In this case, there are four strata: $S_1$ and $S_3$ correspond to $\Z/2$-stabilizations at $p_1$ and $\tau(p_1)$, and $S_2$ and $S_4$ correspond to $\Z/2$-stabilizations at $p_2$ and $\tau(p_2)$. 
    

    We now consider the various flow lines, beginning with flow lines between the degenerate critical points. Suppose that there are $t$ flows from $p_1$ to $p_2$, $c_i$ flows from $p_i$ to $\t(p_i)$ for $i \in \{1,2\}$, as well as $s$ flows from $p_1$ to $\t(p_2)$. Next, assume there are $m$ flows from $p_1$ to index 2 critical points, $n$ flows from index 1 critical points to $p_1$, and likewise, let $k$ and $\ell$ be the number of flows from and to $p_2$. The stabilization $H_1 \ra H_2$ is of type $(\ell + nt + c_2,k+ns,  c_2)$, $H_2 \ra H_3$ is of type $(n,m + s + c_1,  c_1)$, $H_1 \ra H_4$ is of type $(n,m + (t + \ell + c_2^2)s,  c_1 + sc_2^2)$, and $H_4 \ra H_2$ is of type $(\ell + c_2 + t,k + s,  c_2)$. 

    The underlying Heegaard splittings are constructed as follows. Choose $\mu_1$ to be any point in $C_1$ and define $\Sigma_1 = \Sigma_0$. For $\nu \in C_2 \cup S_2 \cup C_3$, we let $q_1^2(\nu)$ be the index 2 critical point which splits off from $p_1$ and for $\eta \in C_3 \cup S_3 \cup C_4$, let $q_2^1(\eta)$ be the index 1 critical point which splits off from $p_2$. By shrinking $\ep$ if necessary, we may assume that $W^u(q_1^2(\nu)) \cap W^s(q_2^1(\eta))$ is empty for all $\nu, \eta \in S_2 \cap C_3 \cap S_3$. Now, fix $\nu \in S_2$ and $\eta \in S_3$. We can obtain a surface $\Sigma_\nu \in \Sigma^R(f_\nu, v_\nu)$ from $\Sigma_0$ by attaching a pair of tubes around $W^u(q_1^2(\nu))$ and $W^s(\tau(q_1^2(\nu)))$; likewise, we construct $\Sigma_\eta \in \Sigma^R(f_\eta, v_\eta)$ from $\Sigma_0$ by attaching a pair of tubes around $W^s(q_2^1(\eta))$ and $W^u(\tau(q_2^1(\eta)))$. Choose arcs $a_2$ and $a_3$ transverse to $S_2$ and $S_3$ at $\nu$ and $\eta$, respectively, and define $\mu_2$ to be the boundary component of $a_2$ in $C_2$, $\mu_3$ the boundary component of $a_2$ in $C_3$, $\mu_3'$ the boundary component of $a_3$ in $C_3$, and $\mu_4$ the boundary component of $a_3$ in $C_4.$ Define $\Sigma_2 = \Sigma_\nu$ and $\Sigma_4 = \Sigma_\eta$. Finally, we construct $\Sigma_3$ from $\Sigma_2$ by attaching a pair of tubes around $W^s(q_2^1(\mu_3))$ and $W^u(\tau(q_2^1(\mu_3)))$ and construct $\Sigma_3'$ from $\Sigma_4$ by attaching a pair of tubes around $W^u(q_1^2(\mu_3'))$ and $W^s(\tau(q_1^2(\mu_3')))$. See \Cref{fig:hd_C_Z2_Z2}.

    \begin{figure}[h]
    \def\svgwidth{.8\linewidth}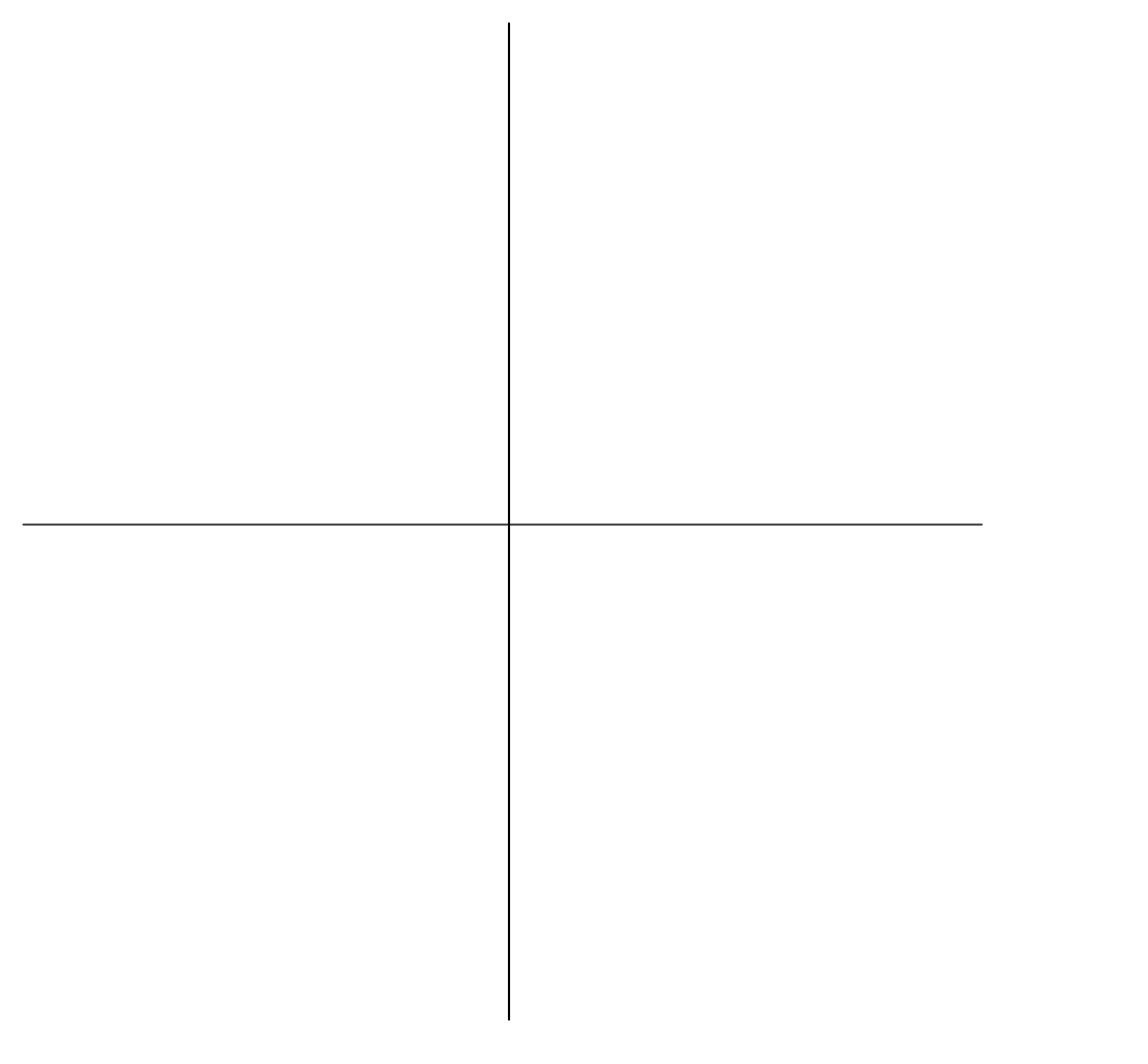
        \caption{The link of a bifurcation of type (C) involving two generalized $\Z/2$-stabilizations. In each quadrant, there are four pairs of arcs which are identified, e.g. the black arcs decorated with a single arrow on the right side of the diagram are glued together. Here, we have color-coded and labeled the various intersection points which correspond to the flow parameters determining the bifurcation type. In this example, $m= 1$, $n = 2$, $k = 2$, $\ell =1$, $c_1 = 1$, $c_2 = 5$, $s = 1$, and $t = 0$.}
    \label{fig:hd_C_Z2_Z2}
    \end{figure}

    Assume now that the orbit types of $p_1$ and $p_2$ differ: concretely assume that $p_1$ corresponds to a $\Z/2$-stabilization and $p_2$ corresponds to a $\{1\}$-stabilization. Suppose that there are $c$ flows from $p_1$ to $p_2$, $r$ flows from $p_1$ to $\tau(p_1)$, $n$ (and $k$) flows from index 1 critical points to $p_1$ (and $p_2$, respectively), $m$ (and $\ell$) flows from $p_1$ (and $p_2$, respectively) to index 2 critical points.

    There are four strata: $S_1$ and $S_3$ are $\Z/2$-stabilizations of types $(n + t\ell + r + t, m, r+t)$ and $(n+r+\ell, m, r)$ while $S_2$ and $S_4$ are $\{1\}$-stabilizations. Unlike the previous case, the gradient $(f_0, v_0)$ is not separable, but we construct the Heegaard diagrams in essentially the same way: we choose some point $\mu_1 \in C_1$ and fix $\Sigma_1 \in \Sigma^R(f_{\mu_1}, v_{\mu_1})$; we then use $\Sigma_1$ to construct surfaces $\Sigma_2$ and $\Sigma_4$ which represent the $\Z/2$- and $\{1\}$-stabilized diagrams, respectively; finally, we use $\Sigma_2$ to construct a surface $\Sigma_3$ which is obtained from it via a $\{1\}$-stabilization, and we use $\Sigma_4$ to construct $\Sigma_3'$  as a $\Z/2$-stabilization. Then, we choose the (non-minimal) link exactly as above, with five vertices and five edges (four of which are transverse to the four codimension-1 strata and one which is contained in $C_3$). See \Cref{fig:hd_C_Z2_1}.

    \begin{figure}[h]
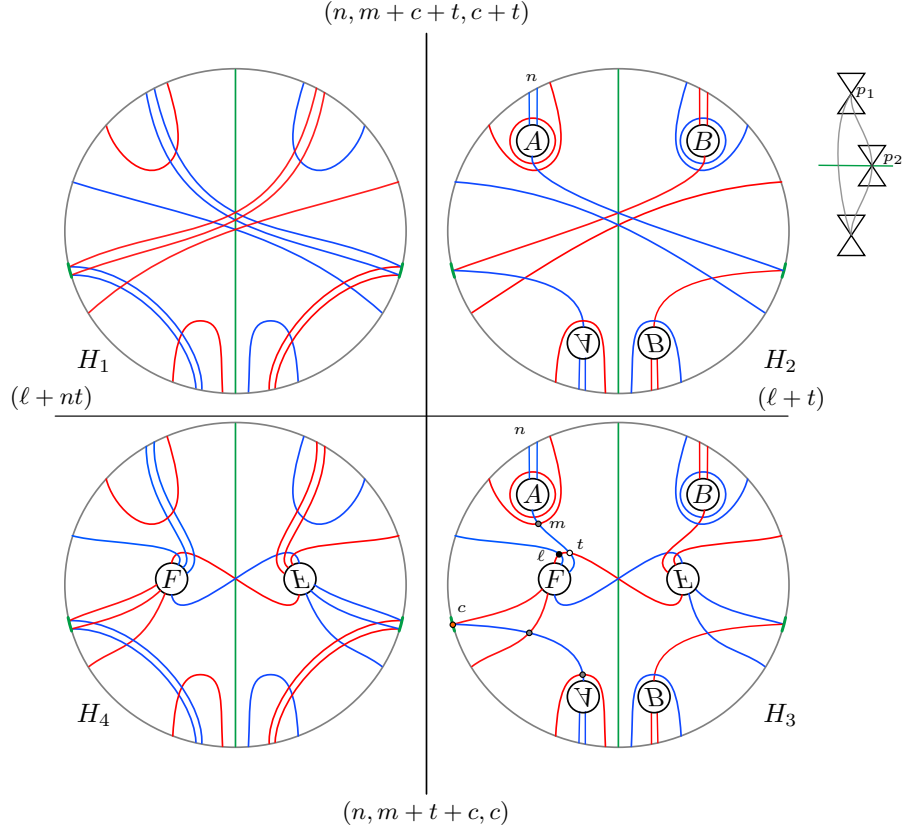

    \def\svgwidth{.8\linewidth}
        \caption{The link of a bifurcation of type (C) involving a generalized $\Z/2$-stabilization and a generalized $\{1\}$-stabilization. Again, there is a pair of green arcs which are identified and represent portions of the fixed point set. In this example, $m = 3$, $n = 2$, $\ell = 1$, $t = 1$, $c = 1$.}
    \label{fig:hd_C_Z2_1}
    \end{figure}

    The case that $p_1$ and $p_2$ both appear in trivial orbits is comparatively much more straightforward; since the two singularities both appear on $C$, there can be no flows connecting the two degenerate critical points. Therefore, the bifurcation diagram has four strata, meeting at the origin. Suppose there are $k$ and $\ell$ flows from index 1 critical points to $p_1$ and $p_2$ respectively (and symmetrically $k$ and $\ell$ flows from $p_1$ and $p_2$ to index 2 critical points). Then, $S_1$ and $S_3$ correspond to $\{1\}$-stabilizations $H_1 \ra H_2$ and $H_4 \ra H_3$ of type $(k)$ and $S_2$ and $S_4$ correspond to stabilizations $H_1 \ra H_4$ and $H_2 \ra H_3$ of type $(\ell)$. The link is obtained in the same way as the previous two cases. See \Cref{fig:hd_c_1_1}.\\

    \begin{figure}[h]
    \def\svgwidth{.8\linewidth}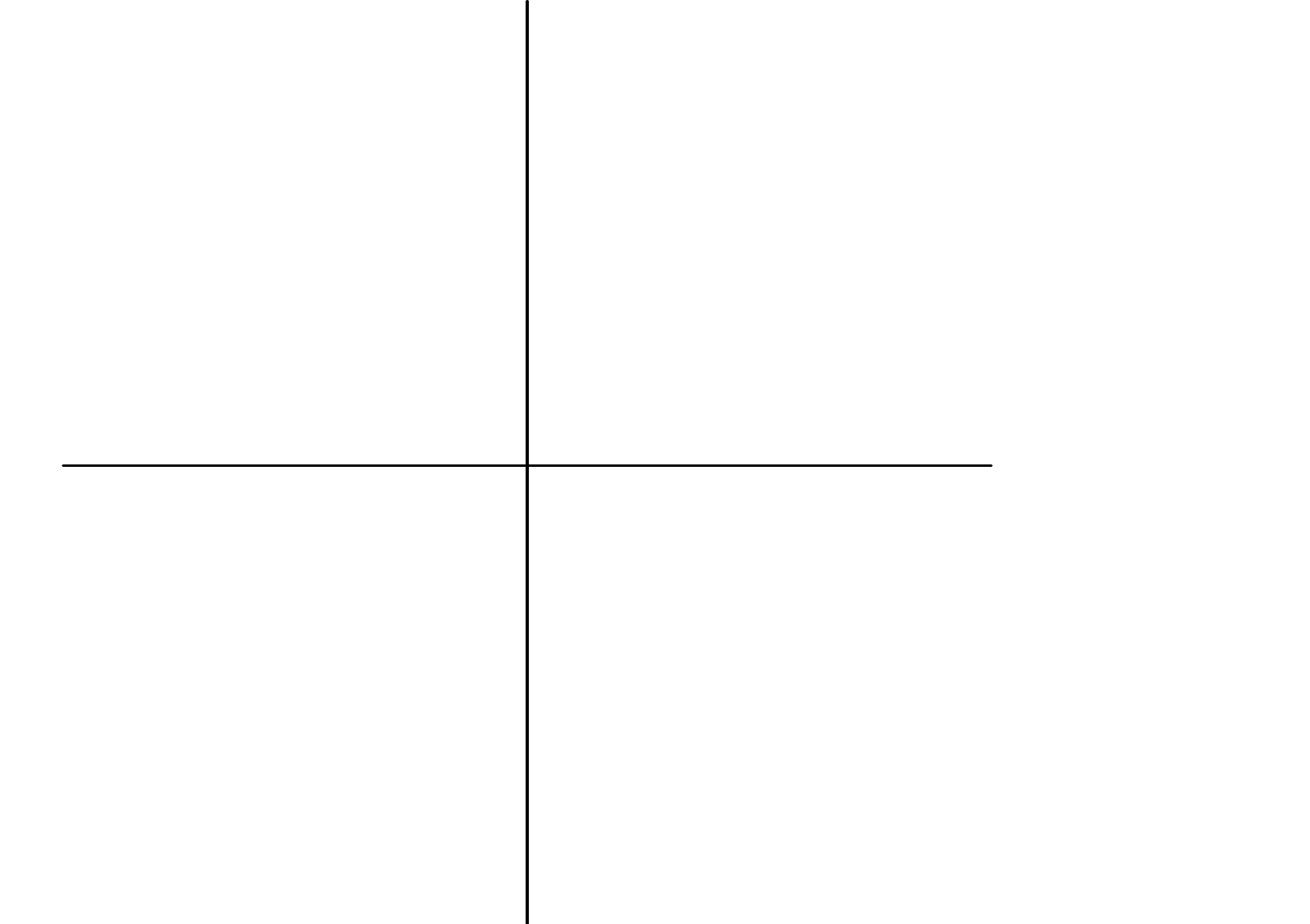
        \caption{The link of a bifurcation of type (C) involving two generalized $\{1\}$-stabilizations. In this example, $k = 1$ and $\ell = 2$.}
    \label{fig:hd_c_1_1}
    \end{figure}

    \noindent \textbf{Type (D) bifurcations:} Next, we consider bifurcations of type (D), i.e., singularities with (real) codimension 2. 
    
    \textbf{(D1):} A pair consisting of an index 1-2-1 ($A_3^+$) and an index 2-1-2 ($A_3^-$) degenerate critical point. In this case, $r = 2$ and on the stabilized side there are three index 1 critical points, $p_1$, $p_2$, and $p_3$, and three index 2 critical $\tau(p_1)$, $\tau(p_2)$, and $\tau(p_3)$. The critical point $p_1$ can cancel with $\t(p_2)$ or $\t(p_3)$ (and symmetrically $\tau(p_1)$ can cancel with $p_2$ or $p_3$). We assume that there are $c_2$ and $c_3$ flows from $p_2$ and $p_3$ to $\t(p_2)$ and $\t(p_3)$ as well as $k$ and $\ell$ flows from $p_2$ and $p_3$ to index 2 critical points (excluding the two flow lines to $\t(p_2)$). Then, the two $\Z/2$-stabilizations passing $S_1$ and $S_2$ are of types $(1,k+c_2,c_2)$ and $(1, \ell+c_3,c_3)$, respectively. See \Cref{fig:hd_d1}.

    \begin{figure}[h]
    \def\svgwidth{.8\linewidth}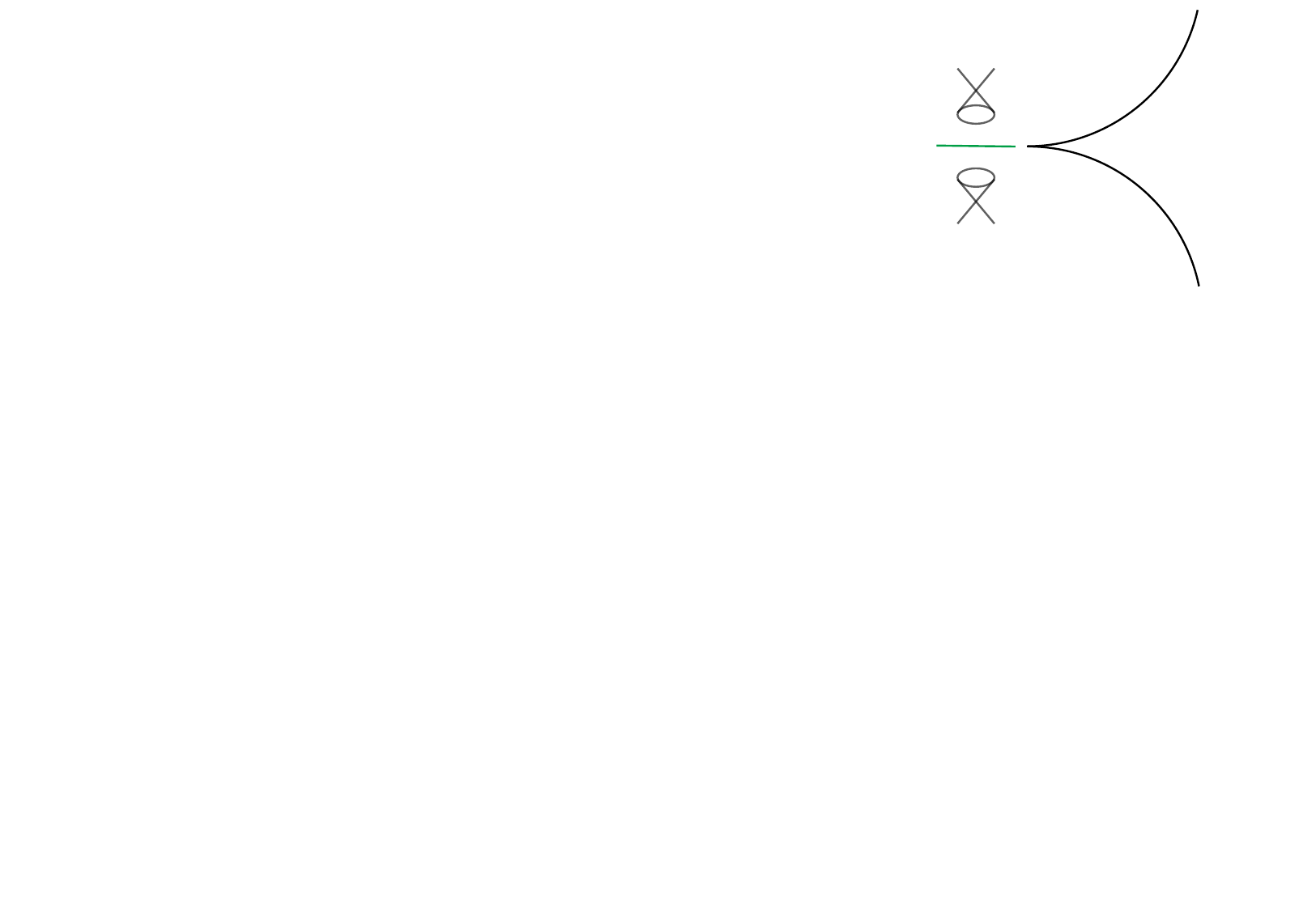
        \caption{The link of a bifurcation of type (D1). In this example, $c_2 = 2$, $c_3 = 0$, $k = 2$, and $\ell = 3$.}
    \label{fig:hd_d1}
    \end{figure}

    \textbf{(D2):} An $A_4$ degenerate critical point. Here, $r = 3$, and in $S_1$ there are four critical points, $p_1$, $p_2$, $\tau(p_1)$, and $\tau(p_2)$, where $p_1$ and $p_2$ are index 1. There is a unique flow line from $\tau(p_2)$ to $p_1$ which is taken by the involution to the unique flow line from $\tau(p_1)$ to $p_2$ (making possible a $\Z/2$-destabilization) while simultaneously, there is a unique flow line from $\tau(p_2)$ to $p_2$ which is fixed setwise by $\tau$ (making possible a $\{1\}$-destabilization). Say there are $k$ and $\ell$ flows from index 2 critical points into $p_1$ and $p_2$ respectively (not including the flow line from $\tau(p_2)$ to $p_2$). Then, the stratum $S_1$ corresponds to a $\Z/2$-destabilization of type $(k + \ell + 1, \ell +2,1)$, the stratum $S_2$ corresponds to a $\{1\}$-stabilization of type $(k)$, and stratum $S_3$ corresponds to a $\{1\}$-stabilization of type $(\ell+1)$. See \Cref{fig:hd_d2}. \\

    \begin{figure}[h]
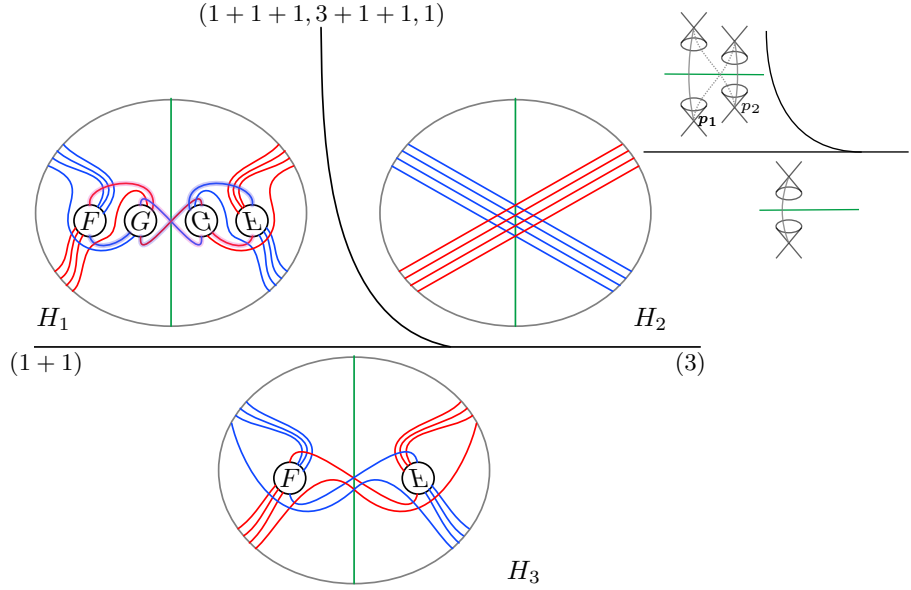

    \def\svgwidth{.8\linewidth}
        \caption{The link of a bifurcation of type (D2). In this example $k = 3$ and $\ell = 1$. In the first region, we have highlighted the pair of symmetric tori which are used in the $\Z/2$-destabilization.}
    \label{fig:hd_d2}
    \end{figure}
    
    \noindent \textbf{Type (E) bifurcations:} Type (E) bifurcations correspond to a failure of real quasi-transversality.
    
    \textbf{(E1):} In this case, we have index 1 critical points $p_1^\mu$ and $p_2^\mu$ and a pair of orbits of tangency from $p_1^0$ to $\tau(p_2^0)$ and $p_2^0$ to $\tau(p_1^0)$. There are two cases, depending on whether or not $p_2^\mu = p_1^\mu$. 
        
    Suppose the two orbits are the same, i.e., $p_2^\mu = p_1^\mu$, and that the two flows out of $p_1^0$ connect with the two flows into $\tau(p_1^0)$. Suppose there are $k$ flows from $p_1$ to index 2 critical points and $\ell$ flows from index 1 critical points to $p_1$. Then, $r = k + \ell$ and all strata correspond to $\Z/2$-handleslides. 
    
    As in \Cref{prop: codim 1 to heegaard moves}, let $\tilde{\Sigma}_0$ be the boundary of a thin regular neighborhood $\tilde{N}_0$ of
    \begin{align*}
         \bigcup \{W^s(p): p\in C_0(f_0) \cup C_1(f_0) \smallsetminus \{p_1^0\} \}.
    \end{align*}
    We then construct the various Heegaard diagrams exactly as in \cite{JTZ_naturality_mapping_class_groups}. Choose values $\nu_i$ in $S_i$ and let $A_i$ be a small tube $\partial\nu(W^s(p_1^{\nu_i})) \smallsetminus N_0$ and let $B_i$ be an annulus around $W^s(p_2^{\nu_i})\cap \tilde\Sigma$. Ensure that $A_i$ is small enough that it is disjoint from the stable and unstable manifolds of $p_2^{\nu_i}$. Let $D_i$ and $D_i'$ be the disks $\nu(W^s(p_1^{\nu_i}))\cap \Sigma$. These are the same diagrams constructed in \cite{JTZ_naturality_mapping_class_groups}, and are not symmetric. To obtain real diagrams, we push all of these along the gradient flow. By shrinking $A_i$ if necessary, arrange that $D_i\cap \tau(D_i)$ and $D_i'\cap \tau(D_i')$ are disjoint and define $\Sigma_i$ to be $(\Sigma_0 \smallsetminus (D_i \cup D_i')\cup \tau(D_i \cup D_i'))\cup A_i \cup B_i$. Compare to \Cref{fig:crossover_splitting_surface}. 


    Finally, to construct the link of the singularity, we choose arcs $a^{i \ra i+1}$ transverse to $S_i$ short enough that $D_i \cap W^s(p_2^\mu) =  \emptyset$ for every $\mu \in a^{i \ra i+1}$. Take $\partial a^{i \ra i+1}= \{\mu_i, \mu_{i+1}'\}$ to be vertices of our link, and complete the link by choosing any arcs connecting $\mu_i$ to $\mu_i'$. The resulting Heegaard diagrams are all identical away from the tubes $A_i$ and $B_i$, and near those tubes, the loop of diagrams is shown in \Cref{fig:hd_e1}.


    The case that $p_1^\mu$ and $p_2^\mu$ are distinct is similar. We construct the Heegaard diagram in the same way, following the construction in \cite{JTZ_naturality_mapping_class_groups}, and pushing forward diagrams by the gradient flow to obtain invariant ones. In this case, however, as the annuli $A_i$ traverse the annulus $B$, there are simultaneously annuli $B_i := \tau(A_i)$ which are looping around $A:= \tau(B)$. 

    Suppose that there are $k$ flows from $p_1$ to index 2 critical points, $r$ flows from $p_2^\mu$ to $\tau(p_2^\mu)$, and $\ell$ additional flows from $p_2^\mu$ to other index 2 critical points. The bifurcation diagram consists of $(k + \ell+r)$ strata which meet at a point. When we resolve the diagram, the $k$ flows out of $p_1$ will glue to $\tau(p_1)$ to become flows from $\tau(p_2^\mu)$ to index-2 critical points, corresponding to handleslides; the $\ell$ flows out of $p_2$ glue to the flow from $\tau(p_1^0)$ to $p_2^0$ and also give rise to handleslides; finally, the $r$ flows from $p_2^0$ to $\tau(p_2^0)$ glue to become crossovers between $p_1^\mu$ and $\tau(p_1^\mu)$. See \Cref{fig:hd_e1} for an example. 

    \begin{figure}[h]
    \def\svgwidth{.8\linewidth}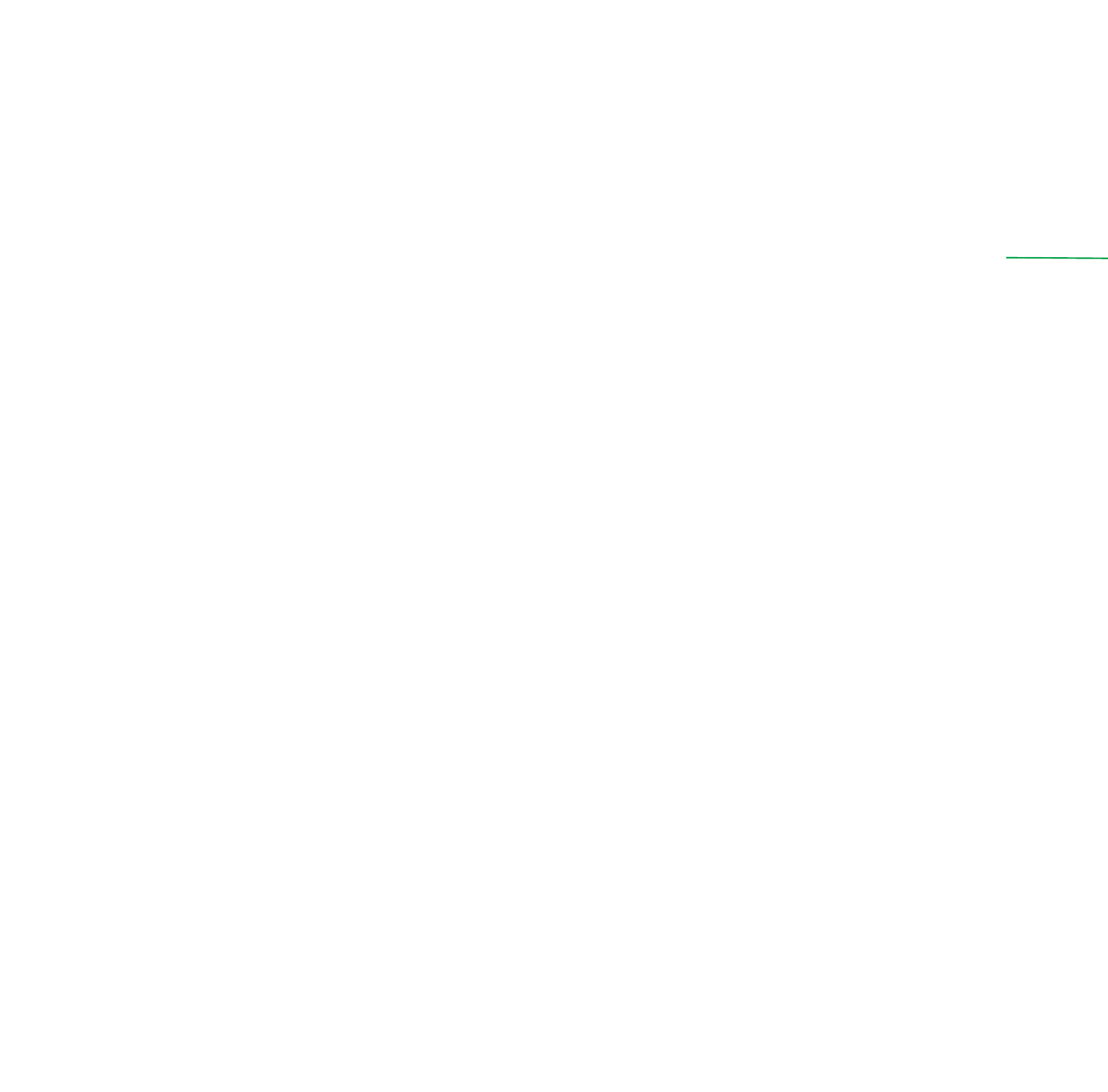
        \caption{The link of a bifurcation of type (E1). Here, we have drawn the handleswap around the annulus $A$. Each time the foot crosses $C$, the annuli $A$ and $B$ intersect, though for clarity we have only drawn $A$. In this example, $k = 2$, $\ell = 3$, and $r = 2$.}
    \label{fig:hd_e1}
    \end{figure}

    \begin{rem}
        We note that the first case is a particular case of the second; we take all $A_i$ to be isotopic to $A$, and all $B_i$ to be isotopic to $B$. In this case, we see that one foot of $A$ is looping around one foot of $B$ while symmetrically the second foot of $B$ loops the second foot of $A$. 
    \end{rem}

    \textbf{(E2):} In this case, $p_1^0$ is a hyperbolic singularity of index 1 with a $\{1\}$-orbit of tangency $\gamma^0$ from $\tau(p_1^0)$ to $p_1^0$. Additionally, there is a flow from $p_1^0$ to an index $2$ critical point $p^*$ which is tangent to the directional manifold of $p_1^{(t,0)}$, where $t \in (-\ep, \ep)$ for $\ep > 0$ small. For $y < 0$ close to $0$, the directional manifold $D(p_1^{(t,y)}, \gamma^\mu)$ partitions the flows from $p_1^0$ to index 2 critical points into two subsets of size $k+1$ and $\ell$, while for small $y > 0$ the directional manifold $D(p_1^{(t,y)},\gamma^\mu)$ partitions the flows into two subsets of size $k$ and $\ell+1$.

    The bifurcation diagram has three strata which meet in a point: $S_1$ and $S_3$ are crossovers of types $(k+1, \ell)$ and $(k, \ell+1)$, respectively, while $S_2$ is a handleslide. See \Cref{fig:hd_e2}.

    \begin{figure}[h]
    \def\svgwidth{.8\linewidth}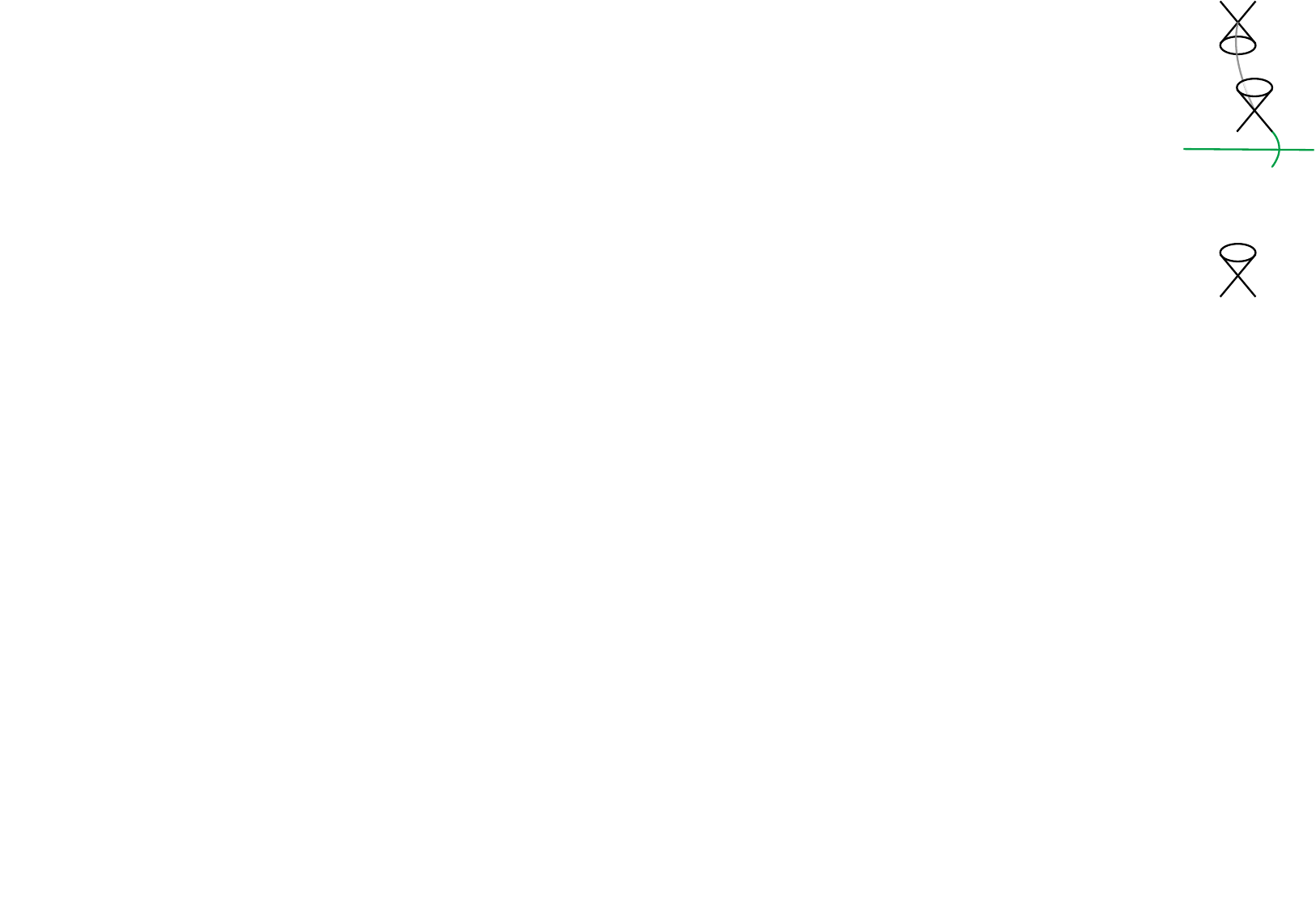
        \caption{The link of a bifurcation of type (E2). In this example, $k =\ell = 2$.}
    \label{fig:hd_e2}
    \end{figure}
    
\end{proof}

\section{Simplifying moves and loops of real Heegaard diagrams}\label{sec:reducing moves}
In this section, we decompose the moves and loops appearing in the previous sections into the simpler moves appearing in the definition of a strong real Heegaard invariant. Throughout, we will work with overcomplete diagrams, only later choosing spanning trees to obtain actual diagrams. 

In the first subsection, we describe resolutions of the basic codimension-1 strata, and discuss some relations between possible resolutions. In the remaining subsections, we return to codimension-2 bifurcations. We will resolve type (A) bifurcations first as well as type (E2) bifurcations;  the remaining codimension-2 bifurcations will typically be reduced to these cases. We will then turn to the resolution of real handleslides, which will ultimately be reduced to the unreal case. Finally, we simplify bifurcations of type (D2).

Let us briefly recall the combinatorial structures from \cite{JTZ_naturality_mapping_class_groups} which are useful in simplifying loops of real Heegaard diagrams. We  state the relevant lemmas without proofs, as they are identical to those in the unreal case.

\begin{definition}\cite[Definition 7.1]{JTZ_naturality_mapping_class_groups} \label{def:polyhedral}
  A \emph{polyhedral decomposition} of $D^2$ is a regular CW decomposition of~$D^2$
  (i.e., the attaching map of every cell is an embedding) such that every closed
  1-cell is smoothly embedded in~$D^2$.

  A \emph{bordered polyhedral decomposition} of $D^2$ is a partition
  of $D^2$ that arises as follows:
  Choose a polyhedral decomposition of~$D^2$ such that every 0-cell in the boundary~$S^1$
  has valence~3, every closed 1-cell not contained in~$S^1$
  intersects~$S^1$ in at most one 0-cell, and every
  2-cell intersects~$S^1$ in at most one 1-cell.
  Then take the union of each open $i$-cell in $S^1$ with the
  neighboring open $(i+1)$-cell in $\text{Int}(D^2)$ for $i \in
  \{0,1\}$ (where an open 0-cell is just a 0-cell).
  We call these \emph{bordered $(i+1)$-cells}.

  Let $\text{sk}_i$ denote the $i$-skeleton. A polyhedral decomposition $\mathcal{P}$ and a bordered polyhedral
  decomposition $\mathcal{R}$ are \emph{dual} if
  $\text{sk}_0(\mathcal{P}) \cap \text{sk}_1(\mathcal{R}) = \emptyset$
  and $\text{sk}_0(\mathcal{R}) \cap \text{sk}_1(\mathcal{P}) = \emptyset$,
  in each 2-cell of $\mathcal{P}$ there is a unique vertex of $\mathcal{R}$,
  and in each (bordered) 2-cell of $\mathcal{R}$ there is a unique vertex of $\mathcal{P}$.
  Furthermore, for each 1-cell~$e$ of~$\mathcal{P}$, we have $|e \cap \text{sk}_1(\mathcal{R})| = 1$,
  and for each (bordered) 1-cell~$e^*$ of~$\mathcal{R}$, we have $|e^* \cap \text{sk}_1(\mathcal{P})| = 1$.
  For an example, see  \cite[Figure 31]{JTZ_naturality_mapping_class_groups}.
\end{definition}

\begin{definition}\cite[Definition 7.2]{JTZ_naturality_mapping_class_groups} \label{def:stratification}
  We say that the partition $\mathfrak{S} = V_0 \sqcup V_1 \sqcup V_2$
  is a \emph{bordered stratification} of the disk $D^2$ if the
  following hold:
  \begin{enumerate}
  \item $V_0$ is a finite set of points in the interior of $D^2$,
  \item $V_1$ is a properly embedded 1-dimensional submanifold-with-boundary of $D^2
    \setminus V_0$, and
  \item each point $x \in V_0$ has a neighborhood $N_x$ such
    that the pair $(N_x, V \cap N_x)$ is homeomorphic to a cone
    $(D^2,I \cdot H)$ for some finite set $H \subset S^1$, where $V = V_0 \cup V_1$,
    and $I \cdot H$ is the union of the line segments connecting the origin with each point of~$H$.
  \end{enumerate}

  A bordered polyhedral decomposition
  $\mathcal{R}$ of $D^2$ is a \emph{refinement of} $\mathfrak{S}$ if
  $\text{sk}_i(\mathcal{R}) \supset V_i$ for $i \in \{0,1\}$.
  We say that a polyhedral decomposition of $D^2$ is \emph{dual to
  $\mathfrak{S}$} if it is dual to some bordered polyhedral
  decomposition $\mathcal{R}$ refining $\mathfrak{S}$.
  For an example, see  \cite[Figure 31]{JTZ_naturality_mapping_class_groups}.
\end{definition}

Of course, a generic 2-parameter family of real gradients $\cF: D^2 \ra \FV(Y, \gamma, \tau)$ gives rise to a bordered stratification of $D^2$ by defining
\begin{align*}
    V_i = \{\mu \in D^2: \cF(\mu) \in \FV_{2-i}(Y, \gamma, \tau)\}
\end{align*}
for $i = 0, 1, 2$.

\begin{definition}\cite[Definition 7.3]{JTZ_naturality_mapping_class_groups}\label{def:adapted}
  A polyhedral decomposition $\mathcal{P}$ of $D^2$ is \emph{adapted
    to the family $\mathcal{F}$} if
  \begin{enumerate}
  \item $\mathcal{P}$ is dual to $\mathfrak{S}(\mathcal{F})$,
  \item \label{item:short} each edge intersecting $V_1$ is so short
    that \Cref{prop: codim 1 to heegaard moves} applies to it,
  \item if $\overline{\mu} \in V_0$ and $\sigma$ is the 2-cell of $\mathcal{P}$
    containing $\overline{\mu}$, then $\partial \sigma$ is a link of $\overline{\mu}$ as in \Cref{thm:codim2-to-heegaard-loops},
  \item every 2-cell $\sigma$ of $\mathcal{P}$ that intersects $V_1$
    but is disjoint from $V_0$ is a quadrilateral, and $\sigma \cap
    V_1$ is an arc connecting opposite sides of $\sigma$,
  \item \label{item:disjoint-edges} any two closed 2-cells of
    $\mathcal{P}$ containing two different points of $V_0$ are
    disjoint, and any two closed 1-cells of $\mathcal{P}$ that
    intersect $V_1 \setminus S^1$ are either disjoint, or they both
    belong to a 2-cell containing a point of $V_0.$
  \end{enumerate}
\end{definition}

\begin{lemma}\cite[Lemma 7.4]{JTZ_naturality_mapping_class_groups} \label{lem:adapted} Let $\mathcal{F} \colon D^2 \to
  \FV(Y,\gamma, \tau)$ be a generic 2-parameter family. Then there exists a
  polyhedral decomposition $\mathcal{P}$ of $D^2$ adapted to
  $\mathcal{F}$. Furthermore, given a triangulation of $S^1$ such that
  each 1-cell contains at most one bifurcation point of $\mathcal{F}$
  and satisfies Condition~(\ref{item:short}) of
  Definition~\ref{def:adapted}, then we can choose $\mathcal{P}$
  such that it extends this triangulation.
\end{lemma}

\begin{definition}\cite[Definition 7.5]{JTZ_naturality_mapping_class_groups} \label{def:coherent-surfaces} Let $\mathcal{F}
  \colon D^2 \to \FV(Y,\gamma, \tau)$ be a generic 2-parameter family and
  $\mathcal{P}$ an adapted polyhedral decomposition. A choice of
  Heegaard surfaces
  \[
  \{\, \Sigma_\mu \in \Sigma^R(\mathcal{F}(\mu)) \,\colon\, \mu \in
  \text{sk}_0(\mathcal{P}) \,\}
  \]
  is \emph{coherent} with $\mathcal{P}$ if, for every edge $e$ of
  $\mathcal{P}$ with $\partial e = \mu - \mu'$, the isotopy diagrams
  $[H(\mathcal{F}(\mu),\Sigma_\mu)]$ and $[H(\mathcal{F}(\mu'),\Sigma_{\mu'})]$
  are related as indicated by the label of $e$. A \emph{surface-enhanced
    polyhedral decomposition of $D^2$ adapted to
    $\mathcal{F}$} is a polyhedral decomposition of $D^2$ adapted to
  $\mathcal{F}$, together with a coherent choice of Heegaard surfaces.
\end{definition}

\begin{lemma}\cite[Lemma 7.6]{JTZ_naturality_mapping_class_groups} \label{lem:coherent} Let $\mathcal{F} \colon D^2 \to
  \FV(Y,\gamma, \tau)$ be a generic 2-parameter family, and suppose that
  $\mathcal{P}$ is a polyhedral decomposition of $D^2$ adapted to
  $\mathcal{F}$.  If we are given Heegaard surfaces $\Sigma_\mu \in
  \Sigma^R(\mathcal{F}(\mu))$ for $\mu \in \text{sk}_0(\mathcal{P}) \cap
  S^1$ such that, for every edge $e$ of
  $\mathcal{P}$ in $S^1$ with $\partial e = \mu - \mu'$, the isotopy diagrams
  $[H(\mathcal{F}(\mu),\Sigma_\mu)]$ and $[H(\mathcal{F}(\mu'),\Sigma_{\mu'})]$
  are related as indicated by the label of $e$, then this can be
  extended to a choice of Heegaard surfaces coherent with
  $\mathcal{P}$.
\end{lemma}

\subsection{Codimension-1 resolutions} First, we break down generalized stabilizations (of both types) and crossovers. 

The case of generalized $\{1\}$-stabilizations is the most straightforward, and follows just as in \cite[Section 7.1]{JTZ_naturality_mapping_class_groups}. Suppose $H'$ differs from $H$ by a generalized $\{1\}$-stabilization of type $(k)$. Such a stabilization can be replaced by a single simple $\{1\}$-stabilization and $k$ real handleslides. The diagram $H'$ is obtained from $H$ by removing a disk $D$ and replacing it with a punctured torus $T$, which contains a new pair of attaching curves $\alpha, \beta$. Fix an orientation on $\alpha$ (which determines an orientation on $\beta = \tau(\alpha)$) and let $\beta_1, \hdots, \beta_k$ be the beta curves which intersect $\alpha$. By sliding each $\beta_i$ over $\beta$ in the direction opposite the orientation on $\alpha$ (and hence the symmetric slides of the $\alpha_i$ over $\alpha$), we obtain a sequence of diagrams
\begin{align*}
    H' = H_0 \ra H_1 \ra \hdots \ra H_k.
\end{align*}
In $H_k$, $\alpha$ and $\beta$ are disjoint from all other attaching curves, and hence is obtained from $H$ by a simple $\{1\}$-stabilization. 

This construction does depend on the orientation of $\alpha$. However, as in \cite[Lemma 7.8]{JTZ_naturality_mapping_class_groups}, there is a 2-parameter family which interpolates between these two choices. See \Cref{fig:resolution_orientations_triv}. Traversing a horizontal path along the top of the diagram corresponds to the resolution given by one choice of orientation and traversing a horizontal path along the bottom of the diagram corresponds to the other. In their situation, there is also a choice of whether to perform slides of the alpha curves first or the beta curves first; this is irrelevant in our situation, as the two slides happen simultaneously. 

\begin{rem}\label{rem:resolving-to-simple-loops}
     Importantly, every codimension-2 bifurcation value in this 2-parameter family corresponds to one of the simple loops appearing in the definition of a strong real Heegaard invariant. 
\end{rem}

In interpolating between resolutions of codimension-1 strata with different resolutions, we will frequently make use of the following subgraphs of $\cG^R$.

\begin{defn}
    An \emph{$\cO$-stabilization slide} is a subgraph of $\cG^R$ of the form 
    \begin{align*}
        \begin{tikzcd}[ampersand replacement = \&]
            H_1 \ar[r,"e"]\ar[dr,"f"] \& H_2 \ar[d,"g"] \\
            \& H_3
        \end{tikzcd}
    \end{align*}
    such that 
    \begin{enumerate}
        \item $H_i = (\Sigma_i, [\bm \alpha_i], [\bm \beta_i])$ are overcomplete isotopy diagrams for $i \in \{1, 2, 3\}$ so that $\Sigma_2 = \Sigma_3$;
        \item the edges $e$ and $f$ are $\cO$-stabilizations, while $g$ is a real-equivalence;
        \item there is a disk $D \sub \Sigma_1$ and a punctured torus $T \sub \Sigma_2 = \Sigma_3$ so that $H_1|_{D}$, $H_2|_{T}$, and $H_3|_T$ are conjugate to the pictures in \Cref{fig:stabilization_slides};
        \item $H_1|_{\Sigma_1\smallsetminus D} = H_2|_{\Sigma_2\smallsetminus T} = H_3|_{\Sigma_3\smallsetminus T}$.
    \end{enumerate}
\end{defn}

\begin{figure}[h]
    \def\svgwidth{.8\linewidth}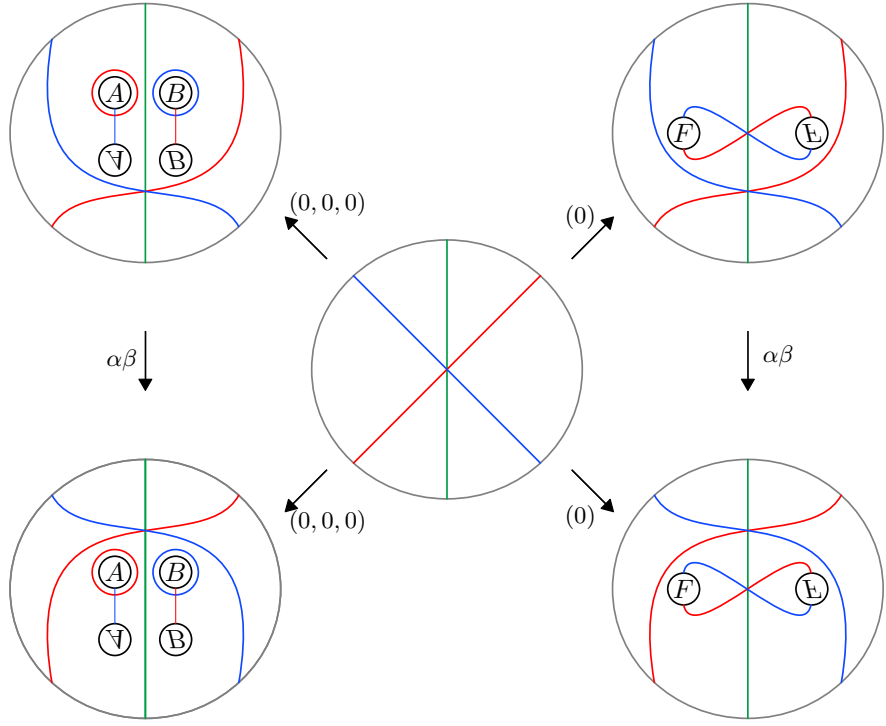
        \caption{The two kinds of stabilization slides.}
    \label{fig:stabilization_slides}
    \end{figure}

    \begin{figure}[h]
    \def\svgwidth{.8\linewidth}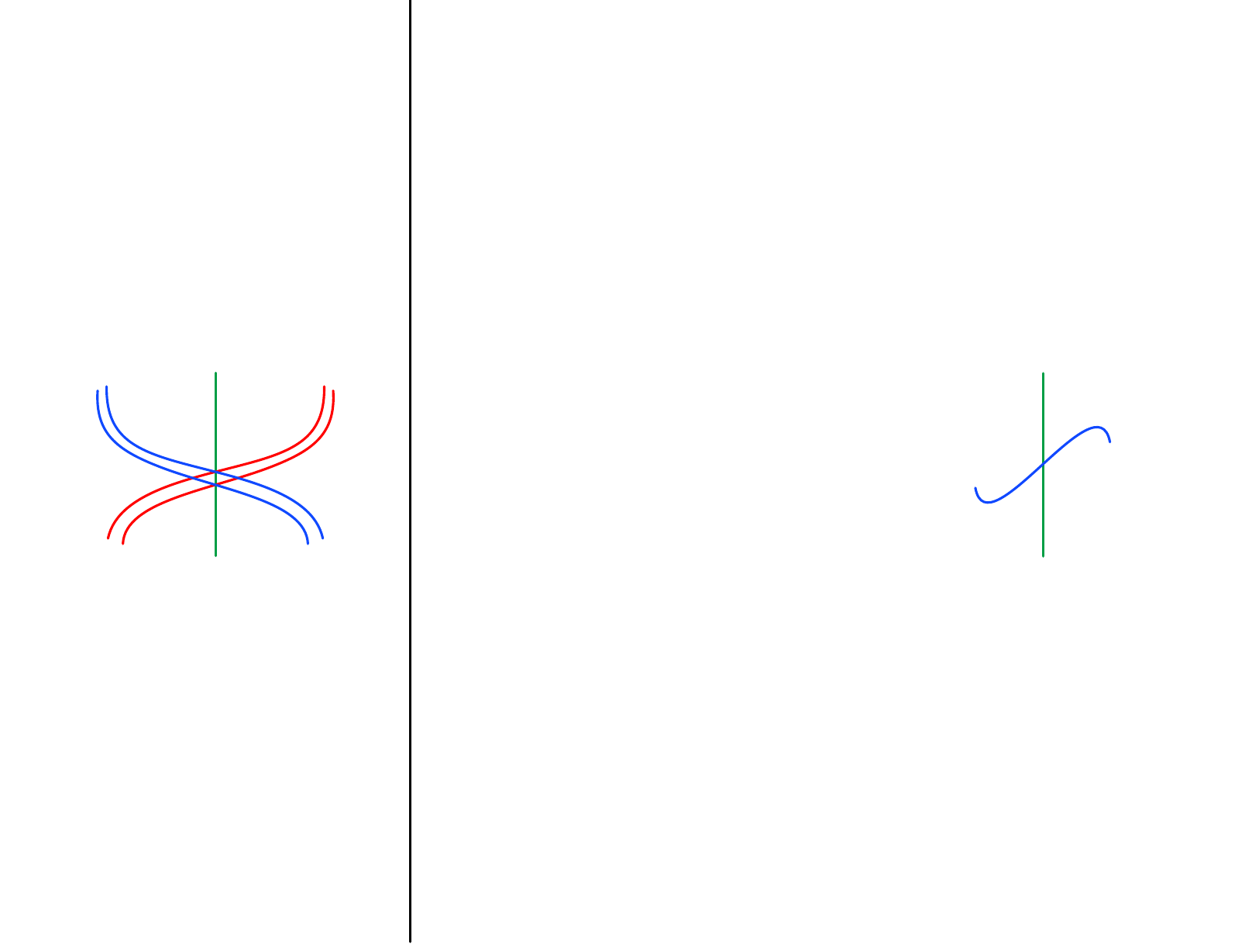
        \caption{Switching the orientation involved in a $\{1\}$-stabilization of type $(k)$. The top and bottom of the figure represent the two ways of resolving the stabilization.}
    \label{fig:resolution_orientations_triv}
    \end{figure}

Now, suppose $H'$ is obtained from $H$ by a crossover of type $(k, \ell)$. Suppose this crossover involves a neighborhood $A$ of a curve $\alpha$ and a neighborhood $B$ of the beta curve $\beta = \tau(\alpha)$. Let $\alpha_1, \hdots, \alpha_k$ and $\alpha_{k+1}, \hdots, \alpha_{k+\ell}$ be the alpha curves which intersect $\beta$ and let $\beta_1, \hdots, \beta_k$ and $\beta_{k+1}, \hdots, \beta_{k+\ell}$ be the beta curves which intersect $\alpha$. An example of such a resolution can be seen in \Cref{fig:resolution_orientations_crossover}, by traversing the strata along the top of the figure. We will refer to this figure throughout the following discussion. First, we perform a simple $\{1\}$-stabilization $H = H_0 \ra H_1$, introducing a new pair of curves $\alpha_0$ and $\beta_0$ which are exchanged by the involution (this is the first stratum in \Cref{fig:resolution_orientations_crossover}; there are two ways this can be done, as the roles of alpha and beta in the new handle can be swapped (in the language of \cite{guth_manolescu2025real}, we can change the framing of $C$ by either $+1$ or $-1)$. Next, we perform a diffeomorphism which drags $A$ and $B$ over the new handle. See the transition between the second and third chambers on the top row of \Cref{fig:resolution_orientations_crossover}. Call the resulting diagram $H_1'$. Next, we perform real handleslides over $\alpha_0$ and $\beta_0$; depending on the framing of the new handle, we either perform $k$ or $\ell$ handleslides. Following the top of \Cref{fig:resolution_orientations_crossover}, we will for now assume we have performed $\ell$ handleslides:
\begin{align*}
    H_{1}' \ra H_{2} \ra \hdots \ra H_{1 + \ell}.
\end{align*} 
We then handleslide $\alpha_0$ over $\alpha$ (and symmetrically $\beta_0$ over $\beta$), giving diagrams $H_{1 + \ell} \ra H_{2 + \ell}$. See the transition between Frames 4 and 5 of \Cref{fig:resolution_orientations_crossover}. Then, we perform $k$ handleslides
\begin{align*}
    H_{2+\ell} \ra \hdots \ra H_{2 +k + \ell},
\end{align*} 
over $\alpha_0$ (and $\beta_0)$ as in Frames 5 and 6 (or $\ell$ handleslides, in the case we have chosen the opposite framing of the $\{1\}$-stabilization). Finally, we perform a simple $\{1\}$-destabilization, giving the last move $H_{2+k+\ell} \ra H_{3+k+\ell}$. Again, this resolution is not unique, but we can interpolate between resolutions, as shown in \Cref{fig:resolution_orientations_crossover}.

    \begin{figure}[h]
    \def\svgwidth{1\linewidth}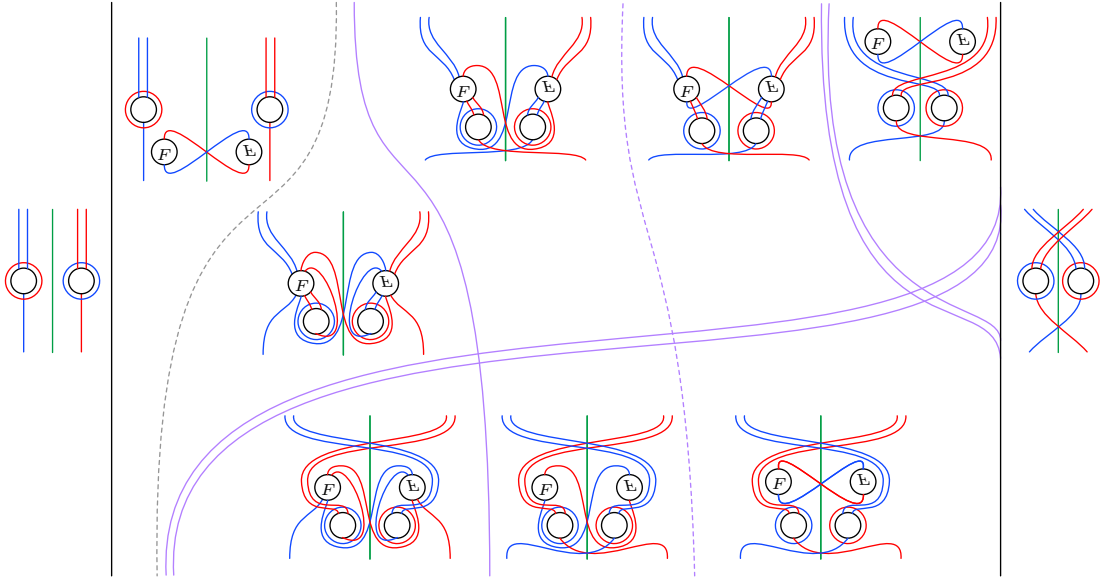
        \caption{Two ways of resolving a crossover of type $(k, \ell)$ and an interpolation between these resolutions. In both cases, there are $k+1+\ell$ real handleslides. Here, and later, the dashed purple curve represents the handleslide involving the curves created in the $\{1\}$-stabilization.}
    \label{fig:resolution_orientations_crossover}
    \end{figure}

Finally, we consider generalized $\Z/2$-stabilizations. It is slightly easier to describe the destabilization. Suppose $H'$ is obtained from $H$ by a $\Z/2$-stabilization of type $(k, \ell, c)$. Recall that there is a punctured torus $T$ in $H'$ which is removed (along with $\tau(T)$) and replaced with a pair of disks $D$ and $\tau(D)$ in $H$ (the two tori may intersect, as might the disks). The torus $T$ contains the two new attaching curves $\alpha$ and $\beta$, which we orient. By traversing $\beta$ starting from $\alpha$ following the orientation, we obtain a sequence of alpha arcs $\alpha_1, \hdots, \alpha_{\ell+c}$ (possibly with multiplicity). The curve $\beta$ intersects exactly $c$ of these $\alpha$-arcs on $C$. Likewise, the orientation of $\alpha$ gives an ordering of the $\beta$-arcs which intersect it, $\beta_1, \hdots, \beta_k$. 

To perform the destabilization, we drag $\beta$ along $\alpha$. We obtain a sequence 
\begin{align*}
    H' = H_0 \ra H_1 \ra \hdots \ra H_{\ell + c},
\end{align*}
where the diagram $H_i \ra H_{i+1}$ corresponds to a handleslide of $\alpha_i$ over $\alpha$ if $\alpha_i$ intersects $\beta$ away from the fixed set and corresponds to a crossover of type $(1, k)$ otherwise. Finally, we perform $k$ handleslides of the $\beta_1, \hdots, \beta_k$ over $\beta$:
\begin{align*}
    H_{\ell+c} \ra H_{1+\ell+c} \ra \hdots \ra H_{k+\ell+c}.
\end{align*}
At this point, $\alpha$ and $\beta$ intersect no other curve in the diagram and can be removed by a simple destabilization. For an example of this simplification, see \Cref{fig:resolution_orientations_Z2}; the sequence of moves is given by traversing a horizontal path at the top of the figure.

As in the case of $\{1\}$-stabilizations, there is a dependence on the orientations of the alpha curves, though we can interpolate between the two; see \Cref{fig:resolution_orientations_Z2}.\\

    \begin{figure}[h]
    \def\svgwidth{.8\linewidth}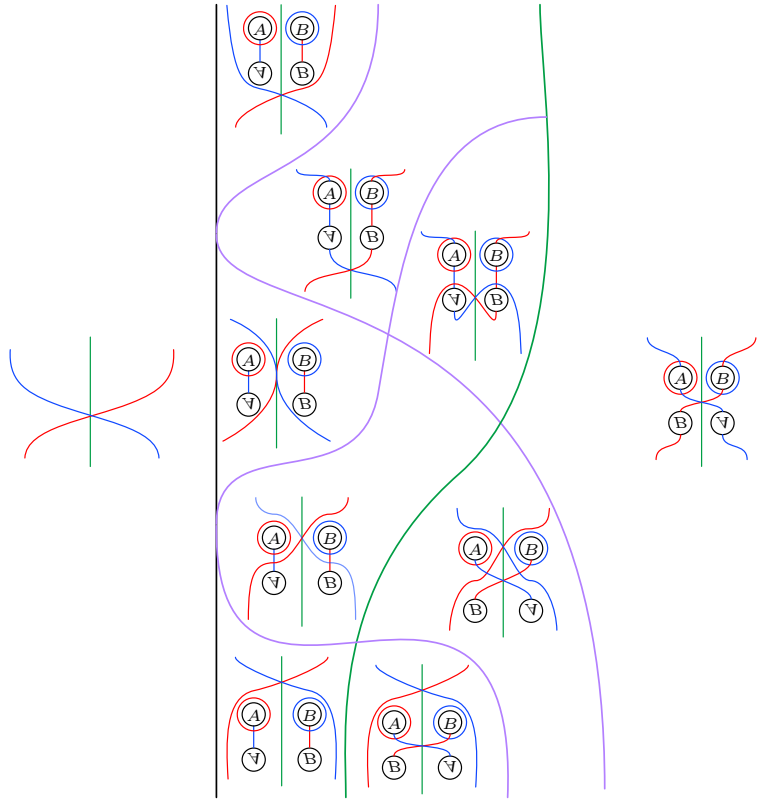
        \caption{Switching the orientation involved in a $\Z/2$-stabilization. The top and bottom of the figure represent two ways of resolving the stabilization.}
    \label{fig:resolution_orientations_Z2}
    \end{figure}

We now turn to the task of resolving the codimension-2 bifurcations appearing in Section~\ref{subsec:gradient-to-HD-codim2} into 2-parameter families whose bifurcation values correspond only to the simple loops of diagrams which appear in the definition of strong real Heegaard invariants. We provide paradigmatic examples for each bifurcation type in \Cref{appendix}. (Since there are many large figures of this kind, we put them in an  appendix to improve readability of the remaining sections.)

\subsection{Bifurcations of type (A)} 

The bifurcations of type (A1) and (A2) involve real handleslides, and do not need to be simplified, just as in the unreal case. On the other hand, bifurcations of types (A3) and (A4) require some work, as they involve real handleslides and crossovers. We must resolve these first, as we have chosen to resolve generalized $\Z/2$-stabilizations in terms of crossovers. 

Bifurcations of type (A4b) are the simplest. We simply resolve the $(k, \ell)$-crossover into a simple $\{1\}$-stabilization, followed by a sequence of $(k+1+\ell)$ real handleslides, and end with a $\{1\}$-destabilization. Since the handleslide and crossover involve different critical points, the resolved diagram consists of a single $\{1\}$-stabilization-handleslide commutation and $(k+\ell+1)$ handleslide commutations. See \Cref{fig:hd_a4b_resolved}.

(A3) is very similar: the two crossovers involve different critical points, so we may resolve them both into a $\{1\}$-stabilization followed by handleslides and a $\{1\}$-destabilization. The resolved diagram involves a $(\{1\}, \{1\})$-stabilization commutation as well as many handleslide and $\{1\}$-stabilization-handleslide commutations. See \Cref{fig:hd_a3_resolved}.

To resolve a bifurcation of type (A4a), we resolve the two crossover strata. The stratum $S_4$ is a crossover of type $(k, \ell)$ whereas $S_2$ is of type $(k +1, \ell)$, hence we must introduce a link of type (A2), i.e. a handleslide pentagon as in \cite[Figure 22]{JTZ_naturality_mapping_class_groups}.  An example is shown in \Cref{fig:hd_a4a_resolved}; the handleslide pentagon is the result of the loop around the central bifurcation point where the five handle-slide strata meet. 

Next, we consider bifurcations of type (A4c), which are the most complicated. We encourage the reader to compare the following discussion to \Cref{fig:hd_a4c_resolved}. Recall that in this setting, there are two alpha curves $\alpha$ and $\alpha'$ which may be resolved either into a pair of crossovers between $\alpha$ and $\beta = \tau(\alpha)$ and between $\alpha'$ and $\beta' = \tau(\alpha')$, or, alternatively, into a real handleslide of $\alpha'$ over $\alpha$ and a crossover between the new alpha curve and its symmetric beta. Let $S_2$, $S_i$, and $S_{i+1}$ be crossover strata of types $(k_3, \ell_3)$,  $(k_2,\ell_2)$, and $(k_1,\ell_1)$ respectively. Recall that $(k_3, \ell_3) = (k_1,\ell_2 + (\ell_1 - k_2))$.

A resolution of a crossover of type $(k, \ell)$ involves a diffeomorphism stratum, a simple $\{1\}$-stabilization, $\ell$ real handleslides, a real handleslide involving the new curves introduced in the stabilization,  $k$ more real handleslides, and then finally a simple $\{1\}$-destabilization. As in the previous section, strata which correspond to handleslides which involve the new curves from the stabilization will be drawn as dashed purple lines rather than solid. We complete the resolution as follows. Glue together the three diffeomorphism strata as in \Cref{fig:hd_a4c_resolved}. There are six stabilization strata which we connect up: the stabilization strata from the $(k_2, \ell_2)$-crossover and the $(k_3, \ell_3)$-crossover are connected, the $(k_2, \ell_2)$-crossover destabilization and $(k_1, \ell_1)$-crossover stabilization are connected, and the $(k_1, \ell_1)$-crossover destabilization and $(k_3, \ell_3)$-crossover destabilization are connected. 

Since $k_1 = k_3$, we can connect the $k_3$ handleslide strata from the resolution of the $(k_3, \ell_3)$-crossover to the $k_1$ handleslide strata from the resolution of the $(k_1,\ell_1)$-crossover. Similarly, the $\ell_2$ handleslide strata from the $(k_2,\ell_2)$-crossover can be connected to the $\ell_2$ corresponding strata in the $(k_3, \ell_3)$-crossover (again, recall that $\ell_3 = \ell_2 + (\ell_1 - k_2)$). 

The remaining strata are more complicated. In the simplest case, when $\ell_1 = k_2$, we can connect the $\ell_1$ strata from the $(k_1,\ell_1)$-crossover to the $k_2$ strata from the $(k_2,\ell_2)$-crossover. This leaves only the strata involving the curves introduced in the stabilization (i.e., the dashed purple strata). We glue these three strata together on the handleslide strata $S_1$; these five strata meet in a point, forming a handleslide pentagon. See \Cref{fig:hd_a4c_resolved} for an illustration in this case. Notice that we have used two different crossover resolutions in this example to simplify matters. We may interpolate between the two kinds of crossover resolutions as in \Cref{fig:resolution_orientations_crossover}.

In general, it may be the case that $\ell_1 > k_2$. Such an example is shown in \Cref{fig:hd_a4c_resolved2}. In this situation, we cannot match up all the handleslide strata in the three crossover resolutions. To describe the resolution, let $\alpha_1$ be the alpha curve associated to the birth-death $p_1$ and $\alpha_2$ be the alpha curve associated to the birth-death $p_2$ (i.e., these are the curves corresponding to the dashed purple strata). There are $(\ell_1 - k_2)$ beta arcs emanating from $\alpha_1$ which do not intersect $\alpha_2$. These are partitioned into two groups, of size $m_1$ and $m_2$ (there are $m_1$ arcs to the left of $\alpha_2$ and $m_2$ to the right). 

These additional arcs create two issues when trying to resolve the bifurcation. First, we must add in additional stabilization slides along the stabilization strata to interpolate between the various crossovers. This is due to the fact that the $m_2$ $\beta$-arcs which do not intersect $\alpha_2$ must be slid over $\beta_2 = \t(\alpha_2)$ before the $(k_2,\ell_2)$ crossover can be performed. On the right side of the diagram, there is a unique component of $\S \smallsetminus \alphas \smallsetminus \betas$ where the stabilization may take place. On the left half of the diagram there are three regions of the diagram in which we can stabilize, created by the handleslides preceding the $(k_2,\ell_2)$ crossover. Therefore, below the topmost stabilization strata, we must add in stabilization slides to interpolate between the two stabilization locations. There is a similar phenomenon for the other two stabilization strata. 

Second, in each crossover resolution, there is a distinguished handleslide which involves the curves $\alpha_0$ and $\beta_0$ introduced by the stabilization (the dashed purple curves). At these strata, $\alpha_0$ slides over either $\alpha_1$ or $\alpha_2$ (depending on which crossover we are near). In the interior of the diagram, there are parameters at which $\beta_0$ \emph{cannot} slide over $\beta_2$, because it is obstructed by the $m_2$ beta arcs which also slide over $\beta_2$ at various points of the resolution. At these points, we must introduce handleslide pentagons between $\beta_0$, $\beta_2$, and the $m_2$ beta arcs (i.e., we must add in singularities of type (A2)). For an example in the general case, see \Cref{fig:hd_a4c_resolved2}.

\subsection{Bifurcations of type (E2)}

Next, we resolve bifurcations of type (E2), as they  are relatively straightforward and will come up in later resolutions. We resolve the two crossovers. Both resolutions begin with a $\{1\}$-stabilization and end with a $\{1\}$-destabilization. In between the (de)-stabilizations, the first crossover involves $\ell$ real handleslides, a distinguished handleslide involving the new alpha and beta curves coming from the stabilization, and then $k$ more real handleslides. The second crossover  involves $\ell+1$ real handleslides, a distinguished handleslide involving the new alpha and beta curves coming from the stabilization, and then $k-1$ more real handleslides. The first $\ell$ slides are identical, as are the last $k-1$. The remaining real handleslides can be resolved into an (A2) singularity, forming a pentagon. See \Cref{fig:hd_e2_resolved}.

\subsection{Bifurcations of type (B)}

We now turn to bifurcations involving stabilizations. Since we have chosen to resolve generalized  $\Z/2$-stabilizations into simple $\Z/2$-stabilizations, handleslides, and crossovers, bifurcations of type (A) will appear in our resolutions. To avoid further complicating the figures, we will not depict the resolutions of crossover strata. Additionally, we will at times reduce to diagrams which contain bifurcations of types (A1)--(A4), which we have already shown how to resolve.

The simplest type (B) bifurcations are those of types (B3) and (B5). Recall that there are two cases, depending on the orbit type of the stabilization. In the trivial case, we simply resolve the $\{1\}$-stabilization stratum and obtain a resolution with no singularities. See \Cref{fig:hd_b3_triv_resolved}.

In the case of $\Z/2$-orbits, we must be more careful. The $S_1$ stratum is a stabilization of type $(k,\ell+t, t)$ while $S_3$ is of type $(k-1,\ell, t)$, so when we resolve the two generalized $\Z/2$-stabilizations, there are an unequal number of handleslide strata. More specifically, in the resolution of the $S_3$ stratum, for each crossover, we must perform an additional handleslide that is not required in the $S_1$ resolution. To account for this, we simply merge the extra handleslide with the corresponding crossover, creating an (A4c) singularity. See \Cref{fig:hd_b3_Z2_resolved} for an example.

Case (B5) is very similar. Stratum $S_1$ is a generalized $\Z/2$-stabilization of type $(m+c+1, n, c+1)$ and $S_3$ is a generalized $\Z/2$-stabilization of type $(m + c, n,c)$. We simply resolve $S_1$, and since the resolution of $S_3$ is a sub-resolution of $S_1$, there is no additional work to do. See \Cref{fig:hd_b5_resolved}.

Bifurcations of type (B1) are also relatively straightforward. In the case of a $\{1\}$-stabilization, the resolution is exactly like the unreal case. We resolve the stabilization into a simple $\{1\}$-stabilization and $k = k_1 + k_2$ real handleslides, though we must introduce $k$ pentagon slides (singularities of type (A2)). For an example, see \Cref{fig:hd_b1_triv_resolved}.

The case of generalized $\Z/2$-stabilizations is very similar. We break the type $(k, \ell, r)$ stabilization into a single simple $\Z/2$-stabilization and then a sequence of $k = k_1 + k_2$ handleslides, then a sequence of crossovers and handleslides. To resolve strata $S_1$ and $S_3$, we again introduce $k$ pentagon slides, just as in the previous case. The remaining handleslides and crossovers are identical for both strata, so no additional care is needed. For an example, see \Cref{fig:hd_b1_Z2_resolved}.

Next, consider bifurcations of type (B2). The case involving a generalized $\Z/2$-stabilization is nearly identical to the resolution in the unreal case. Recall that we have a birth-death singularity $p^0$ with $k$ flows in, $\ell$ flows out, and $c$ flows between $p^0$ and $\tau(p^0)$; additionally, there are $m_1 + m_2$ flows from $\overline{p}^0$ to index 2 critical points. The stabilization strata $S_1$ and $S_3$ are of types $(k+m_1 + c, \ell, c)$ and $(k+m_2 + c, \ell, c)$ respectively; the $m_1$ real handleslides from resolving $S_1$ cannot always be matched up with the $m_2$ real handleslides from resolving $S_3$. Therefore, we remove a small disk $D$ where these $m_1 + m_2$ strata terminate. See \Cref{fig:hd_b2_Z2_resolved}. The loop around this disk corresponds exactly to a real handleswap (i.e., a bifurcation of type (E1)).

The case of a $\{1\}$-stabilization is similar. The relevant parameters are the $k$ flows from index 1 critical points to the birth-death singularity $p^0$ and the $m_1 + m_2$ flows from the hyperbolic critical point $\overline{p}^\mu$. Stratum $S_1$ is a stabilization of type $(k+m_1)$ and $(k+m_2)$. As in the previous case, we resolve $S_1$ and $S_2$ into a simple stabilization followed by a sequence of real handleslides, but the $m_1$ handleslides from resolving $S_1$ may not match up with the $m_2$ handleslides from resolving $S_2$, so we remove a disk $D$. Unlike in the previous case, there is also a crossover stratum which terminates on $\partial D$ as well. Nevertheless, a loop around $\partial D$ is again a real handleswap, which is permitted to involve crossover strata as well as handleslides. See \Cref{fig:hd_b2_triv_resolved}.

Let us consider the case of (B4) bifurcations. The $\{1\}$-stabilization case is straightforward. Recall that there is a birth-death singularity $p^0$ with $m$ flows in and out as well as a tangency between the 1-dimensional stable manifolds of $p_1^0$ and $\tau(p_1^0)$; we assume there are $k_1 + k_2$ flows from $p_1^\mu$ to index 2 critical points partitioned into two sets. Strata $S_1$ and $S_2$ both correspond to $\{1\}$-stabilizations of type $(m)$, and therefore can be resolved simultaneously, and likewise $S_2$ and $S_4$ are both crossover strata of type $(k_1, k_2)$, so can also be resolved simultaneously. See \Cref{fig:hd_b4_triv_resolved} for an example.

The case of generalized $\Z/2$-stabilizations is more subtle. The relevant parameters are as follows: there are $n$ flows out of $p^0$ to index 2 critical points; there are $r_1 + r_2$ flows from $p_1^0$ to $p^0$ partitioned into two subsets by the directional manifold of $p_1^0$ as well as $m$ additional flow lines from additional index 1 critical points; there are also $t$ flow lines from $\tau(p^0)$ to $p^0$; there are $\ell_1 + \ell_2$ flow lines out of $p_1^\mu$ to index 2 critical points. Recall that strata $S_1$ and $S_3$ are generalized $\Z/2$-stabilizations of types $(m + r_1+r_2 + t, n, t)$ and $(m+r_1+r_2+t+r_1(r_1 + \ell_1) + r_2(r_2 + \ell_2), n, t+r_1+r_2)$ while strata $S_2$ and $S_4$ are crossovers of types $(r_1n+\ell_1, r_2n + \ell_2)$ and $(r_1+\ell_1, r_2 + \ell_2)$. 

We will begin by resolving the stabilization strata, and as we will see, this will be sufficient to reduce to previous cases. The resolutions for $S_1$ and $S_3$ are identical except that the resolution for $S_3$ involves additional crossovers (and adjacent handleslides). More precisely, let $\alpha_0$ be the alpha curve associated to the index-1 critical point born from $p^0$ and let $\alpha$ be the alpha curve for $p_1$ which is involved in the crossover. Each time $\alpha$ slides over $\alpha_0$ on the upper half of the diagram, the corresponding sequence on the lower half of the diagram after the crossover involves a sequence of handleslides, a crossover of $\alpha_0$ and $\tau(\alpha_0)$, a second series of real handleslides, and then $\alpha$ slides over $\alpha_0$. These two paths are precisely the loop around a singularity of type (A4c). We simply repeat for each time $\alpha$ slides over $\alpha_0$ (which happens $r_1 + r_2$ times in total). See \Cref{fig:hd_b4_Z2_resolved} for an example of the resolved diagram.

\subsection{Bifurcations of type (C)}

Next, we resolve bifurcations of type (C), which involve simultaneous $\cO$-stabilizations (where $\cO \in \{\Z/2, \{1\}\}$). The simplest case is that of a pair of $\{1\}$-stabilizations. Since there can be no flows between two birth-death singularities appearing in trivial orbits, we simply resolve the two $\{1\}$-stabilizations into a pair of simple $\{1\}$-stabilizations and real handleslides. See \Cref{fig:hd_c_1_1_resolved}.

Next, we consider a generalized $\Z/2$-stabilization and a generalized $\{1\}$-stabilization. The generalized $\{1\}$-stabilization is resolved into a simple $\{1\}$-stabilization and a sequence of handleslides and the generalized $\Z/2$-stabilization is resolved into a simple $\Z/2$-stabilization and a combination of handleslides and crossovers. The key parameter is the number of flows $t$ between the two kinds of birth-death singularities. For each such flow line, we must remove a small disk in the parameter space; the horizontal flow lines correspond to handleslides over the new $\{1\}$-handle, while the vertical flow lines correspond to crossovers of the new $\Z/2$-handles and handleslides of the same alpha and beta curves. The loop around this disk corresponds exactly to a real handleswap (E1); the feet of the $\Z/2$-handles are sliding around the alpha and beta curves introduced in the $\{1\}$-stabilization. See \Cref{fig:hd_C_Z2_1_resolved} for the resolution.

The case of two generalized $\Z/2$-stabilizations is the most complicated. We deal with the $\{1\}$-orbit flow lines between the degenerate critical points just as in the unreal case. For each flow line of this kind, there are strata which cannot be completed, so we remove a small disk from the parameter space; a loop around the removed disk is a handleswap. 

The new phenomenon is due to the $s$ pairs of flow lines from $p_1$ to $\tau(p_2)$. In these cases, the curve $\alpha_2$ associated to the index $1$ critical point born from $p_2$ slides over the curve $\alpha_1$ associated to the index $1$ critical point born from $p_1$. On the half of the parameter space where $\alpha_2$ has not slid over $\alpha_1$, there are many handleslides over $\alpha_1$ involved in resolving the stabilization coming from $p_1$. On the half of the parameter space where $\alpha_2$ has slid over $\alpha_1$, these handleslides do not occur, as they are slid over $\alpha_2$ instead. Similarly, crossovers involving $\alpha_1$ on the left half of the diagram may be traded for crossovers involving $\alpha_2$ on the right-hand side.

To address these issues, we will glue the extra handleslide and crossover strata on the left side of the diagram to the handleslide strata corresponding to the slide of $\alpha_2$ over $\alpha_1$. For the handleslides, this amounts to adding additional (A2) bifurcations and for the crossovers, this amounts to adding additional (A4c) bifurcations. See \Cref{fig:hd_C_Z2_Z2_resolved} for an example.

\subsection{Bifurcations of type (D)}

Bifurcations of type (D) involve codimension-2 singularities. Bifurcation (D1) is a pair consisting of an $A_3^+$ singularity and an $A_3^-$ singularity. We resolve the singularity just as in the unreal case: we introduce a simple $\Z/2$-stabilization together with a sequence of real handleslides and crossovers; the sequences corresponding to the original two strata may not match up, in which case we can resolve the singularity by adding in a real handleswap and a handleslide triangle. See \Cref{fig:hd_d1_resolved}; in the figure, we have only drawn an annular neighborhood of the beta curve corresponding to the singularity, and not its symmetric alpha curve.

Next, we consider a singularity of type (D2). Let $p^0$ be the singularity of type $\{1\}$-$A_4$ and suppose that there are $k$ flows in and out of $p^0$. The loop around the singularity involves two generalized $\{1\}$-stabilizations and then a single generalized $\Z/2$-destabilization. We resolve the first generalized $\{1\}$-stabilization stratum into a simple $\{1\}$-stabilization and $k$ real handleslides; starting from the destabilized side, the generalized $\Z/2$-stabilization is resolved into a simple $\Z/2$-stabilization followed by $k$ handleslides followed by a crossover (which we resolve momentarily). We can connect the $k$ real handleslides from the generalized $\{1\}$-stabilization with the $k$ slides from the generalized $\Z/2$-stabilization. Next, we resolve the crossover into a simple $\{1\}$-stabilization followed by two real handleslides and a simple $\{1\}$-destabilization; we connect the stabilization and destabilization strata. Finally, we resolve the second generalized $\{1\}$-stabilization into a simple $\{1\}$-stabilization and a real handleslide. It remains to extend the resolution over the central disk of \Cref{fig:hd_d2_resolved}.  The six remaining strata cannot be glued together to form the simple loops appearing in the definition of a strong real Heegaard invariant. Therefore, we simply glue the remaining strata at a point, a loop around which is exactly a simple trade loop. See \Cref{fig:hd_d2_trade_loop}.

\subsection{Bifurcations of type (E1)}

Finally, we resolve singularities of type (E1), real handleswaps. Let us briefly recall the set-up and introduce some terminology. 

\begin{defn}\label{def:real handle swaps}
    A \emph{real handleswap of type $(k, \ell, r)$} is a loop of overcomplete diagrams $\cH_0, \hdots, \cH_{k+\ell+r}$, as follows. There is a surface $\Sigma$ such that $\cH_i = (\Sigma, \bm\alpha_i, \bm\beta_i)$ for each $i$ and there is a pair of pants $P$ so that
    \begin{align*}
        \bm \alpha_i \cap (\Sigma \smallsetminus P) = \bm \alpha_j \cap (\Sigma \smallsetminus P) \text{ and }\bm \beta_i \cap (\Sigma \smallsetminus P) = \bm \beta_j \cap (\Sigma \smallsetminus P)
    \end{align*}
    for each $i$ and $j$. Inside $P$, there is a full alpha circle $\alpha_0$. The boundary of $P$ consists of three circles: $A$ which is parallel to $\alpha_0$, $B$, and $Z$. The boundary component $B$ is parallel to a beta curve $\beta_0$, which may or may not be contained in $P$. The collection $(\bm \alpha \smallsetminus \alpha_0)\cap P$ is a collection of arcs connecting $B$ to $Z$, $(\bm \beta \smallsetminus \beta_0)\cap P$ is a collection of arcs connecting $A$ to $Z$, and $C \cap P$ is a collection of arcs connecting $A$ to $Z$. Precisely $k$ of the real Heegaard moves $\cH_{i} \ra \cH_{i+1}$ are given by a real handleslide of an alpha curve over $\alpha_0$ followed by a diffeomorphism dragging $B$ along $\alpha_0$, $\ell$ of the moves are handleslides of beta curves over $\beta_0$ (followed by a diffeomorphism dragging $B$ along $\alpha_0$), and the remaining are crossovers of $\beta_0$ with $\tau(\beta_0)$.\\

    A real handleswap of type $(k, \ell, r)$ of isotopy diagrams is a loop $H_0, \hdots, H_{k + \ell+r}$ such that there are representatives $\cH_i$ for $H_i$ such that $\cH_0, \hdots, \cH_{k+\ell+r}$ is a real $(k, \ell, r)$-handleswap.
\end{defn}

The first step in the simplification process is to reduce to the case of $(k, \ell, 0)$ crossovers. Suppose $S_i$ is a crossover stratum. Each crossover stratum corresponds to an intersection point $p$ between $C$ and $\alpha_0$. Introduce a circular stratum corresponding to a $\{1\}$-stabilization of type (1) around the singularity which removes this intersection point. We resolve this into a simple $\{1\}$-stabilization stratum $S$ near $p$ and then perform a real handleslide of $\alpha_0$ and $\tau(\alpha_0)$ over the new handle. The new stratum $S$ meets $S_i$, introducing a (B2) bifurcation involving the $\{1\}$-stabilization. On the stabilized side of $S$, we have a handleslide stratum, rather than a crossover; in particular, we see a real handleswap loop of type $(k, \ell+1, r-1)$. See \Cref{fig:hd_e1_resolved0}. We repeat this process for each crossover stratum, after which we will have obtained a real handleswap loop of type $(k, \ell+r, 0)$.

Having eliminated all intersections between $\alpha_0$ and $\tau(\alpha_0)$, we can reduce to the case of ordered real handleswaps following the discussion after \cite[Definition 7.9]{JTZ_naturality_mapping_class_groups} by introducing commutation moves between diffeomorphism and handleslide strata (see especially \cite[Figures 41 and 42]{JTZ_naturality_mapping_class_groups}). 

At this point, we are actually in a position to apply the simplification process of \cite[Section 7.3]{JTZ_naturality_mapping_class_groups} equivariantly, as we have almost reduced real handleswaps to disjoint pairs of unreal handleswaps. There is one subtlety, namely that it could be the case that $\alpha_0$ and $\tau(\beta_0)$ intersect (this was the first case considered in the description of (E1) in \Cref{subsec:codim 2 birfurcations}), in which case the local pictures shown in \Cref{fig:hd_e1_resolved1} are actually connected; we identify annular neighborhoods of the two alpha curves as well as annular neighborhoods of the two beta curves. 

More precisely, Juh\'asz-Thurston-Zemke reduce their analysis to the case of (1;1)-handleswaps; this simply means that there is a single beta curve $\beta$ which intersects $\alpha_0$ and a single alpha curve $\alpha$ which intersects $\beta_0$. In their setting, the final step is to consider all alpha arcs that intersect $\alpha$ and slide them over $\alpha_0$, and likewise slide beta arcs intersecting $\alpha$ over $\beta_0$. However, this may be impossible in our setting, as it could be the case that $\beta = \tau(\alpha)$ and these curves intersect. In this case, we cannot apply real handleslides to reduce to the case of simple handleslides. By applying the Juh\'asz-Thurston-Zemke algorithm slightly more carefully we can actually preclude this possibility, as we now explain. 

The first step in their simplification process is to reduce to the case that there is a beta curve which intersects $\alpha_0$ exactly once and an alpha curve which intersects $\beta_0$ exactly once (Juh\'asz-Thurston-Zemke call this reducing to $(k,1;\ell;1)$-handleswaps). They do this by performing a $(k,1)$-stabilization involving $\alpha_0$ and a $(\ell,1)$-stabilization involving $\beta_0$ and then resolving the stabilization strata (see \cite[Figure 54]{JTZ_naturality_mapping_class_groups}). We may apply this step equivariantly as well. The crucial point for us is that in performing the $\Z/2$-stabilization, we introduce a curve $\beta'$ which intersects $\alpha_0$ once and does not intersect $\tau(\beta')$, and it is this curve that appears when we reduce to $(1,1)$-handleswaps. See \Cref{fig:hd_e1_resolved1} for an example. In short, we need to apply step 1 of their simplification process even if we began in the case of a pair of $(1,1)$-handleslides, as we need to ensure the handleswaps are disjoint.

At this point, we can apply the rest of their algorithm equivariantly with no issue to reduce to the case of a simple real handleswap.

\section{Strong real Heegaard Floer invariants have no monodromy}\label{sec:reducing no_monodromy}
We now have the tools to prove \Cref{prop:connected-strong} and \Cref{thm:iso}, which we restate here.

{
\renewcommand{\thelem}{\ref{prop:connected-strong}}
\begin{proposition}
    Any two isotopy diagrams in $\G_{(Y,\g,\tau)}$ can be connected by an oriented path.
\end{proposition}
\addtocounter{lem}{-1}
}
\begin{proof}
    This is the same as \cite[Proof of Proposition 2.36]{JTZ_naturality_mapping_class_groups} and follows from our analysis of 1-parameter families of real gradients. See \Cref{cor:weaklycon}, \Cref{prop: codim 1 to heegaard moves}, and \Cref{lem:sym_6.21}. 
\end{proof}

{
\renewcommand{\thethm}
{\ref{thm:iso}}
\begin{thm}
 Let $\cS^R$ be a set of diffeomorphism types of real sutured manifolds containing $[(Y, \gamma, \tau)]$. Suppose that $F: \cG^R(\cS^R) \ra \cC$ is a strong real Heegaard invariant. Given isotopy diagrams $H, H' \in |\cG^R_{(Y, \gamma, \tau)}|$ and any two oriented paths $\eta$ and $\nu$ in $\cG^R_{(Y, \gamma, \tau)}$ connecting $H$ to $H'$, we have 
    \begin{align*}
        F(\eta) = F(\nu).
    \end{align*}
\end{thm}
\addtocounter{thm}{-1}}

\begin{proof}
    This is nearly identical to the proof in the unreal case \cite[Theorem 2.38]{JTZ_naturality_mapping_class_groups}, so we only sketch the proof, emphasizing the necessary modifications. By the Functoriality Axiom, it suffices to prove that, for any loop $\eta$ of the form 
    \begin{align*}
        H_0 \xra{e_1} H_1\xra{e_2} \hdots \xra{e_{n-1}} H_{n-1}\xra{e_n} H_0,
    \end{align*}
    we have $F(\eta) = \id_{F(H_0)}$. We then construct a generic 2-parameter family $\cF: D^2 \ra \FV(Y, \gamma, \tau)$ and a surface-enhanced polyhedral decomposition $\cP$ of $D^2$ so that along $\partial D^2$ we have the loop $\eta$ just as in \cite{JTZ_naturality_mapping_class_groups}. We then apply the resolution process of \Cref{sec:reducing moves} to obtain a decomposition $\cP'$ with the property that each 2-cell has boundary corresponding to a simple loop appearing in the axioms of a strong real Heegaard invariant: a loop of real equivalences, a loop of equivariant diffeomorphisms, a distinguished rectangle, a simple real handleswap, a stabilization slide (which is just a degenerate distinguished rectangle), or a simple trade. The composition of such a loop of diffeomorphisms is the identity by \Cref{lem:sym_6.24}. The vertices of $\cP'$ may be overcomplete, but we can follow \cite{JTZ_naturality_mapping_class_groups} to obtain a decomposition $\cP''$ decorated by actual real Heegaard diagrams without altering the boundary. Each 2-cell of $\cP''$ corresponds uniquely to a cell (of some dimension) of $\cP'$, and we have to check that $F$ commutes along each of these. The analysis in \cite{JTZ_naturality_mapping_class_groups} carries through in our setting, so it suffices to address the case that $\sigma$ is a 2-cell of $\cP''$ corresponding to a 2-cell $\sigma_0$ in $\cP'$ so that $\partial \sigma_0$ is a simple trade loop. But in this case, the spanning tree plays no essential role. Let $K_1, \hdots, K_5$ be the real isotopy diagrams decorating the vertices of $\partial \sigma_0$. Start with a spanning tree $T^1$ for $\Gamma(K_1)$; as $K_2$ and $K_3$ are obtained by stabilizations, we naturally obtain spanning trees $T^2$ and $T^3$ for $\Gamma(K_2)$ and $\Gamma(K_3)$; a spanning tree $T^4$ for $\Gamma(K_4)$ is obtained by sliding the new edges over one another and one for $\Gamma(K_5)$ is obtained by pushing $T^4$ forward by the diffeomorphism relating $K_4$ and $K_5$. We then take $D_i = H(K_i, T^i)$ and $D_1, \hdots, D_5$ clearly still form a simple trade loop, and hence $F$ commutes.
\end{proof}

\section{Naturality of real Heegaard Floer homology}\label{sec:real hf}
In this section, we establish naturality of real Heegaard Floer homology. First, we review the construction of \cite{guth_manolescu2025real} and our proof of invariance up to isomorphism. There we proved real handleslide invariance by appealing to Perutz \cite{perutz}. For ease of computation, we will reformulate this statement in terms of holomorphic polygon counting maps. Then we will verify the remaining axioms of a strong real Heegaard Floer theory: the most interesting verification is that of the simple trade loop. To verify the computation, we will briefly set up a cylindrical reformulation of real Heegaard Floer homology.

\subsection{Real Lagrangian Floer homology}

We begin with some background on our theory. Let $(M, \omega)$ be a symplectic manifold equipped with an anti-symplectic involution $R: M \ra M$; i.e. a smooth involution satisfying $R^*(\omega) = - \omega.$ Let $L$ be a Lagrangian submanifold of $M$. Let $\cJ$ be the space of time-dependent almost complex structures on $M$ and let $\cJ_R \sub \cJ$ be the subset satisfying the condition
\begin{align*}
    J_t \circ R_* = - R_* \circ J_{1-t}.
\end{align*}
Let us write $M^R$ for the fixed point set of $R$; $M^R$ is a smooth submanifold of $M$ and is, in fact, a Lagrangian submanifold. Therefore, for a path of real-invariant almost complex structures $J_s \in \cJ_R$, we define 
\begin{align*}
    \CFR_{J_s}^\circ(L) := \CF^\circ_{J_s}(L, M^R).
\end{align*}
Note that as usual, we can construct continuation maps
\begin{align}\label{eqn:continuation_J}
    \Psi_{J_s \ra J_s'}: \CFR_{J_s}^\circ(L) \ra \CFR_{J_s'}^\circ(L)
\end{align}
which show these complexes are independent of our choice of almost complex structure. Furthermore, the notation is intended to indicate that the real Floer chain complex associated to the Lagrangian $L$ and involution $R$ is simply the usual Lagrangian Floer chain complex of $L$ and $M^R$; in particular, we impose no symmetry relation on the holomorphic strips which contribute to the differential. However, as $J$ is assumed to be real-invariant, there is a correspondence between holomorphic strips between $L$ and $M^R$ and real invariant strips between $L$ and $R(L)$, i.e. holomorphic maps satisfying: 
\begin{align*}
    u: [0,1] \times \R \ra M, \quad \begin{cases}
        u(\{1\} \times \R) \sub L \\
        u(\{0\} \times \R) \sub R(L)\\
        u(1-s, t) = R \circ u(s, t).
    \end{cases}
\end{align*}
Furthermore, the generators of $\CF_J^\circ(L, M^R)$ (i.e. elements of $L \cap M^R$) are exactly the intersection points of $L \cap R(L)$ which are fixed by $R$. In short, we will frequently make use of the fact that we may either view $\CFR_J^\circ(H)$ as the usual Lagrangian Floer chain complex of $L$ and $M^R$, or the complex generated by $(L \cap R(L))^R$ with differential which counts rigid real invariant holomorphic strips.

\subsection{Real holomorphic polygons}

As is typical in Floer theory, we obtain maps between our Floer complexes by counting holomorphic triangles. Suppose we have two Lagrangian submanifolds $L_1$ and $L_2$ of $(M, \omega)$. Just as there is a correspondence between $J$-holomorphic strips between $M^R$ and $L_1$ (as well as between $M^R$ and $L_2$) with real invariant $J$-holomorphic strips between $R(L_1)$ and $L_1$ (as well as between $R(L_2)$ and $L_2$), there is also a correspondence between $J$-holomorphic triangles between $M^R$, $L_1$, and $L_2$ and invariant $J$-holomorphic rectangles between $L_1$, $L_2$, $R(L_2)$, and $R(L_1)$. 

More generally, let $\P_k$ be a $k$-gon in the plane; we do not fix a complex structure on $\P_k$. Label the edges of $\P_k$ by $e_1, \hdots, e_k$. Let $L_1, \hdots,L_k$ be a collection of Lagrangian subspaces of $(M, \omega)$. A \emph{Whitney $n$-gon} is a map $\phi: \P_n \ra M$ such that $\phi(e_i) \sub L_i$. Let $\cC(\P_k)$ be the space of conformal structures on $\P_k$. Given $\x_i \in L_i \cap L_{i+1}$ for all $i \in \Z/k$ and a class $\psi \in \pi_2(\x_1, \hdots, \x_k)$, we define $\cM(\psi)$ to be the space of holomorphic $k$-gons $\P_k \ra M$  in the class $\psi$. The expected dimension of $\cM(\psi)$ is given by 
\begin{align*}
    \mu(\psi) + \dim \cC(\P_k) = \mu(\psi) + (k-3).
\end{align*}
In our setting, we consider Lagrangian subspaces $L_1, \hdots, L_{k-1}, M^R \sub M$.  Given a class $\psi \in \pi_2(\x_1, \hdots, \x_{k})$ and $J \in \cJ_R$, we define $\cM_R^J(\psi)$ to be the space of $J$-holomorphic $k$-gons $\P_k \ra M$ in the class $\psi$. The expected dimension of $\cM_R(\psi)$ is given by 
\begin{align*}
    \mu_R(\psi) + \dim \cC(\P_k) = \mu_R(\psi) + (k-3).
\end{align*}
The symmetry condition on $J$ sets up our correspondence between Whitney $k$-gons between Lagrangians $L_1, \hdots, L_{k-1}, M^R$ and real invariant $(2k-2)$-gons between $$R(L_{k-1}), \hdots, R(L_1), L_1, \hdots, L_{k-1}.$$ 
We remark that the space of real conformal structures on $\P_{2k-2}$ is $\cC^R(\P_{2k-2})=\cC(\P_k)$. 

We now specialize to the case of triangles (that is, real-invariant rectangles). Let $\Diamond$ be the unit disk in the complex plane with four marked points. The space of conformal structures of the rectangle $\cM(\Diamond)$ is one dimensional, parametrized by the cross ratio. In \cite{os_holodisks}, they consider the moduli space $\cM(\psi)$ of holomorphic maps of the rectangle in the homotopy class $\psi$ into $\Sym^m(\Sigma)$ without fixing the complex structure on $\Diamond$. In our setting, we will fix the almost complex structure on $\Diamond$ which is preserved by complex conjugation (which, of course, agrees with the fact that there is a unique conformal structure on the triangle). In the usual setting, holomorphic rectangle maps are not chain maps; when the cross-ratio goes to 0 or infinity, rectangles break into a pair of triangles meeting at a vertex. However, in our setting, this kind of degeneration cannot occur. Because of the symmetry, at the ends of 1-parameter families of real rectangles, only strips can break off. It follows that the real holomorphic rectangle counting maps 
\begin{align*}
    F_{L_1,L_2}: \CF^\circ(L_2, L_1)\otimes \CFR^\circ(L_1, R(L_1)) \ra \CFR^\circ(L_2, R(L_2)), 
\end{align*}
are actually chain maps in the real setting. We note here that while $F_{L_1,L_2}$ does have four inputs, the input coming from $R(L_1)\cap R(L_2)$ is determined by the one coming from $L_1 \cap L_2$.

As in the case of bigons, there is a simple relationship between the real index of a rectangle and its ordinary index. The space of conformal structures of the rectangle is one dimensional. However, the space of conformal structures on the rectangle which respect the real structure is zero dimensional, as the symmetry forces the cross-ratio to be one. More generally, the space of conformal structures of a $(2k-2)$-gon is $(2k - 5)$-dimensional, while the space of real conformal structures is $(k-3)$-dimensional (since the quotient is a $k$-gon).

\begin{lem}
    Let $\psi \in \pi_2^R(\theta, \x, R(\theta), \y)$ be a homotopy class of invariant rectangles and let $\Tilde{\psi}$ be its image in $\pi_2(\theta, \x, R(\theta), \y)$. Then 
    \begin{align}\label{eqn:index_for_rectangles}
        \ind(\tilde{\psi}) = 2\ind_R(\psi) + \frac{\sigma(L_0,\x) - \sigma(L_1,\y)}{2}.
    \end{align}
\end{lem}
\begin{proof}
    This is similar to \cite[Proposition 2.3]{guth_manolescu2025real}. The index of a holomorphic rectangle $u$ in the class $\psi$ is given by trivializing $u^*TM$ over the disk and considering the loop $\lambda$ of Lagrangian subspaces in $\C^n$ determined by $L_0$, $L_1$, $R(L_0)$, and $R(L_1)$ which in turn determines an integer, as the fundamental group of the Lagrangian Grassmannian $\pi_1(U(n)/O(n)) \cong \Z$. We write $\mu(\lambda) = \ind(\tilde{\psi})$ and $\mu_R(\lambda) = \ind_R(\psi)$. 

    Consider the quantity 
    \begin{align*}
        \delta(\lambda) = 2 \mu_R(\lambda) - \mu(\lambda).
    \end{align*}
    This quantity only depends on the Lagrangian subspaces at the vertices, as concatenating with the generator of $\pi_1(U(n)/O(n))$ increases $\mu(\lambda)$ by two and increases $\mu_R(\lambda)$ by one.

    Both $\ind$ and $\ind_R$ are additive under splicing in bigons at $\x$ and $\y$ and pairs of bigons at $\theta$ and $R(\theta)$. Any configuration of Lagrangians at the vertices can be obtained by splicing in bigons, so it suffices to check the formula holds in a single case. Consider the rectangle $\psi: \D \ra \Sym^2(\Sigma)$ shown in \Cref{fig:equivariant_quadrilateral}. This rectangle has $\ind(\psi) = 0$ (and the corresponding moduli space has expected dimension 1). Indeed, for each component and for each slit length, there is a map from $\mathbb{P}_4$ (of varying of conformal class) to the domain. In the real case, the conformal structure on $\Diamond$ is not permitted to vary, and the cuts must be made equivariantly. Therefore, there is a unique slit length for which there is a holomorphic representative for $\psi$. Hence, $\ind_R(\psi) = 0$ (and so the moduli space of real rectangles has expected dimension 0). In this case, Equation \eqref{eqn:index_for_rectangles} holds. 
\end{proof}

\begin{figure}[h]
\def\svgwidth{.8\linewidth}
\begingroup%
  \makeatletter%
  \providecommand\color[2][]{%
    \errmessage{(Inkscape) Color is used for the text in Inkscape, but the package 'color.sty' is not loaded}%
    \renewcommand\color[2][]{}%
  }%
  \providecommand\transparent[1]{%
    \errmessage{(Inkscape) Transparency is used (non-zero) for the text in Inkscape, but the package 'transparent.sty' is not loaded}%
    \renewcommand\transparent[1]{}%
  }%
  \providecommand\rotatebox[2]{#2}%
  \newcommand*\fsize{\dimexpr\f@size pt\relax}%
  \newcommand*\lineheight[1]{\fontsize{\fsize}{#1\fsize}\selectfont}%
  \ifx\svgwidth\undefined%
    \setlength{\unitlength}{779.52755906bp}%
    \ifx\svgscale\undefined%
      \relax%
    \else%
      \setlength{\unitlength}{\unitlength * \real{\svgscale}}%
    \fi%
  \else%
    \setlength{\unitlength}{\svgwidth}%
  \fi%
  \global\let\svgwidth\undefined%
  \global\let\svgscale\undefined%
  \makeatother%
  \begin{picture}(1,0.57090909)%
    \lineheight{1}%
    \setlength\tabcolsep{0pt}%
    \put(0,0){\includegraphics[width=\unitlength,page=1]{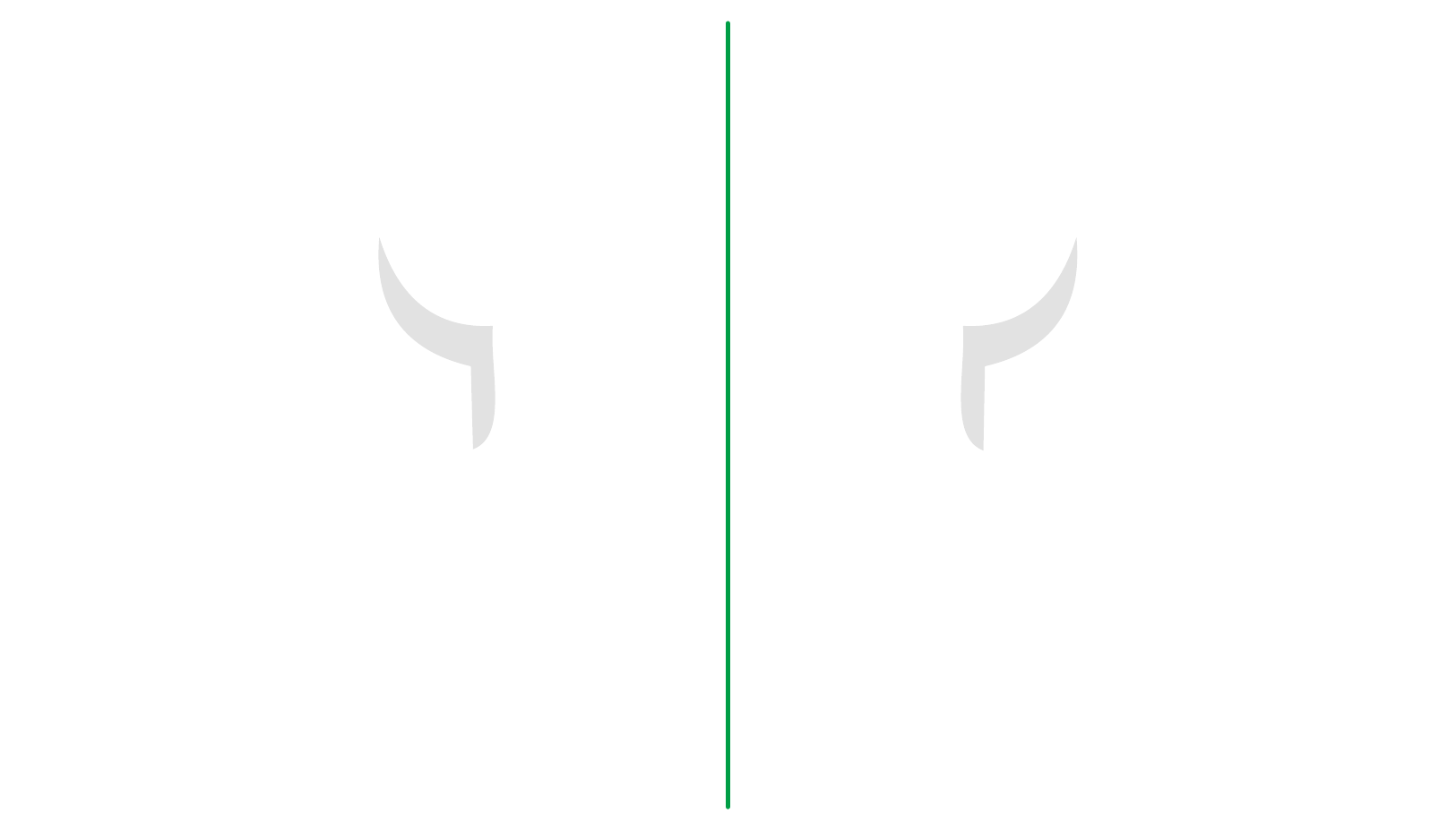}}%
    \put(0.33163918,0.3931234){\color[rgb]{0,0,0}\makebox(0,0)[t]{\smash{\begin{tabular}[t]{c}{$A$}\end{tabular}}}}%
    \put(0,0){\includegraphics[width=\unitlength,page=2]{equivariant_quadrilateral.pdf}}%
    \put(0.33163186,0.13527954){\color[rgb]{0,0,0}\rotatebox{-180}{\makebox(0,0)[t]{\smash{\begin{tabular}[t]{c}{$\reflectbox{$A$}$}\end{tabular}}}}}%
    \put(0,0){\includegraphics[width=\unitlength,page=3]{equivariant_quadrilateral.pdf}}%
    \put(0.66842024,0.39312361){\color[rgb]{0,0,0}\makebox(0,0)[t]{\smash{\begin{tabular}[t]{c}{$B$}\end{tabular}}}}%
    \put(0,0){\includegraphics[width=\unitlength,page=4]{equivariant_quadrilateral.pdf}}%
    \put(0.66841287,0.13527954){\color[rgb]{0,0,0}\rotatebox{-180}{\makebox(0,0)[t]{\smash{\begin{tabular}[t]{c}{$\reflectbox{$B$}$}\end{tabular}}}}}%
    \put(0,0){\includegraphics[width=\unitlength,page=5]{equivariant_quadrilateral.pdf}}%
  \end{picture}%
\endgroup%

    \caption{A class of rectangles admitting a unique (real) holomorphic representative.}
\label{fig:equivariant_quadrilateral}
\end{figure}

\begin{rem}
    We note that a simple modification of this proof shows that for a $2n$-gon $$\psi \in \pi_2^R(\theta_{n-1}, \hdots, \theta_{1}, \x, R(\theta_1), \hdots, R(\theta_{n-1}), \y),$$ we have 
    \begin{align}\label{eqn:index_for_polygons}
        \ind(\tilde{\psi}) = 2\ind_R(\psi) + \frac{\sigma(L_0,\x) - \sigma(L_1,\y)}{2}
    \end{align}
    as well.
\end{rem}

\subsection{Real Heegaard Floer homology} Given a real sutured manifold $(Y, \gamma, \tau)$, the associated real Heegaard Floer complex is defined as a special case of real Lagrangian Floer homology. We choose an embedded real sutured Heegaard diagram $\cH = (\Sigma, \bm \alpha, \bm \beta, \tau)$ representing $(Y, \gamma, \tau)$, and consider $M = \Sym^m(\Sigma)$  equipped with the involution $R: M \ra M$ induced by the restriction of $\tau$ to $\Sigma$; since $\tau$ restricts to an orientation-reversing involution of $\Sigma$, the map $R$ is anti-symplectic. Utilizing the work of Perutz, we may pick an anti-invariant symplectic form $\omega$ on $M$ with respect to which $\Ta = \prod_{i=1}^m \alpha_i$ is a Lagrangian submanifold of $M$. We then define 
\begin{align*}
    \CFR^\circ(\cH):= \CFR^\circ(\Ta, M^R).
\end{align*}
In \cite{guth_manolescu2025real}, we proved that $\CFR^\circ(H)$ is a weak real Heegaard Floer invariant; namely, that it is an invariant up to isomorphism of $(Y, \gamma, \tau)$. This amounts to verifying that real isotopies, handleslides, and stabilizations induce chain homotopy equivalences. In \cite{guth_manolescu2025real}, we defined maps associated to real isotopies and handleslides in terms of continuation maps. However, for the sake of computations, it is more convenient to define these maps in terms of counts of holomorphic rectangles. 

First, we recall a basic fact in the theory.
\begin{lemma}
    Let $(\S, \betas, \betas')$ be an admissible (unreal) Heegaard diagram such that $\betas \sim \betas'$. Then, there is a unique $\SpinC$-structure $\frs_0$ with $c_1(\frs_0) = 0$ as well as a highest-graded element $\Theta_{\beta, \beta'} \in \CF^\circ(\S,\betas,\betas',\frs_0)$, well defined up to boundary.
\end{lemma}

\begin{definition}
    Let $(\Sigma, \alphas', \alphas, \betas, \betas', \tau)$ be an admissible quadruple diagram with $\betas \sim \betas'$. We will write $\Psi^{\alpha \ra \alpha'}_{\beta\ra\beta'}$ for the map
    \begin{align*}
        F_{\alpha', \alpha, \beta, \beta'}(\Theta_{\alpha',\alpha}\otimes \bm {-}\otimes \Theta_{\beta,\beta'}): \CFR^\circ(\S, \alphas, \betas,\tau) \ra \CFR^\circ(\S, \alphas', \betas', \tau).
    \end{align*}
    Note that $\HF^\circ(\Sigma, \alphas', \alphas)$ and $\HF^\circ(\Sigma, \betas, \betas')$ are isomorphic via $\tau$, and in particular, $\tau_*(\Theta_{\beta, \beta'}) = \Theta_{\alpha',\alpha}$.
\end{definition}

Before proceeding, we state the following lemma regarding admissibility. 

\begin{lemma} \label{lem:adm-multi}\cite[Lemma 9.5]{JTZ_naturality_mapping_class_groups} For every $i \in \{\, 1,\dots,k \,\}$, let
  \[
  (\Sigma,\tau(\etas_n),\tau(\etas_{n-1}^i), \hdots, \tau(\etas_0^i), \etas_0^i,\dots,\etas_{n-1}^i,\etas_n, \tau)
  \]
  be a real sutured multi-diagram such that $$(\Sigma,\tau(\etas_{n-1}^i), \hdots, \tau(\etas_0^i), \etas_0^i,\dots,\etas_{n-1}^i,\tau)$$ is admissible.
  Then, there is a translate $\etas_n'$ of $\etas_n$ that is Hamiltonian isotopic to $\etas_n$ such that $$(\Sigma,\tau(\etas_n'),\tau(\etas_{n-1}^i), \hdots, \tau(\etas_0^i), \etas_0^i,\dots,\etas_{n-1}^i,\etas_n', \tau)$$ is
  admissible for every $i \in \{\, 1,\dots, k \,\}$.
\end{lemma}
\begin{proof}
    This follows by combining \cite{Grigsby_Wehrli_colored_jones} and \cite[Lemma 3.34]{guth_manolescu2025real}.
\end{proof}

Suppose that $\cH = (\Sigma, \bm \alpha, \bm \beta, \tau)$ and $\cH' = (\Sigma, \bm \alpha', \bm \beta', \tau)$ are two real sutured Heegaard diagrams for $(Y, \gamma, \tau)$ which differ by a real Hamiltonian isotopy. Then, there is a continuation map
\begin{align*}
    \Gamma_{\bm \beta \ra \bm \beta'}^{\bm \alpha \ra \bm \alpha'}: \CFR^\circ(\cH) \ra \CFR^\circ(\cH'),
\end{align*}
which is a chain homotopy equivalence \cite[Proposition 5.3]{guth_manolescu2025real}. On the other hand, $\CF^\circ(\Ta, \mathbb{T}_{\alpha'})$ has a canonical top-graded class $\Theta_{\alpha',\alpha}$, so there is another map 
\begin{align*}
    \Psi_{\bm \beta \ra \bm \beta'}^{\bm \alpha \ra \bm \alpha'}: \CFR^\circ(\cH) \ra \CFR^\circ(\cH'),
\end{align*}
given by counting rigid real-invariant holomorphic rectangles with one vertex on $\Theta_{\alpha',\alpha}$ and another on $R(\Theta_{\alpha',\alpha})$. The two maps agree up to homotopy. 

\begin{lem}\label{lem:iso-is-rectangle-map}
    Let $H$ and $H'$ be admissible diagrams which differ by a real Hamiltonian isotopy. Then, 
    \begin{align*}
        \Gamma_{\bm \beta \ra \bm \beta'}^{\bm \alpha \ra \bm \alpha'} \sim \Psi_{\bm \beta \ra \bm \beta'}^{\bm \alpha \ra \bm \alpha'}.
    \end{align*}
    In particular, $\Psi_{\bm \beta \ra \bm \beta'}^{\bm \alpha \ra \bm \alpha'}$ is a chain homotopy equivalence and these maps compose under concatenation of isotopies.  
\end{lem}
\begin{proof}
    Since we may instead view $\CFR^\circ(\cH)$ and $\CFR^\circ(\cH')$ as $\CF^\circ(\Ta, M^R)$ and $\CF^\circ(\mathbb{T}_{\alpha'}, M^R)$ respectively, and the maps $\Gamma_{\bm \beta \ra \bm \beta'}^{\bm \alpha \ra \bm \alpha'}$ and $\Psi_{\bm \beta \ra \bm \beta'}^{\bm \alpha \ra \bm \alpha'}$ as maps between them, this follows from the proof in the usual case, which equates the continuation and triangle counting maps; see \cite[Theorem 2.3]{os_holotri} and \cite[Lemma 9.7]{JTZ_naturality_mapping_class_groups}.
\end{proof}

Next, we consider the situation in which $\cH = (\Sigma, \bm \alpha, \bm \beta, \tau)$ and $\cH' = (\Sigma, \bm \alpha', \bm \beta', \tau)$ are two real sutured Heegaard diagrams for $(Y, \gamma, \tau)$ which differ by a real handleslide. Then, by appealing to \cite{perutz}, we once again have a continuation map
\begin{align*}
    \Gamma_{\bm \beta \ra \bm \beta'}^{\bm \alpha \ra \bm \alpha'}: \CFR^\circ(\cH) \ra \CFR^\circ(\cH'),
\end{align*}
which is a chain homotopy equivalence \cite[Proposition 5.4]{guth_manolescu2025real}. Again, $\CF^\circ(\Ta, \mathbb{T}_{\alpha'})$ has a canonical top-graded class $\Theta_{\alpha',\alpha}$, so there is another rectangle counting map 
\begin{align*}
    \Psi_{\bm \beta \ra \bm \beta'}^{\bm \alpha \ra \bm \alpha'}: \CFR^\circ(\cH) \ra \CFR^\circ(\cH').
\end{align*}
These two maps are homotopic.

\begin{lem}\label{lem:handleslide-is-rectangle-map}
    Let $\cH$ and $\cH'$ be admissible diagrams which differ by a real handleslide. Then, 
    \begin{align*}
        \Gamma_{\bm \beta \ra \bm \beta'}^{\bm \alpha \ra \bm \alpha'} \sim \Psi_{\bm \beta \ra \bm \beta'}^{\bm \alpha \ra \bm \alpha'}.
    \end{align*}
    In particular, $\Psi_{\bm \beta \ra \bm \beta'}^{\bm \alpha \ra \bm \alpha'}$ is a chain homotopy equivalence.
\end{lem}
\begin{proof}
    This again follows from the unreal case, since the statement only involves $\Ta$, not $\Tb$ nor $M^R$. Compare with \cite[Remark after Proof of Corollary 1.3]{perutz} or \cite[Proposition 11.4]{Lipshitz_cylindrical}. The continuation map counts strips with dynamic boundary conditions; by degenerating the source, one can see that the continuation map is homotopic to the composition of a monogon map and a triangle counting map. The monogon map picks out the top generator of $\CF^\circ(\Ta, \mathbb{T}_{\alpha'})$, and therefore composing with the triangle counting map is precisely the map associated to the handleslide. 
\end{proof}

Having related the continuation maps to holomorphic rectangle maps in the real setting, we show that if $\cH = (\Sigma, \bm \alpha, \bm \beta, \tau)$ and $\cH' = (\Sigma, \bm \alpha', \bm \beta, \tau)$ are real equivalent, then $\Psi_{\bm \beta \ra \bm \beta'}^{\bm \alpha \ra \bm \alpha'}$ is a chain homotopy equivalence. 

\begin{proposition}\label{prop:Compatibility1}
\hspace{2em}
  \begin{enumerate}
  \item \label{item:psi-iso} Suppose that
    $(\Sigma,\alphas',\alphas,\betas,\betas')$ is an admissible quadruple and we have
    $\betas \sim \betas'$.  Then the map
    \[
    \Psi^{\alphas\ra \alphas'}_{\betas \ra \betas'} \colon
    \CFR^\circ(\Sigma,\alphas,\betas, \t)\ra \CFR^\circ(\Sigma,\alphas',\betas', \t)
    \]
    is a chain homotopy equivalence.
  \item \label{item:psi-compose} These isomorphisms are compatible in
    the sense that if the quadruple diagrams
    $(\Sigma,\alphas',\alphas,\betas,\betas', \tau)$, $(\Sigma,\alphas'',\alphas',\betas',\betas'', \tau)$,
    and $(\Sigma,\alphas'',\alphas,\betas,\betas'', \tau)$ are admissible, then
    \begin{equation*}
      \Psi^{\alphas'\ra \alphas''}_{\betas'\ra\betas''}\circ \Psi^{\alphas\ra\alphas'}_{\betas\ra\betas'} \sim
      \Psi^{\alphas\ra \alphas''}_{\betas\ra\betas''}.
    \end{equation*}
  \end{enumerate}
\end{proposition}
\begin{proof}
    This is an analogue of \cite[Proposition 9.10]{JTZ_naturality_mapping_class_groups}, and the proof is identical, since we may simply replace the role of $\Tb$ with $M^R$.
\end{proof}

Hence, when $(\Sigma,\alphas',\alphas,\betas,\betas', \t)$ is admissible, we have canonical identifications between $\CFR^\circ(\S, \alphas, \betas, \t)$ and $\CFR^\circ(\S, \alphas', \betas', \t)$. In the case that $(\Sigma,\alphas',\alphas,\betas,\betas', \t)$ is not admissible, we make use of the following lemma.

\begin{lem}\label{lem:factor-admissible}
    Suppose $(\S, \alphas, \betas, \t)$ and $(\S, \alphas', \betas', \t)$ are admissible real diagrams and that $\alphas \sim \alphas'$ and therefore $\betas \sim \betas'$. Suppose that $\overline{\alphas}$ and $\overline{\betas}$ are attaching curves isotopic to $\alphas$ and $\betas$ respectively so that the quadruples $(\Sigma,\overline{\alphas},\alphas,\betas,\overline{\betas}, \t)$ and $(\Sigma,\alphas',\overline{\alphas},\overline{\betas},\betas', \t)$ are admissible. Then the map 
    \begin{align*}
        \Psi^{\overline{\alpha} \to \alpha'}_{\overline{\beta} \to \beta'}\circ\Psi^{\alpha \to \overline{\alpha}}_{\beta \to \overline{\beta}}
    \end{align*}
    is a homotopy equivalence. Furthermore, the composite map $\Psi^{\overline{\alpha} \to \alpha'}_{\overline{\beta} \to \beta'}\circ\Psi^{\alpha \to \overline{\alpha}}_{\beta \to \overline{\beta}}$ is independent of $\overline{\alphas}$ and $\overline{\betas}$.
\end{lem}
\begin{proof}
    The proof of \cite[Proposition 9.13]{JTZ_naturality_mapping_class_groups} carries over to our context with no change.
\end{proof}

Therefore, for any quadruple $(\Sigma,\alphas',\alphas,\betas,\betas')$ such that $\alphas \sim \alphas'$ and $\betas \sim \betas'$, there is a canonical homotopy equivalence 
\begin{align*}
    \Phi_{\beta \to \beta'}^{\alpha \to \alpha'}:=\Psi^{\overline{\alpha} \to \alpha'}_{\overline{\beta} \to \beta'}\circ\Psi^{\alpha \to \overline{\alpha}}_{\beta \to \overline{\beta}},
\end{align*}
which is well defined by \Cref{lem:factor-admissible}. In particular, there is a well-defined map $\Psi_{\alpha \ra \alpha}^{\beta \ra \beta}$, which ought to be the identity map. This is the case.

\begin{lemma} \label{lem:identity}
Suppose that the diagram $(\S,\alphas,\betas, \tau)$ is admissible and $\betas \sim \betas'$. Then
  \begin{equation} \label{eqn:id} \Phi^{\alphas \ra \alphas}_{\betas
      \ra \betas} \sim \text{Id}_{\CFR^\circ(\S,\alphas,\betas)}.
  \end{equation}
  In particular, if $\betas \sim \betas'$, then $\Psi^{\alphas' \ra \alphas}_{\betas' \ra \betas} \sim (\Psi^{\alphas \ra \alphas'}_{\betas \ra \betas'})^{-1}$.
\end{lemma}
\begin{proof}
    Let $\alphas'$ be a Hamiltonian translate of $\alphas$ so that $(\Sigma,\alphas',\alphas,\betas,\betas')$ is admissible. Then, by \Cref{prop:Compatibility1}, we have 
    \begin{align*}
        \Phi^{\alphas \ra \alphas}_{\betas
      \ra \betas} \sim \Psi^{\alphas' \ra \alphas}_{\betas' \ra \betas} \circ \Psi^{\alphas \ra \alphas'}_{\betas \ra \betas'}.
    \end{align*}
    But by \Cref{lem:iso-is-rectangle-map}, we have that 
    \begin{align*}
        \Psi^{\alphas' \ra \alphas}_{\betas' \ra \betas} \circ \Psi^{\alphas \ra \alphas'}_{\betas \ra \betas'} \sim \Gamma^{\alphas' \ra \alphas}_{\betas' \ra \betas} \circ \Gamma^{\alphas \ra \alphas'}_{\betas \ra \betas'} \sim \text{Id}_{\CFR^\circ(\S,\alphas,\betas)},
    \end{align*}
    by the naturality of the continuation maps under juxtaposition. 
\end{proof}

Let $H = (\S,A,B, \tau)$ be an isotopy diagram and let $Y_{(\S,A,B, \tau)}$ be
the set of admissible diagrams $(\S,\alphas,\betas, \tau)$ such that
$[\alphas] = A$ and $[\betas] = B$. This is non-empty by
Lemma~\ref{lem:adm-multi}.  It follows from
\Cref{prop:Compatibility1} and \Cref{lem:identity} that the
groups $\CFR^\circ(\S,\alphas,\betas, \tau)$ for $(\S,\alphas,\betas,\tau) \in
Y_{(\S,A,B,\tau)}$, together with the isomorphisms $\Phi^{\alphas \ra
\alphas'}_{\betas \ra \betas'}$ form a transitive system of (curved) chain complexes, which will be denoted $\CFR^\circ(H)$.

\subsection{Real Heegaard Floer homology as a weak real Heegaard invariant} We now show that the maps induced by real equivalences, $\cO$-stabilizations, and diffeomorphisms between isotopy diagrams commute with the maps appearing in this transitive system. We address real equivalences first, since they are the most straightforward.

\begin{lemma}\label{lem:real equivs}
    Let $(\S,\alphas_1,\betas_1, \tau)$, $(\S,\alphas_1',\betas_1', \tau)$, $(\S,\alphas_2,\betas_2, \tau)$, and $(\S,\alphas_2',\betas_2', \tau)$ be admissible diagrams such that $\alphas_1 \sim \alphas_2$ and $\alphas_1' \sim \alphas_2'$. Then, the diagram
    \begin{align*}
        \begin{tikzcd}[ampersand replacement = \&]
            \CFR^\circ(\S,\alphas_1,\betas_1, \tau) \ar[r,"\Psi^{\alphas_1 \ra \alphas_1'}_{\betas_1 \ra \betas_1'}"]  \ar[d,"\Psi^{\alphas_1 \ra \alphas_2}_{\betas_1 \ra \betas_2}"] \& \CFR^\circ(\S,\alphas_1',\betas_1', \tau) \ar[d,"\Psi^{\alphas_1' \ra \alphas_2'}_{\betas_1' \ra \betas_2'}"] \\
            \CFR^\circ(\S,\alphas_2,\betas_2, \tau) \ar[r,"\Psi^{\alphas_2 \ra \alphas_2'}_{\betas_2 \ra \betas_2'}"] \& \CFR^\circ(\S,\alphas_2',\betas_2', \tau) 
        \end{tikzcd}
    \end{align*}
     commutes up to homotopy.
\end{lemma}
\begin{proof}
    This follows immediately from \Cref{prop:Compatibility1}.
\end{proof}

\begin{definition} \label{def:diffeo}
  Let $(\S,\alphas,\betas,\tau)$ be an admissible diagram and $d \colon \S
  \to \S'$ a diffeomorphism.  We write $\alphas' = d(\alphas)$,
  $\betas' = d(\betas)$, and $\tau' = d\circ \tau\circ d^{-1}$.  Then $d$ induces an isomorphism
  \[
  d_* \colon \CFR^\circ(\S, \alphas, \betas,\tau) \to \CFR^\circ(\S', \alphas',
  \betas',\tau'),
  \]
  as follows. We fix a generic path of symmetric almost complex
  structures $J_s$ on $M = \Sym^{|\alphas|}(\Sigma)$. Pushing $J_s$
  forward along $d$ produces an almost complex structure $J_s'$ on
  $M' = \Sym^{|\alphas|}(\Sigma')$. Of course, $d$ induces an isomorphism
  \[
  d_{J_s,J_s'} \colon \CFR^\circ_{J_s}(\S,\alphas,\betas,\tau) \to
  \CFR^\circ_{J_s'}(\S',\alphas', \betas', \tau').
  \]
  Furthermore, since the map $d_{J_s,J_s'}$ commutes with the continuation maps $\Psi_{J_s \to J_s'}$, these diffeomorphism maps descend to a map $d_*$ of transitive systems.
\end{definition}

\begin{lemma} \label{lem:diffeo-triangle} 
  Suppose that $(\S,\alphas',\alphas,\betas,\betas')$ is an admissible quadruple. Then, the diagram
    \begin{align*}
        \begin{tikzcd}[ampersand replacement = \&, column sep = large]
            \CFR^\circ(\S,\alphas,\betas, \tau) \ar[r,"\Psi^{\alphas \ra \alphas'}_{\betas \ra \betas'}"]  \ar[d,"d_*"] \& \CFR^\circ(\S,\alphas',\betas', \tau) \ar[d,"d_*"] \\
            \CFR^\circ(d(\S),d(\alphas),d(\betas), d_*(\tau)) \ar[r,"\Psi^{d(\alphas) \ra d(\alphas')}_{d(\betas) \ra d(\betas')}"] \& \CFR^\circ(d(\S),d(\alphas'),d(\betas'), d_*(\tau))
        \end{tikzcd}
    \end{align*}
    is commutative up to homotopy. Here, $d_*(\tau) = d\circ \tau\circ d^{-1}$.
\end{lemma}
\begin{proof}
    This follows exactly as in \cite[Lemma 9.24]{JTZ_naturality_mapping_class_groups}.
\end{proof}

Importantly, for diffeomorphisms isotopic to the identity, the induced maps are homotopic to maps induced by real equivalences. 

\begin{proposition} \label{prop:continuity} Let $(\S,\alphas,\betas,\tau)$
  be an admissible diagram. Suppose that $d \colon \S \to \S$ is an equivariant
  diffeomorphism equivariantly isotopic to $\text{Id}_\S$, and let $\alphas' =
  d(\alphas)$, $\betas' = d(\betas)$, and $\tau' = d\circ \tau \circ d^{-1}$. Then
  \begin{equation}\label{eqn:continuity}
  d_* = \Phi^{\alphas \to \alphas'}_{\betas \to \betas'} \colon
  \CFR^\circ(\S,\alphas,\betas, \tau) \to \CFR^\circ(\S,\alphas',\betas',\tau').
  \end{equation}
\end{proposition}
\begin{proof}
    This follows just as \cite[Proposition 9.27]{JTZ_naturality_mapping_class_groups}. The point is to decompose $d$ into a composition of diffeomorphisms $H_0 \xra{d_1} \hdots \xra{d_n} H_0$ where $H_{i+1}$ differs from $H_{i}$ by an equivariant Hamiltonian isotopy so that $|\alpha \cap d_i(\alpha)| = 2$ for each $i$. They argue that both $d_*$ and $\Phi^{\alphas \to \alphas'}_{\betas \to \betas'}$ are homotopic to continuation maps. They then construct an explicit homotopy between the continuation maps by constructing a 1-dimensional moduli space whose ends give the desired relation.
\end{proof}

Finally, we turn to stabilizations. Suppose that $\cH' = (\S', \alphas', \betas', \tau')$ is obtained from $\cH = (\S, \alphas, \betas, \tau)$ by a $\{1\}$-stabilization. Then, for a suitable almost complex structure, the map
\begin{align*}
    \sigma_{\cH \ra \cH'}^{\{1\}}: \CFR^\circ(\cH) \ra \CFR^\circ(\cH')
\end{align*}
is given by $\x \mapsto \x \times \{c\}$, where $c$ is the unique new intersection point introduced in the $\{1\}$-stabilization. Similarly, if $\cH'$ is obtained from $\cH$ by a $\Z/2$-stabilization, then, for a suitable almost complex structure, the map
\begin{align*}
    \sigma_{\cH \ra \cH'}^{\Z/2}: \CFR^\circ(\cH) \ra \CFR^\circ(\cH')
\end{align*}
is given by $\x \mapsto \{a\} \times \x \times \{\tau(a)\}$, where $a$ and $\tau(a)$ are the two new intersection points introduced in the $\Z/2$-stabilization. To prove that these induce maps of transitive systems, we appeal to the following result. 

\begin{figure}[h]
\def\svgwidth{.8\linewidth}
\begingroup%
  \makeatletter%
  \providecommand\color[2][]{%
    \errmessage{(Inkscape) Color is used for the text in Inkscape, but the package 'color.sty' is not loaded}%
    \renewcommand\color[2][]{}%
  }%
  \providecommand\transparent[1]{%
    \errmessage{(Inkscape) Transparency is used (non-zero) for the text in Inkscape, but the package 'transparent.sty' is not loaded}%
    \renewcommand\transparent[1]{}%
  }%
  \providecommand\rotatebox[2]{#2}%
  \newcommand*\fsize{\dimexpr\f@size pt\relax}%
  \newcommand*\lineheight[1]{\fontsize{\fsize}{#1\fsize}\selectfont}%
  \ifx\svgwidth\undefined%
    \setlength{\unitlength}{841.88976378bp}%
    \ifx\svgscale\undefined%
      \relax%
    \else%
      \setlength{\unitlength}{\unitlength * \real{\svgscale}}%
    \fi%
  \else%
    \setlength{\unitlength}{\svgwidth}%
  \fi%
  \global\let\svgwidth\undefined%
  \global\let\svgscale\undefined%
  \makeatother%
  \begin{picture}(1,0.34343434)%
    \lineheight{1}%
    \setlength\tabcolsep{0pt}%
    \put(0,0){\includegraphics[width=\unitlength,page=1]{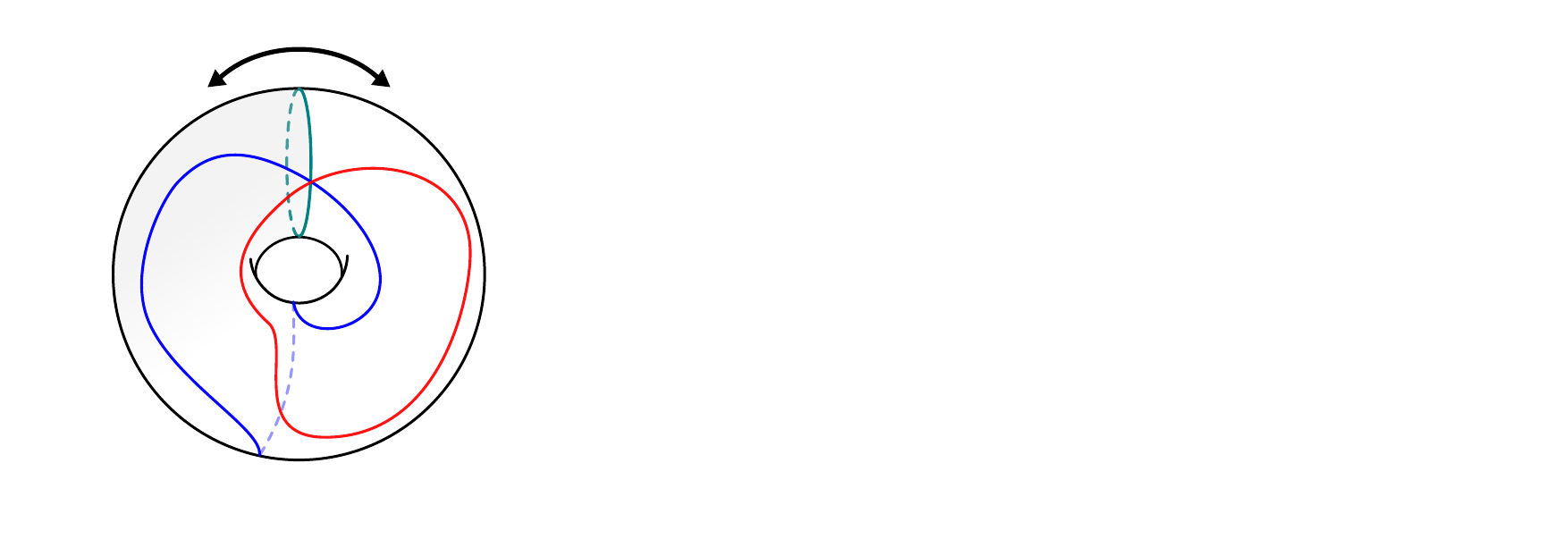}}%
    \put(0.18943641,0.32153186){\color[rgb]{0,0,0}\makebox(0,0)[lt]{\smash{\begin{tabular}[t]{l}{$\tau$}\end{tabular}}}}%
    \put(0.1906339,0.01151504){\color[rgb]{0,0,0}\makebox(0,0)[t]{\smash{\begin{tabular}[t]{c}{$E_{\{1\}}$}\end{tabular}}}}%
    \put(0,0){\includegraphics[width=\unitlength,page=2]{s3_diagrams.pdf}}%
    \put(0.67133067,0.32153186){\color[rgb]{0,0,0}\makebox(0,0)[lt]{\smash{\begin{tabular}[t]{l}{$\tau$}\end{tabular}}}}%
    \put(0.6753501,0.01151504){\color[rgb]{0,0,0}\makebox(0,0)[t]{\smash{\begin{tabular}[t]{c}{$E_{\Z/2}$}\end{tabular}}}}%
    \put(0,0){\includegraphics[width=\unitlength,page=3]{s3_diagrams.pdf}}%
  \end{picture}%
\endgroup%

    \caption{The diagrams $E_{\Z/2}$ and $E_{\{1\}}$ appearing in \Cref{prop:gluing-rectangles-stab}.}
\label{fig:s3_diagrams}
\end{figure}

\begin{prop}\label{prop:gluing-rectangles-stab}
    Let $\cH = (\S, \alphas, \betas, \tau, z)$ be a real Heegaard diagram and $E^{\{1\}}$ be a genus 1, real Heegaard diagram $(\Sigma_0, \alpha_0, \beta_0, \tau_0, z_0)$ for $(S^3, \mathrm{rot})$ shown in the first frame of \Cref{fig:s3_diagrams}. Let $\cH \# E_{\{1\}}$ be the connected sum of $\cH$ and $E_{\{1\}}$ at the points $z$ and $z_0$, extending the two involutions in the obvious way. Let $\psi\in \pi_2^R(\tau(\bm w), \x, \y, \bm w)$ be a class of real-invariant rectangles in $\cH$ and let $\psi_0\in \pi_2^R(\tau(\bm w_0), \x_0, \y_0, \bm w_0)$ be a class of real invariant rectangles in $E_{\{1\}}$ with $n_{z_0}(\psi_0) = 0$. Then, for an appropriate almost complex structure, there is a diffeomorphism
    \begin{align*}
        \cM_R(\psi\#\psi_0) \cong \cM_R(\psi) \times \cM_R(\psi_0).
    \end{align*}
    Further, define $E_{\Z/2} = (\S_0 \sqcup \S_1,\alpha_0\cup \alpha_1, \beta_0\cup \beta_1, z_0 \cup z_1, \tau_{\Z/2})$ where $\Sigma_0$ and $\Sigma_1$ are two copies of a genus 1 surface and $\tau_{\Z/2}$ is the involution exchanging them, swapping the alphas and betas. See the second frame of \Cref{fig:s3_diagrams}. Also, let $z_L$ and $z_R$ be two distinct points of $\Sigma$ near $z$ that are exchanged by $\tau$. Then, let $\cH\#_{\Z/2} E_{\Z/2}$ be the diagram obtained by taking an equivariant connected sum, identifying $z_L$ with $z_0$ and $z_R$ with $z_1$, extending the involution in the obvious way. Let $\psi\in \pi_2^R(\tau(\bm w), \x, \y, \bm w)$ be a class of real-invariant rectangles in $\cH$ and let $\psi_0\in \pi_2^R(\tau(\bm w_0), \x_0, \y_0, \bm w_0)$ be a class of real invariant rectangles in $E_{\Z/2}$ with $n_{z_0}(\psi_0) = 0 = n_{z_1}(\psi_0)$. Then, for an appropriate almost complex structure, there is a diffeomorphism
    \begin{align*}
        \cM_R(\psi\#\psi_0) \cong \cM_R(\psi) \times \cM_R(\psi_0).
    \end{align*}
\end{prop}
\begin{proof}
    This follows just as in \cite{os_properties_apps} and \cite[Theorem 9.4]{os_holotri}. In the first case, we choose holomorphic representatives $u$ and $u_0$, which produce a map to $\Sym^{k+1}(\Sigma\vee \Sigma_0)$. We splice in spheres to construct a map $u\#u_0$ to $\Sym^{k+1}(\Sigma \# \Sigma_0)$ which is approximately holomorphic. Then, we insert a long connected-sum tube; as the length goes to infinity, we can apply the inverse function theorem to find a nearby pseudoholomorphic rectangle. The second case is nearly identical, though we insert two long connected-sum tubes.
\end{proof}

As a consequence, the maps for stabilizations descend to the transitive systems.

\begin{lemma}
  \label{lem:StabilizePhi}
  Let $\cH_1=(\Sigma,\alphas_1,\betas_1, \tau_1)$ and
  $\cH_2=(\Sigma,\alphas_2,\betas_2, \tau_2)$ be two admissible Heegaard
  diagrams such that $\alphas_1 \sim \alphas_2$.  If $\cH_1' = (\S',\alphas_1',\betas_1')$ and $\cH_2' =
  (\S',\alphas_2',\betas_2')$ are $\cO$-stabilizations of $\cH_1$ and
  $\cH_2$ respectively, then the diagram
    \begin{align*}
        \begin{tikzcd}[ampersand replacement = \&]
            \CFR^\circ(\S,\alphas_1,\betas_1, \tau) \ar[r,"\Phi_{\betas_1 \ra \betas_2}^{\alphas_1 \ra \alphas_2}"]  \ar[d,"\sigma_{\cH_1\to \cH_1'}^\cO"] \& \CFR^\circ(\S,\alphas_2,\betas_2, \tau) \ar[d,"\sigma_{\cH_2\to \cH_2'}^\cO"] \\
            \CFR^\circ(\S,\alphas_1',\betas_1', \tau) \ar[r,"\Phi_{\beta_1'
    \to \betas_2'}^{\alphas_1' \ra \alphas_2'}"] \& \CFR^\circ(\S,\alphas_2',\betas_2', \tau) 
        \end{tikzcd}
    \end{align*}
    commutes up to homotopy. In particular, there are well-defined maps $\sigma^\cO_{H \to H'}: \CFR^\circ(H)\to \CFR^\circ(H')$ associated to stabilizations.
\end{lemma}
\begin{proof}
    This follows from \Cref{prop:gluing-rectangles-stab}.
\end{proof}

It follows from the previous lemmas that $\CFR^\circ$ is a weak real Heegaard invariant. We now verify the remaining axioms to upgrade it to a strong invariant.

\subsection{Real Heegaard Floer homology as a strong real Heegaard invariant}\label{sub:HFR-strong} We now verify the axioms from \Cref{def:strong-Heegaard}.
\begin{enumerate}
    \item \textbf{Functoriality:} The real equivalence maps are functorial according to \Cref{prop:Compatibility1} and \Cref{lem:identity}. Functoriality of diffeomorphisms is a consequence of the definition. Similarly, the destabilization map $\sigma^{\cO}_{H' \ra H}$ is defined to be the homotopy inverse of $\sigma^{\cO}_{H \ra H'}$.
    \item \textbf{Commutativity:} For the four kinds of distinguished rectangles, we must have homotopy commutativity. Squares of type \eqref{item:rect-alpha-stab} commute according to \Cref{lem:StabilizePhi}. Squares of type \eqref{item:rect-alpha-diff} commute by \Cref{lem:diffeo-triangle}. Squares of type \eqref{item:rect-stab-stab} commute on the nose for the appropriate choice of almost complex structure. For instance, suppose that $H_1$ and $H_2$ (as well as $H_3$ and $H_4$) differ by a $\{1\}$-stabilization and that $H_1$ and $H_3$ (as well as $H_2$ and $H_4$) differ by a $\Z/2$-stabilization, then we have that 
    \begin{align*}
        \sigma_{H_3 \ra H_4}^{\{1\}}\circ \sigma_{H_1 \ra H_3}^{\Z/2}(\x)
        & = \sigma_{H_3 \ra H_4}^{\{1\}}(\{a\} \times \x \times \{\tau_4(a)\}) \\
        &= \{a\} \times \x \times \{\tau_4(a)\}\times \{c\}\\
        &= \{a\} \times \x \times \{c\} \times \{\tau_4(a)\}\\
        &= \sigma_{H_2 \ra H_4}^{\Z/2}(\x \times \{c\}) \\
        &= \sigma_{H_2 \ra H_4}^{\Z/2}\circ \sigma_{H_1 \ra H_2}^{\{1\}} (\x).
    \end{align*}
    The other two cases are similar. 

    Finally, for squares of type \eqref{item:rect-stab-diff}, we have that $H_2$ is an $\cO$-stabilization of $H_1$ and $H_4$ is an $\cO$-stabilization of $H_3$ and $f(H_1) = H_3$ and $g(H_2) = H_4$. Further, we assume that $f$ and $g$ satisfy $f(\cO\times D) = \cO\times D'$,  $g(\cO\times T) = \cO\times T'$, and $f|_{\S_1 \setminus (\cO\times D)} = g|_{\S_2 \setminus (\cO\times T)}$. In the case of $\{1\}$-stabilizations, we have that $\alphas_2 \cap \betas_2 \cap T = \{c\}$ and $\alphas_4 \cap \betas_4 \cap T' = \{c'\}$; note also that $g(c) = c'$. Hence, for the appropriate choice of almost complex structure, we have 
    \begin{align*}
        g_* \circ \sigma^{\{1\}}_{H_1 \ra H_2}(\x) &= g(\x\times \{c\})\\
        &= g(\x) \times \{g(c)\} \\
        &= f(\x) \times \{c'\} \\
        &= \sigma^{\{1\}}_{H_3 \ra H_4}\circ f_*(\x).
    \end{align*}
    The case of $\Z/2$-stabilizations is identical.

    \item \textbf{Continuity:} Continuity is precisely \Cref{prop:continuity}.
\end{enumerate}

The remaining two axioms (real handleswap and trade invariance) are more subtle, and require explicit counts of holomorphic rectangles. In order to make these computations more tractable, one typically makes use of Lipshitz's cylindrical reformulation of Heegaard Floer homology \cite{Lipshitz_cylindrical}. We will describe the analogue of his approach in the real setting.

\subsection{Cylindrical real Floer homology}

Our treatment here follows \cite{lipshitz_ozsvath:real_bordered}, so we will be terse, only emphasizing the differences between the bordered setup and ours.

Let $W = \Sigma \times [0,1]\times \R$ and define $\pi_{\Sigma}: W \ra \Sigma$ and $\pi_\D: W \ra [0,1]\times \R$ to be the natural projections. An involution $\tau$ of $\Sigma$ induces an involution $\tu$ of $W$ defined by $\tu(x,s,t) = (\t(x), 1-s, t)$. This involution clearly exchanges the cylinders $\bm{\alpha}\times \{1\}\times \R$ and $\bm{\beta}\times \{0\}\times \R$. Fix a point $z_i$ in each region $R_i$ of $\Sigma \smallsetminus (\bm{\alpha}\cup\bm{\beta})$ which satisfies 
\begin{align*}
    \begin{cases}
        \tau(z_i) = z_i & \text{if } \tau(R_i) = R_i, \\
        \tau(z_i) = z_j & \text{if } \tau(R_i) = R_j.
    \end{cases}
\end{align*}
We will write $P$ for this collection of points and $D_i$ for a small disk in $R_i$ centered at $z_i$.

Fix a $\tau$-anti-invariant area form $dA$ on $\Sigma$ and a complex structure $j_\Sigma$ which is tamed by $dA$ and anti-holomorphic, i.e.
\begin{align*}
    j_\Sigma \circ \t_* = -\t_* \circ j_\Sigma. 
\end{align*}
Let $\omega = ds \wedge dt + dA$ be a split symplectic form on $W$. Note that $\tu$ is an anti-symplectic involution: $\tu^* \omega = -\omega.$ In the cylindrical setting, one considers the set $\cJ$ of almost complex structures on $W$ which satisfy
\begin{enumerate}
    \item[(J1)] $J$ is tamed by $\omega$;
    \item[(J2)] In a cylindrical neighborhood of $P\times [0,1]\times \R$, $J = j_\Sigma \times j_\D$ is split.
    \item[(J3)] $J$ is translation invariant in the $\R$-factor.
    \item [(J4)] $J(\partial_t) = \partial_s$
    \item[(J5)] $J$ preserves $T_p \Sigma \times (s, t)$ for all $(s, t) \in [0,1]\times \R$.   
\end{enumerate}
As in \cite{Lipshitz_cylindrical,HHSZ_surgery_exact_invol}, we will sometimes work with a more general class of almost complex structures, which satisfy
\begin{enumerate}
    \item[(J5')] $T_p\Sigma \oplus \{0\}$ is a complex line near $(\bm \alpha \cup \bm \beta)\times [0,1] \times \R$ and $\Sigma \times \{0,\frac{1}{2},1\} \times \R$. 
\end{enumerate}

Let $\cJ_R$ be the subset of $\cJ$ consisting of almost complex structures $J$ satisfying 
\begin{align*}
    J \circ \tu_* = - \tu_* \circ J.
\end{align*}
We refer to these as \emph{symmetric almost complex structures}.

Let $(S, j)$ be a Riemann surface with boundary which contains $m = |\alphas| = |\betas|$ negative punctures $\bm{p} = \{p_1,\hdots,p_m\}$ and $m$ positive punctures $\bm{q} = \{q_1,\hdots,q_m\}$. 
For a symmetric almost complex structure $J$ satisfying (J1)-(J5) (or (J5')), we are interested in $J$-holomorphic maps $u: S \ra W$ which satisfy:
\begin{enumerate}
    \item[(M0)] The source $S$ is smooth.
    \item[(M1)] $u(\partial(S))\subset (\bm{\alpha}\times \{1\}\times \R) \cup (\bm{\beta}\times \{0\}\times \R)$.
    \item[(M2)] There are no components of $S$ on which $\pi_\D\circ u$ is constant.
    \item[(M3)] For each $i$, $u^{-1}(\alpha_i \times \{1\}\times \R)$ and  $u^{-1}(\beta_i \times \{0\}\times \R)$ consist of exactly one component of $\partial S \smallsetminus (\bm{p}\cup \bm{q})$.
    \item[(M4)] $\lim_{w \ra p_i}\pi_\R \circ u(w) = -\infty$ and $\lim_{w \ra q_i}\pi_\R \circ u(w) = +\infty$.
    \item[(M5)] The energy of $u$ is finite.
    \item[(M6)] $u$ is an embedding.
\end{enumerate}

If $u$ is in the class $\phi \in \pi_2(\x, \y)$, we will write $\cM^{cyl}(\phi)$ for the moduli space of $J$-holomorphic curves which satisfy (M0)-(M6). Let $\cM_R^{cyl}(\phi)$ be the subspace of \emph{real holomorphic curves} $u: S \ra W$ for which there is an anti-holomorphic involution $\sigma$ of $S$ which satisfies $u\circ \sigma = \tu \circ u$. We will also at times consider the moduli space $\cM^{cyl}(\phi; S)$ of curves with a fixed source $S$ and the associated subspace $\cM_R^{cyl}(\phi; S)$. 

We say that $J \in \cJ^R$ is \emph{generic} if $\cM_R^{cyl}(\phi; S)$ is transversely cut out. Let $\cB^R$ be the space of triples $(j, J, u)$ where $j$ is a symmetric complex structure on $S$, and $J$ is a symmetric almost complex structure on $W$ and $u$ is a real smooth map from $S$ to $W$. Let $\cE^R \ra \cB^R$ be the bundle of symmetric $(0, 1)$-forms on $S$ valued in $u^*(TW)$, i.e. those $\alpha$ satisfying 
\begin{align*}
    \alpha \circ d \sigma = d \tu \circ \alpha.    
\end{align*}
We write $\cB^R_{j, J}$ for the subspace of $\cB^R$ with fixed $j$ and $J$ and similarly for $\cE^R_{j, J, u}$. The linearization of the $\overline{\partial}$-operator gives a map $T\cB^R \ra \cE^R$. It follows from \cite[Lemma 3.3]{lipshitz_ozsvath:real_bordered} that for any real $u$ and symmetric $j$ and $J$,  the map $D_u\overline{\partial}: T_u \cB^R_{j,J} \ra \cE^R_{j,J,u}$ is well-defined and Fredholm. Let $\cM^R$ be the subspace of $\cB^R$ consisting of those triples $(j, J, u)$ such that $\overline{\partial}_{j,J}(u) = 0$.

An \emph{annoying curve} is a map $u: S \ra W$ such that there is a nonempty open subset of $S$ on which $\pi_\D \circ u$ is constant. 

\begin{prop}
    With respect to a symmetric almost complex structure satisfying (J1)-(J5), the space $\cM^R$ is a smooth manifold away from annoying curves. 
\end{prop}
\begin{proof}
    This is standard; see \cite[Proposition 3.4.1]{McDuffSalamon},\cite[Lemma 3.3]{Lipshitz_cylindrical}, or \cite[Proposition 2.2]{guth_manolescu2025real}. The key point is that according to \cite[Proposition 3.3]{Lipshitz_cylindrical} $u$ is somewhere injective, provided that $u$ is not annoying. Then, we can run the usual transversality argument, obtaining symmetric objects by averaging when needed. See \cite[Proposition 2.2]{guth_manolescu2025real} or \cite[Proposition 3.4]{lipshitz_ozsvath:real_bordered} for details.
\end{proof}

In particular, the moduli spaces $\cM_R^{cyl}(\phi) \subset \cM^R$ are smooth manifolds of dimension $\ind_R(\phi)$ for generic symmetric almost complex structures. In particular, we can define a cylindrical version of real Heegaard Floer homology.

Just as in the unreal setting, there is a tautological correspondence between $j\times j_\D$-holomorphic curves in $W$ and $\Sym^m(j)$-pseudo-holomorphic strips in $M = \Sym^m(\Sigma)$. 

\begin{lem}\label{lem:taut-corr}
    Let $\t$ be an involution on $\Sigma$, let $j$ be a symmetric complex structure on $\Sigma$, and let $J:=\Sym(j)$ be the induced almost complex structure on $\Sym^m(\Sigma)$. There is a one-to-one correspondence between $J$-holomorphic maps
    \begin{align*}
        u: \D \ra \Sym^m(\Sigma)
    \end{align*}
    satisfying $\tu\circ u =u\circ \rho$ and diagrams of the form 
    \begin{align*}
        \begin{tikzcd}[ampersand replacement = \&]
            \& S\ar[rr,"\hat{u}"]\ar[dd] \& \& \Sigma \\
            S\ar[ur,"\sigma"]\ar[dd]\ar[rr,crossing over,"\hat{u}" pos=0.65] \& \& \Sigma \ar[ur,"\t"]\& \\
            \& \D \& \&  \\
            \D \ar[ur,"\rho"]\& \&  \& \\
        \end{tikzcd}
    \end{align*}
    where $S$ is a Riemann surface with real structure, $\hat{u}$ is a $j$-holomorphic map, and $S \ra \D$ is a holomorphic $d$-fold branched cover. The latter data is equivalent to a $j \times j_\D$-holomorphic map $S \ra \Sigma \times [0,1]\times \R$.
\end{lem}
\begin{proof}
    This follows just as in \cite[Lemma 3.6]{os_holodisks} while keeping track of the real structures. 
\end{proof}

Following \cite{Lipshitz_cylindrical}, we will show that the moduli spaces in the cylindrical setting agree with those in the Lagrangian Floer setting. By \Cref{lem:taut-corr}, we can identify the two moduli spaces when $j \times j_\D$ achieves transversality, but this cannot always be achieved. To circumvent this issue, Lipshitz considers a broader class of almost complex structures than those considered by Ozsv\'ath and Szab\'o. Fix an open neighborhood $V_1$ of $z_i \times \Sym^{m-1}(\Sigma)$ and an open neighborhood $V_2$ of the diagonal, $\Delta$. Let $\pi: \Sigma^m \ra \Sym^m(\Sigma)$ be the natural projection and let $\omega_0 = (dA)^m$. Fix a complex structure $j_{\Sigma}$ on $\Sigma$.

\begin{definition}{\cite[Definition 13.1]{Lipshitz_cylindrical}}
    We say that an almost complex structure $\Tilde{J}$ on $\Sym^m(\Sigma)$ is a \emph{quasi-nearly-symmetric almost complex structure} if:
    \begin{enumerate}
        \item $\Tilde{J}$ is tamed by $\pi_*(\omega_0)$ on $\Sym^m(\Sigma)\smallsetminus V_2$.
        \item $\Tilde{J}$ agrees with $\Sym^m(j_\Sigma)$ on $V_1$.
        \item There is some complex structure $j$ on $\Sigma$ such that $\Sym^m(j)$ agrees with $\Tilde{J}$ on $V_2$.
    \end{enumerate}
\end{definition}
We will refer to such almost complex structures as \emph{QNS almost complex structures}, so that when we consider \emph{symmetric QNS almost complex structures} we can avoid referring to ``symmetric quasi-nearly-symmetric almost complex structures''. If $J$ is a symmetric almost complex structure on $W$ satisfying (J1)-(J5), then $J$ determines a path $J_s$ of symmetric complex structures on $\Sigma$; furthermore $\Sym^m(J_s)$ is a path of symmetric QNS almost complex structures on $\Sym^m(\Sigma)$. Conversely, if $\Tilde{J}_t$ is a path of symmetric QNS almost complex structures on $\Sym^m(\Sigma)$, we write $j_t$ for the complex structure on $\Sigma$ which agrees with $\Tilde{J}_t$ near the diagonal. 

\begin{lem}
    The class of paths of symmetric almost complex structures of the form $\Sym^m(J_s)$ is sufficient to achieve transversality for real-invariant strips $u: (\D, \partial \D) \ra (\Sym^m(\Sigma), \bT_\alpha \cup M^R)$. 
\end{lem}
\begin{proof}
    This again follows along the same lines as \cite[Proposition 2.2]{guth_manolescu2025real} and \cite[Proposition 3.4]{lipshitz_ozsvath:real_bordered}; we appeal to the fact that transversality can be achieved in the unreal case and then average to obtain symmetric objects when necessary.
\end{proof}

Therefore, we have a tautological correspondence in the real setting.
\begin{prop}\label{prop:real-taut-corr}
    Equipped with symmetric almost complex structures $J_s$ and $\Sym^m(J_s)$, the moduli spaces $\cM_R^{cyl}(\phi)$ and $\cM_R(\phi)$ agree.
\end{prop}
\begin{proof}
    This follows just as in the unreal case, making use of \cite[Proposition 13.2]{Lipshitz_cylindrical}.
\end{proof}

\begin{rem}
    Having established the equivalence of the two theories, we will henceforth write $\cM_R(\phi)$ rather than $\cM_R^{cyl}(\phi)$.
\end{rem}

\subsection{Real cylindrical boundary degenerations} It will ultimately be more useful to work with the larger class of almost complex structures satisfying (J1)-(J4) and (J5'). This larger class has the advantage that the $\overline\partial$-operator achieves transversality for all curves appearing in codimension-1 degenerations. For this reason it will be useful to have an explicit count of boundary degenerations in the real cylindrical setting.

\begin{defn}\label{def:cyl-boundary-degenerations} (cf.\cite[Definition 7.6]{HHSZ_surgery_exact_invol})
    A \emph{cylindrical beta boundary degeneration} is a smooth map
    \[
    u\colon (S,\partial S)\to (\Sigma\times [0,\infty)\times \R, \betas\times \{0\}\times \R),
    \]
    satisfying the following:
    \begin{enumerate}
    \item[(N-1)] \label{N-1} $u$ is $(j,J)$-holomorphic.
    \item[(N-2)] \label{N-2} $u$ is proper.
    \item[(N-3)] \label{N-3} For each $t\in \R$ and $i\in \{1,\dots, n\}$, $u^{-1}(\beta_i\times \{t\})$ consists of a single point.
    \item[(N-4)] \label{N-4} $u$ has finite energy.
    \item[(N-5)] \label{N-5} $\pi_{\bH}\circ u$ is non-constant on each component of $S$, where $\pi_{\bH}$ denotes the projection onto $[0,\infty)\times \R$.
\end{enumerate}

\begin{definition}\label{def:real-boundary-degenerations}
    A \emph{$\Z/2$-cylindrical boundary degeneration} is a smooth map 
    \[
    u\sqcup (\underline{\t}\circ u) \colon (S\sqcup -S,\partial (S\sqcup -S))\to (\Sigma\times ([0,\infty)\sqcup (-\infty,1])\times \R, (\alphas\cup\betas)\times \{0,1\}\times \R),
    \]
    such that $u$ is a cylindrical beta boundary degeneration.
    
    A \emph{$\{1\}$-cylindrical boundary degeneration} is a smooth map 
    \[
    u \colon (S,\partial S) \to (\Sigma\times (\R\times \R, C\times \{1/2\}\times \R),
    \]
    such that there exists an anti-holomorphic involution $\sigma$ of $S$ satisfying  $u \circ \sigma = \underline{\t} \circ u$.
\end{definition}
    
    If $B$ is a relative 2-cycle in $(\Sigma, \betas)$, we will write $\cN^{\Z/2}_J(\t(B) \sqcup B, \x)$ for the moduli space of $\Z/2$-cylindrical boundary degenerations, and $\widetilde\cN^{\Z/2}_J(\t(B) \sqcup B, \x)$ for the quotient $\cN^{\Z/2}_J(\t(B) \sqcup B, \x)/\Aut(\mathbb{H})$, where $\Aut(\mathbb{H})$ is the group of conformal automorphisms of the half-plane which acts pointwise on the $[0,\infty)\times \R$ factor. In the $\{1\}$-orbit case, we write $\cN^{\{1\}}_J(B, \x)$ for the moduli space of $\{1\}$-cylindrical boundary degenerations, and $\widetilde\cN^{\{1\}}_J(B, \x)$ for the quotient $\cN^{\{1\}}_J(B, \x)/\Aut(\mathbb{H})$. 
\end{defn}

\begin{defn}\label{def:admissible-acs} We say an almost complex structure $J$ on $\Sigma\times [0,\infty)\times \R$ is \emph{admissible} if it satisfies the following:
\begin{enumerate}
\item[(J-1')]\label{J'1} $J$ is tamed by $\omega$.
\item[(J-2')]\label{J'2} $J$ is split on  $(D_1\cup \cdots \cup D_n)\times [0,\infty)\times \R$.
\item[(J-3')]\label{J'3} $J$ is invariant under the action of $\Aut(\bH)$.
\item[(J-4')]\label{J'4} $J \partial /\partial s=\partial/\partial t$.
\item[(J-5')]\label{J'5} $T_p \Sigma \times \{0\}$ is a complex line in a neighborhood of $\betas\times [0,\infty)\times \R$ and $C\times [0,\infty)\times \R$, and along $\Sigma\times \{0\}\times \R$.
\end{enumerate}
\end{defn}

We note that any admissible almost complex structure on $\Sigma \times [0,\infty) \times \R$ clearly induces one on $\Sigma \times ([0,\infty)\sqcup (-\infty,0]) \times \R$. Additionally, Hendricks-Hom-Stoffregen-Zemke show that any almost complex structure on $\Sigma \times [0,1] \times \R$ gives rise to one on $\Sigma \times [0,\infty) \times \R$. 

\begin{prop}\label{prop:boundary-degenerations} (cf. \cite[Proposition 7.9]{HHSZ_surgery_exact_invol}) 
    Let $(\S, \alphas, \betas, \w, \tau)$ be a real Heegaard diagram, and let $J$ be a generic symmetric split almost complex structure on $\S \times ([0, \infty)\cup (-\infty,1]) \times \R$. If $B\cup \tau(B)$ is a real Maslov-index-$2$ class of real boundary degenerations at $\x$, then the moduli space $\widetilde{\cN}_J^R(B\cup \tau(B), \x)$ is transversely cut out and 
    \begin{align*}
        \# \widetilde{\cN}_J^{\Z/2}(\tau(B)\cup B, \x) \equiv 1 \mod 2,
    \end{align*}
    independent of $|\w|$. Similarly, if the class $[\Sigma]$ has real index 2, then
    \begin{align*}
        \# \widetilde{\cN}_J^{\{1\}}([\Sigma], \x) \equiv 1 \mod 2.
    \end{align*}
    In particular, the real cylindrical differential satisfies
    \begin{align*}
        \partial^2 = \begin{cases}
            0 & |\w| = 1\\
            \sum_{w_i \in \w} U_{w_i}^2 & |\w| > 1.
        \end{cases}
    \end{align*}
\end{prop}
\begin{proof}
    The transversality result and the curve count follow by the same proof as \cite[Proposition 7.9]{HHSZ_surgery_exact_invol}; indeed, our $\Z/2$-cylindrical degenerations are simply symmetric pairs of the boundary degenerations appearing in \cite[Proposition 7.9]{HHSZ_surgery_exact_invol}, limiting to curves in $\Sigma \times \{0\}\times \R$ and $\Sigma \times \{1\}\times \R$.  Our $\{1\}$-boundary degenerations are given by copies of $\Sigma$ mapping to $\Sigma \times \{1/2\}\times \R$. 
    
    In the case $|\bm w|>1$, there is a $\Z/2$-boundary degeneration for each of the $|\bm w|$ basepoints, just as in \cite[Theorem 5.5]{os_linkinvts}, represented by $B \cup \t(B)$ for $B \in \S \smallsetminus \betas$. In this case, there are no analogous $\{1\}$-boundary degenerations with which these contributions can cancel, hence they give rise to the curvature term appearing in \cite{guth_manolescu2025real}. 
    
    In the case that $|\bm w| = 1$, the class $[\Sigma]$ has real index 1, and so the class $[\Sigma\cup \t(\S)]$ has real index 2. Therefore the moduli spaces $\# \widetilde{\cN}_J^{\Z/2}(\Sigma\cup \t(\S), \x)$ and $\# \widetilde{\cN}_J^{\{1\}}(\Sigma\cup \t(\S), \x)$ always contribute to $\partial^2(\x)$ and  cancel according to  \cite[Proposition 7.9]{HHSZ_surgery_exact_invol}.
\end{proof}

\subsection{Compactness, transversality, and gluing results}

In the cylindrical setting, holomorphic rectangle maps are replaced by counts of holomorphic curves in $\Sigma \times \Diamond$, where $\Diamond$ is the subset of $\C$ shown in \Cref{fig:cylindrical_quadrilateral}, with four cylindrical ends modeled on $[0,1] \times [0, \infty)$. 

\begin{figure}[h]
\def\svgwidth{.8\linewidth}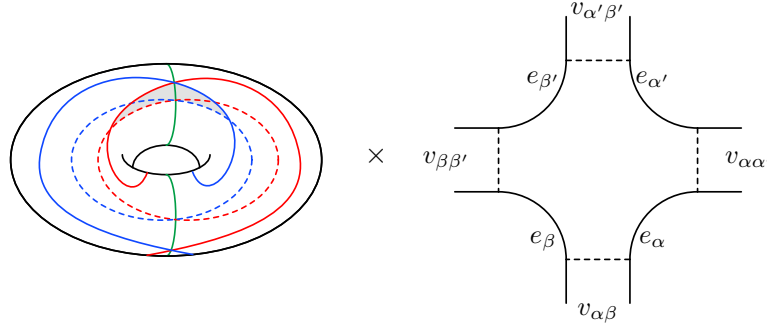
    \caption{The four-manifold $\Sigma \times \Diamond$ used to define the holomorphic quadrilateral maps.}
\label{fig:cylindrical_quadrilateral}
\end{figure}

In this setting we consider almost complex structures on $\Sigma\times \Diamond$ satisfying: 
\begin{enumerate}
\item[$(J'1')$] $J$ is tamed by the split symplectic form on $\Sigma \times \Diamond$.
\item[$(J'2')$] There is a finite collection of points 
    $P \subset \Sigma \setminus (\alphas'\cup \alphas\cup \betas\cup \betas')$,
    with at least one point in each component of  $\Sigma \setminus (\alphas' \cup \alphas\cup \betas\cup \betas')$,
    such that the almost complex structure is split on a product neighborhood of $P \times \Diamond$;
    i.e., $J=\mathfrak{j}_\Sigma \times \mathfrak{j}_\Diamond$.
\item[$(J'3')$] Near  the cylindrical ends of $\Diamond$,
    the almost complex structure~$J$ agrees with cylindrical almost complex structures
    on $\Sigma \times [0,1] \times \R$ satisfying condition $(J5')$ above.
\item[$(J'4')$] The 2-planes $T_d(\{p\}\times \Diamond)$ are complex lines
    of $J$ for all $(p,d) \in \Sigma \times \Diamond$.
\item[$(J'5')$] The 2-planes $T_p(\Sigma \times \{d\})$ are complex lines of $J$
    for $(p,d)$ near $(\alphas' \cup \alphas \cup \betas \cup \betas') \times \Diamond$ and
    for $(p,d) \in \Sigma \times U$, where $U\subset \Diamond$ is an open subset containing
    $\partial \Diamond \setminus \{v_{\alpha\alpha'},v_{\alpha\beta},v_{\beta\beta'},v_{\alpha'\beta'}\}$.
\end{enumerate}

We will write $\cJ_R(\Diamond)$ for the set of symmetric almost complex structures on $\Sigma \times \Diamond$ satisfying (J'1')-(J'5').

\begin{lem}\label{lem:open_mapping}
    Let $J \in \cJ_{R}(\Diamond).$ If $u: S \ra \Sigma \times \Diamond$ is a real $J$-holomorphic  curve and $\pi_\Sigma \circ u$ is nonconstant on a component $S_0$ of $S$, then $\pi_\Sigma \circ u|_{S_0}$ is an open map. Moreover, there are coordinates near any critical point of $\pi_\Sigma \circ u|_{S_0}$ where $\pi_\Sigma \circ u$ has the form $z \mapsto z^k$ for some $k > 0$.
\end{lem}
\begin{proof}
    This follows from \cite[Proposition 3.1]{Lipshitz_cylindrical}.
\end{proof}

A (possibly nodal) holomorphic rectangle in a real quadruple diagram $\cQ = (\Sigma, \bm \alpha', \bm \alpha, \bm \beta, \bm \beta')$ with $d = |\bm \alpha| = |\bm \alpha'| = |\bm \beta| = |\bm \beta'|$ is a map $u: S \ra \Sigma \times \Diamond$ that satisfies the following:
\begin{enumerate}
\item[$(M1)$] $(S,j)$ is a (possibly nodal) Riemann surface with boundary and $4d$ punctures on~$\partial S$.
\item[$(M2)$] $u$ is locally nonconstant and $(j,J)$-holomorphic.
\item[$(M3)$] $u(\partial S)\subset (\alphas'\times e_{\alpha'})\cup  (\alphas\times e_\alpha)\cup (\betas\times e_\beta)\cup (\betas'\times e_{\beta'}) $.
\item[$(M4)$] $u$ has finite energy.
\item[$(M5)$] For each $i \in \{\,1, \dots, d\,\}$ and $\eta \in \{\,\alpha',\alpha,\beta,\beta'\,\}$, the preimage
$u^{-1}(\eta_i \times e_\eta)$ consists of exactly one component of the punctured boundary of~$S$.
\item[$(M6)$] As one approaches the punctures along $\partial S$, the map $u$ converges to a collection of Reeb chords (i.e., intersection points on the Heegaard quadruple) in the cylindrical ends of $\Sigma \times \Diamond$.
\end{enumerate}
For the holomorphic rectangle maps, we also require:
\begin{enumerate}
\item[$(M7)$] $\pi_\Diamond\circ u$ is nonconstant on each component of $S$.
\item[$(M8)$] $S$ is smooth (i.e., not nodal), and the map
    $u \colon S \to \Sigma \times\Diamond$ is an embedding.
\end{enumerate}
If $\psi$ is a homology class of rectangles on a real Heegaard quadruple
$\mathcal{T}$, we write $\mathcal{M}_R(\psi)$ for the moduli space of real
holomorphic maps $u \colon S \to \Sigma \times \Diamond$ that satisfy
$(M1)$--$(M6)$. If $S$ is a decorated surface (possibly with nodes), then we
write $\mathcal{M}_R(\psi,S)$ for the subset of~$\mathcal{M}_R(\psi)$ consisting
of holomorphic curves with underlying source~$S$. We will see that if $\psi$
is a real Maslov-index-$0$ homology class of rectangles, then, for a generic choice of symmetric almost complex structure $J$ in $\cJ_R(\Diamond)$, any real
$J$-holomorphic map $u \colon S \to \Sigma \times \Diamond$ representing $\psi$,
and satisfying $(M1)$--$(M6)$ also satisfies $(M7)$ and $(M8)$.

\begin{defn}\label{def:brokenholrectangle}
    Let $\cQ = (\Sigma, \bm \alpha', \bm \alpha, \bm \beta, \bm \beta')$ be a quadruple diagram. A \emph{broken holomorphic rectangle} on $\cQ$ representing a homology class $\psi$ of rectangles is a collection of (possibly nodal) $(j, J)$-holomorphic real surfaces $(u_1,v_1, \hdots, v_n, w_1, \hdots, w_m)$ such that 
    \begin{enumerate}
        \item[(BR1)] $u_1$ maps into $\Sigma \times \Diamond$ and satisfies axioms (M1) and (M3)-(M6) for rectangles. 
        \item[(BR2)] $v_1, \hdots, v_n$ map into $\Sigma \times [0,1] \times \R$ and satisfy axioms (M1) and (M3)-(M6) for bigons and represent homology classes of disks in the diagrams $(\Sigma, \bm \alpha',  \bm \alpha)$, $(\Sigma, \bm \alpha,  \bm \beta)$, $(\Sigma, \bm \beta,  \bm \beta')$, and $(\Sigma, \bm \alpha',  \bm \beta')$.
        \item[(BR3)] $w_1, \hdots, w_m$ map into $\Sigma \times \Diamond$ and $\Sigma \times [0,1] \times \R$. For every $i \in \{1, \hdots, m\}$, the curve $w_i\bigsqcup R(w_i)$ has $2d$ boundary components, each with a single puncture, and maps onto one set of attaching curves (boundary degenerations).
        \item[(BR4)] The total homology class of the curves $u_1,v_1, \hdots, v_n, w_1, \hdots, w_m$ represents $\psi$.
    \end{enumerate}
\end{defn}

\begin{prop}
    Let $J\in \cJ_R(\Diamond)$. If $u_i$ is a sequence of real $J$-holomorphic rectangles satisfying axioms (M1)-(M6) in $\cM_R(\psi)$ for a homology class $\psi$, then there is a subsequence that converges to a broken holomorphic rectangle. 
\end{prop}
\begin{proof}
    This follows from \cite[Section 7]{Lipshitz_cylindrical}. The real structure prevents rectangles from splitting into triangles. 
\end{proof}

Given symmetric almost complex structures $J$ and $J_0$ on $\Sigma \times \Diamond$ and $\Sigma_0 \times \Diamond$, form their connected sum by removing small disks $D$ and $D_0$ which are fixed setwise by the involution. Let $J(T)$ be the almost complex structure on  $(\Sigma\# \S_0) \times \Diamond$ obtained by inserting a connected sum tube $S^1 \times [0,T]$ between $\S$ and $\S_0$ equipped with the involution $(z, t) \ra (\overline{z}, t)$. The following compactness result will be useful.

\begin{prop}\label{lem:neck-compactness}
    If $u_{T_{i}}$ is a sequence of holomorphic rectangles on $(\Sigma \# \Sigma_0)\times \Diamond$ for a sequence of real-invariant almost complex structures $J(T_i)$ with $T_i \ra \infty$, then a subsequence can be extracted that converges to a triple $(U, V, U_0)$, where $U$ is a broken real holomorphic rectangle in $\Sigma \times \Diamond$ and $U_0$ is a broken real holomorphic rectangle in $\Sigma_0 \times \Diamond$. Moreover, $V$ is a collection of real holomorphic curves mapping into the neck regions $S^1 \times \R \times \Diamond$ or $S^1 \times \R \times [0,1] \times \R$ which are asymptotic to possibly multiply covered Reeb chord orbits of the form $S^1 \times \{d\}$ in $S^1 \times \Diamond$ or $S^1 \times \R \times [0,1]$ for $d \in \Diamond$ or $d \in [0,1] \times \R$, respectively.
\end{prop}

Since this statement does not depend on transversality and the real moduli spaces are subspaces of the unreal ones, we can deduce it from the corresponding result in the unreal theory (which is a consequence of \cite[Appendix A]{Lipshitz_cylindrical}). Compare with \cite[Proposition 9.40]{JTZ_naturality_mapping_class_groups}. 

Fix a point $p$ in $\Sigma$ in the complement of the attaching curves. Let $u: S \ra \Sigma \times \Diamond$ be a real holomorphic curve. By \Cref{lem:open_mapping}, $u$ intersects $p \times \Diamond$ in finitely many points, with positive intersections. Let $(x_1, \hdots,x_{n_p(u)}) =: (\pi_\Sigma \circ u)^{-1}(p) \in \Sym^{n_p(u)}(S)$ and define
\begin{align*}
    \rho^p(u) = (\pi_\Diamond\circ u(x_1),\hdots, \pi_\Diamond\circ u(x_{n_p(u)}))\in\Sym^{n_p(u)}(\Diamond).
\end{align*}

\begin{defn}
    Let $\psi$ be a homology class of real rectangles with $n_p(\psi) = k$. Given $X \sub \Sym^k(\Diamond)$, define 
    \begin{align*}
        \cM_R(\psi, X):= \{u \in \cM_R(\psi): \rho^p(u) \in X\}.
    \end{align*}
    As above, for a particular real source $S$, we write $\cM_R(\psi, X, S)$ for the subspace of holomorphic curves in the homology class of $\psi$ with source $S$ that match $X$ at $p$.
\end{defn}

\begin{lem}\label{lem:somehere-injective}
    Let $\cQ$ be a quadruple diagram and fix a point $p$ in the complement of the attaching curves. Say $\bm d \in (\Sym^k(\Diamond))^R$ is not in the diagonal. Let $u: S \ra \Sigma\times \Diamond$ be a real $J$-holomorphic curve in $\cM_R$ for a symmetric almost complex structure $J \in \cJ_R$ so that $\rho^p(u) = \bm d$. Then each component of $u$ is somewhere injective.
\end{lem}
\begin{proof}
    This follows from \cite[Lemma 9.45]{JTZ_naturality_mapping_class_groups} and \cite[Lemma 3.3]{Lipshitz_cylindrical}.
\end{proof}

For a homology class of holomorphic rectangles $\psi$, the Maslov index $\mu(\psi)$ is the expected dimension of $\cM(\psi)$. If we fix a holomorphic rectangle $u: S \ra \Sigma\times \Diamond$ in the class of $\psi$, the expected dimension of the moduli space $\cM(\psi, S)$ is given by the Fredholm index of the linearization of the $\overline{\partial}$-operator at $u$, which we write as $\ind(\psi, S)$, and satisfies
\begin{align}\label{eqn:immered-index}
    \ind(\psi, S) = \mu(\psi) - 2\sing(u),
\end{align}
where $\sing(u)$ is the order of singularity of $u$; double points in the interior of $S$ contribute $+1$ and double points along the boundary contribute $+1/2$.

If $\psi$ represents a real holomorphic rectangle, then $\mu_R(\psi)$ is the expected dimension of $\cM_R(\psi)$ and if $u: S \ra \S\times \Diamond$ is a real holomorphic representative, then $\ind_R(\psi, S)$ is the expected dimension of $\cM_R(\psi, S)$. By \cite[Proposition 2.3]{guth_manolescu2025real}, we have the relation:
\begin{align*}
    \ind(\psi) - 2\ind_R(\psi) = \mu(\psi) - 2\mu_R(\psi).
\end{align*}
Combining this observation with \eqref{eqn:immered-index}, it follows that 
\begin{align}\label{eqn:real-immered-index}
    \ind_R(\psi, S) = \mu_R(\psi) - \sing(u).
\end{align}

We will also make use of the following transversality result.
\begin{prop}\label{prop:matching-transversality}
    Let $\cQ$ be a quadruple diagram and $p$ a point in the complement of the attaching curves. Suppose $X \subset (\Sym^k(\Diamond))^R$ for some $k$ is a non-empty submanifold which does not intersect the fat diagonal. Suppose that for each $x \in X$, the tuple $x$ has no coordinate in the open set $U \sub \Diamond$ from (J'5') which contains $\partial \Diamond \smallsetminus \{v_{\alpha\alpha'},v_{\alpha\beta},v_{\beta\beta'},v_{\alpha'\beta'}\}$. For a generic symmetric almost complex structure $J$, the moduli space $\cM_R(\psi, S, X)$ is a smooth manifold of dimension 
    \begin{align*}
        \ind_R(\psi, S) - \mathrm{codim}(X).
    \end{align*}
    When $X = (\Sym^k(\Diamond))^R$, the same result holds near any curve $u$ which has no component $T$ on which $\pi_\Diamond \circ u|T$ is constant and has image in $U$ and such that all components of $u$ are somewhere injective. 
\end{prop}
\begin{proof}
    This is a modification of \cite[Proposition 9.47]{JTZ_naturality_mapping_class_groups} and follows by their argument.
\end{proof}

Finally, we state gluing results for real holomorphic rectangles. We will be interested in the setting in which $\psi$ is a homology class of real rectangles on $\cQ$ and $\psi_0$ is a class of real rectangles on $\cQ_0$. We consider the class $\psi \# \psi_0$ in $\cQ \# \cQ_0$. 

\begin{prop}\label{prop: gluing}
    Let $u$ and $u_0$ be real holomorphic rectangles in $\Sigma \times \Diamond$ and $\Sigma_0 \times \Diamond$ of real index 0 and $k$, respectively, representing homology classes $\psi$ and $\psi_0$. Suppose 
    \begin{align*}
        \rho^p(u) = \rho^p(u_0) \in (\Sym^k(\Diamond))^R \smallsetminus \mathrm{Diag}^k(\Diamond)
    \end{align*}
    and that $\cM_R(\psi)$ and $\cM_R(\psi_0)$ are transversely cut out near $u$ and $u_0$. Then, there is a homeomorphism $h$ between a neighborhood of $(u, u_0)$ in the compactified moduli space 
    \begin{align*}
        \overline{\bigcup_T \cM_{R, J(T)}(\psi\# \psi_0)}
    \end{align*}
    and $[0,1)$ so that $h(u, u_0) = \{0\}$.
\end{prop}
\begin{proof}
    This follows immediately from \cite[Appendix]{Lipshitz_cylindrical} together with the dimension counts from \Cref{prop:matching-transversality}.
\end{proof}

\subsection{Simple trade loops}

In this section, we prove that real Heegaard Floer homology satisfies the trade axiom from \Cref{def:strong-Heegaard} \Cref{item:strong-trade}. Most of the work will consist in describing the equivalence $g$ that appears in the trade pentagon. This equivalence splits into three handleslides, shown in \Cref{fig:simple_trade_diagram0,fig:simple_trade_diagram1}.

\begin{figure}[h]
\def\svgwidth{.8\linewidth}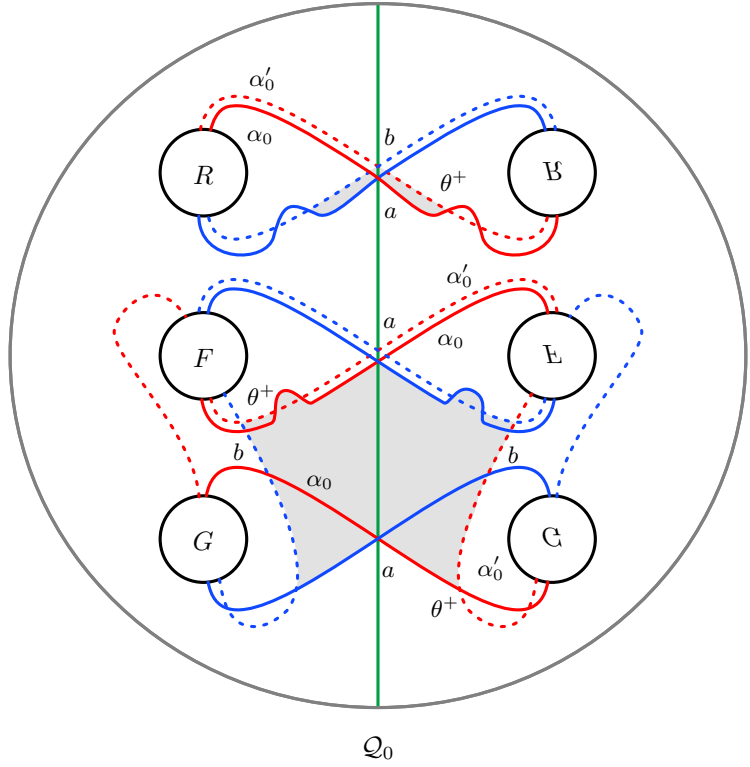
    \caption{A quadruple diagram used for computing the real handleslides in a trade loop.}
\label{fig:simple_trade_diagram0}
\end{figure}

Consider the real quadruple diagram $\cQ_{a,b} = (\Sigma_0, \bm \alpha_b, \bm \alpha_a, \bm \beta_a, \bm \beta_b,\t_0)$ shown in the first frame of \Cref{fig:simple_trade_diagram0}. If $\cQ = (\Sigma, \bm \alpha', \bm \alpha, \bm \beta, \bm \beta',\t)$ is any other real quadruple diagram, we consider the quadruple
\begin{align*}
    \cQ\#\cQ_{a,b} = (\Sigma \#\Sigma_0, \bm \alpha'\cup \bm \alpha_b, \bm \alpha\cup \bm \alpha_a, \bm \beta\cup \bm \beta_a, \bm \beta'\cup \bm \beta_b,\t \# \t_0),
\end{align*}
where the connected sum is taken at points $p\in \Sigma$ and $p_0 \in \S_0$ fixed by $\t$ and $\t_0$, respectively. Let $\bT_{\a_a} \cap \bT_{\b_a} = \bm a$ and $\bT_{\a_b} \cap \bT_{\b_b} = \bm b$  and 
\begin{align*}
    \bT_{\alpha_b}\cap \bT_{\alpha_a} =\{\theta_1^{\ep_1}\theta_2^{\ep_2}\theta_3^{\ep_3}\},
\end{align*}
where $\ep_i \in \{1,-1\}$; let $\Theta = \{\theta_1^+\theta_2^+\theta_3^+\}$ be the top generator. In \Cref{fig:simple_trade_diagram0}, we have labeled all $\theta_i^+$ by $\theta^+$ and all elements of $\bm a$ as $a$ and $\bm b$ as $b$ for simplicity. Since $(\Sigma_0, \bm \alphas_b, \alphas_a)$ represents $\#^3 (S^1 \times S^2)$, we have that $\dim \CFh(\Sigma_0, \bm \alphas_b, \alphas_a) = \dim \HFh(\#^3 (S^1 \times S^2))$, and hence the differential on $\dim \CFh(\Sigma_0, \bm \alphas_b, \alphas_a)$ must vanish.

\begin{lem}
    Consider the quadruple diagram $\cQ_{a,b} =  (\Sigma_0, \bm \alpha_b, \bm \alpha_a, \bm \beta_a, \bm \beta_b, \t_0)$ shown in \Cref{fig:simple_trade_diagram0}. For $\x \in \bT_{\alpha_b} \cap \bT_{\alpha_a}$ and $\psi_0 \in \pi_2^R(\x,\bm a, \tau_0(\x), \bm b)$, 
    \begin{align}\label{eqn:S_0-index-formula}
        \mu_R(\psi_0) = n_{p_0}(\psi_0) + \mu(\x, \Theta)
    \end{align}
    where $\mu(\bm z, \Theta)$ is the (ordinary) relative Maslov grading. 
\end{lem}
\begin{proof}
    This is the analogue of \cite[Lemma 9.50]{JTZ_naturality_mapping_class_groups}. First, assume $\x = \Theta$ and that $\psi_0$ is equal to the union of the small rectangle and octagon shown in \Cref{fig:simple_trade_diagram0}. The real index of $\psi_0$ is $0$ and $n_{p_0}(\psi_0) = 0$, so the claim holds in this case. The equation respects adding multiples of $[\Sigma_0]$. Every real doubly periodic domain has real index 0 and every real quadruply periodic domain is a sum of these. Hence, the equation is preserved under addition of real quadruply periodic domains. 

    Finally, if $\x$ is any other point of $\bT_{\alpha_b} \cap \bT_{\alpha_a}$, then we can choose a (non-invariant) domain $\phi \in \pi_2(\x, \Theta)$. Then $\psi_0 - (\phi - \tau(\phi))$ is an element of $\pi_2^R(\Theta,\bm a, \tau(\Theta), \bm b)$, and so satisfies \eqref{eqn:S_0-index-formula} by the argument above. However, $\mu(\phi)$ is exactly $\mu(\x, \Theta)$; this does not depend on $\phi$ since every periodic domain between $\bm \alpha_b$ and $\bm \alpha_a$ has unreal Maslov index 0. 
\end{proof}

\begin{prop}\label{prop:trade-loop-rectangles}
    Suppose that $\cQ = (\Sigma, \bm \alpha', \bm \alpha, \bm \beta, \bm \beta',\t)$ is an admissible real Heegaard quadruple diagram and let $\cQ \# \cQ_{a,b}$ be the connected sum where $\cQ_{a,b}$ is the quadruple shown in the first frame of \Cref{fig:simple_trade_diagram0}. For a sufficiently stretched real-invariant almost complex structure, the holomorphic rectangle maps satisfy
    \begin{align*}
        F_{\cQ \# \cQ_{a,b}}((\z \times \Theta) \otimes (\y\times \bm a) \otimes ( \t(\z) \times \t_0(\Theta))) = F_{\cQ}( \z \otimes \y \otimes \t(\z)) \times \bm b 
    \end{align*}
    for every $\z \in \bT_{\alpha'} \cap \Ta$ and $\y \in (\Ta \cap \Tb)^R$. 
\end{prop}
\begin{proof}
    We follow the proof strategy for handleswaps \cite[Proposition 9.31]{JTZ_naturality_mapping_class_groups}, pointing out where changes must be made in the real setting. Fix a class $\psi \in \pi_2^R(\bm z, \y, R(\bm z), \bm u)$ of rectangles in $\cQ$ as well as a class $\psi_0 \in \pi_2^R(\x, \bm a,\tau_0(\x), \bm b)$ of rectangles in $\cQ_{a,b}$ with $n_{p}(\psi) = n_{p_0}(\psi_0)$ and form $\psi \# \psi_0$. The index formula combined with the previous lemma implies that
    \begin{align}\label{eqn:index for sum}
        \mu_R(\psi \# \psi_0) = \mu_R(\psi) + \mu_R(\psi_0) - n_{p_0}(\psi_0) = \mu_R(\psi) + \mu(\x, \Theta).
    \end{align}

    Suppose the class $\psi \# \psi_0$ has holomorphic representatives $u_{T_i}$ for a sequence of neck lengths $T_i \ra \infty$. There is a subsequence converging to a broken rectangle $U$ on $\cQ$ representing $\psi$, a broken rectangle $U_0$ on $\cQ_0$ representing $\psi_0$ and a collection of curves $V$ in the connected sum regions $S^1 \times \R \times \Diamond$ or $S^1 \times \R \times [0,1]\times \R$. As we are interested in real index $0$ rectangles, Equation \eqref{eqn:index for sum} implies that 
    \begin{align*}
        0 = \mu_R(\psi\# \psi_0) = \mu_R(\psi) + \mu(\Theta, \Theta) = \mu_R(\psi) + \mu(\Theta, \Theta),
    \end{align*}
    and therefore that $\mu_R([U]) = 0$. 

    Following \cite[Proof of Proposition 9.31]{JTZ_naturality_mapping_class_groups}, we break the remainder of the proof into four claims. The analogous statements are listed in \cite[Equation 9.55]{JTZ_naturality_mapping_class_groups}.

    \begin{enumerate}
        \item[\textbf{Claim 1:}] If $U$ has no constant components, then it consists of a single holomorphic rectangle $u$ of real index 0 satisfying (M1)-(M8).
        \begin{proof}
            The broken curve $U$ has no components mapping into $\Sigma \times [0,1] \times \R$; by transversality, curves into $\Sigma\times [0,1] \times \R$ have index at least 1, which would force components mapping into $\Sigma\times \Diamond$ to have negative index. Hence, $U$ consists of a single real index 0 rectangle.  

            We now claim that $U$ satisfies (M1)-(M6). (M1)-(M4) and (M6) are immediate. The fact that property (M5) is preserved follows just as in \cite[Proposition 9.31, Proof of Claim (1)]{JTZ_naturality_mapping_class_groups}.  Roughly, (M5) holds when the source curve is degenerated along a simple closed curve (although this is impossible by index considerations) and strip breaking is ruled out by the previous paragraph. Hence, it suffices to consider the case that the source curve is pinched along (symmetric pairs) of arcs connecting boundary components of the same type, resulting in boundary degenerations, which increase the real index by at least 2 (compare to \cite[Lemma 4.9]{guth_manolescu2025real}). Hence, by transversality, boundary degenerations are ruled out. Properties (M7) and (M8) also hold; nodal source curves and boundary degenerations have already been ruled out, and we have assumed $U$ contains no constant curves. 
        \end{proof}

    \item[\textbf{Claim 2:}] If $U_0$ has no constant components, it consists of a single rectangle $u_0$ of real Maslov index $k:= n_{p_0}$ satisfying (M1)-(M8).
    \begin{proof}
        As in \cite[Proposition 9.31, Proof of Claim (2)]{JTZ_naturality_mapping_class_groups}, we may assume that $p$ is not a branch point of $\pi_\Diamond\circ v$ for any index 0 rectangle $v$ in $(\Sigma, \alphas', \alphas, \betas,\betas',\t)$ and that $\rho^p(v)$ is not an element of $\mathrm{Diag}^k(\Diamond)$. This forces curves in $V$ to map to $S^1 \times \R \times \Diamond$. 

        By Equation \eqref{eqn:index for sum}, we have that $\mu_R(\psi_0) = k$. A priori, the curve $U_0$ might consist of broken rectangles with nodal curves, boundary degenerations, and holomorphic strips. The asymptotics of $u$ and $V$ agree, so there must be a curve $u_0$ in $U_0$ which is matched with $u$, i.e.
        \begin{align*}
            \rho^p(u) = \rho^{p_0}(u_0).
        \end{align*}
        Since $\rho^p(u)$ is in $\Sym^k(\Diamond)^R\smallsetminus\mathrm{Diag}^k(\Diamond)$, \Cref{lem:somehere-injective} can be applied to guarantee that curve $u_0$ satisfying the matching condition is somewhere injective on each component, and hence Proposition \Cref{prop:matching-transversality} can be applied to show that the moduli space 
        \begin{align*}
            \cM_R(\psi_0, \bm d) = \{u_0 \in \cM_R(\psi_0): \rho^{p_0}(u_0) = \bm d\}
        \end{align*}
        is a smooth manifold of dimension 
        \begin{align*}
            \mu_R(\psi_0) - \mathrm{codim}(\{\bm d\})
        \end{align*}
        for $\bm d \in \{\rho^{p_0}(u): u \text{ holomorphic on } \cQ, \mu_R(u) = 0 \}$. The collection $\{\bm d\}$ is finite, and therefore, we have that 
        \begin{align*}
            \dim \cM_R(\psi_0, \bm d) = n_{p_0}(\psi_0) - n_{p}(u) = 0.
        \end{align*}
    Repeating the arguments from Claim 1, we see that $U_0$ must consist of a single rectangle in $\cM_R(\psi_0, \bm d)$ mapping into $\Sigma_0 \times \Diamond$ satisfying (M1)-(M8).
    \end{proof}
    \item[\textbf{Claim 3:}] $V$ consists entirely of trivial cylinders and there are no constant components of $U$ or $U_0$.
    \begin{proof}
        We follow the arguments in \cite[Proposition 9.31, Proof of Claim (3)]{JTZ_naturality_mapping_class_groups}. By passing to a subsequence, we can assume each $u_{T_i}$ has the same source curve, $\hat{S}$. This surface can be reconstructed from the curves in  $U$, $V$, and $U_0$ and a simple Euler characteristic argument shows that if $V$ did not consist entirely of trivial cylinders, then one would have that
        \begin{align*}
            \chi(\hat{S}) < \chi(S) + \chi(S_0) - 2k.
        \end{align*}
        According to \cite[Section 10.2]{Lipshitz_cylindrical}, we have
        \begin{align*}
            \ind(\psi\#\psi_0,\hat{S}) = \frac{1}{2}(d+2) - \chi(\hat{S}) + 2e(\psi\#\psi_0).
        \end{align*}
        This in turn implies
        \begin{align*}
            \ind(\psi, S) + \ind(\psi_0, S_0) < 2k - 2 + \ind(\psi \# \psi_0, \hat{S}) = 2k-2.
        \end{align*}
        The relation between the real and unreal indices then implies that
        \begin{align*}
            \ind_R(\psi, S) + \ind_R(\psi_0, S_0) < \frac{1}{2}(2k - 2 + \ind_R(\psi \# \psi_0, \hat{S})) = k-1.
        \end{align*}
        However, since we assumed that $\cM_R(\psi)$ and $\cM_R(\psi_0, \bm d)$ are smoothly cut out and are of the expected dimension, this is impossible, as one of $\cM_R(\psi)$ or $\cM_R(\psi_0, \bm d)$ would have to have negative dimension. The same argument also rules out constant components of $U$ and $U_0$.
    \end{proof}
    \end{enumerate}
    
    From these claims, it follows that the sequence $u_{t_i}$ limits to a pair $(u, u_0)$ with $\mu_R(u) = 0$, $\mu_R(\psi_0) = k$, and $\rho^p(u) = \rho^{p_0}(u_0)$. As \Cref{prop: gluing} describes a neighborhood of $(u, u_0)$ in the compactified moduli space
    \begin{align*}
        \overline{\bigcup_T \cM_{R,J(T)}(\psi \# \psi_0)},
    \end{align*}
    it follows that 
    \begin{align*}
        \# \cM_{R,J_T}(\psi\#\psi_0) = \# \{(u,u_0)\in \cM_R(\psi)\times \cM_R(\psi_0): \rho^p(u) = \rho^{p_0}(u_0)\}.
    \end{align*}
    For $\x \in \bT_{\alpha_b} \cap \bT_{\alpha_a}$ and $\bm d \in (\Sym^k(\Diamond))^R$ a generic divisor, define
    \begin{align*}
        \cM_{R, \x}(\bm d) = \coprod_{\substack{\psi_0 \in \pi_2^R(\x, \bm a, \t_0( \x), \bm b)\\ n_{p_0}(\psi_0) = k}} \cM_R(\psi_0,\bm d).
    \end{align*}
    It therefore suffices to prove that 
    \begin{align*}
        \cM_{R, \Theta}(\rho^{p}(u)) \equiv 1 \mod 2.
    \end{align*}
    In fact, we have the following.
    \begin{enumerate}
        \item[\textbf{Claim 4:}] Let $\bm d \in \Sym^k(\Diamond)^R$ be in the complement of the diagonal. Then, $\cM_{R, \Theta}(\bm d)$ is a smoothly cut out zero-dimensional manifold for a generic real almost complex structure. For such a $J$, we have 
    \begin{align*}
        \cM_{R, \Theta}(\bm d) \equiv 1 \mod 2.
    \end{align*}
    \begin{proof}
        For a fixed $\bm d \in \Sym^k(\Diamond)^R$ avoiding the diagonal and a neighborhood of $\partial \Diamond\smallsetminus \{v_{\alpha', \alpha},v_{\alpha, \beta},v_{\beta, \beta'}\}$, transversality for $\cM_{R, \x}(\bm d)$ can be achieved for a generic real almost complex structure satisfying (J1)-(J5), by arguments similar to those above.

        Further, the count of points in the space $\cM_{R, \x}(\bm d)$ is independent of $\bm d$ for a generic $\bm d$. This follows as in \cite[Lemma 9.58]{JTZ_naturality_mapping_class_groups}. In the standard case, one chooses a path of divisors avoiding the diagonal, and shows that the counts of points in the moduli spaces at the two ends agree. In the unreal case, such paths generically miss the diagonal, which is codimension 2. In the present situation, however, $\mathrm{Diag}^k(\Diamond)\cap \Sym^k(\Diamond)^R$ is codimension-1: even worse, $\Sym^k(\Diamond)^R \smallsetminus \mathrm{Diag}^k(\Diamond)$ is disconnected. $\Sym^k(\Diamond)^R$ decomposes into connected strata as
        \begin{align*}
            \Sym^k(\Diamond)^R = \bigcup_{\ell=1}^k \Sym^{k-2\ell}(\Diamond^R) \times \Sym^{\ell}(\Diamond/R).
        \end{align*}
        Each $(\Sym^{k-2\ell}(\Diamond^R) \times \Sym^{\ell}(\Diamond/R))\smallsetminus \mathrm{Diag}^k(\Diamond)$ stratum is connected. Indeed, 
        \[\Sym^q(\Diamond^R) \smallsetminus \mathrm{Diag}^q(\Diamond^R) = \Sym^q(\{1/2\}\times \R)\smallsetminus \mathrm{Diag}^q(\{1/2\}\times \R)
        \] is connected as any \emph{unordered} tuple of points can be connected by a path to any other avoiding the diagonal. Similarly for $$\Sym^q([0,1/2)\times \R) \smallsetminus \mathrm{Diag}^q([0,1/2)\times \R).$$ However, the strata $\Sym^{k-2\ell}(\Diamond^R) \times \Sym^{\ell}(\Diamond/R)$ are glued along facets contained in $\mathrm{Diag}^k(\Diamond)$. 
        
        Therefore, we begin by showing that within a given connected component $(\Sym^{k-2\ell}(\Diamond^R) \times \Sym^{\ell}(\Diamond/R)) \smallsetminus \mathrm{Diag}^k(\Diamond)$, the space $\cM_{R, \x}(\bm d)$ is independent of the choice of divisor. To do so, we count the ends of the 1-dimensional moduli space 
        \begin{align*}
            \cM_{R, \x}(\bm D) = \bigcup_{t \in [0,1]} \cM_{R, \x}(\bm D(t)).
        \end{align*}
        The space $\cM_{R, \x}(\bm D)$ has ends corresponding to $\cM_{R, \x}(\bm D(0))$ and $\cM_{R, \x}(\bm D(1))$ as well as ends corresponding to broken holomorphic rectangles. We claim that ends of the second kind necessarily correspond to index $k$ rectangles breaking into index $k-1$ rectangles matching $\bm D(t)$ and an index 1 strip with multiplicity zero at $p_0$ or into an index $k-1$ rectangle and a constant copy of $\Sigma_0$ along $\Diamond^R$. 
    
        Let $u_i: S_0 \ra \Sigma \times \Diamond$ be a sequence in $\cM_{R, \x}(\bm D)$ representing $\psi_0$. Suppose that $u_i$ degenerates along a $\Z/2$-orbit of disjoint, closed curves in $S_0$. Then, the limiting source has at least two interior nodal singularities which map into $\Sigma_0 \times \Diamond$. Let $S_0'$ be the collection of components of the broken curve and let $\psi_0'$ be its homology class. Since 
        \begin{align*}
            \ind_R(\psi_0', S_0') = \mu_R(\psi_0') - \mathrm{sing}(u),
        \end{align*}
        we have that $\ind_R(\psi_0', S_0') \le \mu_R(\psi_0') - 2 = k - 2$, as $S_0'$ has at least two nodal singularities. We can compute: 
        \begin{align*}
            \dim \cM_R(\psi_0', S_0', \bm D) = \ind_R(\psi_0', S_0') - \mathrm{codim}(\bm D) \le (k-2) - (k-1) = -1.
        \end{align*}
        Hence, by transversality, no such degenerations can occur. In the case that the degeneration occurs along an invariant closed curve (appearing in a trivial orbit), the limiting curve contains an interior nodal singularity which is fixed by the involution. In this case, the argument above fails, and we cannot rule out such degenerations. However, just as in \cite[Lemma 4.12]{guth_manolescu2025real}, such degenerations necessarily come in pairs. Hence, these ends do not contribute to the total $\Z/2$-count of points $\partial \cM_{R, \x}(\bm D)$. Compare with \Cref{prop:boundary-degenerations}.

        Next, we consider curves with boundary nodes. Each such node contributes only $\frac{1}{2}$ to $\mathrm{sing}(u)$, and therefore the index argument above fails. Instead, as in \cite[Lemma 9.58]{JTZ_naturality_mapping_class_groups}, we rule these out using the matching condition. Were such a curve to form, there would be a limiting curve mapping into $\Sigma_0 \times \{x\}$ with $x \in \partial \Diamond$ with nonzero multiplicity on all of $\Sigma_0$.  This is impossible. Since the limiting curve has nonzero multiplicity on all of $\Sigma_0$, $x$ must appear as a point of $\rho^{p_0}(u_0)$, but $\bm D$ misses the boundary of $\Diamond$. Hence, $\rho^{p_0}(u_0)$ cannot coincide with any $\bm d(t)$.

        Finally, we consider the case that $u_i$ converges to a broken rectangle and a nontrivial holomorphic strip. The limiting rectangle $u$ represents a class $\psi$ in $\pi_2^R(\x, \bm a, \t_0(\x), \bm b)$, for $\x \in \bT_{\alpha_b} \cap \bT_{\alpha_a}$. If $\x = \Theta$, then by \eqref{eqn:index for sum}, we have that  $\mu_R(\psi) = k = \mu_R(\psi_0)$, and the remaining curves must have index zero, miss $p_0$, and therefore be constant. By transversality, such $u$ is in the interior of $\cM_{R,\x}(\bm D)$, and therefore does not contribute. If $x = \theta^+_1\theta^+_2\theta^-_3$, $\theta^+_1\theta^-_2\theta^+_3$, or $\theta^-_1\theta^+_2\theta^+_3$ then $u$ has index $k-1$ by \eqref{eqn:index for sum}, so the remaining curves could consist of a pair of holomorphic strips. The count of such ends is trivial, since $\CFh(\Sigma_0, \bm\alpha_b, \bm\alpha_a,p_0)$ has trivial differential. In the remaining cases, if $\x = \theta^+_1\theta^-_2\theta^-_3$, $\theta^-_1\theta^+_2\theta^-_3$, $\theta^-_1\theta^-_2\theta^+_3$, or $\theta^-_1\theta^-_2\theta^-_3$ and $\psi \in \pi_2^R(\x, \bm a, \t_0(\x),\bm b)$ with $n_{p_0}(\psi) = k$, then $\cM_R(\psi, \bm D)$ is empty. Indeed, applying \eqref{eqn:index for sum} once more, we have that $\mu_R(\psi) \in \{ k-2, k-3\}$. By \Cref{prop:matching-transversality}, $\cM_R(\psi,S^, \bm D)$ is a smooth manifold of dimension 
        \begin{align*}
            \ind_R(\psi, S) - (k-1) \le \mu_R(\psi) - (k-1) < 0.
        \end{align*}
        Hence, for a generic symmetric $J$, there are no holomorphic curves representing $\psi$ satisfying the matching condition. This completes the proof that $\cM_{R, (\Theta, \bm a,\tau_0(\Theta),\bm b )}(\bm d_0) = \cM_{R, (\Theta, \bm a,\tau_0(\Theta),\bm b)}(\bm d_1) \mod 2.$

        From here, the proof is the same as in \cite[Lemma 9.58]{JTZ_naturality_mapping_class_groups}. A priori, $\cM_{R, x}(\bm d)$ may depend on the connected component of $\Sym^k(\Diamond) \smallsetminus \mathrm{Diag}^k(\Diamond)$, but it is independent of the divisor within the component. Therefore, suppose that $\cD: [0,\infty) \ra \Sym^k(\Diamond)^R$ is a smooth path in $\Sym^k(\Diamond) \smallsetminus \mathrm{Diag}^k(\Diamond)$ starting at a divisor $\bm d$ where points of $\bm D(T)$ are spaced at least distance $T$ apart and approach $v_{\alpha,\beta}$ as $T \ra \infty$. We consider the space 
        \begin{align*}
            \cM_{R, \Theta}(\bm D) = \bigcup_{T \in [1,\infty)} \cM_{R, \Theta}(\bm D(T)),
        \end{align*}
        which has ends corresponding to $\cM_{R, \Theta}(\bm d)$, ends corresponding to degenerations of holomorphic curves, and ends corresponding to curves which appear as $T \ra \infty$. The argument above applies here to show that only strips can break off at finite times $T$, and the mod two count of these contributions is zero. 

        In the limit, the curve breaks into an index-zero rectangle $\mathfrak r$ and $k$ curves of real Maslov index $1$ satisfying a matching condition. Points of $\bm D(T)$ separate as they approach $v_{\alpha, \beta}$, so in the limit there must be $k$ such curves on $(\Sigma_0, \bm \alpha_0, \bm \beta_0)$, all of the form $e_a +  [\Sigma_0]$, i.e. a constant disk at $a$ and a copy of $\Sigma_0$, each contributing $1$ to the index. Hence, the remaining curve $\frak r$ which maps to $\Sigma_0 \times \Diamond$ must be index 0 and for the matching condition to be satisfied, we must have $n_{p_0}(\frak r) = 0$. As above, one argues the limiting curves satisfy (M1)-(M8).

        Hence, by \cite{Lipshitz_cylindrical}, we have that 
        \begin{align*}
            \# \cM_{R,\Theta}(\bm d) \equiv \left(\# \cM_{R, \bm a}(d) \right)^k \cdot \sum_{\substack{\psi \in \pi_2^R(\Theta, \bm a,\tau_0(\Theta),\bm b)\\ n_{p_0}(\psi) = 0 }} \# \cM_R(\psi),
        \end{align*}
        where $d \in \{1/2\}\times \R$ and $\cM_{R, \bm a}(d)$ is the moduli space of real index 1 curves $u$ on $(\Sigma_0, \alphas_0, \betas_0)$ with $\rho^{p_0} = d$ for a generic symmetric almost complex structure satisfying (J1)-(J4) and (J5'). Furthermore, the count above depended in no way on the component of $\Sym^k(\Diamond) \smallsetminus \mathrm{Diag}^k(\Diamond)$.
    \end{proof}
    \end{enumerate}

\begin{figure}[h]
\def\svgwidth{.8\linewidth}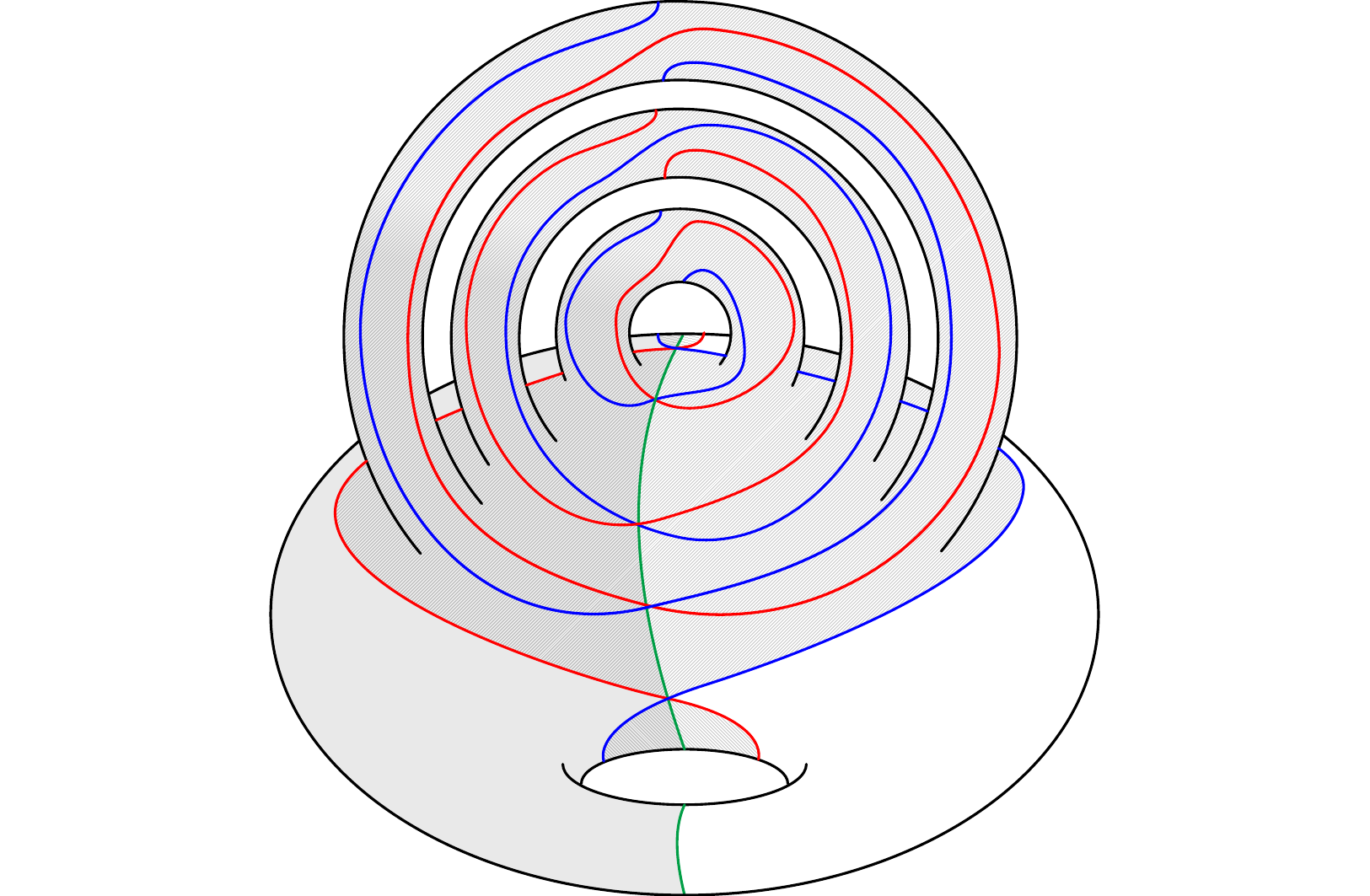
    \caption{A triple stabilization of a real Heegaard diagram for $(\Sigma_2(U \cup U), \tau)$. The two domains contributing to the differential are shaded.
}
\label{fig:triple_stabilization}
\end{figure}
        
        To count the points of $\cM_{R,\bm a}(d)$, we take a genus 1 diagram for the branched double cover of the two-component unlink, equipped with the covering action. There are two real-invariant bigons which contribute to the differential. By performing three $\{1\}$-stabilizations inside one of the bigons, we obtain a class which must admit an odd number of holomorphic representatives, by the invariance of $\widehat\HFR$. See \Cref{fig:triple_stabilization}. By stretching the neck, this count must agree with the count of strips in $\cM_R(e_a + [\Sigma_0])$ which match the connected sum point. Hence, $\#\cM_{R,\bm a}(d) \equiv 1 \mod 2$. Furthermore, it must be that 
    \begin{align*}
        \sum_{\substack{\psi \in \pi_2^R(\Theta, \bm a,\tau_0(\Theta),\bm b)\\ n_{p_0}(\psi) = 0 }} \# \cM_R(\psi) \equiv 1 \mod 2,
    \end{align*}
        since this map is a quasi-isomorphism of 1-dimensional chain complexes. This concludes the proof.
\end{proof}

We now turn to the handleslides in \Cref{fig:hd_d2_trade_loop}; let $\cQ_{b,c}$ and $\cQ_{c,d}$ be the corresponding real quadruple diagrams shown in \Cref{fig:simple_trade_diagram1}. For both diagrams, $\CFh(\alphas_c, \alphas_b)$ has a unique generator in the maximal grading, and we will write $\Theta$ for both. In $\cQ_{b,c}$, we write $\bT_{\alpha_b}\cap \bT_{\beta_b} = \bm b$ and $\bT_{\alpha_c}\cap \bT_{\beta_c} = \bm c$. For $\cQ_{c,d}$, we write $\bT_{\alpha_c}\cap \bT_{\beta_c} = \bm c$ and $\bT_{\alpha_d}\cap \bT_{\beta_d} = \bm d$. For these rectangle counting maps, we have the analogue of \Cref{prop:trade-loop-rectangles}.

\begin{prop}\label{prop:remaining-trade-loops-counts}
    Suppose that $\cQ = (\Sigma, \bm \alpha', \bm \alpha, \bm \beta, \bm \beta',\t)$ is an admissible real Heegaard quadruple diagram and let $\cQ \# \cQ_*$ be the connected sum where $\cQ_*$  is one of the quadruple diagrams shown in \Cref{fig:simple_trade_diagram1} for $*\in\{(b,c),(c,d)\}$. For a sufficiently stretched real-invariant almost complex structure, the holomorphic rectangle maps satisfy
    \begin{align*}
        F_{\cQ \# \cQ_{b,c}}((\z \times \Theta) \otimes (\y\times \bm b) \otimes ( \t(\z) \times \t_1(\Theta))) &\sim F_{\cQ}( \z \otimes \y \otimes \t(\z)) \times \bm c \\
        F_{\cQ \# \cQ_{c,d}}((\z \times \Theta) \otimes (\y\times \bm c) \otimes ( \t(\z) \times \t_2(\Theta))) &\sim F_{\cQ}( \z \otimes \y \otimes \t(\z)) \times \bm d \\
    \end{align*}
    for every $\z \in \bT_{\alpha'} \cap \Ta$ and $\y \in (\Ta \cap \Tb)^R$. 
\end{prop}

\begin{proof}
    The proofs are identical to those of \Cref{prop:trade-loop-rectangles}.
\end{proof}

\begin{figure}[h]
\def\svgwidth{.8\linewidth}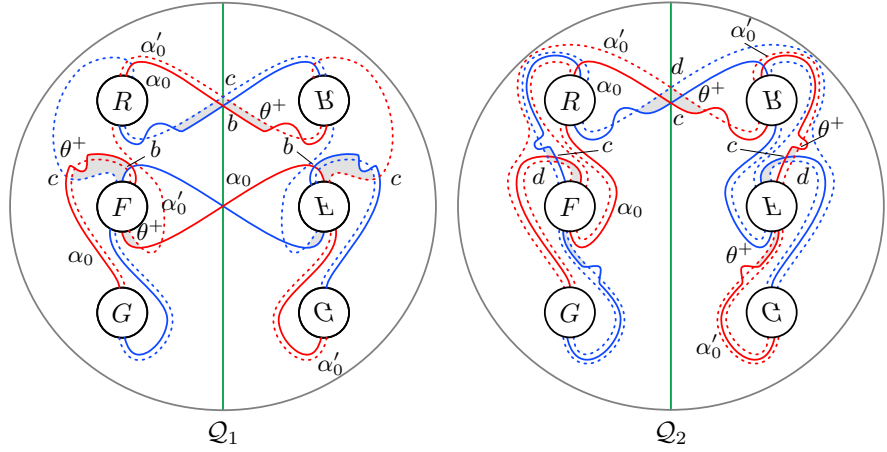
    \caption{The quadruple diagrams used for computing the real handleslides in a trade loop.}
\label{fig:simple_trade_diagram1}
\end{figure}

Let 
    \begin{align*}
    \begin{tikzcd}[ampersand replacement = \&,column sep = small]
        \&\& H_1 \ar[rrd,"e"] \&\& \\
        H_5 \ar[urr,"i"] \&\&\&\& H_2  \ar[dl,"f"] \\
       \& H_4 \ar[ul,"h"] \& \& H_3  \ar[ll,"g"] \&\\
    \end{tikzcd}
    \end{align*}
be a simple trade loop as in \Cref{def:simple-trade}. Recall that edges $e$ and $f$ are $\{1\}$-stabilizations, $g$ is a real equivalence, $h$ is an equivariant diffeomorphism, and $i$ is a $\Z/2$-destabilization. The real equivalence is the composition of the three handleslides appearing in \Cref{fig:simple_trade_diagram0,fig:simple_trade_diagram1}. Therefore, the map $\Phi_g$ can be computed as the composition of the three handleslide maps described in \Cref{prop:trade-loop-rectangles,prop:remaining-trade-loops-counts}.

To prove invariance under simple trade loops, we must prove the following.
\begin{thm}\label{thm:trade-loop-is-id}
    The induced maps $\Phi_{\ast}$ for $\ast \in \{e,f,g,h,i\}$ satisfy
    \begin{align*}
        \Phi_{i}\circ\Phi_{h}\circ\Phi_{g}\circ\Phi_{f}\circ\Phi_{e} \sim \id_{\CFR^\circ(H_1)}.
    \end{align*}
\end{thm}
\begin{proof}
    Let $\cH = (\S, \alphas, \betas, \t)$ be a real sutured Heegaard diagram. Let $\cH_k = \cH \#^k E_{\{1\}}$, where $E_{\{1\}}$ is the genus 1 real Heegaard diagram for $\Sigma_2(U)$ such that each $\cH_k$ is a representative of $[H_k]$ for $k\in \{1, 2, 3\}$. $\cH_4$ is obtained from $\cH_3$ by a real equivalence, which is the composition of three real handleslides. Let $\cH_3(a) = H_3$, $\cH_3(b)$, $\cH_3(c)$, and $\cH_3(d)= \cH_4$ be the intermediate diagrams obtained by the real handleslides shown in \Cref{fig:hd_d2_trade_loop}. The diagram $\cH_5 = \cH_1 \#E_{\Z/2}$ is obtained from $\cH_4$ by a diffeomorphism, $\varphi$, where $E_{\Z/2}$ is the standard genus 2 diagram for $\S_2(U)$. We write $e$ and $\t_0(e)$ for the two symmetric intersection points of $E_{\Z/2}$. Evidently, $\cH_5$ can be $\Z/2$-destabilized to return to $\cH_1.$ According to \Cref{prop:gluing-rectangles-stab}, we can fix symmetric almost complex structures on $\cH_i$ so that the stabilization maps and holomorphic rectangle maps commute. 
    
    In \Cref{fig:simple_trade_diagram0}, there are three intersection points in $\S_0$ which (for simplicity) are all labeled $a$ (similarly for $b$, $c$, $d$, and $\theta^+$ in \Cref{fig:simple_trade_diagram1}). Let $a_1$, $a_2,$ and $a_3$ be the intersection points associated to the handles labeled $R$, $F$, and $G$ respectively. In this way, if $\x \times a_1$ is a generator for $\CFR^\circ(\cH_1)\simeq \CFR^\circ(\cH) \otimes \CFR^\circ(E_{\{1\}})$, then $\Phi_{f}\circ\Phi_{e}(\x \times a_1) = \x \times \bm a$, where $\bm a = \{a_1a_2a_3\}$. Let $\bm b$, $\bm c$, and $\bm d$ be the analogous triples of intersection points in the diagrams $\cH_3(b)$, $\cH_3(c)$, and $\cH_3(d)$  shown  in \Cref{fig:simple_trade_diagram1}. The real equivalence $\Phi_{g}$ is the composition of three real handleslides\footnote{Here we abuse notation: we write $\alphas$ for the alpha curves in all four diagrams $\cH_3(a)$ through $\cH_3(d)$, though implicitly we are applying small Hamiltonian perturbations to pass from one collection to the next.} 
    \begin{align*}
        \Phi_{g} = \Psi^{\alphas \cup \alphas_c \to \alphas \cup \alphas_d}_{\betas \cup \betas_c \to \betas \cup \betas_d} \circ \Psi^{\alphas \cup \alphas_b \to \alphas \cup \alphas_c}_{\betas \cup \betas_b \to \betas \cup \betas_c} \circ \Psi^{\alphas \cup \alphas_a \to \alphas \cup \alphas_b}_{\betas \cup \betas_a \to \betas \cup \betas_b},
    \end{align*}
    where $\alphas_i$ are the alpha curves of $\cH_3(i)$ for $i \in \{a,b,c,d\}$. The curves in $\cH_3(a)$ and $\cH_3(b)$ form a real quadruple diagram $\cQ\#\cQ_{a,b}$. The map $\Psi^{\alphas \cup \alphas_a \to \alphas \cup \alphas_b}_{\betas \cup \betas_a \to \betas \cup \betas_b}$ is defined in terms of holomorphic rectangle counts:
    \begin{align*}
        \Psi^{\alphas \cup \alphas_a \to \alphas \cup \alphas_b}_{\betas \cup \betas_a \to \betas \cup \betas_b}(\x \times \bm a) =  F_{\cQ \# \cQ_{a,b}}((\Theta_{\alphas',\alphas} \times \Theta) \otimes (\x\times \bm a) \otimes ( \t(\Theta_{\alphas',\alphas}) \times \t_0(\Theta))).
    \end{align*}
    According to \Cref{prop:trade-loop-rectangles}, for the appropriate choice of symmetric almost complex structure
    \begin{align*}
        F_{\cQ \# \cQ_{a,b}}((\Theta_{\alphas',\alphas} \times \Theta) \otimes (\x\times \bm a) \otimes ( \t(\Theta_{\alphas',\alphas}) \times \t_0(\Theta))) = F_{\cQ}( \Theta_{\alphas',\alphas} \otimes \x \otimes \t(\Theta_{\alphas',\alphas})) \times \bm b.
    \end{align*}
    Similarly, by \Cref{prop:remaining-trade-loops-counts}, 
     \begin{align*}
        \Psi^{\alphas \cup \alphas_b \to \alphas \cup \alphas_c}_{\betas \cup \betas_b \to \betas \cup \betas_c} ((\Theta_{\alphas',\alphas} \times \Theta) \otimes (\x\times \bm b) \otimes ( \t(\Theta_{\alphas',\alphas}) \times \t_1(\Theta))) &\sim F_{\cQ}( \Theta_{\alphas',\alphas} \otimes \x \otimes \t(\Theta_{\alphas',\alphas})) \times \bm c \\
        \Psi^{\alphas \cup \alphas_c \to \alphas \cup \alphas_d}_{\betas \cup \betas_c \to \betas \cup \betas_d}((\Theta_{\alphas',\alphas} \times \Theta) \otimes (\x\times \bm c) \otimes ( \t(\Theta_{\alphas',\alphas}) \times \t_2(\Theta))) &\sim F_{\cQ}( \Theta_{\alphas',\alphas} \otimes \x \otimes \t(\Theta_{\alphas',\alphas})) \times \bm d.
    \end{align*}
    Combining these facts, we have that 
    \begin{align*}
        \Phi_g(\x \times \bm a) \sim \Psi^{\alphas \to \alphas'}_{\betas \to \betas'}(\x) \times \bm d,
    \end{align*}
    where the curves $\alphas'$ are obtained by a small Hamiltonian perturbation of $\alphas$; the curves $\betas$ and $\betas'$ are symmetric.

    Outside of $\Sigma_0$, we assume that $\varphi$ is isotopic to the identity, taking $\alphas'$ back to the curves $\alphas$. Hence, 
    \begin{align*}
        \Phi_{h}(\x \times \bm d) := \varphi_*(\x \times \bm d)  = (\varphi|_{\Sigma})_*(\x) \times \{a_1\, e\, \t_0(e)\},
    \end{align*}
    where $e$ and $\t_0(e)$ are the symmetric intersection points in $\cH_5 = \cH_1 \# E_{\Z/2}$. 
    Finally, we compute the destabilization map:
    \begin{align*}
        \Phi_i(\x \times \{a_1\, e\, \t_0(e)\}) = (\sigma^{\Z/2}_{\cH_1 \to \cH_5})^{-1}(\x \times \{a_1\, e\, \t_0(e)\})= \x \times a_1.
    \end{align*}
    Putting everything together, we have
    \begin{align*}  \Phi_{i}\circ\Phi_{h}\circ\Phi_{g}\circ\Phi_{f}\circ\Phi_{e}(\x \times a_1) \sim ((\varphi|_{\Sigma})_*\circ \Phi^{\alphas\to \alphas'}_{\betas\to \betas'})(\x) \times a_1.
    \end{align*}
    According to \Cref{prop:continuity}, we have 
    \begin{align*}
        (\varphi|_{\Sigma})_*\circ \Phi^{\alphas\to \alphas'}_{\betas\to \betas'} \sim \id_{\CFR^\circ(\cH)}.
    \end{align*}
    Hence, 
    \begin{align*}
        \Phi_{i}\circ\Phi_{h}\circ\Phi_{g}\circ\Phi_{f}\circ\Phi_{e} \sim \id_{\CFR^\circ(\cH)}\otimes \id_{\CFR^\circ(E_{\{1\}})},
    \end{align*} 
    completing the proof.
\end{proof}

\subsection{Real handleswap loops}

Finally, we turn to real handleswaps. 

Consider the quadruple diagram $\cQ_L = (\Sigma_0, \bm \alpha'_0, \bm \alpha_0, \bm \beta_0, \bm \beta'_0)$ shown in the upper left of \Cref{fig:handleswap_quadrule_diagrams}; let $\cQ_{R}$ be the diagram obtained from $\cQ_L$ by swapping the roles of alpha and beta. If $\cQ = (\Sigma, \bm \alpha', \bm \alpha, \bm \beta, \bm \beta',\t)$ is any real quadruple diagram, we consider the quadruple
\begin{align*}
    \cQ_L\#\cQ\#\cQ_R
\end{align*}
by taking a connected sum at points $p \neq \t(p) \in \Sigma$ and points $p_L$ and $p_R$ in $\Sigma_0$. We extend the involution by swapping the two $\Sigma_0$ summands. Consider the diagram $\cQ_L$, and let $\bT_{\a_0} \cap \bT_{\b_0} = \bm a$ and $\bT_{\a_0'} \cap \bT_{\b_0'} = \bm a'$ and 
\begin{align*}
    \bT_{\alpha_0}\cap \bT_{\alpha_0'} =\{\theta_1^\pm\theta_2^\pm\}.
\end{align*}
Let $\Theta = \{\theta_1^+\theta_2^+\}$ be the top generator. Since $\cQ_L$ represents $\#^2 (S^1 \times S^2)$, we have that $\dim \CFh(\cQ_L) = \dim \HFh(\#^2 (S^1 \times S^2))$, and hence the differential of $\dim \CFh(\cQ_L)$ must vanish.

\begin{figure}[h]
\def\svgwidth{.8\linewidth}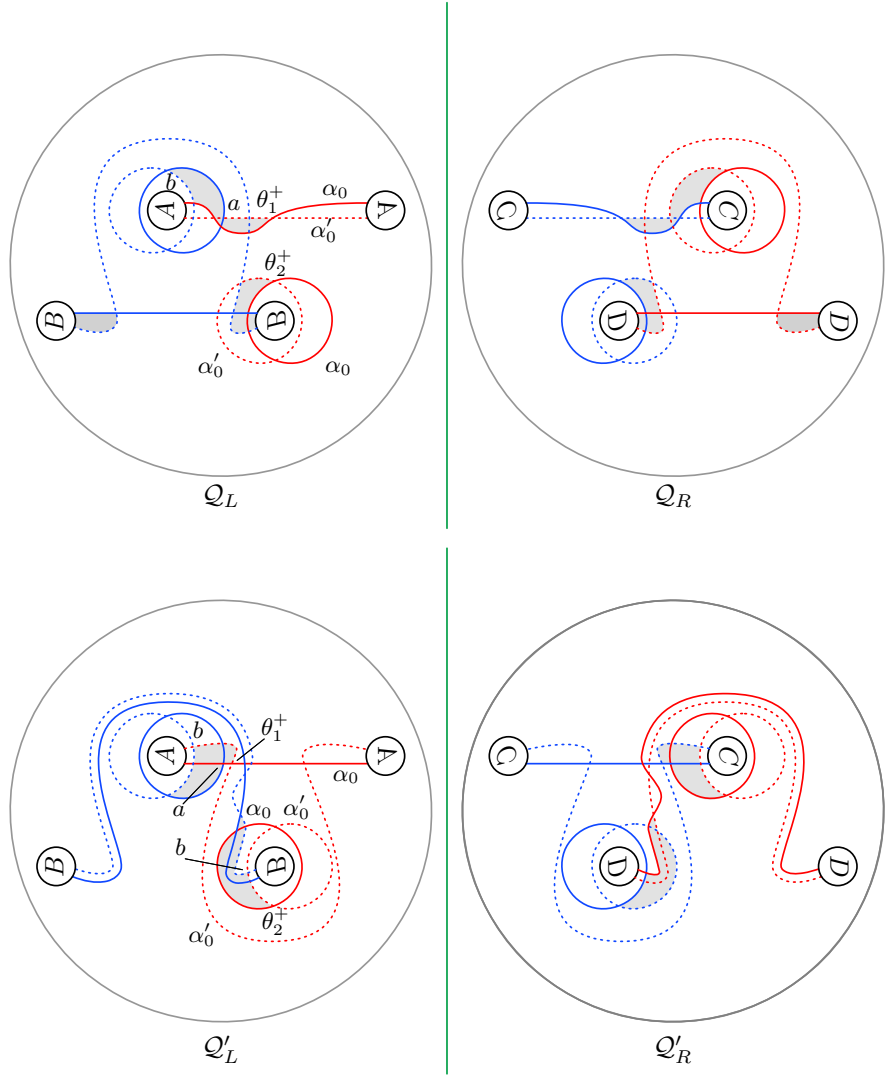
    \caption{The quadruple diagrams used for computing the real handleslides in a real handleswap loop.}
\label{fig:handleswap_quadrule_diagrams}
\end{figure}

\begin{prop}\label{prop:real-handleside-loop-rectangles}
    Suppose that $\cQ = (\Sigma, \bm \alpha', \bm \alpha, \bm \beta, \bm \beta',\t)$ is an admissible real Heegaard quadruple diagram and let $\cQ_L\#\cQ \# \cQ_R$ be the connected sum where $\cQ_L$ and $\cQ_R$ are the quadruple diagrams shown in the top of \Cref{fig:handleswap_quadrule_diagrams}. For a sufficiently stretched real-invariant almost complex structure, the holomorphic rectangle map $F_{\cQ_L\#\cQ \# \cQ_R}$ is given by 
    \begin{align*}
        F_{\cQ_L\#\cQ \# \cQ_R}((\Theta\times \x \times \t(\Theta)) \otimes (\bm a\times \y\times \t(\bm a)) \otimes (\t(\Theta)\times \t(\x) \times \Theta)) \sim \bm a'\times F_{\cQ}( \x \otimes \y \otimes \t(\x)) \times \t(\bm a') 
    \end{align*}
    for every $\x \in \bT_{\alpha'} \cap \Ta$ and $\y \in (\Ta \cap \Tb)^R$. 
\end{prop}
\begin{proof}
    The proof follows by an equivariant neck-stretching argument. Since $\cQ_L$ and $\cQ_R$ are disjoint, the arguments in \cite[Proposition 9.31]{JTZ_naturality_mapping_class_groups} apply here. Indeed, pairs of real-invariant rectangles behave exactly like unreal triangles.
\end{proof}

Let 
\begin{align*}
\begin{tikzcd}[ampersand replacement = \&]
    H_1 \ar[rr,"e"] \& \& H_2\ar[ld,"f"] \\
    \& H_3 \ar[ul,"g"] \& 
\end{tikzcd}
\end{align*}
be a simple real handleswap as in \Cref{def:simple-handleswap}. Recall that edges $e$ and $f$ are real equivalences and $g$ is an equivariant diffeomorphism. The real equivalences are given by the real handleslides appearing in \Cref{fig:handleswap_quadrule_diagrams}, and therefore the map can be computed as the composition of two real handleslide maps. To prove invariance under a simple real handleswap, we establish the following.

\begin{thm}\label{thm:real-handleswap-loop-is-id}
    The induced maps $\Phi_{\ast}$ for $\ast \in \{e,f,g\}$ satisfy
    \begin{align*}
        \Phi_{g}\circ\Phi_{f}\circ\Phi_{e} \sim \id_{\CFR^\circ(H_1)}.
    \end{align*}
\end{thm}
\begin{proof}
    This follows from \Cref{prop:real-handleside-loop-rectangles}. The arguments are the same as in \cite[Theorem 9.30]{JTZ_naturality_mapping_class_groups} in the unreal setting.
\end{proof}

\begin{proof}[Proof of \Cref{thm:naturality}]
    Let $\CFR^\circ$ be any  of the weak real Heegaard Floer theories from Theorems~\ref{thm:HF-weak}, \ref{thm:HFL-weak}, or \ref{thm:RSFH-weak}. Axioms (1)-(3) were verified in \Cref{sub:HFR-strong} and Axioms (4) and (5) were verified in \Cref{thm:trade-loop-is-id,thm:real-handleswap-loop-is-id}.
\end{proof}

\section{Involutive real Heegaard Floer homology}
\label{sec:HFRI}

Heegaard Floer homology has an intrinsic symmetry, given by $\SpinC$-conjugation. In \cite{hendricks_manolescu_Invol}, Hendricks and the second author used this symmetry to define  \emph{involutive Heegaard Floer homology}. We review their  construction. 

For simplicity, fix $(Y, \frs)$, where $\frs$ is a self-conjugate $\SpinC$-structure, i.e.,  $\overline{\frs} = \frs$. Choose a pointed Heegaard diagram $\cH = (\S, \alpha, \beta, z)$ for $Y$. There is a canonical isomorphism 
\begin{align*}
    \eta: \CF^\circ(\cH, \frs) \ra \CF^\circ(\overline{\cH}, \frs),
\end{align*}
where $\overline{\cH}$ is the diagram $(-\S, \betas, \alphas, z)$. Since $\cH$ and $\overline{\cH}$ both represent $Y$, a choice of Heegaard moves connecting the two diagrams gives a homotopy equivalence
\begin{align*}
    \Phi_{\overline{\cH} \ra \cH}: \CF^\circ(\overline{\cH}, \frs) \ra \CF^\circ(\cH, \frs), 
\end{align*}
which is well-defined by naturality. Let
\begin{align*}
    \iota:= \Phi_{\overline{\cH} \ra \cH} \circ \eta: \CF^\circ(\cH, \frs) \to \CF^\circ(\cH, \frs).
\end{align*}
The involutive Heegaard Floer homology $\HFI(Y, \frs)$ is defined to be 
\begin{align*}
    H_*(\Cone(\CF^\circ(\cH, \frs) \xra{Q(1 + \iota)}\CF^\circ(\cH, \frs))[-1]),
\end{align*}
where $Q$ is a formal variable and $[-1]$ denotes a degree shift. 

The same construction can be carried out in the real context. Given a real Heegaard diagram $(\cH, \t) = (\S, \alpha, \beta, z, \t)$ for $(Y, \t)$ and a self-conjugate real $\SpinC$-structure $\frs^R$, there is, again, a canonical map
\begin{align*}
    \eta: \CFR^\circ(\cH, \t, \frs^R) \ra \CFR^\circ(\overline{\cH}, \t, \frs^R).
\end{align*}
The diagrams $(\cH, \t)$ and $(\overline{\cH},\t)$ both represent $(Y, \t)$, so a sequence of real Heegaard moves between them induces a homotopy equivalence 
\begin{align*}
    \Phi_{(\overline{\cH}, \t) \ra (\cH, \t)}: \CFR^\circ(\overline{\cH},\t, \frs^R) \ra \CFR^\circ(\cH,\t, \frs^R).  
\end{align*}
 This gives rise to a homotopy involution 
\begin{align*}
    \iota :=  \Phi_{(\overline{\cH}, \t) \ra (\cH, \t)} \circ \eta: \CFR^\circ(\cH, \t, \frs^R) \to \CFR^\circ(\cH, \t, \frs^R).
\end{align*}
We define $\HFRI^\circ(Y,\t,\frs^R)$, the \emph{involutive real Heegaard Floer homology of $(Y, \t, \frs^R)$}, to be the homology of the mapping cone
\begin{align*}
    \Cone(\CFR^\circ(\cH,\t, \frs^R) \xra{Q(1 + \iota)}\CFR^\circ(\cH, \t,\frs^R))[-1],
\end{align*}
over the ring $\F[Q, v]/(Q^2)$ with $v$ and $Q$ both in degree $-1$.

\begin{proof}[Proof of Theorem~\ref{thm:HFRI}]
This follows just as in the proof of \cite[Theorem 1.1]{hendricks_manolescu_Invol}, making use of the fact that $\Phi_{(\overline{\cH}, \t) \ra (\cH, \t)}$ is well defined up to chain homotopy. This last fact is naturality at the chain level, which was established in the proof of Theorem~\ref{thm:real-nat}.
\end{proof}

\begin{proposition} \label{propn:exact} For $\circ \in \{-, \infty, +, \hat{\phantom{o}} \}$, there is an exact triangle of $\F[v]$-linear maps:
\begin{equation}
\label{eq:Iexact}
\begin{tikzpicture}[baseline=(current  bounding  box.center)]
\node(1)at(0,0){$\HFRI^\circ(Y,\tau,\s^R)$.};
\node(2)at (-3,1){$\HFR^\circ(Y, \tau,\s^R)$};
\node(3)at (3,1){$Q \cdot \HFR^\circ(Y,\tau, \s^R)[-1]$};
\path[->](2)edge node[above]{$Q(1+ \iota_*)$}(3);
\path[->](3)edge (1);
\path[->](1)edge(2);
\end{tikzpicture}
\end{equation}
\end{proposition}

\begin{proof}
This is clear from the definition. Compare \cite[Proposition 4.6]{hendricks_manolescu_Invol}.
\end{proof}

\begin{example} 
\label{ex:Lspace}
    Suppose $(Y, \tau)$ is the double branched cover over a knot $K \subset S^3$ and that $Y$ is an $L$-space (with $\F$-coefficients). The latter condition means that  $\widehat{\HF}(Y, \frs; \F) \cong \F$ for every $\SpinC$
structure $\frs$. Then, by the work of Hendricks \cite[Theorem 1.2]{hendricks:real}, we also have $\widehat{\HFR}(Y, \tau, \frs^R) \cong \F$, for every real $\SpinC$ structure $\frs^R$. If $\frs^R$ is self-conjugate, we obtain
\begin{align*}
    \HFRI^-(Y,\t, \frs^R) & \cong \F[Q, v]/(Q^2),\\
    \widehat{\HFRI} (Y,\t, \frs^R) &\cong \F[Q]/(Q^2). 
\end{align*}
The proof follows that of \cite[Corollary 4.8]{hendricks_manolescu_Invol}, making use of the exact triangle \eqref{eq:Iexact}. To show that the extension of $\HFR^\circ$ by $\HFR^\circ$ splits, we observe that these are free modules over the respective base rings. 
\end{example}

\begin{example} (This was suggested to us by Kristen Hendricks.) Let $Y$ be the Brieskorn sphere $\Sigma(2,3,7)$,
viewed as the double cover of $S^3$ branched along the torus knot $T(3,7)$, and let $\tau$ be the deck transformation. The real manifold $(Y, \tau)$  admits a unique real $\SpinC$ structure, which we drop from the notation. By \cite[Theorem 1.2]{hendricks:real}, there is a spectral sequence whose $E_1$ page is $\widehat{\HF}(Y) \otimes \F[\theta, \theta^{-1}]$, which converges to $\widehat{\HFR}(Y, \tau) \otimes \F[\theta, \theta^{-1}]$, and whose $d_1$ differential is $(1+\iota_*\tau_*)\theta$. By the calculations in \cite[Equation (27)]{OS:AbsGrading},  \cite[Proposition 6.26]{HLS_flexible}, and \cite[Section (6.8)]{hendricks_manolescu_Invol}, we have that $\dim \widehat{\HF}(Y; \F)=3$, with $\tau_*$ being the identity and $\iota_*$ exchanging two generators, so that $1+\iota_*\tau_*$ has rank $1$. From here it follows that 
$$\dim \widehat{\HFR}(Y, \tau) =1.$$
The involution $\iota_*$ must act as the identity on $\widehat{\HFR}(Y, \tau)$. Using the exact triangle \eqref{eq:Iexact}, we deduce that $\widehat{\HFRI}(Y, \tau)$ has dimension $2$. On the other hand, from \cite[Section (6.8)]{hendricks_manolescu_Invol} we see that $\dim \widehat{\HFI}(Y)=4$. Thus, unlike in Example~\ref{ex:Lspace}, in this case we have  
$$  \dim\widehat{\HFRI}(Y, \tau) < \dim \widehat{\HFI}(Y).$$
\end{example}

\clearpage

\appendix 
\section{Resolutions of codimension-2 bifurcations}\label{appendix}

\begin{figure}[h]
    \def\svgwidth{.8\linewidth}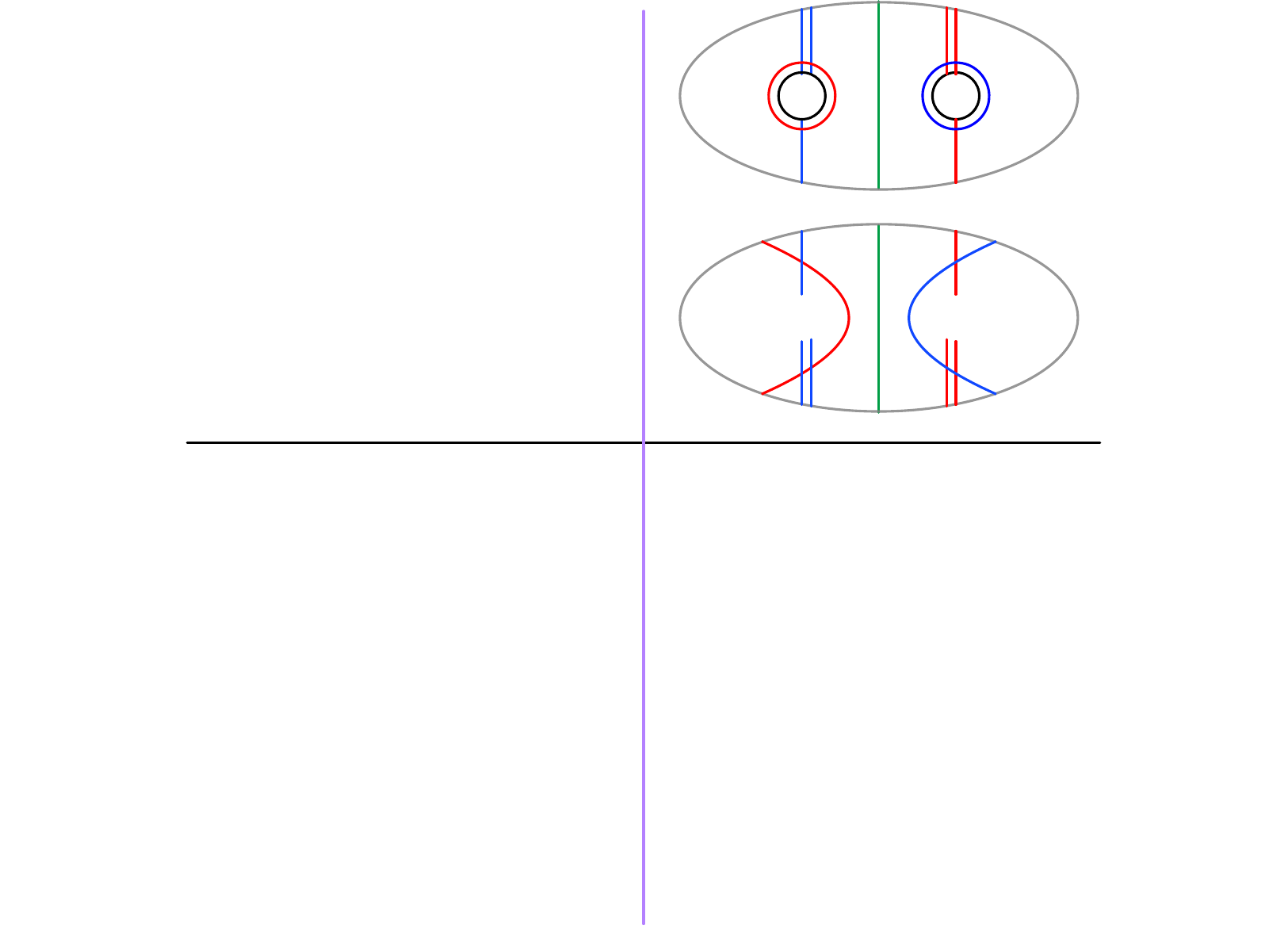
        \caption{Resolving an (A4b) bifurcation.}
    \label{fig:hd_a4b_resolved}
    \end{figure}

\begin{figure}[h]
    \def\svgwidth{.8\linewidth}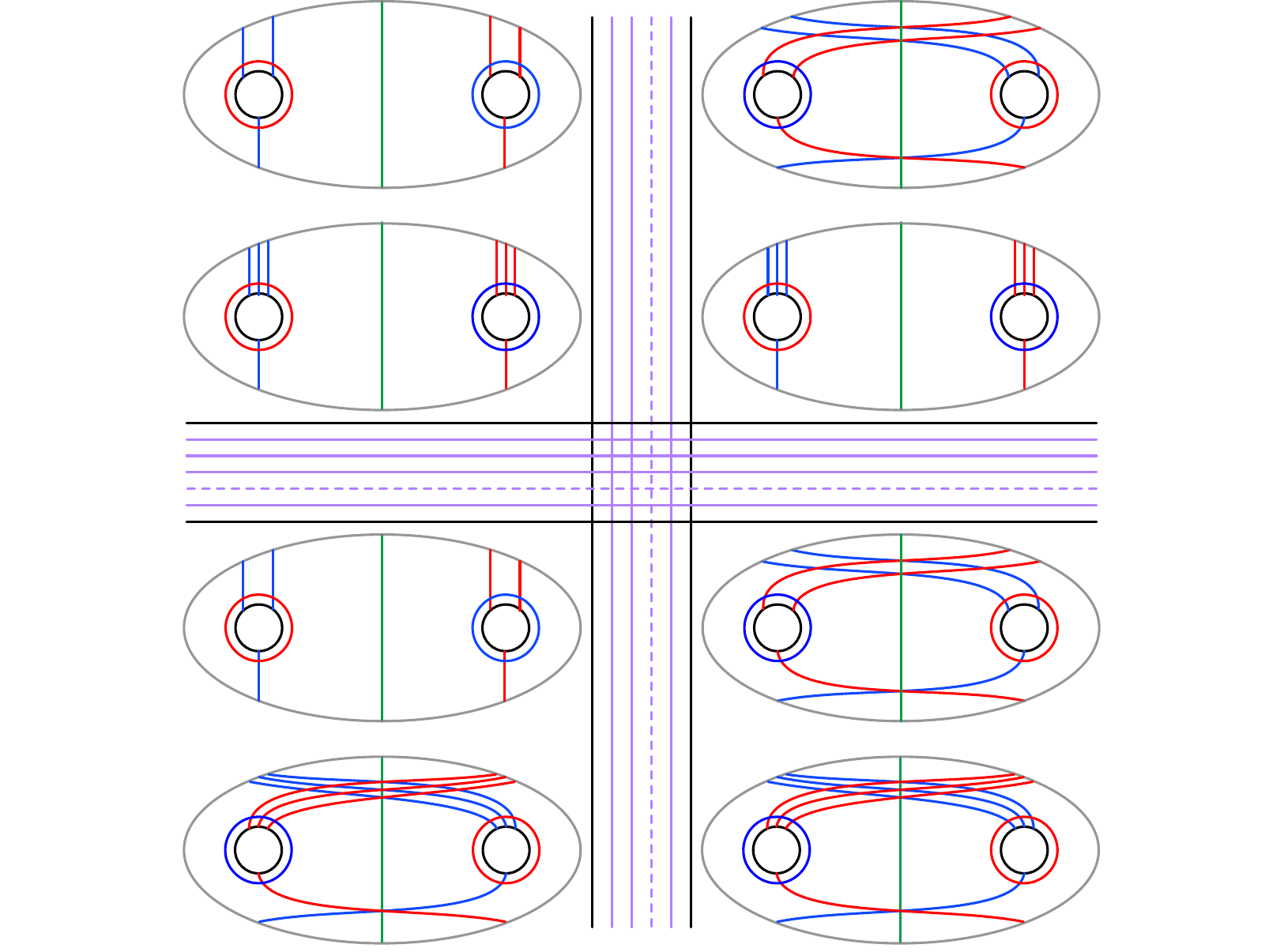
        \caption{Resolving an (A3) bifurcation.}
    \label{fig:hd_a3_resolved}
\end{figure}

\begin{figure}[h]
    \def\svgwidth{.8\linewidth}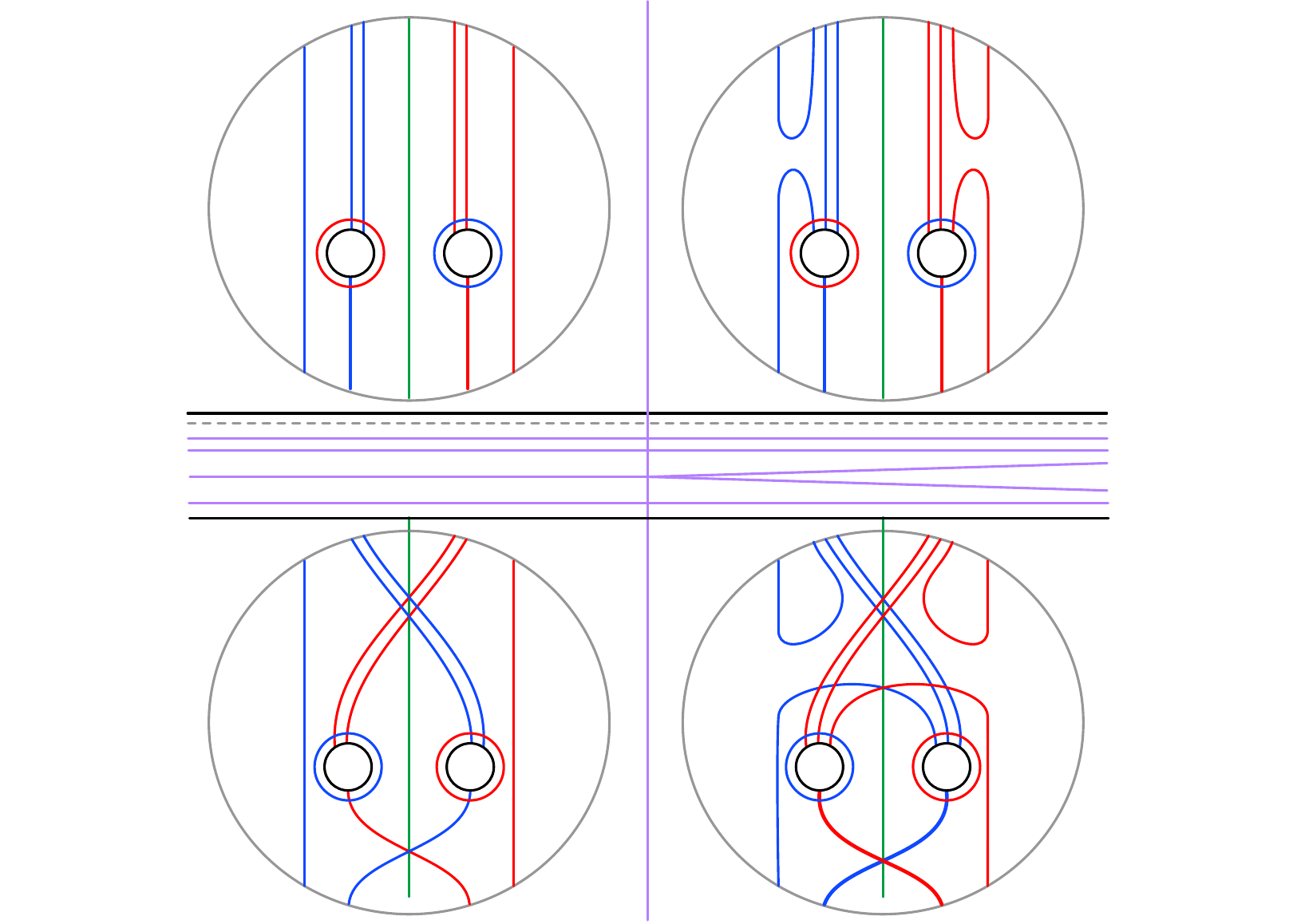
        \caption{Resolving an (A4a) bifurcation.}
    \label{fig:hd_a4a_resolved}
\end{figure}
\begin{figure}[h]
    \def\svgwidth{.8\linewidth}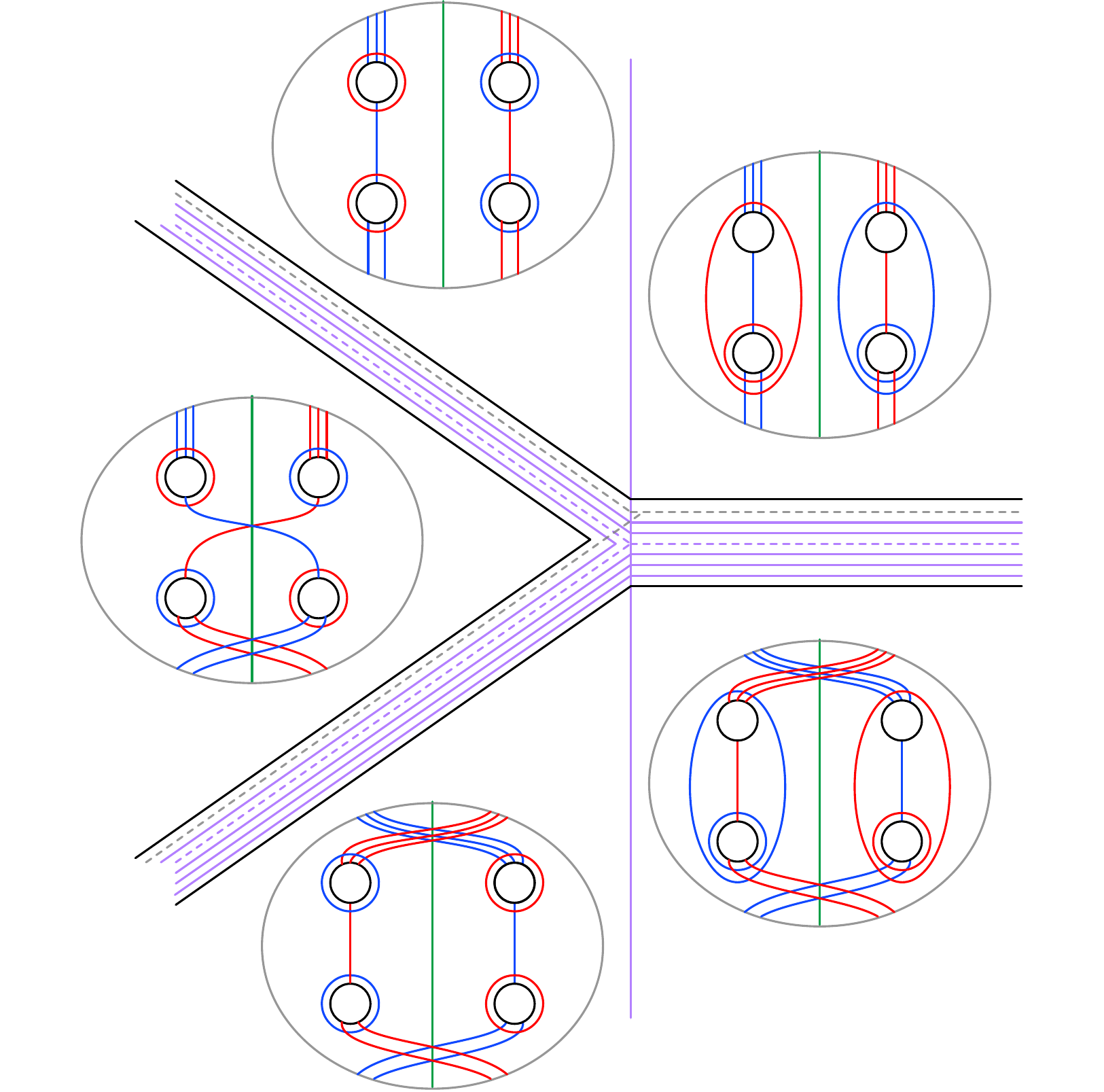
        \caption{Resolving a simple (A4c) bifurcation.}
    \label{fig:hd_a4c_resolved}
\end{figure}
\begin{figure}[h]
    \def\svgwidth{.8\linewidth}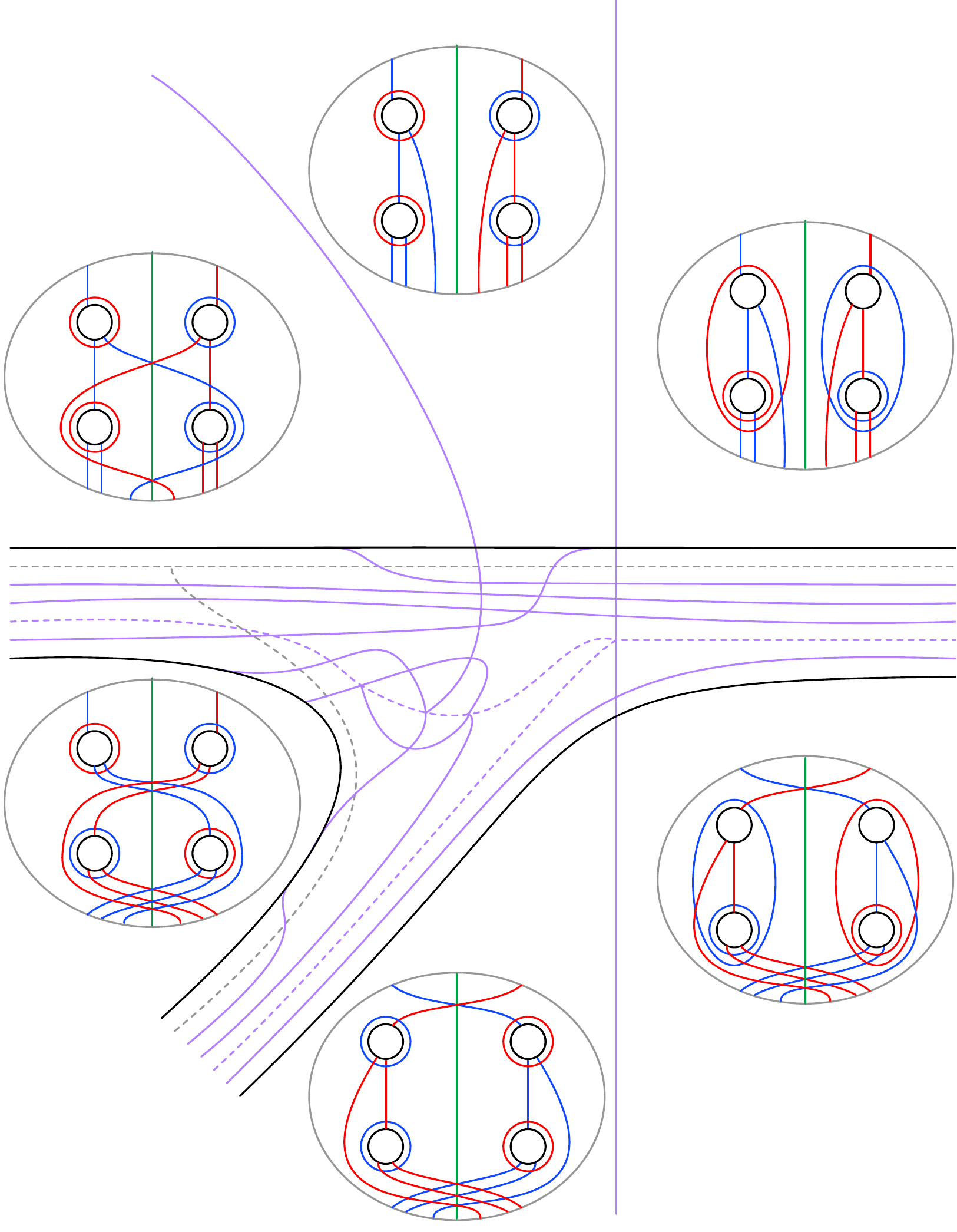
        \caption{Resolving a more complicated (A4c) bifurcation.}
    \label{fig:hd_a4c_resolved2}
\end{figure}
\begin{figure}[h]
    \def\svgwidth{.8\linewidth}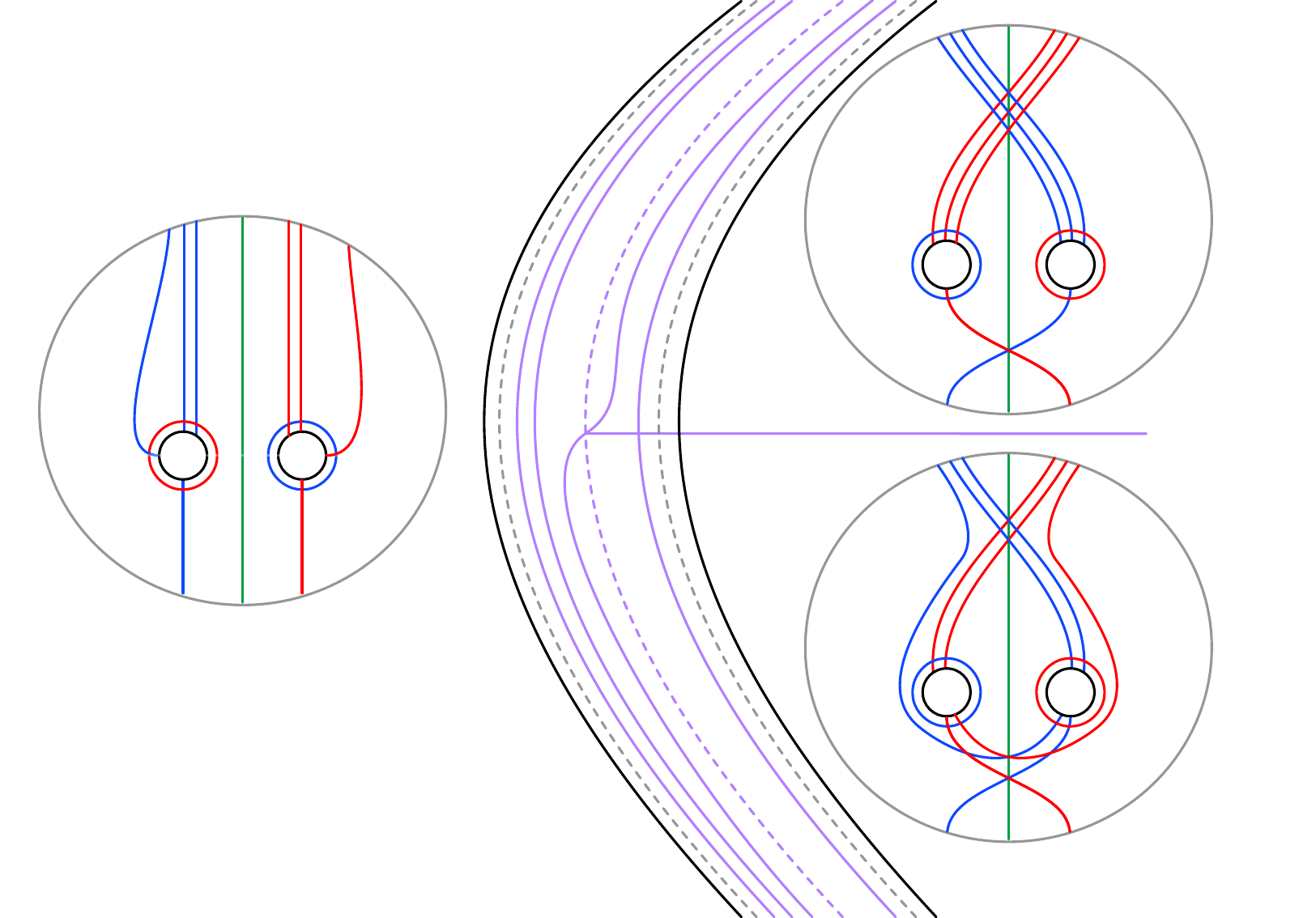
        \caption{Resolving an (E2) bifurcation.}
    \label{fig:hd_e2_resolved}
\end{figure}
\begin{figure}[h]
    \def\svgwidth{.8\linewidth}
\begingroup%
  \makeatletter%
  \providecommand\color[2][]{%
    \errmessage{(Inkscape) Color is used for the text in Inkscape, but the package 'color.sty' is not loaded}%
    \renewcommand\color[2][]{}%
  }%
  \providecommand\transparent[1]{%
    \errmessage{(Inkscape) Transparency is used (non-zero) for the text in Inkscape, but the package 'transparent.sty' is not loaded}%
    \renewcommand\transparent[1]{}%
  }%
  \providecommand\rotatebox[2]{#2}%
  \newcommand*\fsize{\dimexpr\f@size pt\relax}%
  \newcommand*\lineheight[1]{\fontsize{\fsize}{#1\fsize}\selectfont}%
  \ifx\svgwidth\undefined%
    \setlength{\unitlength}{779.52755906bp}%
    \ifx\svgscale\undefined%
      \relax%
    \else%
      \setlength{\unitlength}{\unitlength * \real{\svgscale}}%
    \fi%
  \else%
    \setlength{\unitlength}{\svgwidth}%
  \fi%
  \global\let\svgwidth\undefined%
  \global\let\svgscale\undefined%
  \makeatother%
  \begin{picture}(1,0.72)%
    \lineheight{1}%
    \setlength\tabcolsep{0pt}%
    \put(0,0){\includegraphics[width=\unitlength,page=1]{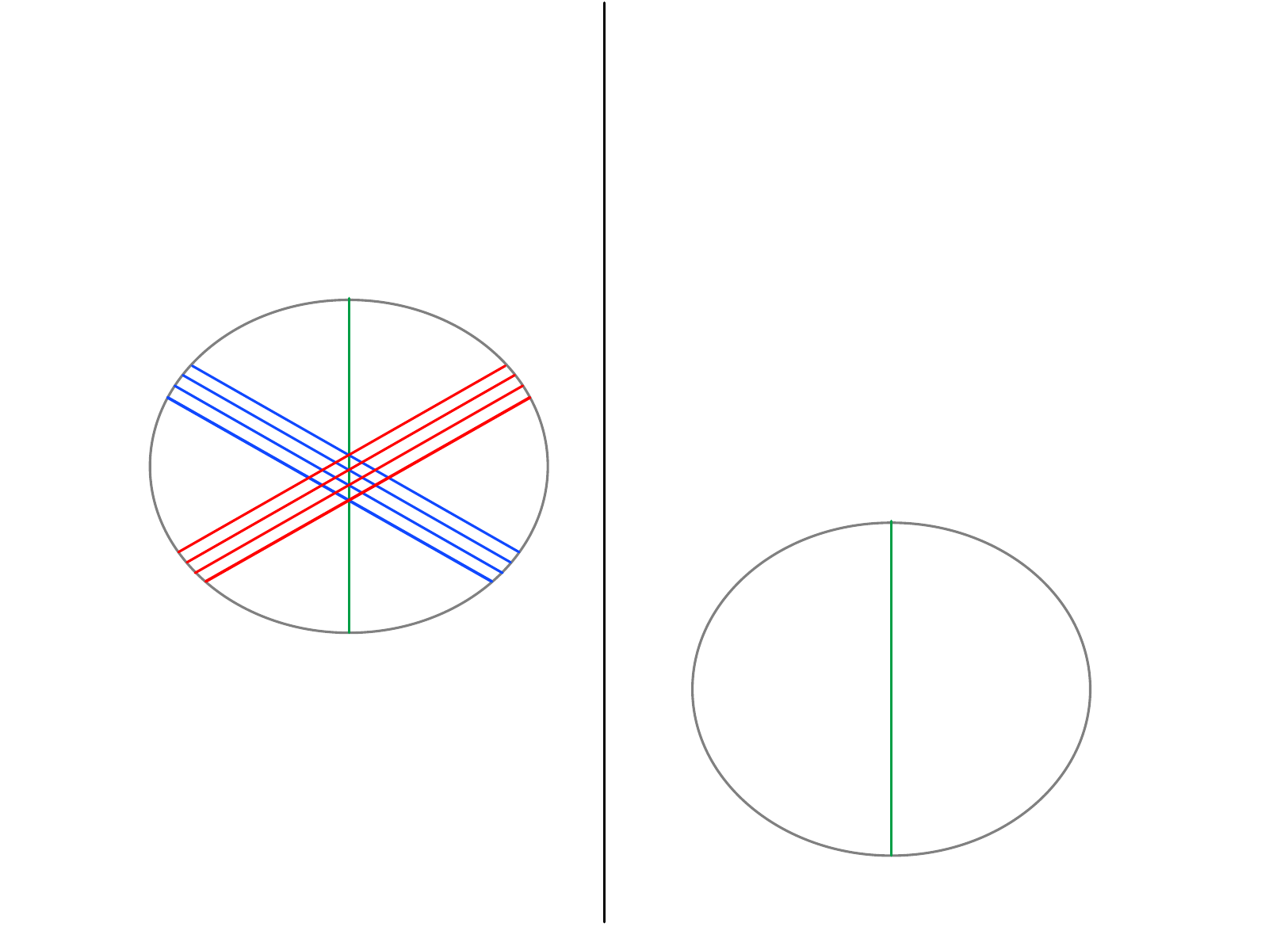}}%
    \put(0.6191155,0.53067368){\color[rgb]{0,0,0}\makebox(0,0)[t]{\smash{\begin{tabular}[t]{c}{$F$}\end{tabular}}}}%
    \put(0,0){\includegraphics[width=\unitlength,page=2]{hd_b3_triv_resolved.pdf}}%
    \put(0.76557202,0.55147016){\color[rgb]{0,0,0}\rotatebox{-180}{\makebox(0,0)[t]{\smash{\begin{tabular}[t]{c}{$\reflectbox{$F$}$}\end{tabular}}}}}%
    \put(0,0){\includegraphics[width=\unitlength,page=3]{hd_b3_triv_resolved.pdf}}%
    \put(0.6191155,0.16506753){\color[rgb]{0,0,0}\makebox(0,0)[t]{\smash{\begin{tabular}[t]{c}{$F$}\end{tabular}}}}%
    \put(0,0){\includegraphics[width=\unitlength,page=4]{hd_b3_triv_resolved.pdf}}%
    \put(0.76557202,0.18586404){\color[rgb]{0,0,0}\rotatebox{-180}{\makebox(0,0)[t]{\smash{\begin{tabular}[t]{c}{$\reflectbox{$F$}$}\end{tabular}}}}}%
    \put(0,0){\includegraphics[width=\unitlength,page=5]{hd_b3_triv_resolved.pdf}}%
  \end{picture}%
\endgroup%

        \caption{Resolving a $\{1\}$-orbit (B3) bifurcation.}
    \label{fig:hd_b3_triv_resolved}
\end{figure}
\begin{figure}[h]
    \def\svgwidth{.8\linewidth}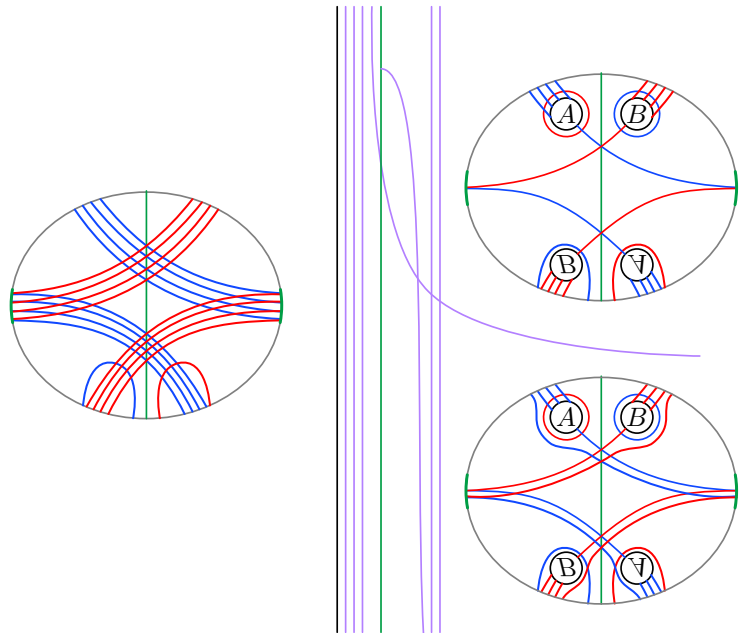
        \caption{Resolving a $\Z/2$-orbit (B3) bifurcation.}
    \label{fig:hd_b3_Z2_resolved}
\end{figure}
\begin{figure}[h]
    \def\svgwidth{.8\linewidth}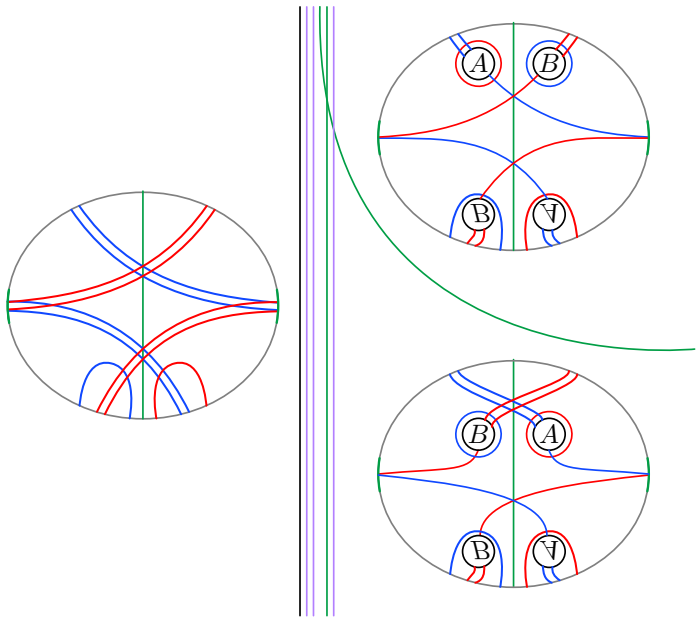
        \caption{Resolving a (B5) bifurcation.}
    \label{fig:hd_b5_resolved}
\end{figure}
\begin{figure}[h]
    \def\svgwidth{.8\linewidth}
\begingroup%
  \makeatletter%
  \providecommand\color[2][]{%
    \errmessage{(Inkscape) Color is used for the text in Inkscape, but the package 'color.sty' is not loaded}%
    \renewcommand\color[2][]{}%
  }%
  \providecommand\transparent[1]{%
    \errmessage{(Inkscape) Transparency is used (non-zero) for the text in Inkscape, but the package 'transparent.sty' is not loaded}%
    \renewcommand\transparent[1]{}%
  }%
  \providecommand\rotatebox[2]{#2}%
  \newcommand*\fsize{\dimexpr\f@size pt\relax}%
  \newcommand*\lineheight[1]{\fontsize{\fsize}{#1\fsize}\selectfont}%
  \ifx\svgwidth\undefined%
    \setlength{\unitlength}{779.52755906bp}%
    \ifx\svgscale\undefined%
      \relax%
    \else%
      \setlength{\unitlength}{\unitlength * \real{\svgscale}}%
    \fi%
  \else%
    \setlength{\unitlength}{\svgwidth}%
  \fi%
  \global\let\svgwidth\undefined%
  \global\let\svgscale\undefined%
  \makeatother%
  \begin{picture}(1,0.88)%
    \lineheight{1}%
    \setlength\tabcolsep{0pt}%
    \put(0,0){\includegraphics[width=\unitlength,page=1]{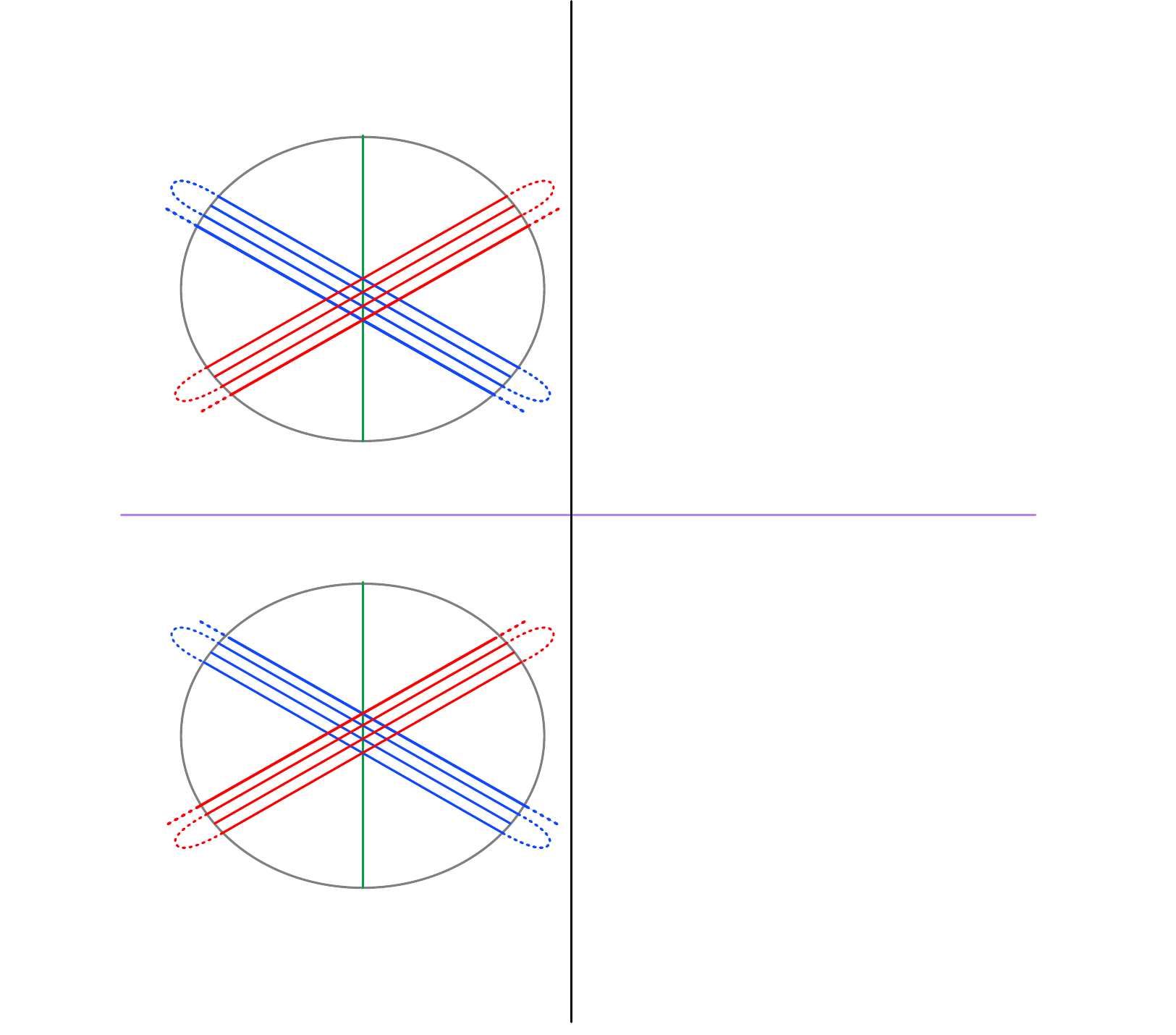}}%
    \put(0.65663276,0.61477603){\color[rgb]{0,0,0}\makebox(0,0)[t]{\smash{\begin{tabular}[t]{c}{$F$}\end{tabular}}}}%
    \put(0,0){\includegraphics[width=\unitlength,page=2]{hd_b1_triv_resolved.pdf}}%
    \put(0.80308928,0.63557251){\color[rgb]{0,0,0}\rotatebox{-180}{\makebox(0,0)[t]{\smash{\begin{tabular}[t]{c}{$\reflectbox{$F$}$}\end{tabular}}}}}%
    \put(0,0){\includegraphics[width=\unitlength,page=3]{hd_b1_triv_resolved.pdf}}%
    \put(0.65661467,0.234738){\color[rgb]{0,0,0}\makebox(0,0)[t]{\smash{\begin{tabular}[t]{c}{$F$}\end{tabular}}}}%
    \put(0.80307119,0.25553445){\color[rgb]{0,0,0}\rotatebox{-180}{\makebox(0,0)[t]{\smash{\begin{tabular}[t]{c}{$\reflectbox{$F$}$}\end{tabular}}}}}%
    \put(0,0){\includegraphics[width=\unitlength,page=4]{hd_b1_triv_resolved.pdf}}%
  \end{picture}%
\endgroup%

        \caption{Resolving a $\{1\}$-orbit (B1) bifurcation.}
    \label{fig:hd_b1_triv_resolved}
\end{figure}
\begin{figure}[h]
    \def\svgwidth{.8\linewidth}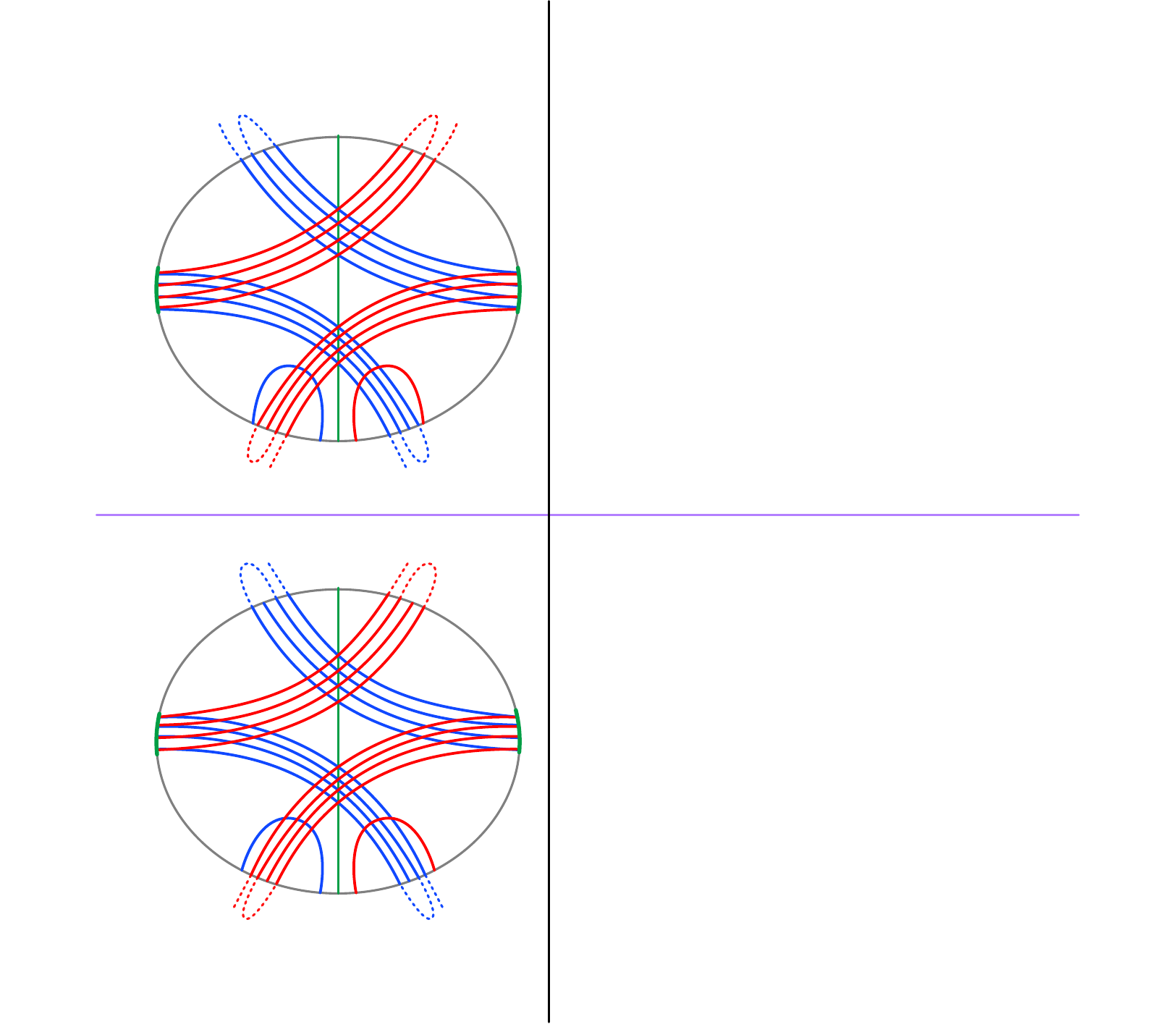
        \caption{Resolving a $\Z/2$-orbit (B1) bifurcation.}
    \label{fig:hd_b1_Z2_resolved}
\end{figure}
\begin{figure}[h]
    \def\svgwidth{.8\linewidth}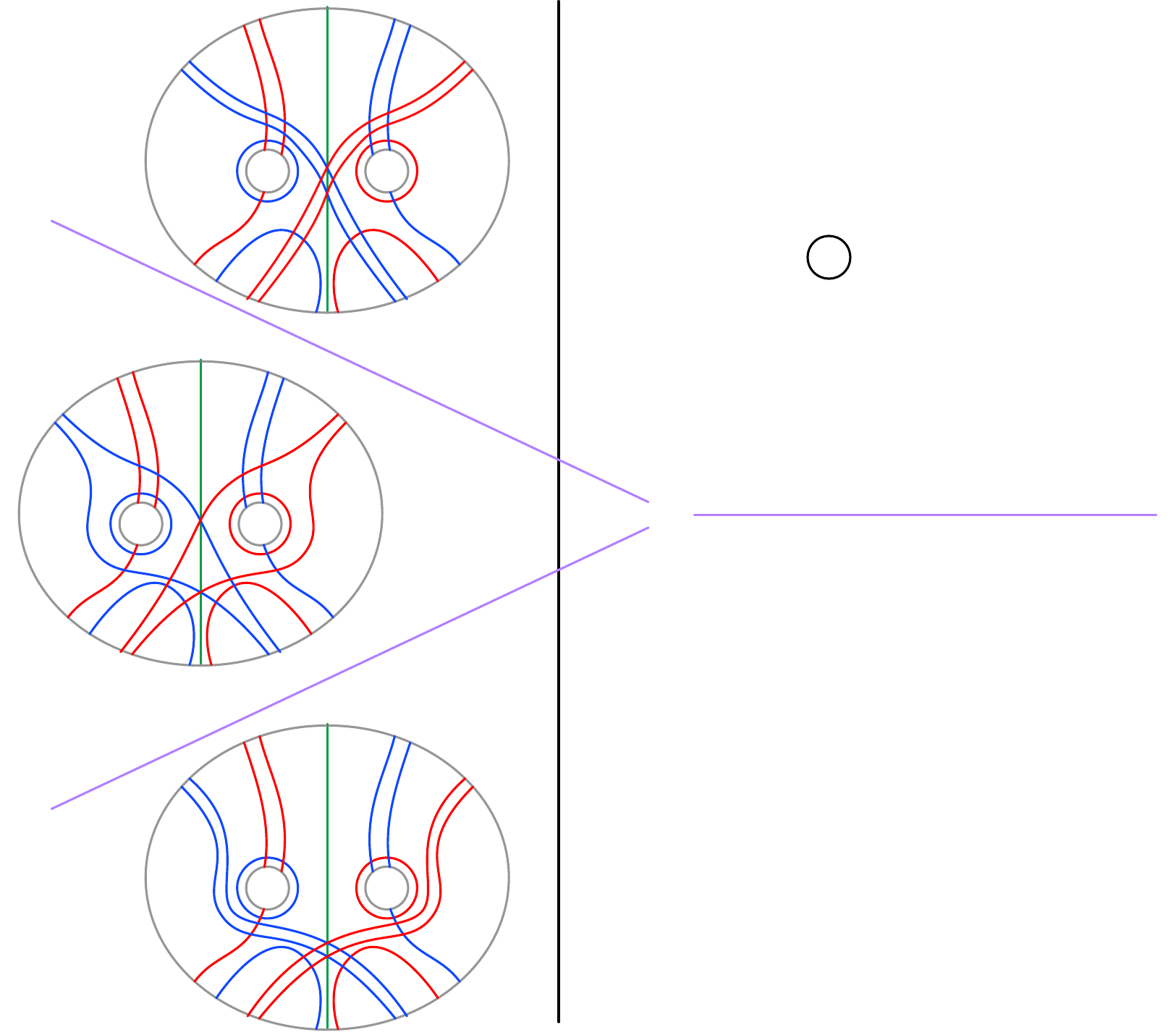
        \caption{Resolving a $\Z/2$-orbit (B2) bifurcation.}
    \label{fig:hd_b2_Z2_resolved}
\end{figure}
\begin{figure}[h]
    \def\svgwidth{.8\linewidth}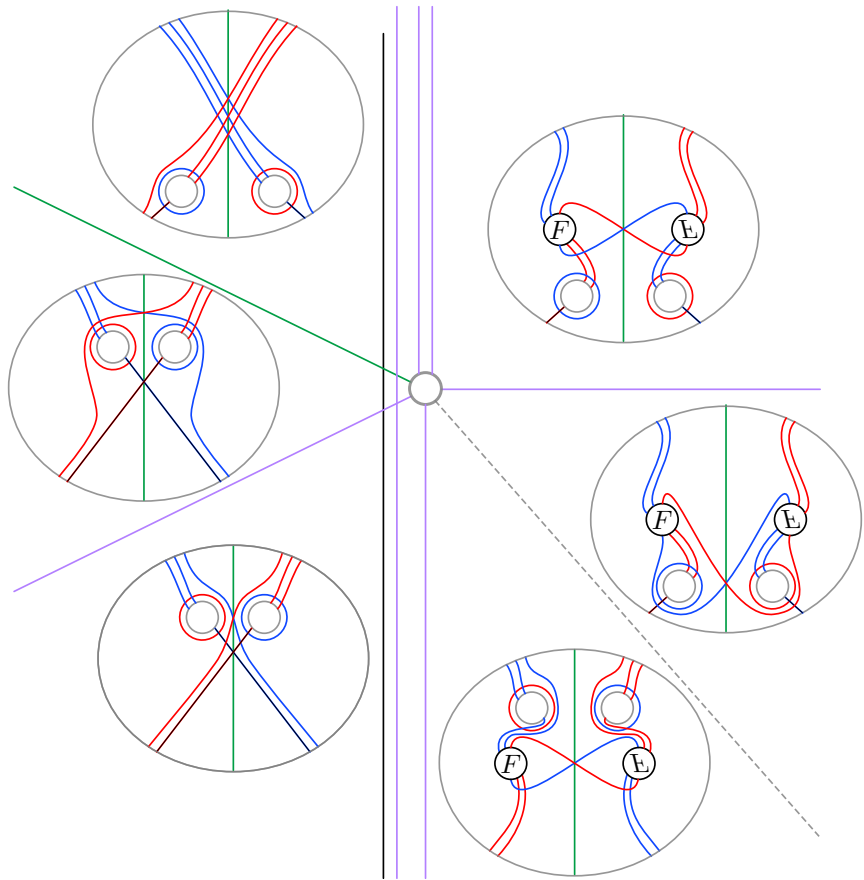
        \caption{Resolving a $\{1\}$-orbit (B2) bifurcation.}
    \label{fig:hd_b2_triv_resolved}
\end{figure}
\begin{figure}[h]
    \def\svgwidth{.8\linewidth}
\begingroup%
  \makeatletter%
  \providecommand\color[2][]{%
    \errmessage{(Inkscape) Color is used for the text in Inkscape, but the package 'color.sty' is not loaded}%
    \renewcommand\color[2][]{}%
  }%
  \providecommand\transparent[1]{%
    \errmessage{(Inkscape) Transparency is used (non-zero) for the text in Inkscape, but the package 'transparent.sty' is not loaded}%
    \renewcommand\transparent[1]{}%
  }%
  \providecommand\rotatebox[2]{#2}%
  \newcommand*\fsize{\dimexpr\f@size pt\relax}%
  \newcommand*\lineheight[1]{\fontsize{\fsize}{#1\fsize}\selectfont}%
  \ifx\svgwidth\undefined%
    \setlength{\unitlength}{779.52755906bp}%
    \ifx\svgscale\undefined%
      \relax%
    \else%
      \setlength{\unitlength}{\unitlength * \real{\svgscale}}%
    \fi%
  \else%
    \setlength{\unitlength}{\svgwidth}%
  \fi%
  \global\let\svgwidth\undefined%
  \global\let\svgscale\undefined%
  \makeatother%
  \begin{picture}(1,0.88)%
    \lineheight{1}%
    \setlength\tabcolsep{0pt}%
    \put(0,0){\includegraphics[width=\unitlength,page=1]{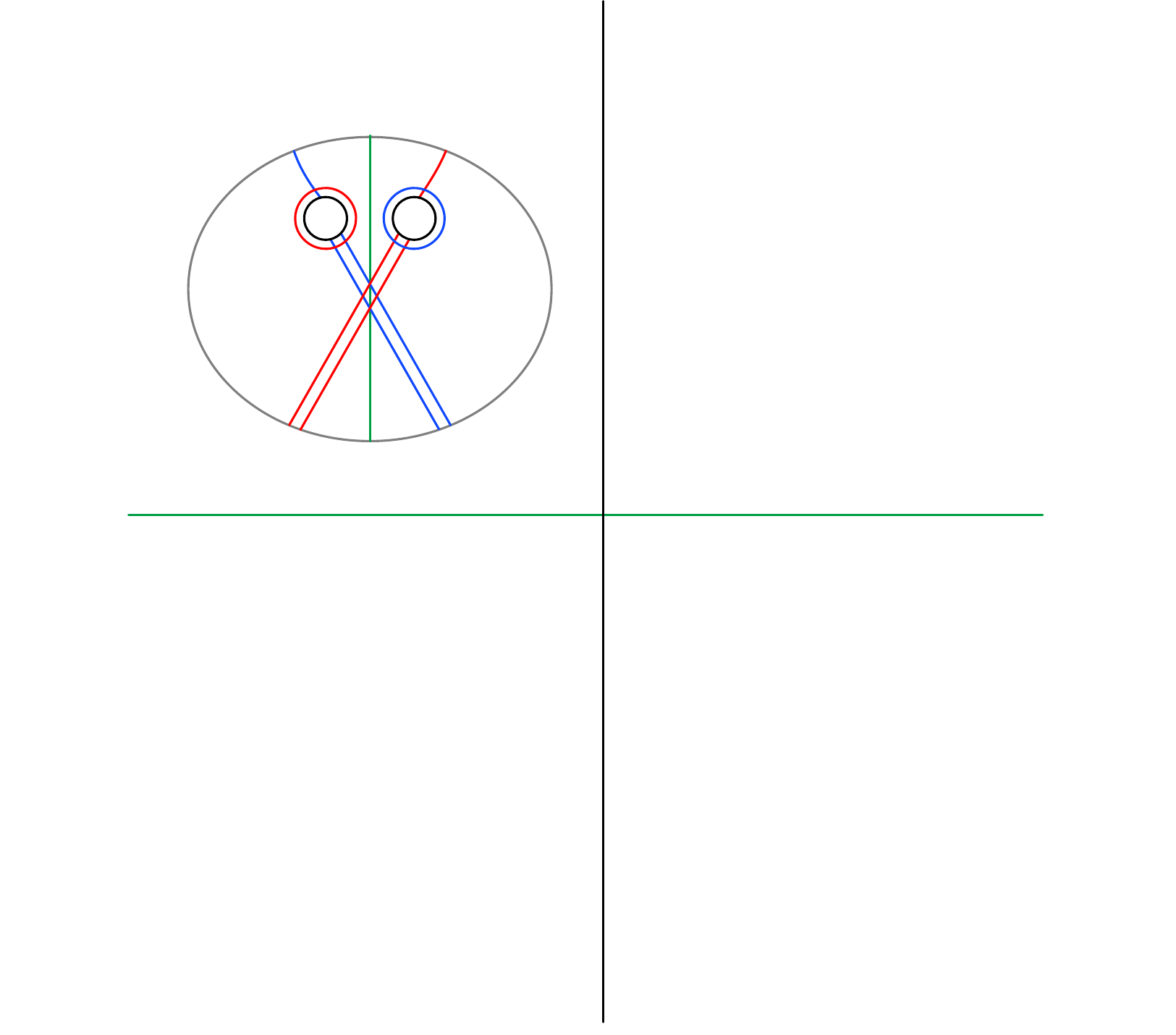}}%
    \put(0.66280995,0.61477603){\color[rgb]{0,0,0}\makebox(0,0)[t]{\smash{\begin{tabular}[t]{c}{$F$}\end{tabular}}}}%
    \put(0,0){\includegraphics[width=\unitlength,page=2]{hd_b4_triv_resolved.pdf}}%
    \put(0.80926647,0.63557251){\color[rgb]{0,0,0}\rotatebox{-180}{\makebox(0,0)[t]{\smash{\begin{tabular}[t]{c}{$\reflectbox{$F$}$}\end{tabular}}}}}%
    \put(0,0){\includegraphics[width=\unitlength,page=3]{hd_b4_triv_resolved.pdf}}%
    \put(0.66280995,0.23377602){\color[rgb]{0,0,0}\makebox(0,0)[t]{\smash{\begin{tabular}[t]{c}{$F$}\end{tabular}}}}%
    \put(0,0){\includegraphics[width=\unitlength,page=4]{hd_b4_triv_resolved.pdf}}%
    \put(0.80926647,0.25457253){\color[rgb]{0,0,0}\rotatebox{-180}{\makebox(0,0)[t]{\smash{\begin{tabular}[t]{c}{$\reflectbox{$F$}$}\end{tabular}}}}}%
    \put(0,0){\includegraphics[width=\unitlength,page=5]{hd_b4_triv_resolved.pdf}}%
  \end{picture}%
\endgroup%

        \caption{Resolving a $\{1\}$-orbit (B4) bifurcation.}
    \label{fig:hd_b4_triv_resolved}
\end{figure}
\begin{figure}[h]
    \def\svgwidth{.8\linewidth}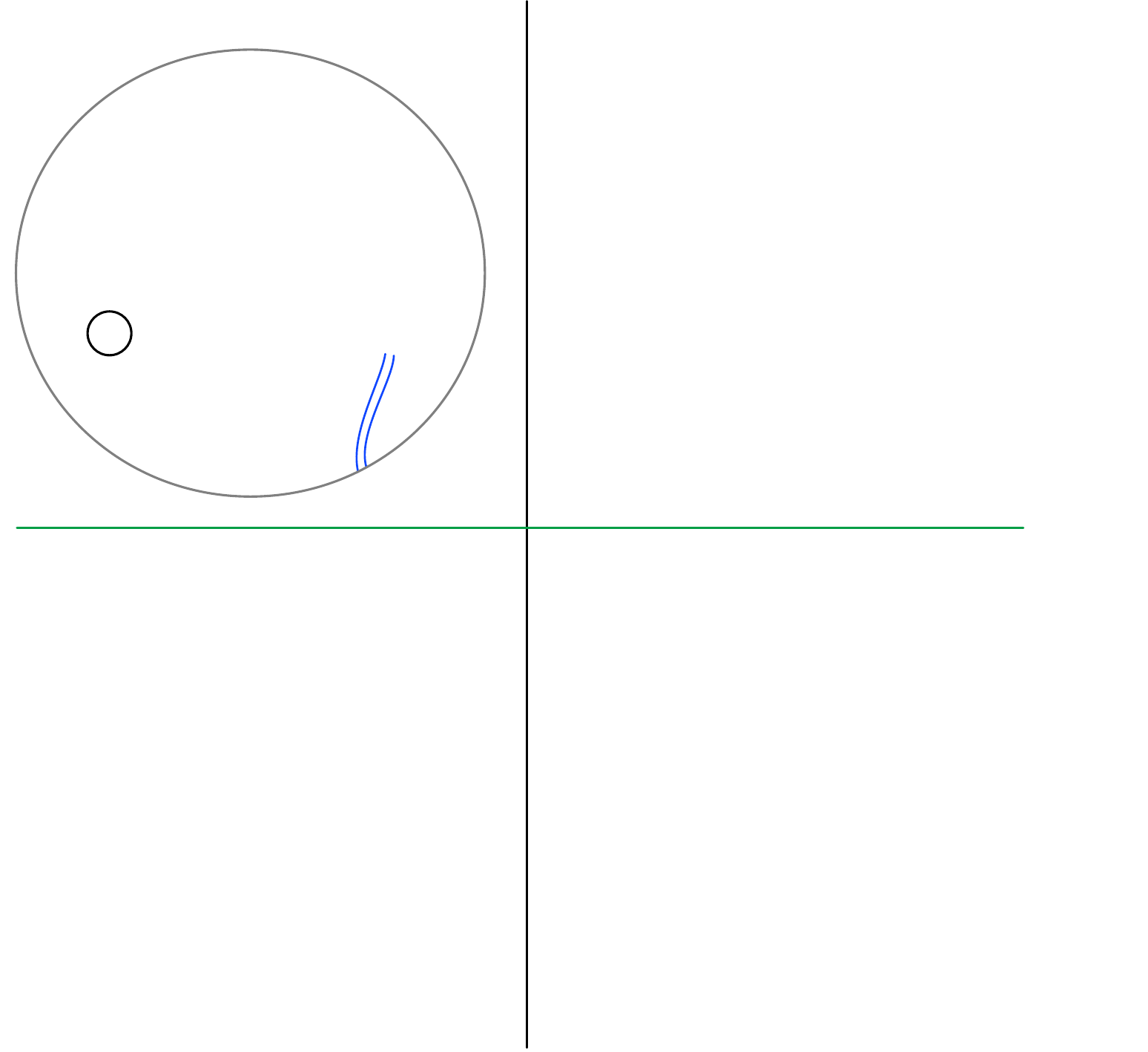
        \caption{Resolving a $\Z/2$-orbit (B4) bifurcation.}
    \label{fig:hd_b4_Z2_resolved}
\end{figure}
\begin{figure}[h]
    \def\svgwidth{.8\linewidth}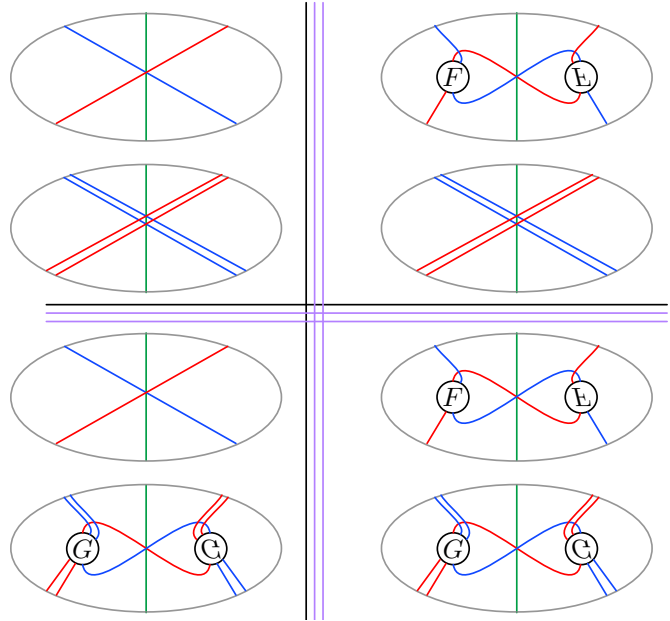
        \caption{Resolving a $\{1\}$-$\{1\}$-orbit bifurcation of type (C).}
    \label{fig:hd_c_1_1_resolved}
\end{figure}
\begin{figure}[h]
    \def\svgwidth{.8\linewidth}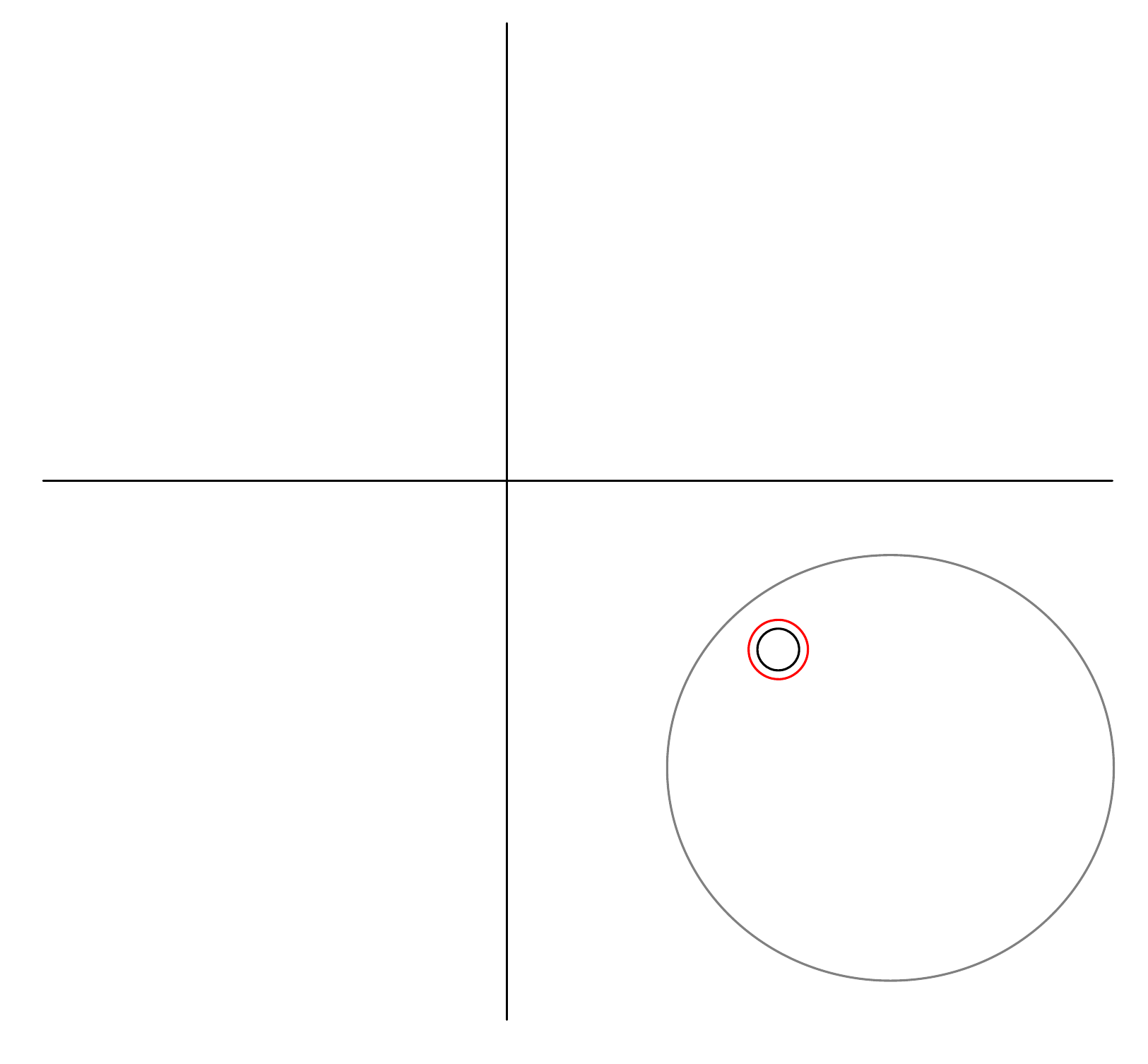
        \caption{Resolving a $\Z/2$-$\{1\}$-orbit bifurcation of type (C).}
    \label{fig:hd_C_Z2_1_resolved}
\end{figure}
\begin{figure}[h]
    \def\svgwidth{.8\linewidth}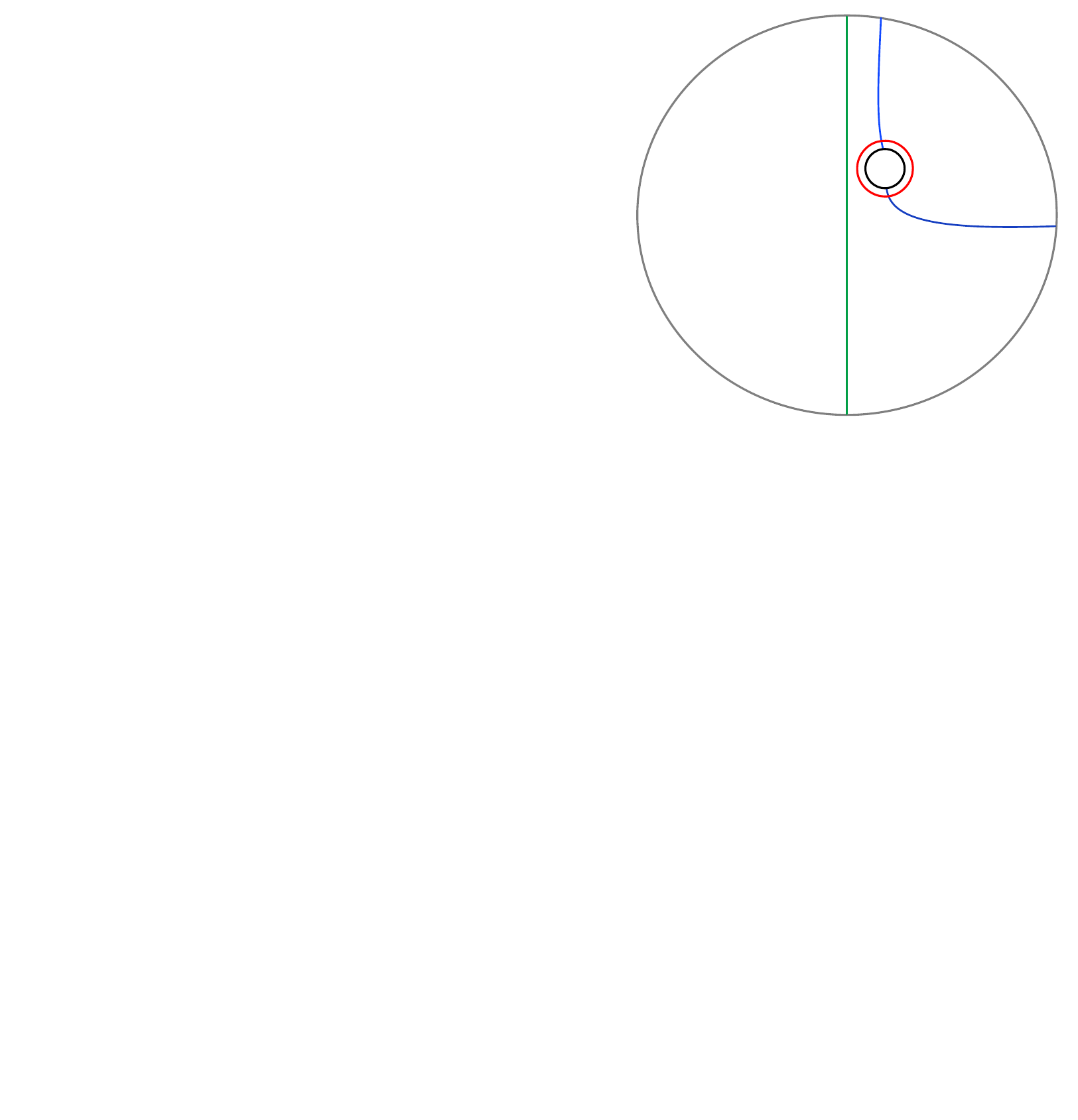
        \caption{Resolving a $\Z/2$-$\Z/2$-orbit bifurcation of type (C).}
    \label{fig:hd_C_Z2_Z2_resolved}
\end{figure}

\begin{figure}[h]
    \def\svgwidth{.8\linewidth}
\begingroup%
  \makeatletter%
  \providecommand\color[2][]{%
    \errmessage{(Inkscape) Color is used for the text in Inkscape, but the package 'color.sty' is not loaded}%
    \renewcommand\color[2][]{}%
  }%
  \providecommand\transparent[1]{%
    \errmessage{(Inkscape) Transparency is used (non-zero) for the text in Inkscape, but the package 'transparent.sty' is not loaded}%
    \renewcommand\transparent[1]{}%
  }%
  \providecommand\rotatebox[2]{#2}%
  \newcommand*\fsize{\dimexpr\f@size pt\relax}%
  \newcommand*\lineheight[1]{\fontsize{\fsize}{#1\fsize}\selectfont}%
  \ifx\svgwidth\undefined%
    \setlength{\unitlength}{779.52755906bp}%
    \ifx\svgscale\undefined%
      \relax%
    \else%
      \setlength{\unitlength}{\unitlength * \real{\svgscale}}%
    \fi%
  \else%
    \setlength{\unitlength}{\svgwidth}%
  \fi%
  \global\let\svgwidth\undefined%
  \global\let\svgscale\undefined%
  \makeatother%
  \begin{picture}(1,0.69818182)%
    \lineheight{1}%
    \setlength\tabcolsep{0pt}%
    \put(0,0){\includegraphics[width=\unitlength,page=1]{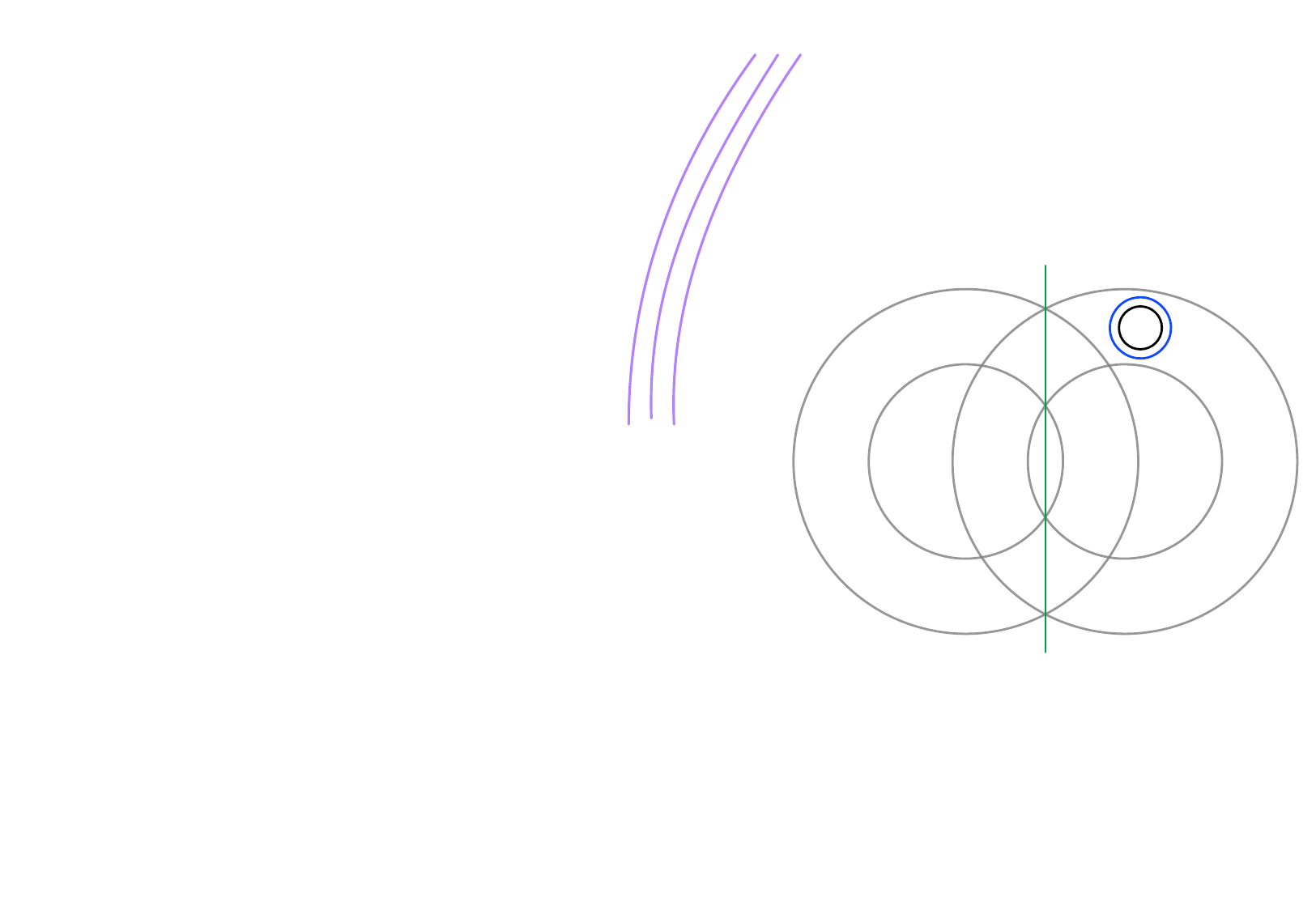}}%
    \put(0.86710378,0.43987767){\color[rgb]{0,0,0}\makebox(0,0)[t]{\smash{\begin{tabular}[t]{c}{$B$}\end{tabular}}}}%
    \put(0,0){\includegraphics[width=\unitlength,page=2]{hd_d1_resolved.pdf}}%
    \put(0.86713515,0.25557677){\color[rgb]{0,0,0}\rotatebox{-180}{\makebox(0,0)[t]{\smash{\begin{tabular}[t]{c}{$\reflectbox{$B$}$}\end{tabular}}}}}%
    \put(0.7228981,0.25557677){\color[rgb]{0,0,0}\rotatebox{-180}{\makebox(0,0)[t]{\smash{\begin{tabular}[t]{c}{$\reflectbox{$A$}$}\end{tabular}}}}}%
    \put(0,0){\includegraphics[width=\unitlength,page=3]{hd_d1_resolved.pdf}}%
    \put(0.72286671,0.43987768){\color[rgb]{0,0,0}\makebox(0,0)[t]{\smash{\begin{tabular}[t]{c}{$A$}\end{tabular}}}}%
    \put(0,0){\includegraphics[width=\unitlength,page=4]{hd_d1_resolved.pdf}}%
  \end{picture}%
\endgroup%

        \caption{Resolving a bifurcation of type (D1).}
    \label{fig:hd_d1_resolved}
\end{figure}

\begin{figure}[h]
    \def\svgwidth{.8\linewidth}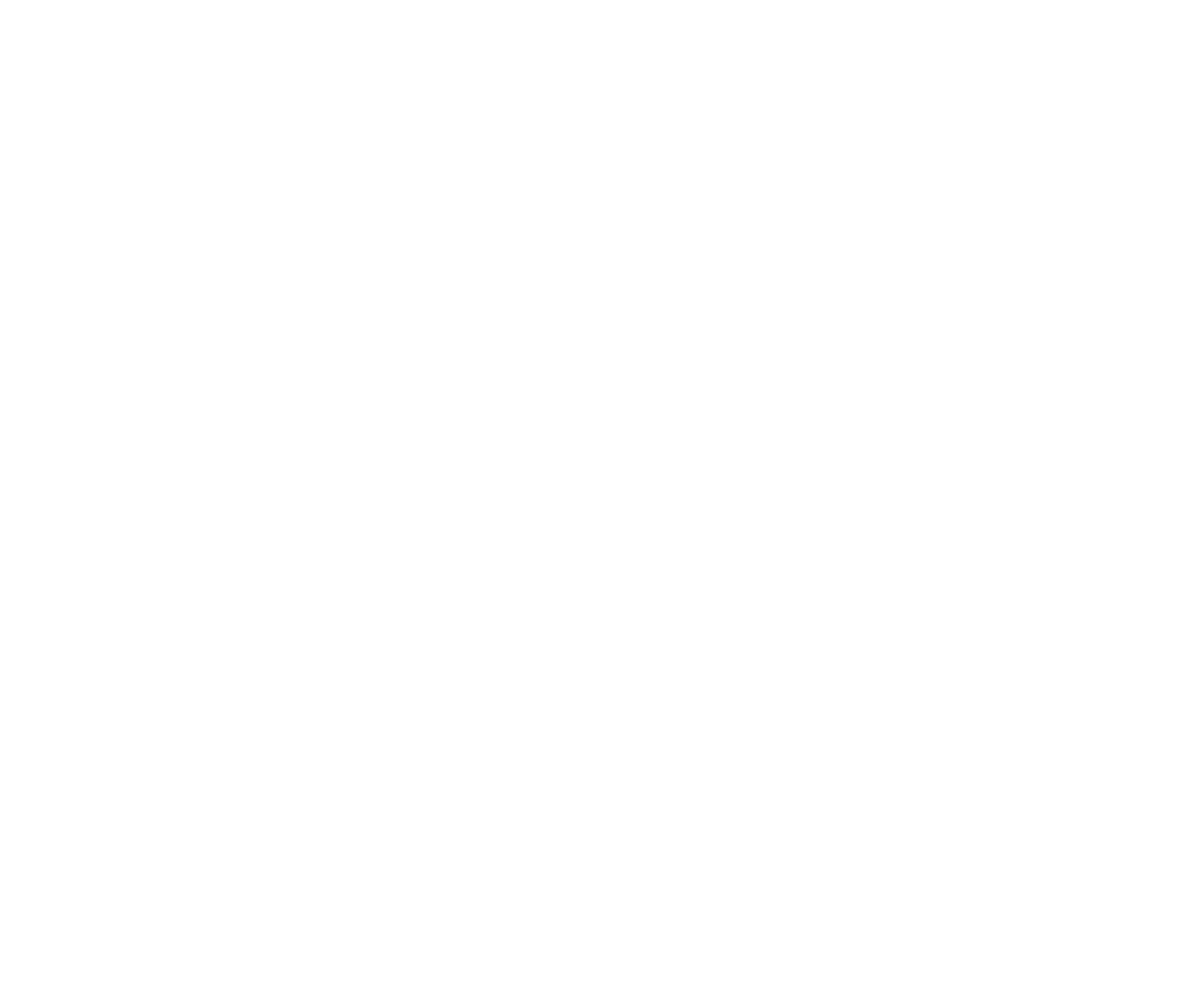
        \caption{Resolving a bifurcation of type (D2).}
    \label{fig:hd_d2_resolved}
\end{figure}

\begin{figure}[h]
    \def\svgwidth{.8\linewidth}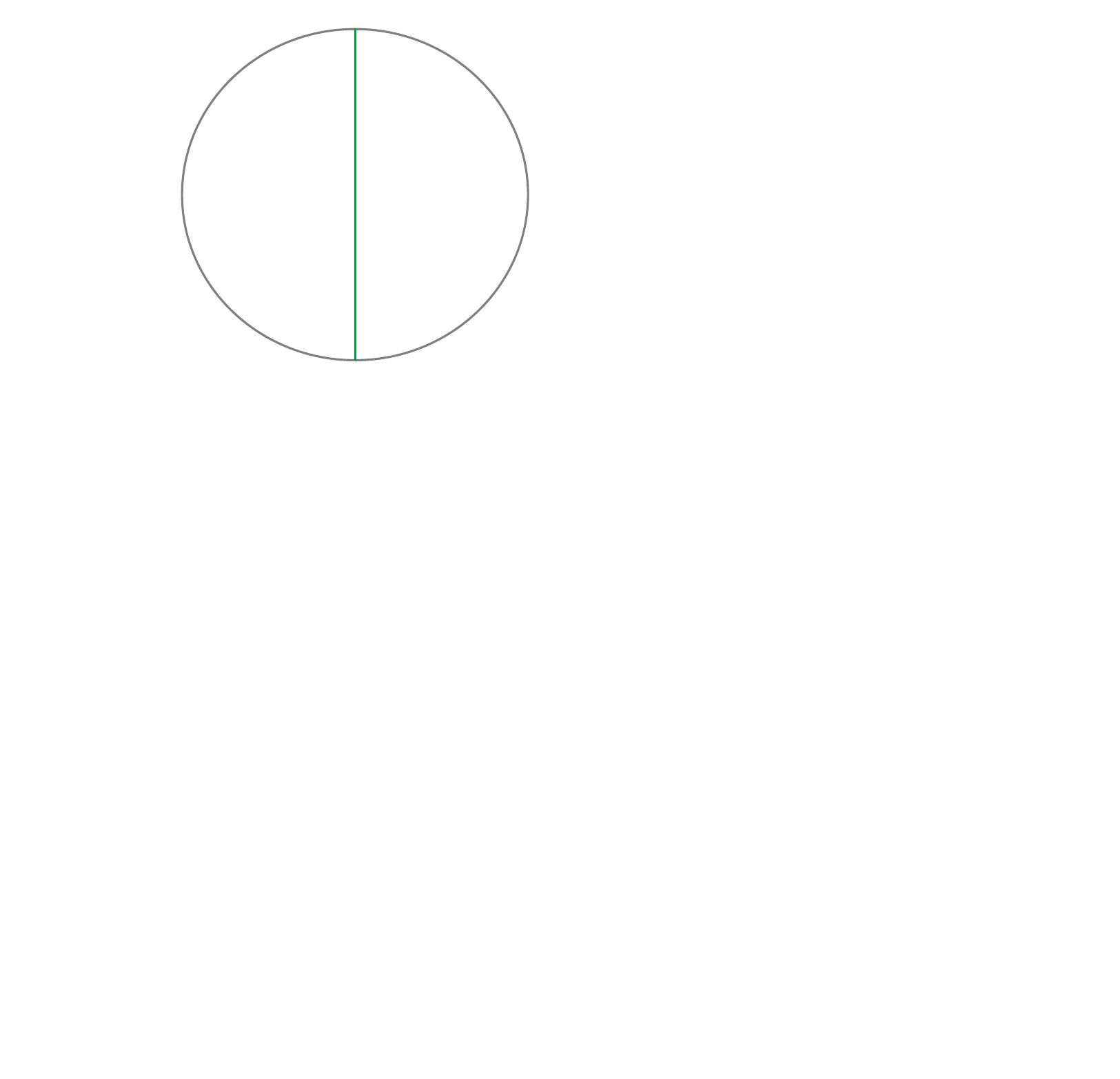
        \caption{Resolving a bifurcation of type (D2).}
    \label{fig:hd_d2_trade_loop}
\end{figure}

\begin{figure}[h]
    \def\svgwidth{.8\linewidth}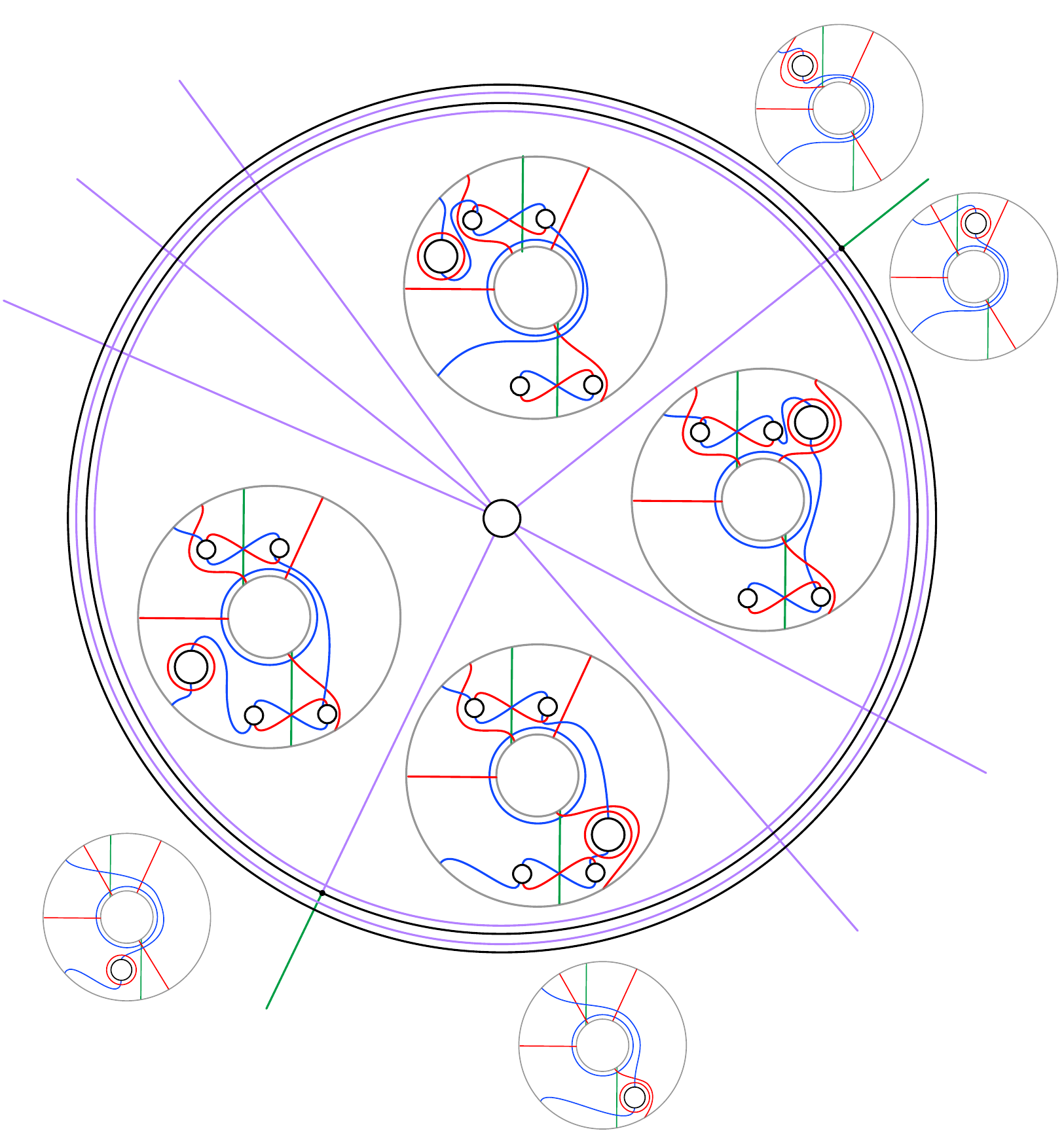
        \caption{Simplifying a real handleswap of type $(k, \ell, r)$ to one of type $(k, \ell, 0)$.}
    \label{fig:hd_e1_resolved0}
\end{figure}
\begin{figure}[h]
    \def\svgwidth{.8\linewidth}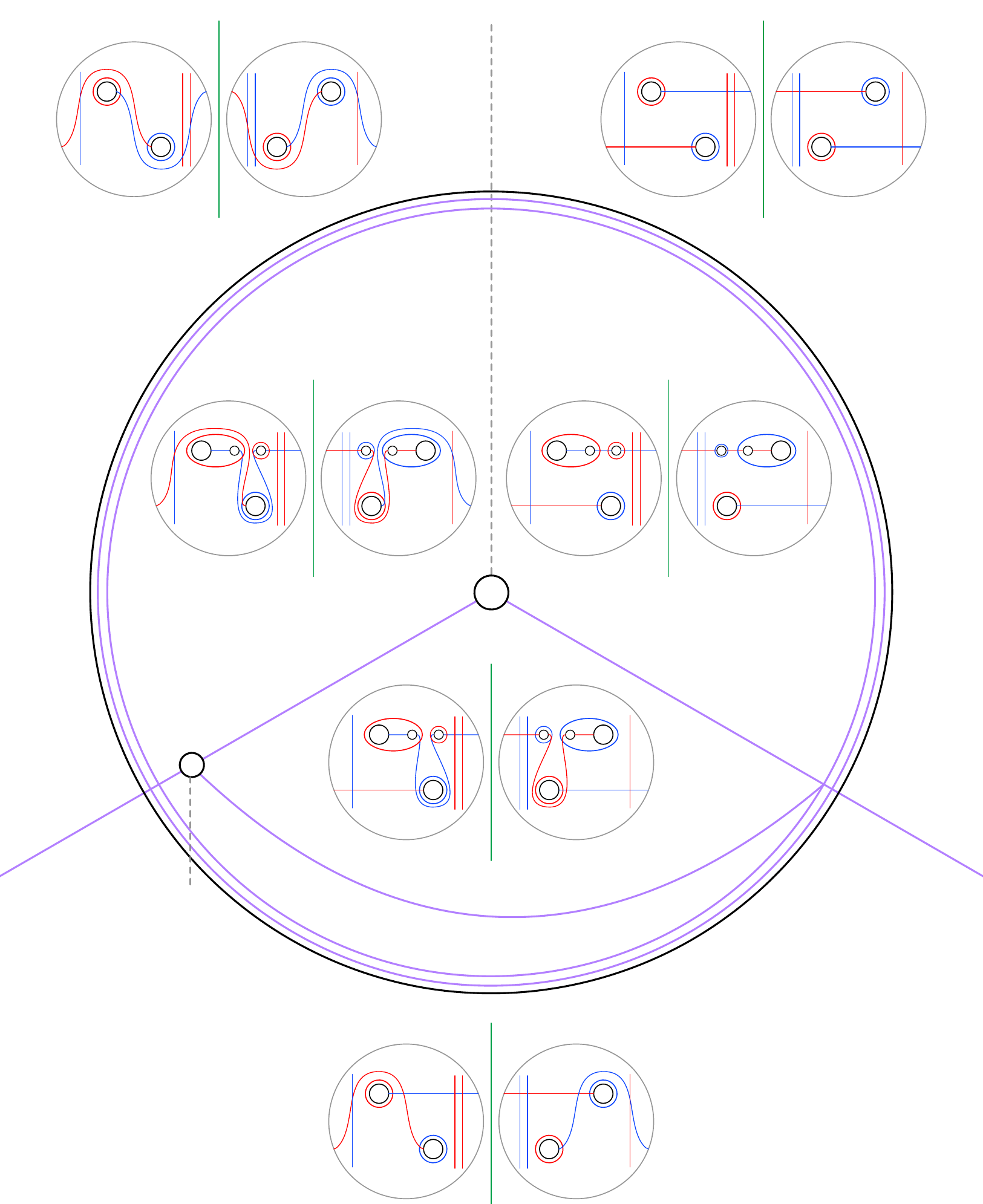
        \caption{Reducing to a pair of unreal handleswaps. Compare with \cite[Figure 54]{JTZ_naturality_mapping_class_groups}.}
    \label{fig:hd_e1_resolved1}
    \end{figure}
\clearpage
\bibliographystyle{amsalpha}
\bibliography{mathbib}

\end{document}